\documentclass[11pt]{article}
\usepackage[all]{xy}
\usepackage{epsfig}

\usepackage{amsfonts,amssymb,amsthm,amsmath,amscd}
\usepackage[mathscr]{eucal}
\usepackage{mathrsfs}
\usepackage{hyperref}
\usepackage{url}
      \makeatletter

\usepackage{titlesec}

\titleformat{\subsection}[runin]
{\normalfont\large\bfseries}{\thesubsection}{1em}{}

\titleformat{\subsubsection}[runin]
{\normalfont\large\bfseries}{\thesubsubsection}{1em}{}

\theoremstyle{plain}
\newtheorem{thm}[subsubsection]{Theorem}
\newtheorem{thm*}{Theorem}

\newtheorem{lem}[subsubsection]{Lemma}
\newtheorem{prop}[subsubsection]{Proposition}
\newtheorem{conj}[subsubsection]{Conjecture}

\theoremstyle{definition}
\newtheorem{defn}[subsubsection]{Definition}

\theoremstyle{remark}
\newtheorem{rem}[subsubsection]{Remark}

\numberwithin{equation}{subsubsection}

\newcommand{\N}{\mathbb N}
\newcommand{\Z}{\mathbb Z}
\newcommand{\Q}{\mathbb Q}
\newcommand{\R}{\mathbb R}
\newcommand{\C}{\mathbb C}
\newcommand{\A}{\mathbb A}

\newcommand{\F}{\mathbb F}
\newcommand{\Fp}{\mathbb{F}_{p}}
\newcommand{\Fpb}{\bar{\mathbb{F}}_{p}}
\newcommand{\Zp}{{\mathbb Z}_p}
\newcommand{\Qp}{{\mathbb Q}_p}
\newcommand{\Ql}{{\mathbb Q}_l}
\newcommand{\Qv}{{\mathbb Q}_v}
\newcommand{\Qb}{\overline{\mathbb Q}}
\newcommand{\Qpb}{\overline{\mathbb{Q}}_p}
\newcommand{\Qlb}{\overline{\mathbb{Q}}_l}
\newcommand{\Qvb}{\overline{\mathbb{Q}}_v}
\newcommand{\kb}{\overline{k}} 
\newcommand{\Lb}{\overline{L}} 
\newcommand{\Qpnr}{\Q_p^{\mathrm{ur}}} 
\newcommand{\Zpnr}{\Z_p^{\mathrm{ur}}}
\newcommand{\nr}{\mathrm{ur}} 
\newcommand{\Gal}{\mathrm{Gal}} 

\newcommand{\Hom}{\mathrm{Hom}}
\newcommand{\Ker}{\mathrm{Ker}}
\newcommand{\Nm}{\mathrm{Nm}} 
\newcommand{\Tr}{\mathrm{Tr}} 
\newcommand{\et}{\text{\'et}} 
\newcommand{\Int }{\mathrm{Int}} 
\newcommand{\Inn}{\mathrm{Int}} 
\newcommand{\Cent}{\mathrm{Cent}} 
\newcommand{\isom}{\stackrel{\sim}{\rightarrow}} 

\newcommand{\lra}{\longrightarrow}
\newcommand{\hra}{\hookrightarrow}
\newcommand{\thra}{\twoheadrightarrow}

\newcommand{\Spec}{\mathrm{Spec}}
\newcommand{\Res}{\mathrm{Res}} 
\newcommand{\cO}{\mathcal{O}} 
\newcommand{\sS}{\mathscr{S}} 
\newcommand{\dS}{\mathbb{S}} 
\newcommand{\Sh}{\mathrm{Sh}} 

\newcommand{\U}{\mathrm{U}}
\newcommand{\SU}{\mathrm{SU}}
\newcommand{\Sp}{\mathrm{Sp}}
\newcommand{\SO}{\mathrm{SO}}
\newcommand{\Spin}{\mathrm{Spin}}
\newcommand{\Gm}{\mathbb{G}_{\mathrm{m}}} 
\newcommand{\ab}{\mathrm{ab}} 
\newcommand{\ad}{\mathrm{ad}} 
\newcommand{\der}{\mathrm{der}} 
\newcommand{\uc}{\mathrm{sc}} 

\mathchardef\mhyphen="2D

\newcommand{\mcB}{\mathcal{B}} 
\newcommand{\mcA}{\mathcal{A}} 
\newcommand{\mbfa}{\mathbf{a}} 
\newcommand{\mbff}{\mathbf{f}} 
\newcommand{\mbfo}{\mathbf{0}} 
\newcommand{\mbfv}{\mathbf{v}}
\newcommand{\mcG}{\mathcal{G}} 
\newcommand{\mro}{\mathrm{o}}
\newcommand{\mbfK}{\mathbf{K}} 
\newcommand{\mbfKt}{\tilde{\mathbf{K}}} 

\newcommand{\mfk}{\mathfrak{k}} 
 
\newcommand{\mfp}{\mathfrak{p}}

\newcommand{\fG}{\mathfrak{G}} 
\newcommand{\fE}{\mathfrak{E}} 
\newcommand{\fD}{\mathfrak{D}} 
\newcommand{\fQ}{\mathfrak{Q}} 
\newcommand{\fP}{\mathfrak{P}} 

\newcommand{\Adm}{\mathrm{Adm}} 

\begin{document}

\title{Admissible morphisms for Shimura varieties with parahoric level}
\author{Dong Uk Lee}
\date{}
\maketitle

\begin{abstract}
In \textit{Shimuravariet\"{a}ten und Gerben} \cite{LR87}, Langlands and Rapoport developed the theory of pseudo-motivic Galois gerb and admissible morphisms between Galois gerbs, with a view to formulating a conjectural description of the $\Fpb$-point set of the good reduction of a Shimura variety with hyperspecial level, as well as to providing potential tools for its confirmation. 
Here, we generalize, and also improve to some extent, their works to parahoric levels when the group is quasi-split at $p$. In particular, we show that every admissible morphism is conjugate to a special admissible morphism, and, when the level is special maximal parahoric, that any Kottwitz triple with trivial Kottwitz invariant, if it satisfies certain natural necessary conditions implied by the conjecture, comes from an admissible pair. As applications, we give effectivity criteria for elliptic stable conjugacy classes and Kottwitz triples, and establish non-emptiness of Newton strata in the relevant cases. Along the way, we fill some gaps in the original work.
\end{abstract}

\tableofcontents

\section{Introduction}

The celebrated conjecture of Langlands and Rapoport, which was stated in \cite{LR87} and a significant progress towards which was recently made by Kisin \cite{Kisin13}, aims to give a group-theoretic description of the set of $\Fpb$-points of the mod-$p$ reduction of a Shimura variety, as provided with Hecke operators and Frobenius automorphism. 
To explain it, let $(G,X)$ be a Shimura datum, and $\mbfK^p\subset G(\A^p)$, $\mbfK_p\subset G(\Qp)$ compact open subgroups; set $\mbfK:=\mbfK_p\times \mbfK^p\subset G(\A_f)$. The original conjecture mainly concerned the good reduction cases, where $\mbfK_p$ is \textit{hyperspecial}, i.e. $\mbfK_p=G_{\Zp}(\Zp)$ for a reductive $\Zp$-group scheme with generic fiber $G_{\Qp}$. 
We also choose a place $\wp$ of the reflex field $E(G,X)$ dividing $p$, and let $\cO_{\wp}$, $\kappa(\wp)$ denote respectively the integer ring of the local field $E(G,X)_{\wp}$ and its residue field.
Then, Langlands and Rapoport conjectured that there exists an integral model $\sS_{\mbfK_p}(G,X)$ of $\Sh_K(G,X)$ over $\cO_{\wp}$, for which there is a bijection 
\begin{equation} \label{eqn:LRconj-ver1}
\sS_{\mbfK_p}(G,X)(\Fpb)\isom \bigsqcup_{[\phi]}S(\phi)
\end{equation}
where
\[S(\phi)=\varprojlim_{\mbfK^p} I_{\phi}(\Q)\backslash X_p(\phi)\times X^p(\phi)/\mbfK^p.\]
To give an idea of what these objects are like, suppose that our Shimura variety $\Sh_{\mbfK_p}(G,X)$ is a moduli space of abelian varieties endowed with a certain prescribed set of additional structures (called $G$-structures, for short), and that there exists an integral model whose reduction affords a similar moduli description (at least over $\Fpb$). Then, roughly speaking, each $\phi$ is supposed to correspond to an isogeny class of abelian varieties with $G$-structure, and the set $S(\phi)$ is to parameterize the isomorphism classes in the corresponding isogeny class. 
More precisely, $X_p(\phi)$ and $X^p(\phi)$ should correspond to the isogenies of $p$-power order and prime-to-$p$ order (say, leaving from a fixed member in the isogeny class $\phi$), respectively, and $X_p(\phi)$ can be also identified with a suitable affine Deligne-Lusztig variety $X(\{\mu_X\},b)_{\mbfK_p}$. 
The term $I_{\phi}(\Q)$ is to be the automorphism group of the isogeny class attached to $\phi$, and thus acts naturally on $X_p(\phi)$ and $X^p(\phi)$.
Moreover, each of the sets $S(\phi)$ carries a compatible action of $G(\A_f^p)$ and the Frobenius $\Phi$ (in $\Gal(\Fpb/\kappa(\wp))$), and the bijection (\ref{eqn:LRconj-ver1}) should be compatible with these actions.

When it comes to precise definition, the most tricky object is the parameter $\phi$. Its precise definition makes use of the language of \textit{Galois gerbs}: A Galois gerb is a gerb, in the sense of \textit{Cohomologie non ab\'elienne} \`a la Giraud, on the \'etale site of a field (with choice of a neutralizing object). This is motivated by the fact (\cite{Milne94}) that there is a well-determined class of Shimura varieties (i.e. Shimura varieties of abelian type) which, in characteristic zero, have a description of their point sets similar to (\ref{eqn:LRconj-ver1}) with the parameter $\phi$ being an abelian motive. For $\Fpb$-points, the parameter $\phi$, called an admissible morphism, is to represent ``a motive over $\Fpb$ with $G$-structure". But, the Tannakian category of Grothendieck motives over $\Fpb$, being non-neutral, is identified with the representation category of a certain Galois gerb (``motivic Galois gerb"), not an affine group scheme, and ``a motive over $\Fpb$ with $G$-structure" is then constructed as a morphism from this motivic Galois gerb to the neutral Galois gerb attached to $G$.

The conjectural description (\ref{eqn:LRconj-ver1}) also determines the finite sets $\sS_{\mbfK}(G,X)(\F_{q^m})$ for each finite extension $\F_{q^m}$ of $\F_q=\kappa(\wp)$, and eventually allows one to obtain a purely group-theoretic formula for their cardinalities, whose knowledge amounts to that of the local zeta function of $\sS_{\mbfK_p}(G,X)_{\kappa(\wp)}$. It is also this formula which was attacked with success in some special cases (mostly of PEL-type), e.g. in \cite{Kottwitz92}. 
In more detail, each subset $S(\phi)^{\Phi^m}$ of $\sS_{\mbfK}(G,X)(\F_{q^m})=[\sS_{\mbfK}(G,X)(\Fpb)]^{\Phi^m}$
is grouped into further subsets indexed by a pair $(\phi,\epsilon)$ consisting of an admissible morphism $\phi$ and a Frobenius descent datum $\epsilon\in \mathrm{Aut}(\phi)$: such pair $(\phi,\epsilon)$, when it satisfies certain conditions, is supposed to correspond to an $\F_{q^m}$-isogeny class of abelian varieties. With such admissible pair $(\phi,\epsilon)$, one associates a triple of group elements (\textit{Kottwitz triple})
\[(\gamma_0;\gamma=(\gamma_l)_{\neq p},\delta)\in G(\Q)\times G(\A_f^p)\times G(L_n)\] 
(defined up to some suitable equivalence relation), where $L_n$ is the unramified extension of $\Qp$ of degree $n=m[\kappa(\wp):\Fp]$. The formula in question then takes the form of a sum, indexed by (equivalence classes of) these triples, of a product of quantities:
\begin{equation} \label{eqn:formulra_for_number_of_pts}
\sum_{(\gamma_0;\gamma=(\gamma_l),\delta)}\iota(\gamma_0)\cdot \mathrm{vol}(I_0(\Q)Z_{\mbfK}/I_0(\A_f))\cdot O_{\gamma}(f^p)\cdot TO_{\delta}(\phi_p).
\end{equation}
We refer the reader to \cite{Kottwitz84b}, \cite{Kottwitz90} for detailed discussion of this formula. Here, the sum runs only over the triples satisfying certain conditions (among which the most important one is the vanishing of the so-called Kottwitz invariant).

A substantial part of the work \cite{LR87} of Langlands and Rapoport is devoted to constructing these objects and establishing their basic properties. The facts thus obtained are needed to carry out the deduction of the formula (\ref{eqn:formulra_for_number_of_pts}), and were also important ingredients in the works on the original conjecture itself, including the recent one by Kisin \cite{Kisin13}.
The main results of \cite{LR87} assume that the level subgroup $\mbfK_p$ is hyperspecial  (so, in particular that $G_{\Qp}$ is unramified). 
Our primary task in this article is to generalize their works to more general parahoric levels so as to allow possibly bad reductions, in the case that $G_{\Qp}$ is quasi-split; some of the results will further assume that $\mbfK_p$ is special maximal parahoric. 
In this generality, the first notable change occurs in definitions (including that of admissible morphism), in some of which instead of a single affine Deligne-Lusztig variety $X(\{\mu_X\},b)_{\mbfK_p}$ which defined the term $X_p(\phi)$, one needs to use a \emph{finite union} of \emph{generalized} affine Deligne-Lusztig varieties $X(w,b)_{\mbfK_p}$:
\[X(\{\mu_X\},b)_{\mbfK_p}:=\bigsqcup_{w\in\Adm_{\mbfKt_p}(\mu_X)} X(w,b)_{\mbfK_p}\]
Here, $\Adm_{\mbfKt_p}(\mu_X)$ is a certain subset of the extended Weyl group $\widetilde{W}$ of $G_{\Qp}$ determined by the datum $(G,X)$. When $\mbfK_p$ is hyperspecial, this reduces to the previous definition.
Meanwhile, the conjecture (\ref{eqn:LRconj-ver1}) itself was extended by Rapoport \cite{Rapoport05} to cover general parahoric levels.

To keep the statements short, we will call the following condition the \textit{Serre condition for the Shimura datum $(G,X)$}
\footnote{to distinguish this from two other similar conditions: first, from the original Serre condition which is applied to a $\Q$-torus $T$ endowed with a cocharacter $\mu\in X_{\ast}(T)$ (cf. Lemma \ref{lem:defn_of_psi_T,mu}), and secondly from what Kisin calls the Serre condition for a torus $T$ (\cite[3.7.3]{Kisin13}).}
: \textit{the center $Z(G)$ of $G$ splits over a CM field and the weight homomorphism $w_X$ is defined over $\Q$.}

Our first main results are generalizations (sometimes, including improvements) of the key properties of the objects introduced above briefly:

\begin{thm} \label{thm:1st_Main_thm} Let $p>2$ be a rational prime.
Let $(G,X)$ be a Shimura datum satisfying the Serre condition. Assume that $G$ is of classical Lie type, and that $G_{\Qp}$ is quasi-split and splits over a tamely ramified extension of $\Qp$. Let $\mbfK_p$ be a parahoric subgroup of $G(\Qp)$. Then, we have the followings.

(1) Any admissible morphism $\phi:\fP\rightarrow \fG_G$ is \emph{special}, namely there exists a special Shimura sub-datum $(T,h)$ and $g\in G(\Qb)$ such that $\Int g\circ\phi=i\circ\psi_{T,h}$, where $i:\fG_T\rightarrow \fG_G$ is the canonical morphism of neutral Galois gerbs induced by the inclusion $T\hra G$. If $\mbfK_p$ is special maximal parahoric, then every such morphism $i\circ\psi_{T,h}$ is admissible.

(2) Suppose that $\mbfK_p$ is special maximal parahoric and $G^{\der}$ is simply connected. Let $\gamma_0\in G(\Q)$ be a rational element, elliptic over $\R$, and whose image in $G^{\ad}(\A_f^p)$ lies in a compact open subgroup of $G^{\ad}(\A_f^p)$. Then, there exists an admissible pair $(\phi,\epsilon)$ with $\epsilon$ stably conjugate to $\gamma_0$ if and only if there exists $\epsilon\in G(\Q)$ stably conjugate to $\gamma_0$ and satisfying the condition $(\ast(\gamma_0))$ of Subsec. \ref{subsubsec:pre-Kottwitz_triple}. If the latter condition holds, there exists a $\mbfK_p$-effective admissible pair $(\phi,\epsilon)$ with $\epsilon$ stably conjugate to $\gamma_0^t$ for some $t\in\N$.

(3) Under the same assumptions as (2), for every Kottwitz triple $(\gamma_0;(\gamma_l)_{l\neq p},\delta)$ with trivial Kottwitz invariant (Def. \ref{defn:Kottwitz_triple}) and such that $\gamma_0$ (equiv. $(\gamma_l)_l$) lies in a compact open subgroup of $G(\A_f^p)$ and $X(\{\mu_X\},\delta)_{\mbfK_p}\neq\emptyset$, there exists an admissible pair $(\phi,\epsilon)$ giving rise to it. There exists an explicit cohomological expression for the number of non-equivalent pairs $(\phi,\epsilon)$ producing a given triple $(\gamma_0;(\gamma_l)_{l\neq p},\delta)$.
\end{thm}

These statements (and their counterpart results in \cite{LR87}) are found respectively in Theorem \ref{thm:LR-Satz5.3}, Lemma \ref{lem:LR-Lemma5.2} (for (1)), Theorem \ref{thm:LR-Satz5.21} (for (2)), and Theorem \ref{thm:LR-Satz5.25} (for (3)).
Being generalizations, these facts are more or less known in the hyperspecial level case.
But the statements (2), (3) also contain some meaningful improvements: strengthening of the original assertions and correction of gaps in the original arguments. More precisely, in (2) the condition $(\ast(\gamma_0))$ being used is different from the original condition appearing in \cite[Satz 5.21]{LR87}, which we find insufficient for that theorem (and the subsequent statements depending on it). Also, we will see that the fact (3) is stronger and more natural than its counterpart result \cite[Satz 5.25]{LR87}, which meantime suffers from the same gap. 
For detailed discussions on these issues, we refer to the above theorems as well as the remarks accompanying them. 

The statement (1) is a fundamental fact about admissible morphisms, and underlies another closely related conjecture (which was proved in the hyperspecial level case by Zink for PEL-types and by Kisin (\cite[Thm. 0.4]{Kisin13}) for Hodge-types) that every isogeny class in $\sS_{\mbfK_p}(G,X)(\Fpb)$ contains a point which is the reduction of a special(=CM) point. 

On the other hand, according to the conjecture (\ref{eqn:LRconj-ver1}), to any $\F_{p^n}$-point ($n$ being a multiple of $[\kappa(\wp):\F_p]$), one should be able to attach a Kottwitz triple $(\gamma_0;(\gamma_l)_{l\neq p},\delta)$ of level $n$ (i.e. with $\delta\in G(L_n)$). For Hodge-type Shimura varieties with hyperspecial level, this was done by Kisin (\cite[Cor. 2.3.1]{Kisin13}), but without having control on the level, one can still attach a Kottwitz triple to any point which is the reduction of a CM point (thus in this case the triple is well-defined only up to powers). Now, the formula (\ref{eqn:formulra_for_number_of_pts}) implies that only those triples whose corresponding summation term is non-zero are \textit{effective}, i.e. arises from an $\F_{p^n}$-point (where $n$ is the level of the triple). For that, we point out that in the formula the only possibly zero quantities are $O_{(\gamma_l)}(f^p)$ (orbital integral) and $TO_{\delta}(\phi_p)$ (twisted orbital integral), and the non-vanishing of each of these quantities can be translated into explicit conditions on, respectively, $(\gamma_l)_{l\neq p}$ and $\delta$. Then, a natural question is whether such necessary conditions for effectivity are also sufficient conditions.
As an application of Theorem \ref{thm:2nd_Main_thm}, we verify a slightly weaker version of this, namely, we show that if a Kottwitz triple  $(\gamma_0;(\gamma_l)_{l\neq p},\delta)$ satisfies such effectivity conditions, then indeed some power $(\gamma_0^t;(\gamma_l^t)_l,\delta)$ of it (where $t\in\N$ and $\delta$ is considered as an element of $G(L_{tn})$) is the triple associated with the reduction of a CM point.

Closely related questions are which (elliptic) stably conjugacy class in $G(\Q)$ and which $\sigma$-conjugacy class in $G(L_n)$  are ``effective" (i.e. can be the classes of the elements $\gamma_0$ and $\delta$ attached to some $\F_{p^n}$-point, respectively). 
This question for (the stable conjugacy class of) an elliptic rational element $\gamma_0\in G(\Q)$ can be regarded as the Honda-Tate theorem in the context of Shimura varieties, while the question for $\delta$ is known as the non-emptiness problem of Newton strata. 

The next two theorems provide answers to both of these questions, while the aforementioned effectivity criteria is given in the first one.  
Here, we consider Shimura varieties of Hodge type and fix an integral model $\sS_{\mbfK}$ over $\cO_{\wp}$ of the canonical model $\Sh_{\mbfK}(G,X)_{E_{\wp}}$ with the extension property that every $F$-point of $\Sh_{\mbfK}(G,X)$ for a finite extension $F$ of $E_{\wp}$ extends uniquely to $\sS_{\mbfK}$ over its local ring (for example, a normal integral model); see \cite{KP15} for a construction of such integral model in general parhoric levels. 

\begin{thm} [Thm. \ref{thm:effectivity_criteria}] \label{thm:2nd_Main_thm}
Keep the assumptions of Theorem \ref{thm:1st_Main_thm}, and suppose further that $(G,X)$ is simply connected.

(1) Let $\mathcal{C}\subset G(\Q)$ be a stable conjugacy class. Then, some power of $\mathcal{C}$ contains the relative Frobenius of the reduction of a CM point of $\Sh_{\mbfK}(G,X)(\C)$ if and only if a (possibly different) power of $\mathcal{C}$ contains some $\gamma_0\in G(\Q)$ which satisfies the condition $(\ast(\gamma_0))$ of Subsec. \ref{subsubsec:pre-Kottwitz_triple} and lies in a compact open subgroup of $G(\A_f^p)$.

(2) Let $(\gamma_0;(\gamma_l)_{l\neq p},\delta)$ be a Kottwitz triple with trivial Kottwitz invariant and such that $\gamma_0$ (equiv. $(\gamma_l)_l$) lies in a compact open subgroup of $G(\A_f^p)$ and
\[Y_p:=\{x\in G(L)/\mbfKt_p\ |\ \sigma^n x=x,\ \mathrm{inv}_{\mbfKt_p}(x,\delta\sigma x)\in \Adm_{\mbfKt_p}(\{\mu_X\})\}\neq \emptyset.\]
Then, there exists a special Shimura datum $(T,h)$ such that the reduction of the CM point $[h,1\cdot\mbfK]\in \Sh_{\mbfK}(G,X)(\Qb)$ has the associated Kottwitz triple equal to $(\gamma_0;(\gamma_l)_{l\neq p},\delta)$, up to powers.
\end{thm}

We point out that the conditions in (2) are the aforementioned conditions implied by non-vanishing of the quantities $O_{\gamma}(f^p)$, $TO_{\delta}(\phi_p)$ in the formula (\ref{eqn:formulra_for_number_of_pts}), (cf. \cite[$\S$1.5]{Kottwitz84b}).

\begin{thm} [Thm. \ref{thm:non-emptiness_of_NS}]  \label{thm:3rd_Main_thm}
Keep the assumptions of Theorem \ref{thm:1st_Main_thm}, and further assume that $G_{\Qp}$ splits over a \emph{cyclic} tame extension of $\Qp$.
Let $\mbfK_p$ be a (not necessarily special) parahoric subgroup of $G(\Qp)$ and put $\mbfK=\mbfK_p\mbfK^p$ for a compact open subgroup $\mbfK^p$ of $G(\A_f^p)$. 

(1) Then, for any $[b]\in B(G_{\Qp},\{\mu_X\}$) (Subsec. \ref{subsubsec:B(G,{mu})}), there exists a special Shimura sub-datum $(T,h\in\Hom(\dS,T_{\R})\cap X)$ such that for any $g_f\in\mbfK^p$, the reduction in $\sS_{\mbfK}\otimes\Fpb$ of the special point $[h,g_f\cdot\mbfK]\in \Sh_{\mbfK}(G,X)(\Qb)$ has the $F$-isocrystal represented by $[b]$.

(2) The reduction $\sS_{\mbfK}(G,X)\otimes\Fpb$ has non-empty ordinary locus if and only if $\wp$ has absolute height one (i.e. $E(G,X)_{\wp}=\Q_p$). 
\end{thm}

These results generalize Theorem 4.3.1 and Corollary 4.3.2 of \cite{Lee14} in the hyperspecial cases.

We make some comments on the various assumptions appearing in this article.
The running assumption, which will be effective except in some general discussions, is that $G_{\Qp}$ is quasi-split. Equally universal assumption, although it is not needed for the important Theorem \ref{thm:1st_Main_thm}, (1), is that $\mbfK_p$ is special maximal parahoric.
These two assumptions are somewhat forced on us because for proofs we follow closely the original line of arguments.
The assumption that $G$ is of classical Lie type and splits over a tame extension of $\Qp$ is less serious, even though it appears in many places (including all the statements of Theorem \ref{thm:1st_Main_thm}). It is imposed whenever we invoke Prop. \ref{prop:existence_of_elliptic_tori_in_special_parahorics}, which assumes it. In turn, this assumption shows up in that proposition, simply because we were only able to verify that proposition for classical groups that splits over a tamely ramified extension of $\Qp$. Therefore, if Prop. \ref{prop:existence_of_elliptic_tori_in_special_parahorics} could be extended beyond such cases, we can relax that restriction accordingly (or even purge that assumption of this paper).

Finally, we remark that our sign convention is the same as that of Langlands-Rapoport in \cite{LR87}; so for example, it is opposite to that of Kisin in \cite{Kisin13}.

This article is organized as follows. 
The second section is of preliminary nature, devoted to a review of some basic objects, including the Newton and Kottwitz maps (defined for algebraic groups over $p$-adic fields), parahoric groups (in the Bruhat-Tits theory), extended Weyl groups, and the $\{\mu\}$-admissible set. 
In the third section, we attempt to give a self-contained overview of the notions of Galois gerbs, the pseudo-motivic Galois gerb, admissible morphisms, Kottwitz triples, and admissible pairs, following closely the original source \cite{LR87}. We also give a statement of the Langlands-Rapoport conjecture, as formulated by Rapoport \cite[$\S$8]{Rapoport05} so as to cover parahoric levels. Along the way, we extend results on special admissible morphisms to (special maximal) parahoric subgroups, under the assumption that $G_{\Qp}$ is quasi-split. 
In the fourth section, we prove Theorem \ref{thm:1st_Main_thm}, (1) above, namely that every admissible morphism is conjugate to a special admissible morphism (in our case of general parahoric level), as well as the fact that every admissible pair is nested in a special Shimura datum. 
For some other potential applications in mind, we spilt the proof into a few steps and formalize each of them into a separate proposition (incorporating slight improvements). The results in this section are in large part translations of the original results, except for reorganization (with small improvements) and the non-trivial generalization of \cite[Lemma 5.11]{LR87} in Lemma \ref{lem:LR-Lemma5.11}, which is also a key input in our proof of the non-emptiness of Newton strata.
In the final section, we prove the remaining statements (2) and (3) of Theorem \ref{thm:1st_Main_thm}, and the effectivity criteria of elliptic stable conjugacy classes and Kottwitz triples. Here, as remarked before, we point out and fix some gaps in the original proof of the corresponding results in the hyperspecial level case. 
Throughout this article, we will mainly work with the pseudo-motivic Galois gerb by assuming the aforementioned Serre condition for the Shimura datum $(G,X)$. But still some definitions use the quasi-motivic Galois gerb, so we provide its definition in the appendix. Furthermore, in this work, a certain result in the Bruhat-Tits theory, whose hyperspecial case was established in our previous work \cite[Lemma A.0.4]{Lee14}, plays a key role. We provide its proof in the appendix.

\textbf{Acknowledgement}
This work was supported by IBS-R003-D1. The author would like to thank M. Rapoport and C.-L. Chai for their interests in this work and encouragement.

\textbf{Notations}

Throughout this paper, $\Qb$ denotes the algebraic closure of $\Q$ inside $\C$. 

For a (connected) reductive group $G$ over a field, we let $G^{\uc}$ be the universal covering of its derived group $G^{\der}$, and for a (linear algebraic) group $G$, $Z(G)$, and $G^{\ad}$ denote its center, and the adjoint group $G/Z(G)$, respectively.

For a group $I$ and an $I$-module $A$, we let $A_I$ denote the quotient group of $I$-coinvariants: $A_I=A/\langle ia-a\ |\ i\in I, a\in A\rangle$. For an element $a\in A$, we write $\underline{a}$ for the image of $a$ in $A_I$. In case of need for distinction, sometimes we write $\underline{a}_A$.

For a finitely generated abelian group $A$, we denote by $A_{\mathrm{tors}}$ its subgroup of torsion elements. Also, for a locally compact abelian group $A$, we let $X_{\ast}(A)$, $X^{\ast}(A)$, and $A^D$ denote the (co)character groups $\Hom(\C^{\times},A)$, $\Hom(A,\C^{\times})$, and the Pontryagin dual group $\Hom(A,\Q/\Z)$, respectively.

In this article, the letter $L$ denotes primarily the completion of the maximal unramified extension (in a fixed algebraic closure $\Qpb$) of $\Qp$. On few occasions, it also denotes a finite CM extension of $\Q$, in which case, if one also needs a notation for the completion of the maximal unramified extension of $\Qp$, we will use the german letter $\mfk$ (the original notation of Langlands-Rapoport).


\section{Parahoric subgroups and $\mu$-admissible set}


\subsection{Kottwitz maps and Newton map}

In this section, we briefly recall the definitions of the Kottwitz maps and the Newton map. We refer to \cite{Kottwitz97}, \cite{Kottwitz85}, \cite{RR96}, and references therein for further details.

\subsubsection{The Kottwitz maps $w_G$, $v_G$, $\kappa_{G}$} \label{subsubsec:Kottwitz_hom}

Let $L$ be a strictly henselian discrete valued field and set $I:=\Gal(\overline{L}/L)$.
For any connected reductive group $G$ over $L$, Kottwitz (\cite[$\S$7]{Kottwitz97}) constructs a group homomorphism
\[w_G:G(L)\rightarrow X^{\ast}(Z(\widehat{G})^{I})=\pi_1(G)_{I}.\]
Here, $\widehat{G}$ denotes the Langlands dual group of $G$, $\pi_1(G)=X_{\ast}(T)/\Sigma_{\alpha\in R^{\ast}}\Z\alpha^{\vee}$ is the fundamental group of $G$ (\`a la Borovoi) (i.e. the quotient of $X_{\ast}(T)$ for a maximal torus $T$ over $F$ of $G$ by the coroot lattice), and $\pi_1(G)_I$ is the (quotient) group of coinvariants of the $I$-module $\pi_1(G)$. This map $w_G$ is sometimes denoted by $\widetilde{\kappa}_G$, e.g. in \cite{Rapoport05}. When $G^{\der}$ is simply connected (so that $\pi_1(G)=X_{\ast}(G^{\ab})$ for $G^{\ab}=G/G^{\der}$), $w_G$ factors through $G^{\ab}$: $w_G=w_{G^{\ab}}\circ p$, where $p:G\rightarrow G^{\ab}$ is the natural projection (\cite[7.4]{Kottwitz97}).

There is also a homomorphism 
\[v_G:G(L)\rightarrow \Hom(X_{\ast}(Z(\widehat{G}))^I,\Z)\] 
sending $g\in G(L)$ to the homomorphism $\chi\mapsto \mathrm{val}(\lambda(g))$ from $X_{\ast}(Z(\widehat{G}))^I=\Hom_L(G,\Gm)$ to $\Z$, where $\mathrm{val}$ is the usual valuation on $L$, normalized so that uniformizing elements have valuation $1$. 

There is the relation: $v_G=q_G\circ w_G$, where $q_G$ is the natural surjective map 
\[q_G:X^{\ast}(Z(\widehat{G})^{I})=X^{\ast}(Z(\widehat{G}))_I \rightarrow \Hom(X_{\ast}(Z(\widehat{G}))^I,\Z).\]
The kernel of $q_G$ is the torsion subgroup of $X^{\ast}(Z(\widehat{G}))_I$, i.e. $\Hom(X_{\ast}(Z(\widehat{G}))^I,\Z)\cong \pi_1(G)_I/\text{torsions}$; in particular, $q_G$ is an isomorphism if the coinvariant group $X^{\ast}(Z(\widehat{G}))_I$ is free (e.g. the $I$-module $X^{\ast}(Z(\widehat{G}))$ is trivial or more generally \textit{induced}, i.e. has a $\Z$-basis permuted by $I$).

For example, when $G$ is a torus $T$, we have that $\langle \chi,w_T(t)\rangle= \mathrm{val}(\chi(t))$ for $ t\in T(L)$, $\chi\in X^{\ast}(T)^I$, where $\langle\ ,\ \rangle$ is the canonical pairing between $X^{\ast}(T)^I$ and $X_{\ast}(T)_I$.

Now suppose that $G$ is defined over a local field $F$, i.e. a finite extension of $\Qp$ (in a fixed algebraic closure $\Qpb$). Let $L$ be the completion of the maximal unramified extension $F^{\nr}$ of $F$ in $\Qpb$ and let $\sigma$ denote the Frobenius automorphism on $L$ which fixes $F$ and induces $x\mapsto x^q$ on the residue field of $L$ ($\cong\Fpb$), where the residue field of $F$ is $\F_q$. In this situation, the maps $v_{G_L}$, $w_{G_L}$ each induce a notable map.

First, as $v_{G_L}$ and $w_{G_L}$ commute with the action of $\Gal(F^{\nr}/F)$, by taking $H^0(\Gal(F^{\nr}/F),-)$ on both sides of $v_{G_L}$, we obtain a homomorphism
\[\lambda_G:G(F)\rightarrow \Hom(X_{\ast}(Z(\widehat{G}))^I,\Z)^{\langle\sigma\rangle}\cong(\pi_1(G)_I/\text{torsions})^{\langle\sigma\rangle},\]
where $I\cong\Gal(\overline{F}/F^{\nr})$. This is the map introduced in \cite[$\S$3]{Kottwitz84b} and used in \cite{LR87} (there, denoted by the same symbol) when $G$ is \emph{unramified} over $F$, in which case the canonical action of $I$ on $Z(\widehat{G})$ is trivial, so the target becomes $X^{\ast}(Z(\widehat{G}))^{\Gamma_F}$ ($\Gamma_F:=\Gal(\overline{F}/F)$) (and also $w_{G_L}=v_{G_L}$).

Next, let $B(G)$ denote the set of $\sigma$-conjugacy classes:
\[B(G):=G(L)/\stackrel{\sigma}{\sim},\]
where two elements $b_1$, $b_2$ of $G(L)$ are said to be \textit{$\sigma$-conjugate}, denoted $b_1\stackrel{\sigma}{\sim} b_2$, if there exists $g\in G(L)$ such that $b_2=gb_1\sigma(g)^{-1}$. 
Then, by taking $H^1(\Gal(F^{\nr}/F),-)$, $w_{G_L}$ induces a map of sets 
\[\kappa_G:B(G)\rightarrow X^{\ast}(Z(\widehat{G})^{\Gamma_F})=\pi_1(G)_{\Gamma_F}: \kappa_G([b])=\overline{w_{G_L}(b)}\]
Here, for $b\in G(L)$, $[b]$ denotes its $\sigma$-conjugacy class, and for $x\in \pi(G)_{I}$, $\overline{x}$ denotes its image under the natural quotient map $\pi(G)_{I}\rightarrow \pi(G)_{\Gamma_F}$. This map is functorial in $G$. For further details, see \cite[7.5]{Kottwitz97}.

\subsubsection{The Newton map $\nu_G$}

Let $\mathbb{D}$ denote the protorus $\varprojlim\Gm$ with the character group $\Q=\varinjlim\Z$.
For an algebraic group $G$ over a $p$-adic local field $F$, we put
\[\mathcal{N}(G):=(\Hom_{L}(\mathbb{D},G)/\Int(G(L)))^{\sigma}\]
(the subset of $\sigma$-invariants in the set of $G(L)$-conjugacy classes of $L$-rational quasi-cocharacters into $G_L$). We will use the notation $\overline{\nu}$ for the the conjugacy class of $\nu\in\Hom_{L}(\mathbb{D},G)$. 

For every $b\in G(L)$, Kottwitz \cite[$\S$4.3]{Kottwitz85} constructs an element $\nu=\nu_b\in\Hom_L(\mathbb{D},G)$ uniquely characterized by the property that there are an integer $s>0$, an element $c\in G(L)$ and a uniformizing element $\pi$ of $F$ such that:
\begin{itemize}\addtolength{\itemsep}{-4pt}
\item[(i)] $s\nu\in\Hom_L(\Gm,G)$.
\item[(ii)] $\Int (c)\circ s\nu$ is defined over the fixed field of $\sigma^s$ in $L$.
\item[(iii)] $c\cdot (b\rtimes\sigma)^s\cdot c^{-1}=c\cdot 
(s\nu)(\pi)\cdot c^{-1}\rtimes \sigma^{s}$.
\end{itemize}
In (iii), the product (and the equality as well) take place in the semi-direct product group $G(L)\rtimes\langle\sigma\rangle$. We call $\nu_b$ the \textit{Newton homomorphism} attached to $b\in G(L)$. 

When $G$ is a torus $T$, $\nu_b=\mathrm{av}\circ w_{T_L}(b)$, where $\mathrm{av}:X_{\ast}(T)_I\rightarrow X_{\ast}(T)_{\Q}^{\Gamma_F}$ is ``the average map" $X_{\ast}(T)_I\rightarrow X_{\ast}(T)_{\Gamma_F}\rightarrow X_{\ast}(T)_{\Q}^{\Gamma_F}$ sending $\underline{\mu}\ (\mu\in X_{\ast}(T))$ to $|\Gamma_F\cdot\mu|^{-1} \sum_{\mu'\in \Gamma_F\cdot\mu}\mu'$ (cf. \cite[Thm. 1.15, (iii)]{RR96}). Hence, it follows that if $T$ is split by a finite Galois extension $K\supset F$, for $b\in T(L)$, $[K:F]\nu_b\in X_{\ast}(T)$ and that $\langle \chi,\nu_b\rangle= \mathrm{val}(\chi(b))$ (especially $\in\Z$) for every $F$-rational character $\chi$ of $T$.

The map $b\mapsto \nu_b$ has the following properties.
\begin{itemize}\addtolength{\itemsep}{-4pt}
\item[(a)] $\nu_{\sigma(b)}=\sigma(\nu_b)$.
\item[(b)] $gb\sigma(g)^{-1}\mapsto \Int (g)\circ \nu,\ g\in G(L)$.
\item[(c)] $\nu_b=\Int (b)\circ\sigma(\nu_b)$.
\end{itemize}
It follows from (b) and (c) that $\nu_G:G(L)\rightarrow \Hom_L(\mathbb{D},G)$ gives rise to a map $\overline{\nu}_G:B(G)\rightarrow \mathcal{N}(G)$, which we call the \textit{Newton map}. This can be also regarded as a functor from the category of connected reductive groups to the category of sets (endowed with partial orders defined as below):
\begin{equation*}\overline{\nu}:B(\cdot)\rightarrow \mathcal{N}(\cdot)\ ;\
\overline{\nu}_{G}([b])=\overline{\nu}_{b},\quad
b\in[b].\end{equation*}

\subsubsection{} For a connected reductive group $G$ over an arbitrary (i.e. not necessarily $p$-adic) field $F$, let $\mathcal{BR}(G)=(X^{\ast},R^{\ast},X_{\ast},R_{\ast},\Delta)$ be the based root datum of $G$: we may take $X^{\ast}=X^{\ast}(T)$, $X_{\ast}=X_{\ast}(T)$ for a maximal $F$-torus $T$ of $G$ and $R^{\ast}\subset X^{\ast}(T)$, $R_{\ast}\subset X_{\ast}(T)$ are respectively the roots and the coroots for the pair $(G,T)$ with a choice of basis $\Delta$ of $R^{\ast}$ (whose choice corresponds to that of a Borel subgroup $B$ over $\overline{F}$).
Let $\overline{C}\subset (X_{\ast})_{\Q}$ denote the closed Weyl chamber associated with the root base $\Delta$. It comes with a canonical action of $\Gamma_F:=\Gal(\overline{F}/F)$ on $\overline{C}$. 

For a cocharacter $\mu\in\Hom_{\overline{F}}(\Gm,G)$ lying in $\overline{C}$, we set 
\[\overline{\mu}:=|\Gamma_F\cdot\mu|^{-1} \sum_{\mu'\in\Gamma_F\cdot\mu}\mu'\quad\in\overline{C}.\]
Here, the orbit $\Gamma_F\cdot\mu$ is obtained using the canonical Galois action on $\overline{C}$. Once a Weyl chamber $\overline{C}$ (equivalently, a Borel subgroup $B$ or a root base $\Delta$) is chosen, $\overline{\mu}$ depends only on the $G(\overline{F})$-conjugacy class $\{\mu\}$ of $\mu$.

Suppose that $\mu\in X_{\ast}(T)\cap\overline{C}$. 
As $X_{\ast}(T)=X^{\ast}(\widehat{T})$ for the dual torus $\widehat{T}$ of $T$, regarded as a character on $\widehat{T}$, we can restrict $\mu$ to the subgroup $Z(\widehat{G})^{\Gamma_F}$ of $\widehat{T}$, obtaining an element 
\begin{equation} \label{eqn:mu_natural}
\mu^{\natural}\in X^{\ast}(Z(\widehat{G})^{\Gamma_F})=\pi_1(G)_{\Gamma_F}.
\end{equation}
Again, $\mu^{\natural}$ depends only on the $G(\overline{F})$-conjugacy class $\{\mu\}$ of $\mu$.
Alternatively, $\mu^{\natural}$ equals the image (sometimes, also denoted by $\underline{\mu}$) of $\mu\in X_{\ast}(T)$ under the canonical map $X_{\ast}(T)\rightarrow \pi_1(G)_{\Gamma_F}$.

\subsubsection{The set $B(G,\{\mu\})$} \label{subsubsec:B(G,{mu})} 
Again, let us return to a $p$-adic field $F$. We fix a closed Weyl chamber $\overline{C}$ (equiv. a Borel subgroup $B$ over $\overline{F}$). 
Suppose given a $G(\overline{F})$-conjugacy class $\{\mu\}$ of cocharacters into $G_{\overline{F}}$.
Let $\mu$ be the representative of $\{\mu\}$ in $\overline{C}$; so we have $\overline{\mu}=\overline{\mu}(G,\{\mu\})\in\overline{C}$ and $\mu^{\natural}\in X^{\ast}(Z(\widehat{G})^{\Gamma_F})$. 
We define a finite subset $B(G,\{\mu\})$ of $B(G)$ (cf. \cite[Sec.6]{Kottwitz97}, \cite[Sec.4]{Rapoport05}): 
\[B(G,\{\mu\}):=\left\{\ [b]\in B(G)\ |\quad \kappa_{G}([b])=\mu^{\natural},\quad \overline{\nu}_{G}([b])\preceq \overline{\mu}\ \right\},\]
where $\preceq$ is the natural partial order on the closed Weyl chamber $\overline{C}$ defined by that $\nu\preceq \nu'$ if $\nu'-\nu$ is a nonnegative linear combination (with \emph{rational} coefficients) of simple coroots in $R_{\ast}(T)$ (\cite{RR96}, Lemma 2.2). 
One knows (\cite[4.13]{Kottwitz97}) that the map 
\[(\overline{\nu},\kappa):B(G)\rightarrow \mathcal{N}(G)\times X^{\ast}(Z(\widehat{G})^{\Gamma_F})\] 
is injective, hence $B(G,\{\mu\})$ can be identified with a subset of $\mathcal{N}(G)$.


\subsection{Parahoric subgroups}

Our main references here are \cite{Rapoport05}, \cite{HainesRapoport08}, \cite{HainesRostami10}, as well as \cite{BT72}, \cite{BT84}.

\subsubsection{}  \label{subsubsec:parahoric}
Let $G$ be a connected reductive group $G$ over a strictly henselian discrete valued field $L$. Let $\mcB(G,L)$ be the Bruhat-Tits building of $G$ over $L$ (cf. \cite{Tits79}, \cite{BT72}, \cite{BT84}). Then, a \textit{parahoric subgroup} of $G(L)$ is a subgroup of the form
\[ K_{\mbff}=\mathrm{Fix}\ \mbff \cap \Ker\ w_G\]
for a facet $\mbff$ of $\mcB(G,L)$. Here, $\mathrm{Fix}\ \mbff$ denotes the subgroup of $G(L)$ fixing $\mbff$ pointwise and $w_G$ is the Kottwitz map (\ref{subsubsec:Kottwitz_hom}). When $\mbff$ is an alcove of $\mcB(G,L)$ (i.e. a maximal facet), the parahoric subgroup is called an \textit{Iwahori} subgroup. A \textit{special maximal parahoric subgroup} of $G(L)$ is the parahoric subgroup attached to a special point in $\mcB(G,L)$.
More precisely, choose a maximal split torus $A$ of $G$ and let $\mcA(A,L)$ be the associated apartment; let $\mcA(A^{\ad},L)$ be the apartment in $\mcB(G^{\ad},L)$ corresponding to the image $A^{\ad}$ of $A$ in $G^{\ad}$. Then, there exists a canonical simplicial isomorphism (\cite[1.2]{Tits79})
\[\mcA(A,L)\cong \mcA(A^{\ad},L)\times X_{\ast}(Z(G))_{\Gamma_F}\otimes\R.\]
Then, every special point in $\mcA(A,L)$ is of the form $\{\mbfv\}\times x$ for a unique special \emph{vertex} $\mbfv$ of $\mcA(A^{\ad},L)$ (in the sense of \cite[1.9]{Tits79}) and some $x\in X_{\ast}(Z(G))_{\Gamma_F}\otimes\R$.

The original definition of pararhoic subgroups by Bruhat-Tits (\cite[5.2.6]{BT84}, cf. \cite[3.4]{Tits79}) uses group schemes. With every facet $\mbff$ of $\mcB(G,L)$ they associate a smooth group scheme $\mcG_{\mbff}$ over $\Spec(\cO_L)$ with generic fiber $G$ such that $\mcG_{\mbff}(\cO_L)=\mathrm{Fix}\ \mbff$. Also, there exists an open subgroup $\mcG_{\mbff}^{\mathrm{o}}$ with the same generic fiber $G$ and the connected special fiber. Then, the parahoric subgroup attached to $\mbff$ by Bruhat-Tits is $\mcG_{\mbff}^{\mathrm{o}}(\cO_L)$. It is known (\cite[Prop.3]{HainesRapoport08}) that they coincide:
\[K_{\mbff}=\mcG_{\mbff}^{\mathrm{o}}(\cO_L).\]

Now suppose that $G$ is defined over a local field $F$, as before given as a finite extension of $\Qp$ in $\Qpb$.
Again, $L$ denotes the completion of the maximal unramified extension of $F$ in $\Qpb$ and let $\sigma$ be the Frobenius automorphism. Let $\mcB(G,L)$ (resp. $\mcB(G,F)$) be the Bruhat-Tits building of $G$ over $L$ (resp. over $F$); as $G$ is defined over $F$, $\mcB(G,L)$ carries an action of $G(L)\rtimes\langle\sigma\rangle$ and $\mcB(G,F)$ is identified with the set of fixed points of $\mcB(G,L)$ under $\langle\sigma\rangle$ (\cite[5.1.25]{BT84}). This procedure of taking $\sigma$-fixed points $\mbff\mapsto \mbff^{\sigma}$ gives a bijection from the set of $\sigma$-stable facets in $\mcB(G,L)$ to the set of facets in $\mcB(G,F)$.

A \textit{parahoric subgroup} of $G(F)$ is by definition $\mcG_{\mbff}^{\mathrm{o}}(\cO_L)^{\sigma}$ for a $\sigma$-stable facet $\mbff$ of $\mcB(G,L)$. A \textit{special maximal parahoric subgroup} of $G(F)$ is $\mcG_{\mbff}^{\mathrm{o}}(\cO_L)^{\sigma}$ for a special point $\mbff\in\mcB(G,F)$ (see \cite[1.9]{Tits79} for the definition of a special point).

\subsubsection{Extended affine Weyl group} \label{subsubsec:EAWG}

Let $G$ be a connected reductive group $G$ over a strictly henselian discrete valued field $L$. Let $S$ be a maximal split $L$-torus of $G$ and $T$ its centralizer; $T$ is a maximal torus since $G_L$ is quasi-split by a well-known theorem of Steinberg. Let $N=N_G(T)$ be the normalizer of $T$. The \textit{extended affine Weyl group} (or \textit{Iwahori Weyl group}) associated with $S$ is the quotient group
\[\tilde{W}:=N(L)/T(L)_1,\]
where $T(L)_1$ is the kernel of the Kottwitz map $w_T:T(L)\rightarrow X_{\ast}(T)_{I}$. As $w_T$ is surjective, it is an extension of the relative Weyl group $W_0:=N(L)/T(L)$ by $X_{\ast}(T)_I$: 
\begin{equation} \label{eqn:EAWG1}
0\rightarrow X_{\ast}(T)_I\rightarrow \tilde{W}\rightarrow W_0\rightarrow 0.
\end{equation}
The normal subgroup $X_{\ast}(T)_I$ is called the \textit{translation subgroup} of $\tilde{W}$, and any $\lambda\in X_{\ast}(T)_I$, viewed as an element in $\tilde{W}$ in this way, will be denoted by $t^{\lambda}$ (\textit{translation element}).

This extension splits by choosing a special vertex $\mbfv$ in the apartment corresponding to $S$, namely if $K=K_{\mbfv}\subset G(L)$ is the associated parahoric subgroup, the subgroup
\[\tilde{W}_K:=(N(L)\cap K)/T(L)_1\]
of $\tilde{W}$ projects isomorphically to $W_0$, and thus gives a splitting
\[\tilde{W}=X_{\ast}(T)_I\rtimes \tilde{W}_K.\]

For two parahoric subgroups $K$ and $K'$ associated with facets in the apartment corresponding to $S$, there exists an isomorphism
\begin{equation} \label{eqn:parahoric_double_coset}
K\backslash G(L)/ K'\cong \tilde{W}_K\backslash \tilde{W}/ \tilde{W}_{K'}.
\end{equation}

Let $S^{\uc}$ (resp. $T^{\uc}$, $N^{\uc}$) be the inverse image of $S$ (resp. $T$, $N$) in the universal covering $G^{\uc}$ of $G^{\der}$; then, $S^{\uc}$ is a maximal split torus of $G^{\uc}$ and $T^{\uc}$ (resp. $N^{\uc}$) is its centralizer (resp. the normalizer of $N^{\uc}$). 
The natural map $N^{\uc}(L)\rightarrow N(L)$ induces an injection $X_{\ast}(T^{\uc})_I\hra X_{\ast}(T)_I$ and presents
the extended affine Weyl group associated with $(G^{\uc},S^{\uc})$
\[W_a:=N^{\uc}(L)/ T^{\uc}(L)_1\]
as a normal subgroup of the extended affine Weyl group $\tilde{W}$ (attached to $S$) such that the translation subgroup $X_{\ast}(T)_I$ maps onto the quotient $\tilde{W}/W_a$ with kernel $X_{\ast}(T^{\uc})_I$:
\begin{equation} \label{eqn:EAWG2}
0\rightarrow W_a\rightarrow \tilde{W}\rightarrow X_{\ast}(T)_I/X_{\ast}(T^{\uc})_I\rightarrow 0.
\end{equation}
Note that the quotient group $X_{\ast}(T)_I/X_{\ast}(T^{\uc})_I$ is in a natural way a subgroup of $\pi_1(G)_I$.
The group $W_a$ can be also regarded as an affine Weyl group attached to some reduced root system. 

This extension (\ref{eqn:EAWG2}) also splits by the choice of an alcove in the apartment $\mcA(S,L)$ of $S$. More precisely, the extended affine Weyl group $\tilde{W}$ (resp. the affine Weyl group $W_a$) acts transitively (resp. simply transitively) on the set of alcoves in $\mcA(S,L)$, hence when we choose a base alcove $\mbfa$ in $\mcA(S,L)$, 
\begin{equation} \label{eqn:splitting_of_EAWG2}
\tilde{W}=W_a\rtimes \Omega_{\mbfa},
\end{equation}
where $\Omega_{\mbfa}$ is the normalizer of $\mbfa$; $\Omega_{\mbfa}$ will be often identified with $X_{\ast}(T)_I/X_{\ast}(T^{\uc})_I$.

Finally, suppose that there is an automorphism $\sigma$ of $L$ such that $L$ is the strict henselization of its fixed field $L^{\natural}$ and that $G$ is defined over $L^{\natural}$. Then, we can find a $L^{\natural}$-torus $S$ such that $S_L$ becomes maximal split $L$-torus and a maximal $L^{\natural}$-torus $T$ which contains $S$; set $N$ to be the normalizer of $T$. Then $\sigma$ acts on the extended Weyl group $\tilde{W}$ in the obvious way. Moreover, if $K_{\mbfv}\subset G(L)$ is the parahoric subgroup attached to a $\sigma$-stable facet $\mbfv$, then
the subgroup $\tilde{W}_{K_{\mbfv}}$ is stable under $\sigma$. We refer the reader to \cite[Remark 9]{HainesRapoport08} for a ``descent theory" in this situation.

\subsubsection{The $\{\mu\}$-admissible set} \label{subsubsec:mu-admissible_set}
As before, let $G$ be a connected reductive group $G$ over a complete discrete valued field $L$ with algebraically closed residue field. Let $W=N(\Lb)/T(\Lb)$ be the absolute Weyl group. Let
$\{\mu\}$ be a $G(\Lb)$-conjugacy class of cocharacters of $G$ over $\Lb$. We use $\{\mu\}$ again to denote the corresponding $W$-orbit in $X_{\ast}(T)$. Let us choose a Borel subgroup $B$ over $L$ containing $T$ (which exists as $G_L$ is automatically quasi-split), and let $\mu_B$ be the unique representative of $\{\mu\}$ lying in the associated absolute closed Weyl chamber in $X_{\ast}(T)_{\R}$. Then, the $W_0$-orbit of the image $\underline{\mu_B}$ of $\mu_B$ in $X_{\ast}(T)_I$ is well-determined, since any two Borel subgroups over $L$ containing $T$ are conjugate under $G(L)$. We denote this $W_0$-orbit by $\Lambda(\{\mu\})$:
\[ \Lambda(\{\mu\}):=W_0\cdot\underline{\mu_B}\ \subset\ X_{\ast}(T)_I.\]
 It is known (\cite[Lemma 3.1]{Rapoport05}) that the image of $\Lambda(\{\mu\})$ in the quotient group $X_{\ast}(T)_I/X_{\ast}(T^{\uc})_I$ consists of a single element, which we denote by $\tau(\{\mu\})$. 

Let us now fix an alcove $\mbfa$ in the apartment corresponding to $S$. This determines a Bruhat order on the affine Weyl group $W_a$ which extends to the extended Weyl group $\tilde{W}=W_a\rtimes \Omega_{\mbfa}$ (\ref{eqn:splitting_of_EAWG2}), (\cite[$\S$1]{KR00}). Also, when $K\subset G(L)$ is a parahoric subgroup associated with a facet of $\mbfa$, it induces a Bruhat order on the double coset space $\tilde{W}_K\backslash \tilde{W}/\tilde{W}_K$ (\cite[$\S$8]{KR00}). We will denote all these orders by $\leq$; this should not cause much confusion.

\begin{defn} \label{defn:mu-admissible_subset}
The \textit{$\{\mu\}$-admissible subset} of $\tilde{W}$ is 
\[\mathrm{Adm}(\{\mu\})=\{w\in\tilde{W}\ |\ w\leq t^{\lambda}\text{ for some }\lambda\in\Lambda(\{\mu\})\},\]
and the \textit{$\{\mu\}$-admissible subset} of $\tilde{W}_K\backslash \tilde{W}/\tilde{W}_K$ is 
\[\mathrm{Adm}_K(\{\mu\})=\{w\in\tilde{W}_K\backslash \tilde{W}/\tilde{W}_K\ |\ w\leq \tilde{W}_Kt^{\lambda}\tilde{W}_K\text{ for some }\lambda\in\Lambda(\{\mu\})\}.\]
\end{defn}
One knows (\cite[(3.8)]{Rapoport05}) that $\mathrm{Adm}_K(\{\mu\})$ is the image of $\mathrm{Adm}(\{\mu\})$ under the natural map $\tilde{W}\rightarrow \tilde{W}_K\backslash \tilde{W}/\tilde{W}_K$.

\begin{prop}
Suppose that $G$ splits over $L$ (thus $S=T$) and $K$ is a special maximal parahoric subgroup. Then, 
\[\mathrm{Adm}_K(\{\mu\})=\{\nu\in X_{\ast}(S)\cap\overline{C}\ |\ \nu\stackrel{!}{\leq} \mu\},\]
where $\mu$ denotes the representative in $\overline{C}$ of $\{\mu\}$.
If furthermore $\{\mu\}$ is minuscule, $\mathrm{Adm}_K(\{\mu\})$ consists of a single element, i.e. $\{\mu\}\in X_{\ast}(T)/W$ itself.
\end{prop}

Here, $\nu\stackrel{!}{\leq} \mu$ means that $\mu-\nu$ is a sum of simple coroots with non-negative \emph{integer} coefficients. See Prop. 3.11 and Cor. 3.12 of \cite{Rapoport05} for a proof.


\section{Pseudo-motivic Galois gerb and admissible morphisms}

This section is devoted to a review of the theory of pseudo-motivic Galois gerbs and admissible morphisms, as explained in \cite{LR87}. In addition to this original source \cite{LR87}, our main references are \cite{Milne92},  \cite{Kottwitz92}, \cite{Reimann97}.

\subsection{Galois gerbs}
We review the notion of Galois gerbs as used by Langlands-Rapoport in \cite[$\S$2]{LR87} (cf. \cite[$\S$4]{Breen94}, \cite[$\S$8]{Rapoport05}, \cite[Appendix B]{Reimann97}).
 
Let $k$ be a field of characteristic zero (which will be for us either a global or a local field) and $\kb$ an algebraic closure. For an affine group scheme $G=\Spec A$ over a Galois extension $k'\subset\kb$ of $k$ and $\sigma\in\Gal(k'/k)$, 
an automorphism $\kappa$ of $G(k')$ is said to be \textit{$\sigma$-linear} if there is a $\sigma$-linear automorphism $\kappa'$ of the algebra $A$ such that
\[\kappa'(f)(\kappa(g))=\sigma(f(g)),\quad f\in A,\ g\in G(k').\]
The simplest example is given by the natural action of $\Gal(k'/k)$ on $G(k')$, when $G$ is defined over $k$. In this article, we will be concerned mainly with the following kind of Galois gerbs, which will be called \textit{algebraic}. 
For $\sigma\in\Gal(k'/k)$, let
\[\sigma_{k'}:G(k')\rightarrow (\sigma^{\ast}G)(k')\] 
be the unique map for which $f\otimes 1(\sigma_{k'}(g))=\sigma(f(g))$ holds for $f\in A,\ g\in G(k')$, where $f\otimes1\in A\otimes_{k',\sigma}k'$.
Then, for any algebraic isomorphism $\theta$ of $k'$-group schemes from $\sigma^{\ast}G$ to $G$, 
the automorphism $\theta\circ\sigma_{k'}$ of $G(k')$ is $\sigma$-linear, since then one can take $\kappa:=\theta\circ\sigma_{k'}$ and $\kappa'(f):=(\theta^{\ast})^{-1}(f\otimes1)$ (Here, $\theta^{\ast}:A\isom A\otimes_{k',\sigma}k'$ denotes the associated map on the structure sheaf). 
We will call such $\sigma$-linear automorphism of $G(k')$ \textit{algebraic}. 
Hence, one can identify an algebraic $\sigma$-linear isomorphism $\kappa(\sigma)$ with an algebraic $k'$-isomorphism $\theta(\sigma):\sigma^{\ast}(G)\isom G$ via $\kappa(\sigma)=\theta(\sigma)\circ\sigma_{k'}$.

\begin{defn} \label{defn:Galois_gerb}
Let $k'\subset \kb$ be a Galois extension of $k$. A \textit{$k'/k$-Galois gerb} is an extension of topological groups
\[1\lra G(k')\lra \fG\lra \Gal(k'/k)\lra 1,\]
where $G$ is an affine smooth group scheme (i.e. a linear algebraic group) over $k'$ and $G(k')$ (resp. $\Gal(k'/k)$) has the discrete (resp. the Krull) topology, such that
\begin{itemize}
\item[(i)] for every representative $g_{\sigma}\in\fG$ of $\sigma\in \Gal(k'/k)$, the automorphism $\kappa(\sigma):g\mapsto g_{\sigma}gg_{\sigma}^{-1}$ of $G(k')$ is algebraic $\sigma$-linear.
\item[(ii)] for some finite sub-extension $k\subset K\subset k'$, there exists a continuous section
\[\Gal(k'/K)\lra \fG\ :\ \sigma\mapsto g_{\sigma}\] 
which is a group homomorphism.
\end{itemize}
\end{defn}

In the presence of (i), the condition (ii) means that the family $\{\theta(\sigma):\sigma^{\ast}(G)\isom G\}$ of isomorphisms associated with $\Int (g_{\sigma})$ is a $k'/K$-descent datum on $G$: the homomorphism property of (ii) gives the cocycle condition of descent datum. 
Thus the section $\sigma\mapsto g_{\sigma}\ (\sigma\in\Gal(k'/K))$ determines a $K$-structure on $G$ and accordingly an action of $\Gal(k'/K)$ on $G(k')$. 
This Galois action is nothing other than $\theta(\sigma)\circ\sigma_{k'}$,
\footnote{Suppose that the $K$-structure is given by an isomorphism $\alpha:G_0\otimes_{K}k'\isom G$ for an algebraic $K$-group $G_0$. Then, it gives rise to a descent isomorphism $\alpha\circ\sigma^{\ast}(\alpha^{-1}):\sigma^{\ast}(G)\isom G$ and a Galois action $\sigma(g)=\alpha\circ \sigma(\alpha^{-1}(g))$ on $G(k')=G_0(k')$. Since the descent isomorphism was $\theta(\sigma)$, we have $\sigma(g)=\alpha\circ \sigma^{\ast}(\alpha^{-1})(\sigma_{k'}(g))=\theta(\sigma)\circ\sigma_{k'}(g)$.}
namely, we have the relation 
\begin{equation} \label{eq:conjugation=Galois_action}
g_{\sigma}gg_{\sigma}^{-1}=\sigma(g),\quad \sigma\in \Gal(k'/K),
\end{equation}
where $\sigma(g)$ is the just mentioned action of $\sigma\in\Gal(k'/K)$ on $G(k')$. In other words, the conditions (i), (ii) imply that over some finite Galois extension $K\subset k'$ of $k$, there exists a group-theoretic section $\sigma\mapsto \rho_{\sigma}$, via which the pull-back to $\Gal(k'/K)$ of $\fG$ becomes a semi-direct product $G(k')\rtimes \Gal(k'/K)$, with the action of $\Gal(k'/K)$ on $G(k')$ (i.e. the conjugation action of $\Gal(k'/K)$ on $G(k')$ via the section) being the natural Galois acton resulting from a $K$-structure on $G$.

We remark that our definition of Galois gerb is equivalent to that of affine smooth gerb
\footnote{in the sense of Giruad, or in the sense of the theory of Tannakian categories, namely, a stack in groupoids over an \'etale site which is locally nonempty and locally connected, cf. \cite[Appendix]{Milne92}, \cite[Ch.II]{DMOS82}, \cite[2.2]{Breen94}.}
on the \'etale site $\Spec(k)_{\et}$ \textit{equipped with a neutralizing object over $\Spec(K)$}.
\footnote{The $2$-category of such affine gerbs endowed with a distinguished neutralizing object is equivalent to the $2$-category of affine $\Spec(K)/\Spec(k)$-groupoid schemes, acting transitively on $\Spec(K)$, (\cite{Milne92}, Appendex, Prop. A.15). 
For this reason, Milne insists to call Galois gerbs in our sense groupoids (\cite{Milne92}, \cite{Milne03}). 
But any two neutralizing local objects become isomorphic over $\kb$, thus a gerb $\fG$ (as a stack) is uniquely determined by its associated groupoid $(\fG,x\in \mathrm{Ob}(\fG(\kb)))$, up to conjugation by an element of $\mathrm{Aut}(x)=\fG^{\Delta}(\kb)$. Hopefully, this justifies our decision to stick to the original terminology of Langlands-Rapoport.}
For a detailed discussion of this relation, we refer to \cite[p.152-153]{LR87}, \cite[$\S$4]{Breen94}.

We call the group scheme $G$ the \textit{kernel} of $\fG$ and write $G=\fG^{\Delta}$. A \textit{morphism} between $k'/k$-Galois gerbs $\varphi:\fG\rightarrow\fG'$ is a continuous map of extensions which induces the identity on $\Gal(k'/k)$ and an algebraic homomorphism on the kernel groups. Two morphisms $\phi_1$ and $\phi_2$ are said to be \textit{conjugate} if there exists $g'\in G'(k')$ with $\phi_2=\Int (g')\circ \phi_1$. With every linear algebraic group $G$ over $k$, the semi-direct product gives a gerb
\[\fG_G=G(k')\rtimes\Gal(k'/k).\]
We call it the \textit{neutral gerb} attached to $G$.

For two successive Galois extensions $k\subset k'\subset k''\subset \kb$, 
any $k'/k$-Galois gerb $\fG$ gives rise to a $k''/k$-Galois gerb $\fG'$, by first pulling-back the extension $\fG$ by the surjection $\Gal(k''/k)\twoheadrightarrow\Gal(k'/k)$ and then pushing-out via $G(k')\rightarrow G(k'')$. In this situation, we will call $\fG'$ the \textit{inflation} to $k''$ of the $k'/k$-Galois gerb $\fG$; this terminology will be justified when we relate Galois gerbs with commutative kernels to Galois cohomology. 
We also call a $\kb/k$-Galois gerb, simply a Galois gerb over $k$. It follows from definition that any Galois gerb over $k$ is the inflation of a $k'/k$-Galois gerb for some finite Galois extension $k\subset k'\subset \kb$ and that every morphism between $k'/k$-Galois gerbs induces a morphism between their inflations to $k''$ for any sub-extension $k''\subset \kb$.

For two morphisms of $k'/k$-Galois gerbs $\phi_1,\phi_2:\fG\rightarrow\fG'$, there exists a $k$-scheme $\underline{\mathrm{Isom}}(\phi_1,\phi_2)$, whose set of $R$-points, for a $k$-algebra $R$, is given by
\[\underline{\mathrm{Isom}}(\phi_1,\phi_2)(R)=\{g\in\fG'^{\Delta}(k'\otimes_kR)\ |\ \Int (g)\circ\phi_1=\phi_2\},\]
where the equality is considered in the push-out of $\fG'$ by $\fG'^{\Delta}(k')\rightarrow \fG'^{\Delta}(k'\otimes_kR)$. 
When, $\phi_1=\phi_2$, this is a $k$-group scheme which we also denote by $I_{\phi_1}=\mathrm{Aut}(\phi_1)$.

\subsubsection{Galois gerbs defined by $2$-cocyles with values in commutative affine group schemes}

In our work (as well as in the work of Langlands-Rapoport), 
besides the neutral gerbs attached to arbitrary algebraic groups, all the nontrivial Galois gerbs have as associated kernel \textit{commutative affine group schemes (in fact, (pro-)tori) defined over base fields}. In such cases, a Galois gerb has an explicit description in terms of $2$-cocycles (in the usual sense) on the absolute Galois group of the base field with values in the geometric points of given commutative affine group scheme endowed with the natural Galois action.

Recall that for a group $H$ and an $H$-module $A$ (i.e. an abelian group with $H$-action), a normalized%
\footnote{i.e. $e_{h,1}=e_{1,h}=1$ for all $h\in H$.}
$2$-cocyle (or ``factor set") $(e_{h_1,h_2})$ on $H$ with values in $A$ gives rise to an extension of $H$ by $A$:
\[1\rightarrow A\rightarrow E\stackrel{p}{\rightarrow}H\rightarrow1\] 
with property 
\begin{equation} \label{eq:extension_with_given_conjugation_action}
e\cdot a\cdot e^{-1}=p(e)(a)\quad (e\in E,a\in A),
\end{equation}
where the right action of $p(e)$ on $A$ is the given one.
Explicitly, $E$ is generated by $A$ and $\{e_h\}_{h\in H}$ ($a\mapsto 0,e_h\mapsto h$ giving the projection $E\rightarrow H$) with relations 
\[e_h\cdot a\cdot e_h^{-1}=h(a)\ (a\in A,h\in H),\qquad e_{h_1,h_2}=e_{h_1}\cdot e_{h_2}\cdot e_{h_1h_2}^{-1},\quad e_1=1\] ($e_{h_1,h_2}\in Z^2(H,A)$ guarantees the associativity of the resulting composition law). 
Two extensions $E$, $E'$ of $H$ by $A$ with the property (\ref{eq:extension_with_given_conjugation_action}) are said to be \textit{isomorphic} if there exists an isomorphism $E\isom E'$ which restricts to identity on $A$ and also induces identity on $H$. 
Then, this construction gives a bijection of pointed sets between $H^2(H,A)$ and the set of isomorphisms classes of group extensions of $H$ by $A$ with the induced conjugation action of $H$ on $A$ being the given one. Here, the distinguished points are the cohomology class of the trivial $2$-cocycle and the semi-direct product, respectively. For two isomorphic extensions $E,E'$ with the property (\ref{eq:extension_with_given_conjugation_action}), we say that two isomorphisms $f_1,f_2:E\isom E'$ 
which induce identities on $A$ and $H$ are \textit{equivalent} (or \textit{conjugate}) if $f_2=\Int  a\circ f_1$ for some $a\in A$, where $\Int  a$ is the conjugation automorphism of $E'$. Then, there is a natural action of $H^1(H,A)$ on the set of equivalence classes of isomorphisms from $E$ to $E'$, which makes the latter set into a torsor under $H^1(H,A)$. 

Now, suppose that $G$ is a \textit{separable} commutative affine group scheme over $k$: then $G$ is the inverse limit of a strict system of commutative algebraic groups indexed by $(\N,\leq)$ (cf. \cite[2.6]{Milne03}). If $k'\subset\kb$ is a Galois extension of $k$, $G(k')$ is endowed with the inverse limit topology (for algebraic group $Q$, $Q(k')$ is given the discrete topology) and we get a continuous action of $\Gal(k'/k)$ on $G(k')$ provided by the $k$-structure of $G$. 
Then, for any continuous $2$-cocycle $(e_{\rho,\tau})\in Z^2_{cts}(\Gal(k'/k),G(k'))$, the resulting extension 
\[1\rightarrow G(k')\rightarrow \fE_{k'}\rightarrow \Gal(k'/k)\rightarrow 1\] 
is a $k'/k$-Galois gerb.
Indeed, the condition (i) of Definition \ref{defn:Galois_gerb} is obvious, and for (ii), we note that
since $G(k')$ has discrete topology, any class in $H^2_{cts}(\Gal(k'/k),G(k'))$ lies in 
$H^2(\Gal(K/k),G(K))$ for a \emph{finite} Galois extension $K$ of $k$, so becomes trivial when restricted to $\Gal(k'/K)$.
Furthermore, by pulling-back along $\Gal(\kb/k)\twoheadrightarrow\Gal(k'/k)$ and push-out via $G(k')\hra G(\kb)$, we obtain a Galois gerb $\fE$ over $k$
\[1\rightarrow G(\kb)\rightarrow \fE\rightarrow \Gal(\kb/k)\rightarrow 1,\]
which we called the inflation of $\fE_{k'}$ to $\kb$. 
Now, one can verify that the corresponding cohomology class in $H^2_{cts}(\Gal(\kb/k),G(\kb))$ is indeed the image of $(e_{\rho,\tau})\in H^2_{cts}(\Gal(k'/k),G(k'))$ under the inflation map $H^2_{cts}(\Gal(k'/k),G(k'))\rightarrow H^2_{cts}(\Gal(\kb/k),G(\kb))$.

\subsection{Pseudo-motivic Galois gerb}

\subsubsection{Local Galois gerbs} \label{subsubsec:Local_Galois_gerbs}
Here, we define a Galois gerb $\fG_v$ over $\Q_v$ for each place $v$ of $\Q$. 

For $v\neq p,\infty$, we define $\fG_v$ to be the trivial Galois gerb $\Gal(\overline{\Q}_v/\Q_v)$:
\[\begin{array} {ccccccccc}
1 & \rightarrow & 1 & \rightarrow & \Gal(\overline{\Q}_v/\Q_v) & \rightarrow & \Gal(\overline{\Q}_v/\Q_v) & \rightarrow & 1.
\end{array}\]

For $v=\infty$, the cocycle $(d_{\rho,\gamma})\in Z^2(\Gal(\C/\R),\C^{\times})$ 
\[d_{1,1}=d_{1,\iota}=d_{\iota,1}=1,\quad d_{\iota,\iota}=-1,\]
where $\Gal(\C/\R)=\{1,\iota\}$, represents the fundamental class in $H^2(\Gal(\C/\R),\C^{\times})$ (\cite{Milne13}). We set $\fG_{\infty}$ to be the (isomorphism class of) Galois gerb defined by this cocycle (or its cohomology class):
\[\begin{array} {ccccccccc}
1 & \rightarrow & \C^{\times} & \rightarrow & \fG_{\infty} & \rightarrow & \Gal(\C/\R) & \rightarrow & 1,
\end{array}\]
So,
$\fG_{\infty}$ has generators $\C^{\times}$ and $w=w(\iota)$ (lift of $\iota$) satisfying that
\[w(\iota)^2=-1\in\C^{\times},\ \text{ and}\quad wzw^{-1}=\iota(z)=\overline{z}\ (z\in\C^{\times}).\]

For $v=p$ also, for any finite Galois extension $K$ of $\Qp$ in $\Qpb$, there is the fundamental class in $H^2(\Gal(K/\Qp),K^{\times})$ (\cite{Milne13}).
For unramified extension $L_n/\Qp$, it is represented by the cocycle: for $0\leq i,j<n$, 
\begin{equation} \label{eq:canonical_fundamental_cocycle}
d_{\overline{\sigma}^i,\overline{\sigma}^j}=\begin{cases}
p^{-1} & \text{if } i+j\geq n, \\
1 & \text{ otherwise, }
\end{cases}
\end{equation} 
where $\overline{\sigma}\in \Gal(L_n/\Qp)$ is the \emph{arithmetic} Frobenius. 

We let $\fG_{p,K}^K$ be the corresponding (isomorphism class of) $K/\Qp$-Galois gerb and $\fG_p^K$ the Galois gerb over $\Qp$ obtained from $\fG_{p,K}^K$ by inflation:
\footnote{Note that our notations for these gerbs differ from those of \cite{Kisin13}: our $\fG^K_{p,K}$ (resp. $\fG^\mbfK_p$) is his $\fG^\mbfK_p$ (resp. $\tilde{\fG}^\mbfK_p$).}
\[\xymatrix{
1 & \rightarrow & K^{\times} & \rightarrow & \fG_{p,K}^K & \rightarrow & \Gal(K/\Qp) & \rightarrow & 1\\ 
1 & \rightarrow & K^{\times} \ar@{=}[u] \ar@{^{(}->}[d] & \rightarrow & \pi_K^{\ast}\fG_{p,K}^K \ar[u] \ar[d] & \rightarrow & \Gal(\Qp/\Qp) \ar@{->>}[u]^{\pi_K} \ar@{=}[d] & \rightarrow & 1\\
1 & \rightarrow & \Gm(\Qpb) & \rightarrow & \fG_p^K & \rightarrow & \Gal(\Qp/\Qp) & \rightarrow & 1.}
\]
Here, $\pi_K^{\ast}\fG_{p,K}^K$ is the pull-back of $\fG_{p,K}^K$ along $\pi_K:\Gal(\Qp/\Qp)\twoheadrightarrow \Gal(K/\Qp)$, and $\fG_p^K$ is the push-out of $\pi_K^{\ast}\fG_{p,K}^K$ along $K^{\times}\hra \Gm(\Qpb)$.

For each $K\subset K'\subset\Qpb$ ($K'$ being a Galois extension of $\Qp$ containing $K$), there exists a homomorphism
\[\fG_p^{K'}\rightarrow \fG_p^K\]
which, on the kernel, is given by $z\mapsto z^{[K':K]}$ (\cite[p.119]{LR87}, 
\cite[Remark B1.2]{Reimann97}). By passing to the inverse limit over $K\supset \Qp$, we obtain a pro-Galois gerb $\fG_p$ over $\Qp$ with kernel $\mathbb{D}=\varprojlim\Gm$ (the protorus over $\Qp$ with character group $X^{\ast}(\mathbb{D})=\Q$). 

For each Galois extension $K\subset\Qpb$ of $\Qp$, we make a choice of a normalized cocycle $(d_{\tau_1,\tau_2}^K)$ on $\Gal(K/\Qp)$ with values in $K^{\times}$ defining $\fG_{p,K}^K$, and fix a section $\tau\mapsto s_{\tau}^K$ to the projection $\fG_{p,K}^K\rightarrow\Gal(K/\Qp)$ with the property that
\[s_{\tau_1}^Ks_{\tau_2}^K=d_{\tau_1,\tau_2}^Ks_{\tau_1\tau_2}^K,\qquad s_{1}^K=1.\]
Since $\fG_p^K$ is obtained from $\fG_{p,K}^K$ by inflation, this gives rise to a section to $\fG_p^K \rightarrow \Gal(\Qpb/\Qp)$, which we also denote by $s^K$.%
\footnote{At the moment, we do not require that one can choose the sections $\tau\mapsto s_{\tau}^K$ in a compatible way for extensions $K\subset K'\subset\Qpb$, which will be however the case (cf. the proof of Theorem \ref{thm:pseudo-motivic_Galois_gerb}).}
By Hilbert 90, any such section $s^K$ is uniquely determined up to conjugation by an element of $K^{\times}$.

\subsubsection{Dieudonn\'e gerb} 
We also need an unramified version of the Galois gerbs $\fG^K_{p,K}$, $\fG_p$. 
Let $\Qpnr$ be the maximal unramified extension of $\Qp$ in $\Qpb$.
For $n\in\N$, we denote by $\fD_n=\fD_{L_n}$ the inflation to $\Qpnr$ of the $L_n/\Qp$-Galois gerb $\fG^{L_n}_{p,L_n}$.
As before, for every pair $m|n$, there exists a homomorphism $\fD_n\rightarrow\fD_m$ which, on the kernel, is given by $z\mapsto z^{n/m}$ (cf. 
\cite{Reimann97}, Remark B1.2). By passing to the inverse limit, we get a pro-$\Qpnr/\Qp$-Galois gerb $\fD$ over $\Qp$ with kernel $\mathbb{D}$. We call $\fD$ the \textit{Dieudonn\'e gerb}. Obviously, the Galois gerb $\fG_p^{L_n}$ (resp. the pro-Galois gerb $\fG_p$) is (equivalent to) the inflation to $\Qpb$ of $\fD_n$ (resp. $\fD$). Again, a choice of a section to $\fG_{p,L_n}^{L_n} \rightarrow \Gal(L_n/\Qp)$ made above gives us a section to $\fD_n \rightarrow \Gal(\Qpnr/\Qp)$ which is again denoted by $s^{L_n}$.

\subsubsection{Unramified morphisms} \label{subsubsec:cls}

For any (connected) reductive group $H$ over $\Qp$, there exists a canonical map
\[\mathrm{cls}_H:\Hom_{\Qpb/\Qp}(\fG_p,\fG_H)\rightarrow B(H),\]
where $B(H)$ is the set of $\sigma$-conjugacy classes of elements in $H(L)$. 

Let $K\subset \Qpb$ be a finite Galois extension. Recall that we fixed a normalized cocycle $(d_{\tau_1,\tau_2}^K)\in Z^2(\Gal(K/\Qp),K^{\times})$ defining $\fG^K_{p,K}$ as well as a section $s^K$ to the projection $\fG^K_{p,K}\rightarrow\Gal(K/\Qp)$ with the property that $s_{\tau_1}^Ks_{\tau_2}^K=d_{\tau_1,\tau_2}^Ks_{\tau_1\tau_2}^K$ and $s_{1}^K=1$; one uses the same notations for the induced cocycle defining $\fG_p^K$ and the induced section $\Gal(\Qpb/\Qp)\rightarrow \fG_p^K$.
A morphism $\theta:\fG_p^K\rightarrow \fG_H$ is said to be \textit{unramified} (with respect to the chosen section $s^K$) if $\theta(s_{\tau}^K)=1\rtimes\tau$ for all $\tau\in\Gal(\Qpb/\Qpnr)$. Note that if $K$ is unramified and $\theta^{\Delta}:\mathbb{G}_{\mathrm{m},\Qpb}\rightarrow H_{\Qpb}$ is defined over $\Qpnr$, this definition does not depend on the choice of the section $s^K$.
A morphism $\theta:\fG_p\rightarrow \fG_H$ is then said to be \textit{unramified} if $\theta$ factors through $\fG_p^K$ for some finite Galois extension $K$ of $\Qp$ such that the induced map $\fG_p^K\rightarrow\fG_H$ is unramified in the just defined sense. 

For a connected reductive group $H$ over $\Qp$, we introduce the associated neutral $\Qpnr/\Qp$-Galois gerb by
\[\fG_H^{\nr}:=H(\Qpnr)\rtimes\Gal(\Qpnr/\Qp).\]

\begin{lem} \label{lem:unramified_morphism}
(1) For any morphism $\theta^{\nr}:\fD\rightarrow\fG_H^{\nr}$ of $\Qp^{\nr}/\Qp$-Galois gerbs, its inflation $\overline{\theta^{\nr}}:\fG_p\rightarrow \fG_H$ to $\Qpb$ is an unramified morphism.

(2) For every morphism $\theta:\fG_p\rightarrow \fG_H$ of $\Qpb/\Qp$-Galois gerbs, there is a morphism $\theta^{\nr}:\fD\rightarrow\fG_H^{\nr}$ of $\Qp^{\nr}/\Qp$-Galois gerbs whose inflation to $\Qpb$ is conjugate to $\theta$. More precisely, if $\theta$ factors through $\fG_p^K$ for a finite extension $K$ of $\Qp$, there is a morphism $\theta^{\nr}:\fD_n\rightarrow\fG_H^{\nr}$ with $n=[K:\Qp]$ whose inflation to $\Qpb$ is conjugate to $\theta$.

(3) A morphism $\theta^{\nr}$ in (2) is determined uniquely up to conjugation by an element of $H(\Qpnr)$.

(4) For every unramified morphism $\theta:\fG_p\rightarrow \fG_H$ of $\Qpb/\Qp$-Galois gerbs, if $b\in H(\Qpb)$ is defined by $\theta(s_{\widetilde{\sigma}})=b\rtimes \widetilde{\sigma}$ for an element $\widetilde{\sigma}\in \Gal(\Qpb/\Qp)$ lifting $\sigma$, we have that $b\in H(\Qpnr)$ and $b$ does not depend on the choice of $\widetilde{\sigma}$.
\end{lem}

\begin{proof} 
(1) Suppose that $\theta^{\nr}$ factors through $\fD_n$ so that $\overline{\theta^{\nr}}$ factors through $\fG_p^{L_n}$. The $\Qpb/\Qp$-Galois gerb $\fG_p^{L_n}$ is obtained from $\Qpnr/\Qp$-Galois gerb $\fD_n$ by pull-back along $\pi:\Gal(\Qpb/\Qp)\thra \Gal(\Qpnr/\Qp)$, followed by push-out along $\Gm(\Qpnr)\hra\Gm(\Qpb)$. 
To show that $\overline{\theta^{\nr}}$ is unramified, we may consider the morphism $\pi^{\ast}\theta^{\nr}: \pi^{\ast}\fD_n\rightarrow \pi^{\ast}\fG_H^{\nr}$ obtained by pull-back only, as the section to $\fG_p^{L_n}\thra \Gal(\Qpb/\Qp)$ induced, via inflation, from a section to $\fD_n\thra \Gal(\Qpnr/\Qp)$  lands in (the image in the push-out of) the pull-back $\pi^{\ast}\fD_n$. 
But, the pull-back $\pi^{\ast}\fD_n$ is also obtained as the pull-back of the $L_n/\Qp$-Galois gerb $\fG^{L_n}_{p,L_n}$ along the surjection $\Gal(\Qpb/\Qp)\thra \Gal(L_n/\Qp)$, followed by push-out along $\Gm(L_n)\hra\Gm(\Qpnr)$. Then, as the section $s^{L_n}:\Gal(\Qpb/\Qp)\rightarrow \fG_p^{L_n}$ is induced from a section $\Gal(L_n/\Qp)\rightarrow \fG^{L_n}_{p,L_n}$, we have that $s^{L_n}_{\tau}=1$ for all $\tau\in\Gal(\Qpb/L_n)$.
This proves the claim, since by definition $\pi^{\ast}\fD_n\subset \fD_n\times \Gal(\Qpb/\Qp)$ and the pull-back $\pi^{\ast}\theta^{\nr}$ is defined on the second factor $\Gal(\Qpb/\Qp)$ as the identity.

(2) This is Lemma 2.1 of \cite{LR87} (cf. first paragraph on p.167 of loc.cit). The second assertion is shown in the proof of loc.cit.

(3) In general, for any two unramified morphisms $\theta,\theta':\fG_p\rightarrow\fG_H$, if $\theta'=\mathrm{Int}(g_p)\circ \theta$ for some $g_p\in H(\Qpb)$, then it must be that $g_p\in H(\Qpnr)$, since for every $\tau\in\Gal(\Qpb/\Qpnr)$, 
\[1\rtimes\tau=\theta'(s_{\tau}^{K})=g_p\theta(s_{\tau}^{K})g_p^{-1}=g_p(1\rtimes\tau)g_p^{-1}=g_p\tau(g_p^{-1})\rtimes\tau.\]
Here, $K\subset \Qpb$ is some finite Galois extension of $\Qp$ for which both $\theta$ and $\theta'$ factor through $\fG_p^K$.

(4) Let $\theta^{\nr}:\fD\rightarrow\fG_H^{\nr}$ be a morphism of $\Qp^{\nr}/\Qp$-Galois gerbs whose inflation to $\Qpb$ is conjugate to $\theta$. By the proof of (3), we have that $\theta=g_p\overline{\theta^{\nr}}g_p^{-1}$ for some $g_p\in H(\Qpnr)$. So, if $\theta^{\nr}$ factors through $\fD^{L_n}$ for $n\in\N$, $\theta(s_{\widetilde{\sigma}}^{L_n})=g_p\overline{\theta^{\nr}}(s_{\widetilde{\sigma}}^{L_n})g_p^{-1}=g_p\theta^{\nr}(s_{\sigma}^{L_n})g_p^{-1}\in H(\Qpnr)$. Here, $s^{L_n}$ denotes both the section to $\fD^{L_n}\rightarrow \Gal(\Qpnr/\Qp)$ chosen before and the induced section to $\fG_p^{L_n}\rightarrow \Gal(\Qpb/\Qp)$. The second equality is easily seen to follow from the definition of the inflation $\overline{\theta^{\nr}}$ of a morphism $\theta^{\nr}$.
If $(g_p',{\theta^{\nr}}')$ is another pair with $\theta=g_p'\overline{{\theta^{\nr}}'}(g_p')^{-1}$, then we have that \[g_p\theta^{\nr}(s_{\sigma}^{L_n})g_p^{-1}=\theta(s_{\widetilde{\sigma}}^{L_n})=g_p'{\theta^{\nr}}'(s_{\sigma}^{L_n})(g_p')^{-1},\]
so $\theta(s_{\widetilde{\sigma}}^{L_n})=g_p\theta^{\nr}(s_{\sigma}^{L_n})g_p^{-1}$ is independent of the choice of $\widetilde{\sigma}$ as well as that of the pair $(g_p,\theta^{\nr})$.
\end{proof}

\begin{rem}
(1) For a morphism $\theta:\fG_p\rightarrow \fG_H$ of $\Qpb/\Qp$-Galois gerbs, if one chooses a morphism $\theta^{\nr}$ as in (2) and $\theta^{\nr}(s_{\sigma})=b\rtimes\sigma$ for $b\in H(\Qp^{\nr})$, then by (3) the $\sigma$-conjugacy class of $b$ in $H(L)$ is uniquely determined by $\theta$. Also, any other choice $s_{\sigma}'$ of the preimage $s_{\sigma}$ gives the same $\sigma$-conjugacy class, since $s_{\sigma}'=us_{\sigma}\sigma(u^{-1})$ for some $u\in \cO_L^{\times}$ (cf. \cite{LR87}, second paragraph on p.167).

(2) Suppose that $\theta$ is itself unramified, and let $b_1\in H(\Qpnr)$ be defined by $\theta(s_{\widetilde{\sigma}})=b_1\rtimes\widetilde{\sigma}$ for \emph{some} lift $\widetilde{\sigma}\in \Gal(\Qpb/\Qp)$ of $\sigma$. 
Also, let $b\in H(\Qp^{\nr})$ be defined as in (1) for some choice of $\theta^{\nr}:\fD\rightarrow\fG_H^{\nr}$ ($\Qp^{\nr}/\Qp$-Galois gerb morphism). Then (again by Lemma \ref{lem:unramified_morphism}, (3)) the $\sigma$-conjugacy classes of $b$ and $b_1$ are equal.

(3) In \cite[Remark B.12]{Reimann97}, Reimann uses some specific $s_{\sigma}\in \fD$, namely  there exists a unique $s_{\sigma}\in \fD$ whose image $s_{\sigma}^n$ in $\fD_n$, for every $n$, is $p^{-[1/n]}\in \Gm(\Qp^{\nr})\subset \fD_n$ (i.e. equals $p^{-1}$ if $n=1$, or otherwise is $1$),
\footnote{The definition of this element in \cite[(3.3.3)]{Kisin13} is wrong: he asserts that the image of $s_{\sigma}^n$ in $\fG_p^{\Q_{p^n}}$ is $p^{-1}\in \fG_p^{\Q_{p^n}\Delta}$, but it must be $p^{-[1/n]}\in \fG_p^{\Q_{p^n}\Delta}$.}
and one can check that indeed this element maps to $\sigma$ under $\fD\rightarrow\Gal(\Qpnr/\Qp)$. 
\end{rem}

Now, the map $\mathrm{cls}_H$ in question is $\theta\mapsto \overline{b(\theta)}\in B(H)$. Note that this map $\mathrm{cls}$ gives the same element in $B(H)$ for all morphisms $\fG_p\rightarrow\fG_H$ lying in a single equivalence class.

\begin{lem} \label{lem:Newton_hom_attached_to_unramified_morphism}
Let $H$ be a connected reductive group over $\Qp$ and $\theta:\fD\rightarrow \fG_H^{\nr}$ a morphism of $\Qpnr/\Qp$-Galois gerbs. Let $b\in H(\Qpnr)$ be defined by $\theta(s_{\widetilde{\sigma}})=b\rtimes \widetilde{\sigma}$ as in Lemma \ref{lem:unramified_morphism}, (4). Suppose that $\theta$ factors through $\fD^{L_n}$. Then,
the Newton homomorphism $\nu_{b}$ attached to $b\in H(\Qpnr)$ (in the sense of \cite{Kottwitz85}, $\S$4.3) is equal to the quasi-cocharcter
\[-\frac{1}{n}\theta^{\Delta}\quad \in\Hom_{L}(\mathbb{D},G),\]
where $\theta^{\Delta}$ denotes the restriction of $\theta$ to the kernel $\Gm$ of $\fD^{L_n}$.
\end{lem}
\begin{proof} See \textit{Anmerkung} on p.197 of \cite{LR87}.

\subsubsection{The Weil-number protorus and the pseudo-motivic Galois gerb} \label{subsubsec:pseudo-motivic_Galois_gerb}

In \cite{LR87}, Langlands and Rapoport work with two kinds of Galois gerbs, the quasi-motivic Galois gerb and the pseudo-motivic Galois gerb. The latter is the Galois gerb whose associated Tannakian category is supposed to be the Tannakian category of Grothendieck motives over $\Fpb$ (\cite{LR87}, $\S$4). The former's major role in loc.cit. is for formulation of the conjecture for the most general Shimura varieties (beyond those with simply-connected derived groups). Here, we will work mainly with the pseudo-motivic Galois gerb. According to \cite[Lemma B3.9]{Reimann97}, this is harmless, at least when the Serre condition for $(G,X)$ holds (i.e. $Z(G)$ splits over a CM field and the weight homomorphism $w_X$ is defined over $\Q$), e.g. if the Shimura datum is of Hodge-type.

Since this hypothesis $(\dagger)$ will be effective largely in most of the statements and our use of the quasi-motivic Galois gerb will be limited to formulation of some definitions, here we discuss the pseudo-motivc Galois gerb in detail and postpone the discussion of the quasi-motivic Galois gerb to the appendix.

The pseudo-motivic Galois gerb is a Galois gerb over $\Q$, which is also the projective limit of Galois gerbs $\fP(K,m)$ over $\Q$, indexed by CM fields $K\subset\Qb$ Galois over $\Q$ and $m\in\N$. The kernel $P(K,m)$ of $\fP(K,m)$ is a torus over $\Q$ whose character group consists of certain Weil numbers. Here, we give a brief review of their constructions. We begin with $P(K,m)$. 
As before, we fix embeddings $\Qb\hra\Qlb$, for every place $l$ of $\Q$. 

Recall that for a power $q$ of a rational prime $p$ and an integer $\nu\in\Z$, a \textit{Weil $q$-number of weight $\nu=\nu(\pi)$} is an algebraic number $\pi$ such that $\rho(\pi)\overline{\rho(\pi)}=q^{\nu}$ for every embedding $\rho:\Q(\pi)\hra\C$. When $K$ is a field containing $\pi$, then for every archimedean place $v$ of $K$, one has
\begin{equation} \label{eq:Weil-number_archimedean_condition}
|\pi|_v=|\prod_{\sigma\in\Gal(K_v/\Q_{\infty})}\sigma\pi|_{\infty}=q^{\frac{1}{2}[K_v:\R]\nu}.
\end{equation}
Here, $|x+\sqrt{-1}y|_v=x^2+y^2$ if $K_v=\C$, while if $K_v=\R=\Q_{\infty}$, $|x|_v$ is the usual absolute value $|x|_{\infty}$ on $\R$ (hence the first equality always holds for any $\pi\in K$).

\begin{defn} \label{defn:Weil-number_torus}
Let $K\subset\Qb$ be a CM-field which is finite, Galois over $\Q$ and $m\in\N$.

(1) The group $X(K,m)$ consists of the Weil $q=p^m$-numbers $\pi$ in $K$ (for some weight $\nu=\nu_1(\pi)$) with the following properties.
\begin{itemize}
\item[(a)] For each prime $v$ of $K$ above $p$, there is $\nu_2(\pi,v)\in\Z$ with
\[|\pi|_v=|\prod_{\sigma\in\Gal(K_v/\Qp)}\sigma\pi|_p=q^{\nu_2(\pi,v)}.\]
\item[(b)] At all finite places outside $p$, $\pi$ is a unit.
\end{itemize}

(2) Let $X^{\ast}(K,m)$ be the quotient of $X(K,m)$ (which is finitely generated by Dirichlet unit theorem) by the finite group of roots of unity contained therein (so that $X^{\ast}(K,m)$ is torsion free).  
Let $P(K,m)$ be the $\Q$-torus whose character group $X^{\ast}(P(K,m))$ is $X^{\ast}(K,m)$.
\end{defn}

The point of the condition (a), while the first equality is always true (for any $\pi\in K$), is that $|\pi|_v$ is an \emph{integral} power of $q$ (which however may well depend on $v$). One also has 
\[\nu_2(\pi,v)+\nu_2(\pi,\overline{v})=-[K_v:\Qp]\nu_1(\pi),\]
since $\pi\overline{\pi}=q^{\nu_1}$ ($K$ being a CM field, the complex conjugation $\overline{\cdot}$ of $K$ lies in the center of $\Gal(K/\Q)$).

If necessary, to avoid any misunderstandings, we write $\chi_{\pi}$ for the character of $P(K, m)$ which corresponds to a Weil number $\pi\in X(K, m)$. Recall that we fixed embeddings $\Qb\rightarrow\Qpb$, $\Qb\rightarrow\C$. Let $K\subset\Q$ be a Galois CM-field and $v_1$, $v_2$ the thereby determined archimedean and $p$-adic places of $K$, respectively. Then, one can readily see that there exist cocharacters $\nu_1^K$, $\nu_2^K$ in $X_{\ast}(K, m)=X_{\ast}(P(K, m))$ with following properties: 
\begin{eqnarray} \label{eqn:cocharacters_nu^K}
\langle\chi_{\pi},\nu_1^K\rangle&=&\nu_1(\pi), \\
\langle\chi_{\pi},\nu_2^K\rangle&=&\nu_2(\pi,v_2). \nonumber
\end{eqnarray}
A priori, $\nu_1^K$ and $\nu_2^K$ are defined over respectively $K_{v_1}=\C$ and $K_{v_2}\subset\Qpb$, but one readily sees from their definition that
they are defined over respectively $\Q$ and $\Qp$.
Furthermore, for $K\subset K'$ and $m|m'$ (divisible), there exist maps of tori over $\Q$
\[\phi_{K,K'}:P(K',m)\rightarrow P(K,m),\quad \phi_{m,m'}:P(K,m')\rightarrow P(K,m)\]
induced by $\phi_{K,K'}^{\ast}(\pi)=\pi$ and $\phi_{m,m'}^{\ast}(\pi)=\pi^{m'/m}$ for $\pi\in X^{\ast}(K,m)$, and they satisfy that $\phi_{m,m'}(\nu_i')=\nu_i$ and $\phi_{K,K'}(\nu_i')=[K'_{v_i'}:K_{v_i}]\nu_i$ (\cite[p.141]{LR87}). Let $P^K:=\varprojlim_{m|m'}P(K,m)$. This protorus over $\Q$ is in fact a torus (\cite{LR87}, Lemma 3.8), with character group $X^{\ast}(P^K)=\varinjlim X^{\ast}(K,m)$ and which splits over $K$. Let $\nu_1^K$,$\nu_2^K$ be the induced cocharacters of $P^K$. 

The triple $(P^K,\nu_1^K,\nu_2^K)$ is characterized by a universal property:

\begin{lem}  \label{lem:Reimann97-B2.3}
(1) For every CM-field $K\subset\Qb$ which is Galois over $\Q$, $(P^K,\nu_1^K,\nu_2^K)$ is an initial object in the category of all triples $(T,\nu_{\infty},\nu_p)$ where $T$ is a $\Q$-torus which splits over $K$, and, $\nu_{\infty}$ and $\nu_p$ are cocharacters of $T$ defined over $\Q$ and $\Qp$, respectively, and such that 
\[\Tr_{K/K_0}(\nu_p)+[K_{v_2}:\Q_p]\nu_{\infty}=0,\]
where $K_0$ is the totally real subfield of $K$ of index $2$. 

(2) There exists a set $\{\delta_n\}$ with $m|n$, $n$ sufficiently large, of distinguished elements in $P(K,m)(\Q)$ such that for every $\pi\in X^{\ast}(K,m)=X^{\ast}(P(K,m))$, 
\[\chi_{\pi}(\delta_n)=\pi^{\frac{n}{m}},\]
($\frac{n}{m}$ should be divisible by the torsion order of $X(K,m)$) and that, when $K\subset K'$ and $m|m'$ (divisible),
\[\phi_{m,m'}(\delta_n)=\delta_n,\quad \phi_{K,K'}(\delta_n)=\delta_n.\]
Moreover, the set $\{\delta_m^k\ |\ k\in\Z\}$ is Zariski-dense in $P(K,m)$.
\end{lem}

\begin{proof} 
For (1), see \cite[B2.3]{Reimann97}. For (2), if the subset $\{\pi_1,\cdots,\pi_r\} \subset X(K,m)$ forms a basis of $X^{\ast}(K,m)$ (up to torsions) with dual basis $\{\pi_1^{\vee},\cdots,\pi_r^{\vee}\} \subset X(K,m)^{\vee}=X_{\ast}(P(K,m))$, we set $\delta_n:=\sum \pi_i^{n/m}\otimes\pi_i^{\vee}\in \Gm(\Qb)\otimes X_{\ast}(P(K,m))=P(K,m)(\Qb)$, then it clearly satisfies the required properties, cf. \cite[p.142]{LR87}. The last property is Lemma 5.5 of \cite{LR87}; it is stated for a different torus $Q(K,m)$, but the proof carries over to $P(K,m)$. 
\end{proof} 

Set $P:=\varprojlim_K P^K$ (protorus). It is equipped with two morphisms $\nu_1:=\varprojlim_K\nu_1^K:\mathbb{G}_{\mathrm{m}}\rightarrow P$ (defined over $\Q$), $\nu_2:=\varprojlim_K\nu_2^K:\mathbb{D}\rightarrow P_{\Qp}$ (defined over $\Qp$). Often, $\nu_1$ and $\nu_2$ are also denoted by $\nu_{\infty}$ and $\nu_p$, respectively.

\begin{thm} \label{thm:pseudo-motivic_Galois_gerb}
(1) There exists a Galois gerb $\fP$ over $\Q$ together with morphisms $\zeta_v:\fG_v\rightarrow \fP(v)$ for all places $v$ of $\Q$ such that
\begin{itemize}
\item[(i)] $(\fP^{\Delta},\zeta_{\infty}^{\Delta},\zeta_{p}^{\Delta})=(P_{\Qb},(\nu_1)_{\C},(\nu_2)_{\Qpb})$, the identifications being compatible with the Galois actions of $\Gal(\overline{\Q}/\Q)$, $\Gal(\C/\R)$, and $\Gal(\Qpb/\Qp)$ respectively;
\item[(ii)] the morphisms $\zeta_v$, for all $v\neq \infty,p$, are induced by a section of $\fP$ over $\Spec(\overline{\A_f^p}\otimes_{\A_f^p}\overline{\A_f^p})$;
\end{itemize}
where $\overline{\A_f^p}$ denotes the image of the map $\Qb\otimes_{\Q}\A_f^p\rightarrow \prod_{l\neq\infty,p}\Qlb$. 

(2) If $(\fP',(\zeta_v'))$ is another such system, there exists an isomorphism $\alpha:\fP\rightarrow \fP'$ such that, for all $v$, $\zeta_v'$ is isomorphic to $\alpha\circ\zeta_v$, and any to $\alpha$'s arising in this way are isomorphic.

(3) There is a surjective morphism $\pi:\fQ\rightarrow\fP$ such that, for all $l$, $\zeta_l^P$ is algebraically equivalent to $\pi\circ\zeta_l^Q$ (\cite[Def. B1.1]{Reimann97}), where $\fQ$ is the quasi-motivic Galois gerb (cf. Appendix \ref{sec:quasi-motivic_Galois_gerb}).
\end{thm}

\begin{proof} 
(1) and (2): In \cite[$\S$3]{LR87}, Langlands and Rapoport first define, for each CM Galois field $K$, a Galois gerb $\fP^K$ with kernel $P^K$ which, for every place $v$ of $\Q$, is equipped with morphisms  $\zeta_v=\zeta_v^{K_w}:\fG_v^{K_w}\rightarrow \fP^K(v)$ whose restrictions to the kernels are $\nu_v^K$ for $v=\infty,p$, where $w$ is a place of $K$ above $v$. Then, they define $\fP$ as the projective limit of $\fP^K$'s; this requires choosing a place of $\Qb$ above each place $v$ of $\Q$. 
The construction of $\fP^K$ is a direct consequence of their Satz 2.2, but that of $\fP$ is more delicate: for example, for a projective system of algebraic tori $\{T_n\}_{n\in\N}$ over a field $F$, the natural map $H^2_{cts}(F,\varprojlim T_n)\rightarrow \varprojlim_n H^2(F,T_n)$ is not bijective in general (cf. \cite[Prop. 2.8]{Milne03}).
A proper treatment of the construction of $\fP$ can be found in \cite{Milne03}, 
(see also the proof of Theorem B 2.8 of \cite{Reimann97}, where Reimann constructs the quasi-motivic Galois gerb $\fQ$, but the whole arguments should carry over to $\fP$ too, since all the corresponding cohomological facts remain valid).
In more detail, for $v=p,\infty$, let $d_v^K$ be the image in $H^2(\Qv,P^K)$ of the fundamental class of the field extension $K_v/\Qv$ under the map $\nu_v^K$, where $K_v$ denotes (by abuse of notation) the completion of $K$ at the place induced by the embedding $\Qb\hra\Qvb$ (so, $K_{\infty}=\C$). Then, 
the Galois gerb $\fP^K$ corresponds to a cohomology class in $H^2(\Q,P^K)$ with image
\[(0,d_p^K,d_{\infty}^K)\in H^2(\A^{\{p,\infty\}},P^K)\times H^2(\Qp,P^K)\times H^2(\Q_{\infty},P^K).\]
The same statement holds for $\fP$, too (cf. \cite[$\S$4]{Milne03}). 
The work of Langland and Rapoport \cite[$\S$3]{LR87} and Milne \cite[(3.5b)]{Milne03} show that there exists a unique element in $H^2(\Q,P^K)$ with that property.
Then, by showing that the canonical maps 
\[H^2_{cts}(\Q,P)\rightarrow \varprojlim_K H^2(\Q,P^K),\quad H^1_{cts}(\Q,P)\rightarrow H^1_{cts}(\A,P)\] are isomorphisms (\cite{Milne03}, Prop. 3.6, Prop. 3.10), 
Milne concludes the existence of $\fP$ as required.  
The statement (3) is proved in \cite{Reimann97}, Theorem B 2.8. 
\end{proof}

\begin{rem} \label{rem:comments_on_zeta_v}
(1) As was remarked in the proof, to construct $\zeta_v$ (for a place $v$ of $\Q$), we need to choose a place a place $w$ of $K$ for each CM field $K$ Galois over $\Q$, in a compatible manner. From now on, when we talk about the pair $(\fP,(\zeta_v)_v)$, we will understand that such choice was already made. Clearly, it is enough to fix an embedding $\Qb\hra\Qvb$.

(2) As was also pointed out in the proof, for every CM field $K$ Galois over $\Q$ and each place $v$ of $\Q$, by construction, $\zeta_v$ induces a morphism $\fG_p^{K_w}\rightarrow \fP^K(v)$ of Galois $\Qv$-gerbs, where $w$ is the pre-chosen place of $K$ above $v$ (cf. \cite[Satz 2.2]{LR87}).

(3) The proof also establishes the existence of a (constinous) section to the projection $\fP\rightarrow\Gal(\Qb/\Q)$. We fix one and denote it by $\rho\mapsto q_{\rho}$.
\end{rem}

\subsection{The morphism $\psi_{T,\mu}$ and admissible morphisms}  

We consider the $\Q$-pro-torus $R:=\varprojlim_{L}\mathrm{Res}_{L/\\Q}(\mathbb{G}_{\mathrm{m},L})$ ($L$ running through the set of all Galois extensions of $\Q$ inside $\Qb$). Its character group $X^{\ast}(R)$ is naturally identified with the set of all continuous maps $f:\Gal(\Qb/\Q)\rightarrow\Z$, where the Galois action is given by $\rho(f)(\tau)=f(\rho^{-1}\tau),\ \forall\rho,\tau\in\Gal(\Qb/\Q)$.

\begin{lem} \label{lem:defn_of_psi_T,mu}
(1) For any $\Q$-torus $T$ and every cocharacter $\mu$ of $T$, there exists a unique homomorphism $\xi:R\rightarrow T$ such that $\mu=\xi\circ\mu_0$, where $\mu_0\in X_{\ast}(R)$ is defined by $\langle f, \mu_0\rangle =f(id)\in\Z$ for $f\in X^{\ast}(R)$. 

We define 
\[\psi_{T,\mu}:\fQ\rightarrow\fG_T\]
to be the composite of $\psi:\fQ\rightarrow\fG_R$ (cf. Theorem \ref{thm:Reimann97-B.2.8}) and the morphism $\fG_R\rightarrow \fG_T$ induced by $\xi:R\rightarrow T$.

(2) $\psi_{T,\mu}$ factors through $\fP$ if $\mu$ satisfies the Serre condition: if $K$ is the field of definition of $\mu$, 
\[(\rho-1)(\iota+1)\mu=(\iota+1)(\rho-1)\mu=0,\quad \forall\rho\in \Gal(K/\Q)\]
(e.g. if $T$ splits over a CM-field and the \textit{weight} $\mu\cdot\iota(\mu)$ of $\mu$ is defined over $\Q$).
If furthermore $K$ splits $T$, $\psi_{T,\mu}$ factors through $\fP^K$.

(3) The restriction $\psi_{T,\mu}^{\Delta}: \fQ^{\Delta}=Q_{\Qb}\rightarrow \fG_T^{\Delta}=T_{\Qb}$ of $\psi_{T,\mu}$ to the kernels is defined over $\Q$.
\end{lem}

\begin{proof} 
For (1) and (2), see \cite{Reimann97}, Definition B 2.10 and Remark B 2.11. The last statement of (2) follows from the very construction of $\psi_{T,\mu}$ in Satz 2.2, 2.3 of \cite{LR87} (which is equivalent to that of Reimann, \cite{Reimann97}, at least when it factors through $\fP$).

For (3), it is enough to show that the morphism $\psi:\fQ\rightarrow \fG_R$ is defined over $\Q$. But, this restriction is the homomorphism  $\psi^{\Delta}$ in (\ref{eq:morphism_from_Q_to_R}) (Theorem \ref{thm:Reimann97-B.2.8}), from which the claim is obvious. 
\end{proof}

Let $v$ a place of $\Q$ (mainly, one of $p,\infty$), $T$ a torus over $\Qv$, and $\mu\in X_{\ast}(T)$. 
Suppose that $T$ splits over a finite Galois extension $F$ of $\Qv$. Set 
\[\nu^F:=\sum_{\tau\in\Gal(F/\Q_{v})}\tau\mu,\]and let 
\[1\rightarrow F^{\times}\rightarrow W_{F/\Q_{v}}\rightarrow\Gal(F/\Q_{v})\rightarrow 1\] be the Weil group extension of $F/\Q_{v}$ (cf. \cite{Tate79}); we fix a section $s^{F}_{\rho}$ to the projection $W_{F/\Q_{v}}\rightarrow \Gal(F/\Q_{v})$ so that $d^F_{\rho,\tau}:=s_{\rho}\rho(s_{\tau})s_{\rho\tau}^{-1}$ is a cocycle defining $W_{F/\Q_{v}}$.

\begin{defn}  \label{defn:psi_T,mu}
We define $\xi_{\mu,F}^F:W_{F/\Q_{v}}\rightarrow T(F)\rtimes\Gal(F/\Q_{v})$ by
\begin{eqnarray*}
\xi_{\mu,F}^F(z)&=&\nu^F(z)\quad (z\in F^{\times}),\\
\xi_{\mu,F}^F(s^F_{\rho})&=&\prod_{\tau\in\Gal(F/\Q_{v})}\rho\tau\mu(d^F_{\rho,\tau})\rtimes\rho.
\end{eqnarray*}
\end{defn}

One easily checks that $\xi_{\mu,F}^F$ is a homomorphism (cf. \cite[p.134]{LR87}, \cite{Milne92}, Lemma 3.30 - Example 3.32). By obvious pull-back and push-out, one gets a morphism of Galois gerbs over $\Q_v$: 
\[\xi_{\mu}^F:\fG_v^{F}\rightarrow\fG_T,\] 
(where for $v\neq p$, we set $\fG_v^{F}$ to be $\fG_v=\Gal(\Qvb/\Qv)$) and further, by passing to the projective limit, a morphism of Galois gerbs over $\Q_v$:
\[\xi_{\mu}:\fG_v\rightarrow\fG_T,\]
which does not depend on the choice of a field $F$ splitting $T$.
These maps are independent, up to conjugation by an element of $T(\Qvb)$, of the choice of section $s_{\rho}$.

\begin{lem} \label{lem:properties_of_psi_T,mu}
(1) If $v=p$ and $F$ is unramified, $\xi_{\mu}$ is unramified (in the sense of Subsubsection \ref{subsubsec:cls}), and if $\xi_{\mu}(s_{\sigma})=b_{\mu}\rtimes \sigma$ for $b_{\mu}\in T(L)$, one has $\overline{b_{\mu}}=\overline{\mu(p^{-1})}$ in $B(T_{\Qp})$. 

(2) Suppose that $T$ is a torus defined over $\Q$, split over a finite Galois extension $K$ of $\Q$. For each $v=\infty,p$, let $\xi_{\pm\mu}$ be the morphism defined above for $(T_{\Qv},F=K_w,\pm\mu)$.
Then, $\psi_{T,\mu}(\infty)\circ \zeta_{\infty}$ is conjugate to $\xi_{\mu}$, and $\psi_{T,\mu}(p)\circ \zeta_{p}$ is conjugate to $\xi_{-\mu}$. For $v\neq \infty,p$, $\psi_{T,\mu}(\infty)\circ \zeta_{v}$ is conjugate to the canonical neutralization of $\fG_T(v)$. 
\end{lem}

\begin{proof}
(1) See Lemma 4.3 of \cite{Milne92}. (2) This follows from the construction of $\psi_{T,\mu}$ and $\fP(K,m)$, cf. \cite{LR87}, Satz 2.3 and $\S$3 (esp. (3.i)). 
\end{proof}

\subsubsection{Shimura data} Let $(G,X)$ be a Shimura datum. 
For a morphism $h:\dS\rightarrow G_{\R}$ in $X$, the associated \emph{Hodge cocharacter}
\[\mu_{h}:\mathbb{G}_{\mathrm{m}\C}\rightarrow G_{\C}\]
is the composite of $h_{\C}:\dS_{\C}\rightarrow G_{\C}$ and the cocharacter of $\dS_{\C}\cong\mathbb{G}_{\mathrm{m}\C}\times \mathbb{G}_{\mathrm{m}\C}$ corresponding to the identity embedding $\C\hookrightarrow\C$. Let $\{\mu_X\}$ denote the $G(\C)$-conjugacy class of cocharacters of $G_{\C}$ containing $\mu_h$ (for any $h\in X$). For a maximal torus $T$ of $G_{\Qb}$, we can consider $\{\mu_X\}$ as an element of $X_{\ast}(T)/W$. Alternatively, when we fix a based root datum $\mathcal{BR}(G,T,B)$, $c(G,X)$ has a unique representative in the associated closed Weyl chamber $\overline{C}(T,B)$, hence will be also identified with this representative: 
\begin{equation*} \label{eq:representataive_Hodge_cocharacter}
\{\mu_X\}\in \overline{C}(T,B).
\end{equation*}
The \emph{reflex field} $E(G,X)$ of a Shimura datum $(G,X)$ is the field of definition of $c(G,X)\in\mathcal{C}_G(\Qb)$, i.e. the fixed field of the stabilizer of $c(G,X)$ in $\Gal(\Qb/\Q)$; so a reflex field, which is a finite extension of $\Q$, is always a subfield of $\C$. When $T$ is a torus, the reflex field $E(T,\{h\})$ is just the smallest subfield of $\Qb\subset\C$ over which the single
morphism $\mu_h$ is defined.

For each $j\in\N$, we denote by $L_j$ the unramified extension of degree $j$ of $\Qp$ in $\Qpb$, and by $\Qpnr$ the maximal unramified extension of $\Qp$ in $\Qpb$. We let $L$ and $\sigma$ denote the completion of $\Qpnr$ and the absolute Frobenius on it, respectively.

\subsubsection{Strictly monoidal categories $G/\widetilde{G}(\kb)$, $\fG_{G/\widetilde{G}}$}

In order to have a satisfactory formalism without the simply-connected derived group condition, Kisin \cite[(3.2)]{Kisin13} introduced certain strictly monoidal categories. Recall (cf. \cite[App. B]{Milne92}) that a \textit{crossed module} is a group homomorphism $\alpha:\tilde{H}\rightarrow H$ together with an action of $H$ on $\tilde{H}$, denoted by ${}^h\tilde{h}$ for $h\in H$, $\tilde{h}\in \tilde{H}$, which lifts the conjugation action on itself (i.e. 
$\alpha({}^h\tilde{h})=h\alpha(\tilde{h})h^{-1}$ for $h\in H$, $\tilde{h}\in \tilde{H}$) and such that the induced action of $\tilde{H}$ on itself is also the conjugation action (i.e. ${}^{\alpha(\tilde{g})}\tilde{h}=\tilde{g}\tilde{h}\tilde{g}^{-1}$ for $\tilde{g},\tilde{h}\in \tilde{H}$).  A crossed module $\tilde{H}\rightarrow H$ gives rise to a strictly monoidal category $H/\tilde{H}$. Its underlying category is the groupoid whose objects are the elements of $H$ and whose morphisms are given by $\Hom(h_1,h_2)=\{\tilde{h}\in\tilde{H}\ |\ h_2=\alpha(\tilde{h})h_1\}$; thus the set of morphisms is identified with the set $\tilde{H}\times H$.
The monoidal structure $\otimes$ on this groupoid is given on the objects by the group multiplication on $H$ and on the set of morphisms $\tilde{H}\times H$ by the semi-direct product for the action of $H$ on $\tilde{H}$:
\[(\tilde{h}_1,h_1)\otimes (\tilde{h}_2,h_2):=(\tilde{h}_1{}^{h_1}\tilde{h}_2,h_1h_2).\]

We may regard any group $H$ as the strictly monoidal category $H=H/\{1\}$.

Let $k$ be a field with an algebraic closure $\kb$, and $G$ a connected reductive group over $k$. Here, we will use the notation $\tilde{G}$ for the simply connected cover of $G^{\der}$ (which was denoted previously by $G^{\uc}$). Then, the commutator map $[\ ,\ ]:G\times G\rightarrow G$ factors through $[\ ,\ ]:G^{\ad}\times G^{\ad}\rightarrow G$. In particular, as $G^{\ad}=\tilde{G}^{\ad}$, we get a map $[\ ,\ ]:G^{\ad}\times G^{\ad}\rightarrow \tilde{G}$ (\cite[2.0.2]{Deligne77}).
It follows that the conjugation action of $\tilde{G}$ on itself extends to an action of $G$, and thus
the natural map $\tilde{G}\rightarrow G$ has a canonical crossed module structure.
We write $G/\widetilde{G}(\kb)$ for the resulting strictly monoidal category $G(\kb)/\widetilde{G}(\kb)$, and 
$\fG_{G/\widetilde{G}}$ for the strictly monoidal category $\fG_G/\tilde{G}(\kb)$.

\subsubsection{Admissible morphisms}
Let $(G,X)$ be a Shimura datum with reflex field $E\subset\C$. We fix an embedding $\Qb\hra\Qvb$ for every  place $v$. Suppose given a parahoric subgroup $\mbfK_p\subset G(\Qp)$; there exists a unique $\sigma$-stable parahoric subgroup $\mbfKt_p$ of $G(L)$ such that $\mbfK_p=\mbfKt_p\cap G(\Qp)$.

Fix $h\in X$. Then, there exists a homomorphism of $\C/\R$-Galois gerbs 
\[\xi_{\infty}:\fG_{\infty}\rightarrow\fG_G(\infty)\]
defined by $\xi_{\infty}(z)=w_h(z)=\mu_h\cdot\overline{\mu_h}(z),\ z\in\C^{\times}$ and $\xi_{\infty}(w)=\mu_h(-1)\rtimes \iota$, where $w=w(\iota)$. Clearly, its equivalence class depends only on $X$.

For $v\neq\infty, p$, we have the canonical section $\xi_v$ to $\fG_G(v)\rightarrow\Gal(\overline{\Q}_v/\Q_v)$:
\[\xi_v:\fG_v=\Gal(\overline{\Q}_v/\Q_v) \rightarrow\fG_G(v)\ :\ \rho\mapsto 1\rtimes\rho.\] 

For a cocharacter $\mu$ of $G$, we consider the composite of morphisms of strictly monoidal categories \[\mu_{\tilde{\ab}}:\Gm(\Qb)\stackrel{\mu}{\rightarrow} G(\Qb) \rightarrow G/\tilde{G}(\Qb).\]

For a cocharacter $\mu$ of $G$, the composite
\[\psi_{\mu_{\widetilde{\ab}}}:\fQ \stackrel{i\circ\psi_{\mu}}{\rightarrow}\fG_G\rightarrow \fG_G/\tilde{G}(\kb).\]
(of morphisms of strictly monoidal categories) depends only on the $G(\Qb)$ conjugacy class of $\mu$; 
One easily verifies that this is equal to the morphism denoted by the same symbol in \cite[(3.3.1)]{Kisin13}.

\begin{defn} \label{defn:admissible_morphism} \cite[p.166-168]{LR87}
A morphism $\phi:\fQ\rightarrow\fG_G$ is called \textit{admissible} if
\begin{itemize}
\item[(1)] The composite
\[\phi_{\widetilde{\ab}}:\fQ\stackrel{\phi}{\rightarrow}\fG_G\rightarrow \fG_G/\tilde{G}(\kb)\] 
is conjugate to the composite $\Psi_{\mu_{\widetilde{\ab}}}:\fQ \stackrel{i\circ\psi_{\mu}}{\rightarrow}\fG_G\rightarrow \fG_G/\tilde{G}(\kb)$.
\item[(2)] For every place $v\neq p$ (including $\infty$), the composite $\phi\circ\zeta_v$ is conjugate to $\xi_v$ (by an element of $G(\Qlb)$).
\item[(3)] For some (equiv. any) $b\in G(L)$ in the $\sigma$-conjugacy class $\mathrm{cls}(\phi\circ\zeta_p)\in B(G)$ (\ref{subsubsec:cls}),
the following set (which is a union of affine Deligne-Lusztig varieties) $X(\{\mu_X\},b)_{\mbfK_p}$ is non-empty:
\[X(\{\mu_X\},b)_{\mbfK_p}:=\{g\in G(L)/\mbfKt_p\ |\ \mathrm{inv}_{\mbfKt_p}(g,b\sigma(g))\in\Adm_{\mbfKt_p}(\{\mu_X\})\}.\]
Here, $\Adm_{\mbfKt_p}(\{\mu_X\})$ is the $\{\mu_X\}$-admissible subset (Def. \ref{defn:mu-admissible_subset}) defined for the parahoric subgroup $\mbfKt_p\subset G(L)$ attached to $\mbfK_p$, and
\[\mathrm{inv}_{\mbfKt_p}: G(L)/\mbfKt_p \times G(L)/\mbfKt_p \rightarrow \mbfKt_p\backslash G(L)/\mbfKt_p \cong \tilde{W}_{\mbfKt_p}\backslash \tilde{W}/ \tilde{W}_{\mbfKt_p}\] 
is defined by $(g_1\mbfKt_p,g_2\mbfKt_p)\mapsto \mbfKt_pg_1^{-1}g_2\mbfKt_p$,
cf. (\ref{eqn:parahoric_double_coset}).
\end{itemize}
\end{defn}

Suppose that an admissible morphism $\phi:\fQ\rightarrow\fG_G$ factors through the pseudo-motivic Galois gerb $\fP$. 
As $G$ is an algebraic group, it further factors through $\fP(K,m)$ for some CM field $K$ Galois over $\Q$ and $m\in\N$.

\begin{rem} \label{rem:Kisin's_defn_of_admissible_morphism}
This definition of admissible morphism is slightly different from the original definition by Langlands-Rapoport \cite[p.166]{LR87}. Instead of the condition (1) here, which was introduced by Kisin \cite[(3.3.6)]{Kisin13}, they require the equality $pr\circ\phi=pr\circ i\circ\psi_{T,\mu}$, where $pr:\fG_G\rightarrow \fG_{G^{\ab}}$ is the natural map; so, these conditions are identical if $G^{\der}$ is simply connected. This original condition however turns out to be adequate only in their set-up, i.e. when $G^{\der}$ is simply connected. For example, Satz 5.3 of \cite{LR87} shows that for such $G$, every admissible morphism (in the original sense) is conjugate to a special admissible morphism (in the sense of Lemma \ref{lem:LR-Lemma5.2} below). But, they also show (\cite{LR87}, $\S$6, the first example) that this fact is not true in general with the simply connected condition, cf. \cite{Milne92}, Remark 4.20. In contrast, with the new condition, it turns out to always hold (as shown by Kisin in the hyperspecial level case, and by Thm. \ref{thm:LR-Satz5.3} below for general parahoric levels when $G_{\Qp}$ is quasi-split).
\end{rem}

The following lemma was proved by Langlands-Rapoport for unramified $T$ (cf. \cite[4.3]{Milne92}).

\begin{lem} \label{lem:unramified_conj_of_special_morphism}
Let $T$ be a torus over $\Qp$, split by a finite Galois extension $K$ of $\Qp$, say of degree $n$, and $\mu\in X_{\ast}(T)$. Let $K_1$ be the composite in $\Qpb$ of $K$ and $L_n$ (the unramified extension of $\Qp$ in $\Qpb$ of degree $n=[K:\Qp]$). Then, $\xi_{\mu}^{K_1}$ factors through $\fG^{L_n}_p$: let $\xi_{\mu}^{L_n}:\fG^{L_n}_p\rightarrow \fG_{T}$ be the resulting morphism.
When $\xi_p'$ is an unramified conjugate of $\xi_{\mu}^{L_n}$, 
\[\xi_p'(s_{\rho}^{L_n})=\Nm_{K/K_0}(\mu(\pi^{-1}))\rtimes \rho,\] 
up to conjugation by an element of $T(\Qpnr)$.
Here, $K_0=K\cap L_n\subset K$ is the maximal unramified subextension of $\Qp$, $\pi$ is a uniformizer of $K$, and $\rho$ is any element in $\Gal(\Qpb/\Qp)$ whose restriction to $L_n$ is the Frobenius automorphism $\sigma$.
Moreover, we have the equality in $X_{\ast}(T)_{\Gamma(p)}$:
\[\kappa_{T}(\Nm_{K/K_0}(\mu(\pi^{-1})))=-\underline{\mu},\] 
where $\underline{\mu}$ is the image of $\mu\in X_{\ast}(T)$ in $X_{\ast}(T)_{\Gamma(p)}$ ($\Gamma(p)=\Gal(\Qpb/\Qp)$).
\end{lem}

\begin{proof} 
Let $d^{L_n}_{\rho,\tau}$ denote the canonical fundamental $2$-cocycle defined in (\ref{eq:canonical_fundamental_cocycle}) which represents the fundamental class $u_{L_n/\Qp}=[1/n]\in H^2(L_n/\Qp)\cong \frac{1}{n}\Z/\Z$. Also, for $F=K_1$ and $L_n$, fix a section $s_{\rho}^{F}:\Gal(F/\Qp)\rightarrow W_{F/\Qp}$ to $1\rightarrow F^{\times}\rightarrow W_{F/\Qp}\rightarrow \Gal(F/\Qp)\rightarrow 1$ which induces a $2$-cocycle on $\Gal(F/\Qp)$
\[d^{F}_{\rho,\tau}:=s^{F}_{\rho}\rho(s^{F}_{\tau})(s^{F}_{\rho\tau})^{-1}\in F^{\times}\] 
representing the fundamental class $u_{F/\Qp}\in H^2(F/\Qp)\cong \frac{1}{[F:\Qp]}\Z/\Z$ (for $F=L_n$, the induced $2$-cocycle is required to be the canonical one).
Thus there exists a function $b:\Gal(K_1/\Qp)\rightarrow K_1^{\times}$ such that
\begin{equation} \label{eqn:inflations_of_two_cocycles}
(d^{K_1}_{\rho,\tau})^{[K_1:L_n]}\cdot\partial(b)_{\rho,\tau}=d^{L_n}_{\rho|_{L_n},\tau|_{L_n}},
\end{equation}
where $\partial(b)_{\rho,\tau}:=b_{\rho}\rho(b_{\tau})b_{\rho\tau}^{-1}$.
In terms of these generators, the natural map $p_{L_n,K_1}:\fG^{K_1}_{\Qp}\rightarrow \fG^{L_n}_{\Qp}$ is defined by
\[z\mapsto z^{[K_1:K]}\ (z\in\Qpb^{\times}),\quad s^{K_1}_{\rho}\mapsto \ b_{\rho}^{-1}s^{L_n}_{\rho}.\]
Then, the morphisms $\xi_{\mu}^{K_1}, \xi_{\mu}^{L_n}\circ p_{L_n,K_1}:\fG^{K_1}_{\Qp}\rightarrow\fG_T$ differ from each other by conjugation with an element of $T(\Qpb)$. 

Recall that for $(T,\mu,K_1)$, $\xi_{\mu}^{K_1}:\fG_p\rightarrow \fG_T$ is induced, via obvious pull-back and push-out, from a map $\xi_{\mu,K_1}^{K_1}:W_{K_1/\Qp}\rightarrow T(K_1)\rtimes \Gal(K_1/\Qp)$: for $a\in K_1^{\times}$ and $\rho\in \Gal(K_1/\Qp)$, \[\xi_{\mu,K_1}^{K_1}:a\cdot s_{\rho}^{K_1} \mapsto \nu^{K_1}(a)\cdot c^{K_1}_{\rho}\rtimes\rho,\] where $\nu^{K_1}=\Nm_{K_1/\Qp}\mu\in \Hom_{\Qp}(\Gm,T)$ and $c^{K_1}_{\rho}=\prod_{\tau_1\in\Gal(K_1/\Qp)}(\rho\tau_1\mu)(d^{K_1}_{\rho,\tau_1})$.
Now, for any $x\in T(\Qpb)$, if we define $\psi'_x:\fG^{L_n}_p\rightarrow \fG_T$ by 
\[z\mapsto \nu^K(z),\quad s^{L_n}_{\rho}\mapsto \nu^K(b_{\rho})\cdot c^{K_1}_{\rho}\cdot x\cdot\rho(x^{-1}),\]
where $\nu^K=\Nm_{K/\Qp}\mu$, then it is clear that $\psi'_x\circ\pi_{K_1,L_n}=\Int (x)\circ \xi^{K_1}_{\mu}$. This proves the first claim. Since $\psi'_x=\Int x\circ\psi'_1$, the second statement will follow if there exists $x\in T(\Qpb)$ such that $\psi'_x(s^{L_n}_{\rho})$ equals $\Nm_{K/K_0}(\mu(\pi^{-1}))$ whenever $\rho|_{L_n}=\sigma$. This means that the two elements $\nu^K(b_{\rho})\cdot c^{K_1}_{\rho}$, $\Nm_{K/K_0}(\mu(\pi^{-1}))$ of $T(L')$ have the same image under $\kappa_{T_K}:B(T_K)\isom X_{\ast}(T)_{\Gal(\Qpb/K)}$, where $L'$ is the completion of the maximal unramified extension of $K$ in $\Qpb$ and $B(T_K)$ is the set of $\sigma$-conjugacy classes in $T(L')$ (with respect to the Frobenius automorphism of $L'/K$).
But, as $\kappa_{T_K}$ is induced from $w_{T_{L'}}:T(L')\rightarrow X_{\ast}(T)_{\Gal(\overline{L'}/L')}=X_{\ast}(T)$ (Subsec. \ref{subsubsec:Kottwitz_hom}), in turn it suffices to show the equality of the images under $w_{T_{L'}}$ of $c^{K_1}_{\rho}\cdot\nu^{K}(b_{\rho})$ and $\Nm_{K/K_0}(\mu(\pi^{-1}))$ when $\rho|_{\Qpnr}=\sigma$.

Choose a set of representatives $\Gamma_1\subset\Gal(K_1/\Qp)$ for the family of left cosets $\Gal(K_1/\Qp)/\Gal(K_1/L_n)$ (so that restriction to $L_n$ gives a bijection $\Gamma_1\isom\Gal(L_n/\Qp)$) and $\rho\in\Gal(K_1/\Qp)$ such that $\rho|_{L_n}=\sigma$. Then, we get
\begin{eqnarray*}
\prod_{\tau_1\in\Gal(K_1/\Qp)}(\rho\tau_1\mu)(\mathrm{Inf}_{L_n}^{K_1}(d^{L_n/\Qp})_{\rho,\tau_1})&=&\prod_{\tau\in\Gamma_1}\prod_{\gamma\in\Gal(K_1/L_n)}(\rho\tau\gamma\mu)(\mathrm{Inf}_{L_n}^{K_1}(d^{L_n/\Qp})_{\rho,\tau\gamma})\\
&=&\rho\prod_{\tau\in\Gal(L_n/\Qp)}(\tau(\Nm_{K_1/L_n}\mu))(d^{L_n}_{\rho|_{L_n},\tau}) \\
&=&\prod_{0\leq i\leq n-1}(\sigma^{i+1}(\Nm_{K_1/L_n}\mu))(d^{L_n}_{\sigma,\sigma^i}) \\
&=&(\Nm_{K_1/L_n}\mu)(p^{-1})=(\Nm_{K/K_0}\mu)(p^{-1}).
\end{eqnarray*}
Here, the last equality $\Nm_{K_1/L_n}\mu=\Nm_{K/K_0}\mu$ (in $X_{\ast}(T)$) holds since $\mu$ is defined over $K$ and restriction to $K$ is a bijection $\Gal(K_1/L_n)\isom\Gal(K/K_0)$.
Then, by taking $\prod_{\tau\in\Gal(K_1/\Qp)}(\rho\tau\mu)$ on both sides of (\ref{eqn:inflations_of_two_cocycles}), we obtain 
\[(c^{K_1}_{\rho}\cdot\nu^{K}(b_{\rho}))^{[K_1:K]}\cdot\rho(f)f^{-1}=(\Nm_{K/K_0}\mu)(p^{-1}),\]
where $f=\prod_{\tau\in\Gal(K_1/\Qp)}\tau\mu(e_{\tau})$. 
Now applying $w_{T_{L'}}$ to both sides, we get 
\begin{eqnarray*}
[K_1:K] w_{T_{L'}}(c^{K_1}_{\rho}\cdot\nu^{K}(b_{\rho}))&=& w_{T_{L'}}((\Nm_{K/K_0}\mu)(p^{-1}))\\
&\stackrel{(\ast)}{=}&[K:K_0] w_{T_{L'}}(\Nm_{K/K_0}(\mu(\pi^{-1}))).
\end{eqnarray*}
Due to the property \cite[(7.3.2)]{Kottwitz97} of the map $w$, the equality $(\ast)$ is deduced from the following stronger formula (comparing the images under $w_{T_L}$, instead of $w_{T_{L'}}$):
\begin{equation} \label{eqn:comparison_of_two_norms}
 w_{T_L}(\Nm_{K/K_0}(\mu)(p))=[K:K_0] w_{T_L}(\Nm_{K/K_0}(\mu(\pi))).
\end{equation}
Here, $\Nm_{K/K_0}(\mu)\in X_{\ast}(T)^{\Gal(K/K_0)}$ so $\Nm_{K/K_0}(\mu)(p)\in \mathrm{Im}(K_0^{\times}\rightarrow T(K_0))$, while $\Nm_{K/K_0}(\mu(\pi))$ is the image of $\mu(\pi)\in T(K)$ under the norm map $T(K)\rightarrow T(K_0)$. 
To show this formula, by functoriality for tori $T$ endowed with a cocharacter $\mu$, it is enough to prove this formula in the universal case $T=\Res_{K/\Qp}\Gm$ and $\mu=\mu_K$, the cocharacter of $T_K=(\Gm)^{\oplus \Hom(K,K)}$ corresponding to the identity embedding $K\hra K$. Note that in this case $w_{T_L}=v_{T_L}$ as $X_{\ast}(T)$ is an induced $\Gal(K/\Qp)$-module (Subsec. \ref{subsubsec:Kottwitz_hom}).
For any extension $E\supset K$, Galois over $\Qp$, there exists a canonical isomorphism
$T_{E}\cong (\mathbb{G}_{\mathrm{m},E})^{\oplus\Hom(K,E)}$ (product of copies of $\mathbb{G}_{\mathrm{m},E}$, indexed by $\Hom(K,E)$) such that $\tau\in\Gal(E/\Qp)$ acts on $T(E)=(E^{\times})^{\oplus\Hom(K,E)}$ by
 $\tau (x_{\rho})_{\rho\in \Hom(K,E)}=(\tau (x_{\rho}))_{\tau\circ\rho}$. Then, $\mu_K=(f_{\rho})_{\rho}\in\prod_{\rho\in\Hom(K,K)}X_{\ast}(\Gm)$, where $f_{\rho}=1\in X_{\ast}(\Gm)=\Z$ if $\rho$ is the inclusion $K\hra E$, and $f_{\rho}=0$ otherwise. So, $\Nm_{K/K_0}(\mu_K)$ is $(f_{\rho})_{\rho}\in\prod_{\rho}X_{\ast}(\Gm)$, where $f_{\rho}=1$ if $\rho|_{K_0}$ is the inclusion $K_0\hra K$, and $f_{\rho}=0$ otherwise, and similarly
$\Nm_{K/K_0}(\mu_K(\pi))=(x_{\rho})_{\rho}$, where $x_{\rho}=\rho(\pi)$ if $\rho|_{K_0}=(K_0\hra K)$, and $x_{\rho}=1$ otherwise.
It follows that the element of $T(K_0)$
\[\Nm_{K/K_0}(\mu_K)(p)\cdot \Nm_{K/K_0}(\mu_K(\pi))^{-[K:K_0]},\]
lies in the maximal compact subgroup of $T(K_0)$,  which is nothing but $\Ker(v_{T_L})\cap T(K_0)$. This proves the equation (\ref{eqn:comparison_of_two_norms}).

Finally, the equality $\kappa_{T}(\Nm_{K/K_0}(\mu(\pi^{-1})))=-\underline{\mu}= $ in $X_{\ast}(T)_{\Gamma(p)}$ is proved in \cite[2.5]{Kottwitz85} (be careful that here we are considering $\kappa_{T}$, not $\kappa_{T_K}$). Or, as in the proof of the equation (\ref{eqn:comparison_of_two_norms}) above, we reduce the proof to the universal case $(T,\mu)=(\Res_{K/\Qp}\Gm,\mu_K)$, in which case the claim becomes a consequence of (\ref{eqn:comparison_of_two_norms}). 
This completes the proof. 
\end{proof}

\begin{lem} \label{lem:LR-Lemma5.2}
Suppose that $G_{\Qp}$ is quasi-split and $\mbfK_p$ is a special maximal parahoric subgroup. Then,
for any special Shimura datum $(T,h:\dS\rightarrow T_{\R})$ satisfying the Serre condition (e.g. if $T$ splits over a CM field and the weight homomorphism $\mu_h\cdot\iota(\mu_h)$ is defined over $\Q$), the morphism $i\circ\psi_{T,\mu_h}:\fP\rightarrow \fG_T\hra\fG_G$ (Lemma \ref{lem:defn_of_psi_T,mu}) is admissible, where $i:\fG_T\rightarrow\fG_G$ is the canonical morphism defined by the inclusion $i:T\hra G$.
\end{lem}

Such admissible morphism $i\circ\psi_{T,\mu_h}$ will be said to be \textit{special}; in our use of this notation, $i$ will be often spared its explanation (or sometimes will be even omitted). This fact was proved in \cite[Lemma 5.2]{LR87} for hyperspecial level subgroups.

\begin{proof} 
The only nontrivial condition in Def. \ref{defn:admissible_morphism} is (3). Let $L$ be a finite Galois extension of $\Q$ splitting $T$ and $v_2$ the place of $L$ induced by the chosen embedding $\Qb\hra\Qpb$.
Put 
\[\nu_p:=(\xi_{-\mu_h}^{L_{v_2}})^{\Delta}=-\sum_{\sigma\in\Gal(L_{v_2}/\Qp)}\sigma\mu_h\quad (\in\Hom_{\Qp}(\Gm,T_{\Qp})),\]
and let $J$ be the centralizer in $G_{\Qp}$ of the image of $\nu_p$. Then, $J$ is a $\Qp$-Levi subgroup of $G_{\Qp}$ which is also quasi-split as $G_{\Qp}$ is so (Lemma \ref{lem:specaial_parahoric_in_Levi}, (1)). Hence, according to Lemma \ref{lem:specaial_parahoric_in_Levi}, there exists $g\in G(\Qp)$ such that $gJ(L)g^{-1}\cap \mbfKt_p$ is a special maximal parahoric subgroup of $gJ(L)g^{-1}$, where $\mbfKt_p\subset G(L)$ is the special maximal parahoric subgroup associated with $\mbfK_p$: it is enough that for a maximal split $\Qp$-torus $S$ of $G_{\Qp}$ contained in $J$, the apartment $\mcA(\Int g(S),\Qp)$ contains a special point in $\mcB(G,\Qp)$ giving $\mbfK_p$. Set $J':=\Int g(J)$.
Then, by Prop. \ref{prop:existence_of_elliptic_tori_in_special_parahorics} (cf. Remark \ref{rem:properties_of_certain_elliptic_tori_in_special_parahorics}), there exists an elliptic maximal $\Qp$-torus $T'$ of $J'$ such that $T'_{\Qpnr}$ contains (equiv. is the centralizer of) a maximal $\Qpnr$-split $\Qpnr$-torus, say $S'_1$, of $J'_{\Qpnr}$ and that the (unique) parahoric subgroup $T'(L)_1=\Ker\ w_{T'_L}$ of $T'(L)$ is contained in $J'(L)\cap \mbfKt_p$. Let $\mu'$ be the cocharacter of $T'$ that is conjugate to $\Int g(\mu_h)$ under $J'(\Qpb)$ and such that it lies in the closed Weyl chamber of $X_{\ast}(T')$ associated with a Borel subgroup of $G_L$ (defined over $L$) containing $T'_L$.

Then, $\Int g\circ\xi_{-\mu_h}=\xi_{-\Int g(\mu_h)}$ and $\xi_{-\mu'}$ are equivalent as homomorphisms from $\fG_p$ to $\fG_{J'}$. This can be proved by the original argument in \cite[Lemma 5.2]{LR87}, but one can also resort to Lemma \ref{lem:equality_restrictions_to_kernels_imply_conjugacy} below as follows. First, we observe that when we take $L$ to be large enough so that $L_{v_2}$ splits $T'$ as well, the following two cocharacters $\Int g(\nu_p)\in X_{\ast}(\Int g(T_{\Qp}))$, $\nu_p'\in X_{\ast}(T')$ are equal:
\begin{equation} \label{eqn:equality_of_two_cochar}
\Int g(\nu_p):=-\sum_{\sigma\in\Gal(L_{v_2}/\Qp)}\sigma (g\mu_hg^{-1})\qquad =\qquad \nu_p':=-\sum_{\sigma\in\Gal(L_{v_2}/\Qp)}\sigma\mu'.
\end{equation}
Indeed, they both map into the center of $\Int g(J)$: this is clear for $\Int g(\nu_p)$ as $J=\Cent(\nu_p)$, while $\nu_p'$ maps into a split $\Qp$-subtorus of $T'$, so into $Z(J')$ (as $T'$ is elliptic in $J'$). So, their equality can be checked after composing them with the natural projection $J'\rightarrow J'^{\ab}=J'/J'^{\der}$, but this is obvious since $\Int g(\mu_h)$ is conjugate to $\mu'$ under $J'(\Qpb)$.
Next, we conclude by applying Lemma \ref{lem:equality_restrictions_to_kernels_imply_conjugacy} to the two morphisms $\Int (vg)\circ\xi_{-\mu_h}$, $\xi_{-\mu'}$, where $v\in J'(\Qpb)$ satisfies $\Int (vg)(T_{\Qp})=T'$.

But, Lemma \ref{lem:unramified_conj_of_special_morphism} tells us that for an unramified conjugate $\xi_p'$ of $\xi_{-\mu'}:\fG_p\rightarrow\fG_{T'}$ under $T'(\Qpb)$, we have that for $\mathrm{inv}_{T'(L)_1}:T'(L)/T'(L)_1\times T'(L)/T'(L)_1\rightarrow T'(L)/T'(L)_1\cong X_{\ast}(T')_{\Gal(\Qpb/\Qpnr)}$, 
\[\mathrm{inv}_{T'(L)_1}(1,\xi_p'(s_{\rho}))=\underline{\mu'},\]
where $\rho\in\Gal(\Qpb/\Qp)$ is a lift of the Frobenius automorphism $\sigma$.
Because $T'(L)_1\subset \mbfKt_p$, it follows that $\mathrm{inv}_{\mbfKt_p}(x_0,\xi_p'(s_{\rho})x_0)$ ($x_0:=1\cdot\mbfKt_p$) is equal to the image of $t^{\underline{\mu'}}$ in $\tilde{W}_{\mbfK_p}\backslash \tilde{W}/ \tilde{W}_{\mbfK_p}\cong X_{\ast}(T')_{\Gal(\Qpb/\Qpnr)}/\tilde{W}_{\mbfK_p}$. It remains to show that $\tilde{W}_{\mbfK_p}t^{\underline{\mu'}} \tilde{W}_{\mbfK_p}\in\Adm_{\mbfKt_p}(\{\mu_X\})$. 

For that, let $\tilde{W}=N(L)/T'(L)_1$ denote the extended affine Weyl group, where $N$ is the normalizer of $T'$ (note that $T'_L$ contains a maximal split $L$-torus $(S'_1)_L$). By our choice of $T'$ (and as $\mbfK_p$ is special), $\tilde{W}$ is a semi-direct product $X_{\ast}(T')_{\Gal(\Qpb/\Qpnr)}\rtimes \tilde{W}_{\mbfKt_p}$, where $\mbfKt_p$ is the special maximal parahoric subgroup of $G(L)$ corresponding to $\mbfK_p$ and $\tilde{W}_{\mbfKt_p}=(N(L)\cap \mbfK_p)/T'(L)_1$ (which maps isomorphically onto the relative Weyl group $W_0=N(L)/T'(L)$). To fix a Bruhat order on $\tilde{W}$, we choose a $\sigma$-stable alcove $\mbfa$ in the apartment $\mcA(S'_1,L)$ containing a special point, say $\mbfo$, corresponding to $\mbfK_p$. The choice of $\mbfa$ and $\mbfo$ give the semi-direct product decomposition $\tilde{W}=W_a\rtimes\Omega_{\mbfa}$, where $W_a$ is the (extended) affine Weyl group of $(G^{\uc},T'^{\uc})$ and $\Omega_{\mbfa}\subset \tilde{W}$ is the normalizer of $\mbfa$, and a reduced root system ${}^{\mbfo}\Sigma$ such that $W_a$ is the affine Weyl group $Q^{\vee}({}^{\mbfo}\Sigma)\rtimes W({}^{\mbfo}\Sigma)$, (cf. Subsec. \ref{subsubsec:EAWG}). Also, the choice of $\mbfa$ fixes a set of simple affine roots for $W_a$, in particular a set ${}^{\mbfo}\Delta$ of simple roots for ${}^{\mbfo}\Sigma$, whose corresponding coroots $\{\alpha^{\vee}\ |\ \alpha\in {}^{\mbfo}\Delta\}$ form a basis of $Q^{\vee}({}^{\mbfo}\Sigma)\cong X_{\ast}(T'^{\uc})_{\Gal(\Qpb/\Qpnr)}$. Then there exists a set $\Delta\subset X^{\ast}(T')$ of simple roots for the root datum of $(G,T')$ with the property that every $\alpha\in{}^{\mbfo}\Delta$ is the restriction of some $\tilde{\alpha}\in\Delta$.
Let $\overline{C}_{\mbfa}\subset X_{\ast}(S'_1)_{\R}$ and $\overline{C}\subset X_{\ast}(T')_{\R}$ be the associated closed Weyl chambers.
It follows that $\pi(\overline{C})\subset \overline{C}_{\mbfa}$, where $\pi$ is the natural surjection $X_{\ast}(T')_{\R}\rightarrow  (X_{\ast}(T')_{\Gal(\Qpb/\Qpnr)})_{\R}=X_{\ast}(S'_1)_{\R}$.
Now, let $\mu_0\in X_{\ast}(T')$ be the conjugate of $\mu_h$ lying in $\overline{C}$ (so its image $\underline{\mu_0}$ in $(X_{\ast}(T')_{\Gal(\Qpb/\Qpnr)})_{\R}$ lies in $\overline{C}_{\mbfa}$). 
Then, it suffices to show that $t^{\underline{\mu'}}\leq t^{\underline{\mu_0}}$ in $\tilde{W}$. 
But, as $\mu'=w\mu_0$ for some $w\in N(\Qpb)/T'(\Qpb)$ and $\mu_0\in \overline{C}$, we have that
$\mu_0-\mu'=\sum_{\tilde{\alpha}\in\Delta}n_{\tilde{\alpha}}\tilde{\alpha}^{\vee}\in X_{\ast}(T'^{\uc})$ with $n_{\tilde{\alpha}}\in\Z_{\geq0}$ (cf. \cite[2.2]{RR96}). 
Since for each $\tilde{\alpha}\in\Delta$ $\pi(\tilde{\alpha}^{\vee})\in \R_{\geq0}\alpha^{\vee}$ for some $\alpha\in {}^{\mbfo}\Delta$, $t^{\underline{\mu_0}}- t^{\underline{\mu'}}=\pi(\mu_0)-\pi(\mu')\in C_{\mbfa}^{\vee}:=\{\sum_{\beta\in{}^{\mbfo}\Delta}c_{\beta}\beta^{\vee}\ |\ c_{\beta}\in\R_{\geq0}\}$, as wanted. 
\end{proof}

\subsection{The Langlands-Rapoport conjeture} \label{subsec:Langalnds-Rapoport conjeture}
In this subsection, we give a formulation of the Langlands-Rapoport conjecture for parahoric levels, following Kisin \cite[(3.3)]{Kisin13} and Rapoport \cite[$\S$9]{Rapoport05}.

For $v\neq p,\infty$, set $X_v(\phi):=\{g_v\in G(\Qvb)\ |\ \phi(v)\circ\zeta_v=\Int(g_v)\circ\xi_v\}$,
and \[X^p(\phi):=\prod_{v\neq\infty,p}{}^{'}\ X_v(\phi),\]
where $'$ denotes the restricted product of $X_v(\phi)$'s as defined in the line 15-26 on p.168 of \cite{LR87}. By condition (2) of Def. \ref {defn:admissible_morphism}, $X^p(\phi)$ is non-empty  (cf. \cite[B 3.6]{Reimann97}), and  
is a right torsor under $G(\A_f^p)$. 

To define the component at $p$, put $\mbfK_p(\Qpnr):=\mcG_{\mbff}^{\mathrm{o}}(\Zpnr)$ for the parahoric group scheme $\mcG_{\mbff}^{\mathrm{o}}$ over $\Zpnr$ attached to $\mbfK_p$ (Subsec. \ref{subsubsec:parahoric}) and for $b\in G(L)$, let $\mbfKt_p\cdot b\cdot \mbfKt_p$ denote the invariant $\mathrm{inv}_{\mbfKt_p}(1,b)$. We also recall that for $\theta_g^{\nr}:\fD\rightarrow\fG^{\nr}_{G_{\Qp}}$ (a morphism of $\Qpnr|\Qp$-Galois gerbs), $\overline{\theta_g^{\nr}}$ denotes its inflation to $\Qpb$, and $s_{\sigma}\in \fD$ is the lift of $\sigma\in\Gal(\Qpnr/\Qp)$ chosen in (\ref{subsubsec:cls}). Then, we set
\begin{eqnarray*}
X_p(\phi):=\{ g\mbfK_p(\Qpnr) \in G(\Qpb)/\mbfK_p(\Qpnr) & | & \phi(p)\circ \zeta_p=\Int(g)\circ \overline{\theta_g^{\nr}}\text{ for some }\theta_g^{\nr}:\fD\rightarrow\fG^{\nr}_{G_{\Qp}} \text{ s.t. } \\
& &\mbfKt_p\cdot b_g\cdot \mbfKt_p\in \Adm_{\mbfKt_p}(\{\mu_X\}),\text{ where }\theta_g^{\nr}(s_{\sigma})=b_g\rtimes\sigma \}.
\end{eqnarray*}
There is also the following explicit description for $X_p(\phi)$. 
Choose $g_0\in G(\Qpb)$ with $g_0\mbfK_p(\Qpnr)\in X_p(\phi)$. Then, using $g_0$ as a reference point, we obtain a bijection
\begin{equation} \label{eqn:X_p(phi)=AffineDL}
X_p(\phi)\isom X(\{\mu_X\},b_{g_0})_{\mbfK_p}\ :\ h\mbfK_p(\Qpnr)\mapsto g_0^{-1}h\mbfKt_p,
\end{equation}
where $b_{g_0}\rtimes\sigma=\theta_{g_0}^{\nr}(s_{\sigma})$ for $\theta_{g_0}^{\nr}$ with $\phi(p)\circ \zeta_p=\Int(g_0)\circ \overline{\theta_{g_0}^{\nr}}$.
Indeed, for $h\in G(\Qpb)$, if $\Int h^{-1}\circ\phi(p)\circ\zeta_p=\Int(g_0^{-1}h)^{-1}\circ \overline{\theta_{g_0}^{\nr}}$ is unramified, $g:=g_0^{-1}h\in G(\Qpnr)$ (cf. proof of Lemma \ref{lem:unramified_morphism}, (3)), and  
\[b_{h}\rtimes\sigma=\theta_{h}^{\nr}(s_{\sigma})=\Int(g^{-1})\circ \theta_{g_0}^{\nr}(s_{\sigma})=g^{-1}b_{g_0}\sigma(g)\rtimes\sigma,\]
and by definition, $g\mbfKt_p\in X(\{\mu_X\},b_{g_0})_{\mbfK_p}$ if and only if $\mbfKt_p\cdot g^{-1}b_{g_0}\sigma(g)\cdot\mbfKt_p\in \Adm_{\mbfKt_p}(\{\mu_X\})$. 
So, $h\in X_p(\phi)$ if and only if $g_0^{-1}h \mbfKt_p\in X(\{\mu_X\},b_{g_0})_{\mbfK_p}$.  
Note that there are natural left (or right) actions of $Z(\Qp)$ on $X_p(\phi)$ and $X(\{\mu_X\},b_{g_0})_{\mbfK_p}$ compatible with the above bijection: $z\in Z(\Qp)$ sends $g\in X(\{\mu_X\},b_{g_0})_{\mbfK_p}$ (resp. $g\mbfK_p(\Qpnr)\in X_p(\phi)$) to $zg$ (resp. to $zg\mbfK_p(\Qpnr)$).

The absolute Frobenius automorphism $F=\theta_g^{\nr}(s_{\sigma})$ on $G(L)$ (sending $g\in G(L)$ to $b\sigma(g)$) induces an action on $G(L)/\mbfKt_p$ as the facet in $\mcB(G_L)$ defining $\mbfKt_p$ is stable under $\sigma$. To see that, choose a maximal ($\Qp$-)split $\Qp$-torus $S$, and a maximal $\Qpnr$-split $\Qp$-torus $S_1$ containing $S$ (which exists \cite[5.1.12]{BT84}). With $T$, $N$ being the centralizer of $S_1$ and its normalizer as usual, the $r$-th Frobenius automorphism 
\[\Phi=F^r\ (r=[\kappa(\wp):\Fp])\ :\ g\mapsto (b_{g_0}\rtimes\sigma)^r(g)=b_{g_0} \sigma(b_{g_0}) \cdots \sigma^{r-1}(b_{g_0})\cdot\sigma^r(g)\]
leaves $X(\{\mu_X\},b_{g_0})_{\mbfK_p}$ stable.
Indeed, one easily checks that $(\Phi g)^{-1}\cdot b_{g_0}\cdot\sigma(\Phi g)=\sigma^r(g^{-1}b_{g_0}\sigma(g))$ and that $\Adm_{\mbfKt_p}(\{\mu_X\})$ is stable under the action of $\sigma^r$ on $\tilde{W}_{\mbfKt_p}\backslash \tilde{W}/\tilde{W}_{\mbfKt_p}$ since $\mu$ is defined over $E_{\wp}$.
Obviously, $\Phi$ commutes with the action of $Z(\Qp)$. 
One readily sees that when translated to an action on $X_p(\phi)$, $\Phi$ acts by $\Phi(g\mbfK_p(\Qpnr))=g\cdot \Nm_{L_r/\Qp}b_g\cdot\mbfK_p(\Qpnr)$; as $b_{gk}=k^{-1}b_g\sigma(k)$ for $k\in G(\Qpnr)$, 
this does not depend on the choice of a representative $g$ in the coset $g\mbfK_p(\Qpnr)$.

Let $I_{\phi}:=\mathrm{Aut}(\phi)$; this is an algebraic group over $\Q$ such that
\begin{equation*}
I_{\phi}(\Q)=\{g\in G(\Qb)\ |\ \mathrm{Int}(g)\circ\phi=\phi\}.
\end{equation*}
It naturally acts on $X^p(\phi)$ from the left and commutes with the (right) action of $G(\A_f^p)$. It also acts on $X_p(\phi)$ via the map (which depends on the choice of $g_0\in X_p(\phi)$)
\[\Int(g_0):I_{\phi}(\Q)\rightarrow \mathrm{Aut}(\theta_{g_0}^{\nr})(\Qp)=\{g\in G(L) |\ \mathrm{Int}(g)\circ\theta_{g_0}^{\nr}=\theta_{g_0}^{\nr}\}\] which is given by $\alpha\mapsto g_0^{-1}\alpha g_0$, and clearly this action commutes with $Z(\Qp)\times \Phi$, too. Finally, we define
\[S(\phi):=\varprojlim_{\mbfK^p} I_{\phi}(\Q)\backslash (X^p(\phi)/\mbfK^p)\times X_p(\phi),\]
where $\mbfK^p$ runs over the compact open subgroups of $G(\Q^p)$. 
This set is equipped with an action of $G(\A_f^p)\times Z(\Qp)$ and a commuting action of $\Phi$, and as such is determined, up to isomorphism, by the equivalence class of $\phi$.

\begin{conj} \label{conj:Langlands-Rapoport_conjecture_ver1}
Suppose that $\mbfK_p\subset G(\Qp)$ is a parahoric subgroup.
Then, there exists an integral model $\sS_{\mbfK_p}(G,X)$ of $\Sh_{\mbfK_p}(G,X)$ over $\cO_{E_{\wp}}$ for which there exists a bijection 
\[\sS_{\mbfK_p}(G,X)(\Fpb)\isom \bigsqcup_{[\phi]}S(\phi)\]
compatible with the actions of $Z(\Qp)\times G(\A_f^p)$ and $\Phi$, where $\Phi$ acts on the left side as the $r$-th (geometric) Frobenius. Here, $\phi$ runs over a set of representatives for the equivalence classes of admissible morphisms $\fQ\rightarrow\fG_G$.
\end{conj}

\begin{rem}
(1) The original conjecture was made under the assumption that $G^{\der}$ is simply connected (due to the expectation that only special admissible morphisms are to contribute to the $\Fpb$-points), cf. Remark \ref{rem:Kisin's_defn_of_admissible_morphism}.
 
(2) In \cite[Conj. 9.2]{Rapoport05}, Rapoport gave another version of this conjecture, using a different definition of admissible morphisms, where the condition (3) of Def. \ref{defn:admissible_morphism} is replaced by the more natural, from the group-theoretical viewpoint, and a priori weaker condition (3') that \textit{the $\sigma$-conjugacy class $\mathrm{cls}_{G_{\Qp}}(\phi(p))$ of $b_{\phi}$ lies in $B(G,\{\mu_X\})$}. Our theorem \ref{thm:non-emptiness_of_NS} together with Theorem A of \cite{He15} establishes the equivalence of these two versions: before, it was known that (3) $\Rightarrow$ (3'). 
\end{rem}

\subsection{Kottwitz triples and Kottwitz invariant} 
Our main references for the material covered here are \cite[p.182-183]{LR87}, \cite[$\S2$]{Kottwitz90}, and \cite{Kottwitz92}.
Here, we make the assumption that $G^{\der}$ is simply connected, one of whose consequences (of heavy use here) is that the centralizer of a semisimple element is connected.

\subsubsection{Kottwitz triple} \label{subsubsec:pre-Kottwitz_triple}
A Kottwitz triple is a triple $(\gamma_0;\gamma=(\gamma_l)_{l\neq p},\delta)$ of elements satisfying certain conditions, where
\begin{itemize} \addtolength{\itemsep}{-4pt} 
\item[(i)] $\gamma_0$ is a semi-simple element of $G(\Q)$ that is elliptic in $G(\R)$, defined up to conjugacy in $G(\overline{\Q})$;
\item[(ii)] for $l\neq p$, $\gamma_l$ is a semi-simple element in $G(\Q_l)$, defined up to conjugacy in $G(\Q_l)$, which is conjugate to $\gamma_0$ in $G(\overline{\Q}_l)$;
\item[(iii)] $\delta$ is an element of $G(L_n)$ (for some $n$), defined up to $\sigma$-conjugacy in $G(L)$, such that the norm $\Nm_n\delta$ of $\delta$  is conjugate to $\gamma_0$ under $G(\overline{L})$, where $\Nm_n\delta:=\delta\cdot\sigma(\delta)\cdots\sigma^{n-1}(\delta)\in G(\Qp)$.
\end{itemize}

There are two conditions to be satisfied by such triple. 
To explain the first one, put 
\[I_0:=\Cent_G(\gamma_0)\] 
(centralizer of $\gamma_0$ in $G$); as $G^{\der}$ is simply connected, it is a connected group. Then, for every place $v$ of $\Q$, we now construct an algebraic $\Q_v$-group $I(v)$ and an inner twisting 
\[\psi_v:(I_0)_{\Qvb}\rightarrow I(v)_{\Qvb}\] 
over $\Qvb$. 
First, for each finite place $v\neq p$ of $\Q$, set $I(v):=\Cent_{G_{\Q_v}}(\gamma_v)$. For any $g_v\in G(\overline{\Q}_v)$ with $g_v\gamma_0g_v^{-1}=\gamma_v$, the restriction of $\Int (g_v)$ to $(I_0)_{\Qv}$ gives us an inner twisting $\psi_v:(I_0)_{\Qv}\rightarrow I(v)$, which is well defined up to inner automorphism of $I_0$. For $v=p$, we define an algebraic $\Qp$-group $I(p)$ by
\[I(p):=\{x\in \Res_{L_n/\Qp}(G_{L_n})\ |\ x^{-1}\delta\theta(x)=\delta\},\]
where $\theta$ is the $\Qp$-automorphism of $\Res_{L_n/\Qp}(G_{L_n})$ induced by the restriction of $\sigma$ to $L_n$. For more details we refer to \cite{Kottwitz82}, where $\theta$ and $I(p)$ are denoted by $s$ (on p. 801) and $I_{s\delta}$ (on p. 802), respectively. Then, Lemma 5.8 of loc. cit. provides an inner twisting $\psi_p:(I_0)_{\Qpb}\rightarrow (I(p))_{\Qpb}$, which is canonical up to inner automorphisms of $I_0$; note that
$\gamma_0\in G(\Qp)$ lies in the stable conjugacy class of $\delta$, as $\gamma_0$ and $\Nm_{L_n/\Qp}(\delta)$ are conjugate under $G(\Qpb)$ (cf. loc.cit. $\S5$).
Finally, at the infinite place, we choose an elliptic maximal torus $T_{\R}$ of $G_{\R}$ containing $\gamma_0$, as well as $h\in X\cap \Hom(\dS,T_{\R})$. We twist $I_0$ using the Cartan involution  $\Int (h(i))$ on $I_0/Z(G)$, and get an inner twisting $\psi_{\infty}:(I_0)_{\C}\rightarrow I(\infty)_{\C}$. Here, $I(\infty)/Z(G)$ is anisotropic over $\R$.

\begin{defn}  \label{defn:Kottwitz_triple}
A triple $(\gamma_0;(\gamma_l)_{l\neq p},\delta)$ as in (i) - (iii) is called a Kottwitz triple of level $n=m[\kappa(\wp):\Fp]$ ($m\geq1$) if it satisfies the following two conditions (iv), ($\ast(\delta)$):
\begin{itemize} \addtolength{\itemsep}{-4pt}
\item[(iv)] There exists a triple $(I,\psi,(j_v))$ consisting of a $\Q$-group $I$, an inner twisting $\psi:I_0\rightarrow I$ and for each place $v$ of $\Q$, an isomorphism $j_v:(I)_{\Qv}\rightarrow I(v)$ over $\Q_v$, unramified almost everywhere, such that $j_v\circ\psi$ and $\psi_v$ differ by an inner automorphism of $I_0$ over $\Qb_v$. 
\item[($\ast(\delta)$)] the image of $\overline{\delta}$ under the Kottwitz homomorphism $\kappa_{G_{\Q_p}}:B(G_{\Q_p})\rightarrow \pi_1(G_{\Q_p})_{\Gamma(p)}$ (\ref{subsubsec:Kottwitz_hom}) is equal to $\mu^{\natural}$ (defined in (\ref{eqn:mu_natural})).
\end{itemize}
\end{defn}

\begin{rem} \label{rem:Kottwitz_triples}
1) Two triples $(\gamma_0;(\gamma_l)_{l\neq p},\delta)$, $(\gamma_0';(\gamma_l')_{l\neq p},\delta')$ as in (\ref{subsubsec:pre-Kottwitz_triple}) with $\delta,\delta'\in G(L_n)$ for (iii)
are said to be \textit{equivalent}, if $\gamma_0$ is stably conjugate to $\gamma_0'$, $\gamma_l$ is conjugate to $\gamma_l'$ in $G(\Ql)$ for each $l\neq p,\infty$, and $\delta$ is $\sigma$-conjugate to $\delta'$ in $G(L_n)$. Then, for two such equivalent triples, one of them is a Kottwitz triple of level $n$ if and only if the other one is so.

2) According to \cite[p.172]{Kottwitz90}, the condition (iv) is satisfied if the Kottwitz invariant $\alpha(\gamma_0;\gamma,\delta)$ (to be recalled below) is trivial.
\end{rem}

We will also consider the following condition:
\begin{itemize}
\item[$(\ast(\gamma_0))$] Let $H$ be the centralizer in $G_{\Qp}$ of the maximal $\Qp$-split torus in the center of $\Cent_{G}(\gamma_0)_{\Qp}$. Then, there exist a maximal $\Q$-torus $T$ of $\Cent_{G}(\gamma_0)$, elliptic over $\R$, and $\mu\in X_{\ast}(T)\cap \{\mu_X\}$, such that $\Nm_{K/\Qp}\mu$ maps into the center of $\Cent_{G}(\gamma_0)_{\Qp}$, where $K$ is a finite Galois extension of $\Qp$ over which $\mu\in X_{\ast}(T)$ is defined, and that
\[\lambda_H(\gamma_0)=n\underline{\mu}\] 
for some $n\in\N$, where $\lambda_H:H(\Qp)\rightarrow (\pi_1(H)_I/\text{torsions})^{\langle\sigma\rangle}$ is the map from Subsec. \ref{subsubsec:Kottwitz_hom} (restriction of $v_{H_L}$ to $H(\Qp)$) and $\underline{\mu}$ denotes the image of $\mu$ in $\pi_1(H)_I/\text{torsions}$.
\end{itemize}

\begin{rem} \label{rem:condition_(ast(gamma_0))}
1) Note that $H$ is a $\Qp$-Levi subgroup of $G_{\Qp}$ with the center being the maximal $\Qp$-split torus in the center of $\Cent_{G}(\gamma_0)_{\Qp}$ ($H$ is the centralizer of its center). In particular, the first sub-condition in $(\ast(\gamma_0))$ is the same as that $\Nm_{K/\Qp}\mu$ maps into the center of $H$, and $H^{\ab}=H/H^{\der}$ is $\Qp$-split. So, as $H^{\der}$ is simply-connected by assumption, the canonical action of $\Gal(\Qpb/\Qp)$ on $\pi_1(H)$ is trivial.
Sometimes we will refer to the number $n$ appearing in $(\ast(\gamma_0))$ as the \textit{level} of that condition (e.g., we will say that the condition $(\ast(\gamma_0))$ holds for $\gamma_0$ with level $n$).

2) A similar condition also appears in \cite[p.183]{LR87} (in the hyperspecial case), where it is denoted by ($\ast(\epsilon)$). There are two differences between their condition ($\ast(\epsilon)$) and ours $(\ast(\gamma_0))$.
In ($\ast(\epsilon)$), there is no requirement that $\Nm_{K/\Qp}\mu$ maps into the center of $\Cent_{G}(\gamma_0)_{\Qp}$, but instead it demands that $\mu$ is defined over the unramified extension $L_n$ of $E_{\wp}$ of degree $m$. The main use of the original condition in \cite{LR87} is 
Satz 5.21 and the subsequent statements depending on it. This theorem Satz 5.21 asserts that in order for a rational element $\gamma_0\in G(\Q)$ to admit an admissible pair $(\phi,\epsilon)$ with $\epsilon$ stably conjugated to $\gamma_0$, it is necessary and sufficient that $\gamma_0$ is elliptic at $\R$ and satisfies their condition ($\ast(\epsilon)$). To the author's opinion,
\footnote{although I will be happy to be pointed out to be wrong.}
their condition ($\ast(\epsilon)$) is insufficient to prove this claim. And, later we will prove this claim with our condition $(\ast(\gamma_0))$ (Thm. \ref{thm:LR-Satz5.21}). 
But, this is not the only value of the new condition. In fact, it is a quite natural condition, as it is satisfied by every Kottwitz triple that is expected, by the conjecture of Langlands-Rapoport, to arise from an $\Fpb$-point (Thm. \ref{thm:LR-Satz5.21}) and is also a crucial condition for an elliptic stable conjugacy class which is $l$-adic unit for every $l\neq p$ to fulfill, in order for it to contain the relative Frobenius of the reduction of a CM-point (Thm. \ref{thm:effectivity_criteria}). In contrast, in Satz 5.25 of \cite{LR87} (theorem corresponding to the former statement), it is not made clear whether an effective Kottwitz triple would satisfy the original condition $(\ast(\epsilon))$.
\end{rem}

\subsubsection{Kottwitz invariant} \label{subsec:Kottwitz_invariant}

We recall the definition of the Kottwitz invariant, cf. \cite{Kottwitz90}, $\S2$. 
Recall out convention for the notations $X_{\ast}(A)$, $X^{\ast}(A)$, and $A^D$: for a locally compact abelian group $A$, they denote the (co)character groups $\Hom(\C^{\times},A)$, $\Hom(A,\C^{\times})$, and the Pontryagin dual group $\Hom(A,\Q/\Z)$, respectively.
Note that the centralizer $I_0$ of $\gamma_0$ in $G$ is a connected group reductive since $\gamma_0$ is semisimple and $G^{\der}$ is simply connected. 
The exact sequence 
\[1\rightarrow Z(\widehat{G})\rightarrow Z(\widehat{I}_0)\rightarrow Z(\widehat{I}_0)/Z(\widehat{G}) \rightarrow 1\]
induces a homomorphism (\cite[Cor.2.3]{Kottwitz84a})
\[\pi_0((Z(\widehat{I}_0)/Z(\widehat{G}) )^{\Gamma})\rightarrow H^1(\Q,Z(\widehat{G})),\]
where $\Gamma:=\Gal(\Qb/\Q)$. Let $\ker^1(\Q,Z(\widehat{G}))$ denote the kernel of the map
$H^1(\Q,Z(\widehat{G}))\rightarrow \prod_v H^1(\Q_v,Z(\widehat{G}))$, and
define $\mathfrak{K}(I_0/\Q)$ to be the subgroup of $\pi_0((Z(\widehat{I}_0)/Z(\widehat{G}))^{\Gamma})$ consisting of elements whose image in $H^1(\Q,Z(\widehat{G}))$ lies in $\ker^1(\Q,Z(\widehat{G}))$. This is known to be a finite group. Also, since $\gamma_0$ is elliptic, there is an identification
\[\mathfrak{K}(I_0/\Q)=\left(\cap_v Z(\widehat{I}_0)^{\Gamma(v)}Z(\widehat{G})\right)/Z(\widehat{G}).\]
The Kottwitz invariant $\alpha(\gamma_0;\gamma,\delta)$ is then a certain character of $\mathfrak{K}(I_0/\Q)$ (i.e. an element of $X^{\ast}(\mathfrak{K}(I_0/\Q))=\mathfrak{K}(I_0/\Q)^D$). It is defined as a product over all places of $\Q$:
\[\alpha(\gamma_0;\gamma,\delta)=\prod_v \beta_v(\gamma_0;\gamma,\delta)|_{\cap_v Z(\widehat{I}_0)^{\Gamma(v)}Z(\widehat{G})},\]
Here, for each place $v$, $\beta_v(\gamma_0;\gamma,\delta)$ is a character on $Z(\widehat{I}_0)^{\Gamma(v)}Z(\widehat{G})$ which is the unique extension of another character $\alpha_v(\gamma_0;\gamma,\delta)$ on $Z(\widehat{I}_0)^{\Gamma(v)}$ with the restriction to $Z(\widehat{G})$ being 
\[
\beta_v|_{Z(\widehat{G})}=
\begin{cases}
\quad \mu_1 &\text{ if }\quad  v=\infty \\
\ -\mu_1 &\text{ if }\quad  v=p \\
\text{ trivial } &\text{ if }\quad  v\neq p,\infty.
\end{cases}
\]
Finally, $\alpha_v(\gamma_0;\gamma,\delta)\in X^{\ast}(Z(\widehat{I}_0)^{\Gamma(v)})$ is defined as follows: For a place $v\neq p,\infty$, choose $g\in G(\Qvb)$ with $g\gamma_0g^{-1}=\gamma_v$. Then $\tau\mapsto g^{-1}\tau(g)$ is a cocycle of $\Gamma(v)$ in $I_0(\Qvb)$, and its cohomology class lies in $\ker[H^1(\Qv,I_0)\rightarrow H^1(\Qv,G)]$. In view of the canonical isomorphism $H^1(\Qv,H)\isom\pi_0(Z(\widehat{I}_0)^{\Gamma(v)})^D$, this gives the desired character of $Z(\widehat{I}_0)^{\Gamma(v)}$. For $v=p$, choose $c\in G(L)$ with $c\gamma_0 c^{-1}=\Nm_n(\delta)$, and set $b:=c^{-1}\delta\sigma(c)$. Then, one readily sees that $b\in I_0(L)$ and its $\sigma$-conjugacy class in $B((I_0)_{\Qp})$ is well-defined (independent of the choice of $c$). Then, we define $\alpha_p(\gamma_0;\gamma,\delta)$ be the image of this class under the canonical map $B((I_0)_{\Qp})\rightarrow X^{\ast}(Z(\widehat{I}_0)^{\Gamma(p)})$. For $v=\infty$, we let $\alpha_{\infty}(\gamma_0;\gamma,\delta)$ to be the image of $\mu_h$ in $\pi_1(I_0)_{\Gamma(\infty)}\cong X^{\ast}(Z(\widehat{I}_0)^{\Gamma(\infty)})$ for some $h\in X$ factoring through a maximal torus $T$ of $G$ containing $\gamma_0$ and elliptic over $\R$. Again, the character defined this way is independent of all the choices made.
One can check that the product $\prod_v \beta_v(\gamma_0;\gamma,\delta)$ is trivial on $Z(\widehat{G})$.

\subsection{Admissible pairs and associated Kottwitz triples}

Recall that our parahoric subgroup $\mbfK_p\subset G(\Qp)$ is defined by a $\sigma$-stable facet $\mbfa$ in the building $\mcB(G,L)$ which also gives rise to a parahoric subgroup $\mbfKt_p\subset G(L)$ with $\mbfK_p=\mbfKt_p\cap G(\Qp)$. To formulate the definition of admissible pair for general parahoric subgroups, we need to consider the orbit space $G(L)/\mbfKt_p$ endowed with the obvious action of the semi-direct product $G(L)\rtimes\langle\sigma\rangle$, instead of the Bruhat-Tits building $\mcB(G,L)$ that was used in the hyperspecial level case; note that unless $\mbfK_p$ is hyperspecial (in which case $\mbfKt_p$ equals the whole stabilizer $\mathrm{Stab}(\mbfo)\subset G(L)$), $G(L)/\mbfKt_p$ is not a subset of $\mcB(G,L)$ in any natural manner.

\begin{defn} (\cite[p.189]{LR87}) \label{def:admissible_pair}
A pair $(\phi,\epsilon)$ is called \textit{admissible (of level $n=[\kappa(\wp):\Fp]m$)} if 
\begin{itemize}
\item[(1)] $\phi:\fQ\rightarrow \fG_G$ is admissible in the sense of Def. \ref{defn:admissible_morphism}.
\item[(2)] $\epsilon\in I_{\phi}(\Q)(=\{g\in G(\Qb)\ |\ \mathrm{Int}(g)\circ\phi=\phi\})$.
\item[(3)] There are $\gamma=(\gamma(v))\in G(\A_f^p)$, $y\in X^p(\phi)$, and $x$ in the $G(L)$-orbit $G(L)/\mbfKt_p$ 
of the point $x^0:=1\cdot \mbfKt_p$ such that
\[\epsilon y=y\gamma\quad \text{ and }\quad \epsilon x=\Phi^mx.\]
\end{itemize}
\end{defn}

Some comments on the condition (3) are in order.
Recall (Lemma \ref{lem:unramified_morphism}) that if $\xi_p'=\Int u\circ\xi_p\ (u\in G(\Qpb))$ is an unramified conjugate of $\xi_p=\phi(p)\circ\zeta_p:\fG_p\rightarrow \fG_G(p)$ and $\xi_p'(s_{\widetilde{\sigma}})=b\rtimes\widetilde{\sigma}$, where $\widetilde{\sigma}\in\Gal(\Qpb/\Qp)$ is a (fixed) lifting of the Frobenius automorphism $\sigma$ of $\Qpnr$ and $\tau\mapsto s_{\tau}$ is the chosen section to the projection $\fG_p\rightarrow\Gal(\Qpb/\Qp)$, then $b$ and $\epsilon':=u\epsilon u^{-1}\in G(\Qpb)$ in fact belong to $G(\Qpnr)$ (as $\xi_p'(s_{\tau})=1\rtimes\tau$ for every $\tau\in\Gal(\Qpb/\Qpnr)$, one has $\epsilon'\rtimes\tau=\epsilon'\cdot \xi_p'(s_{\tau})=\xi_p'(s_{\tau})\cdot \epsilon'=\tau(\epsilon')\rtimes\tau$ for all $\tau\in\Gal(\Qpb/\Qpnr)$).
Then, by existence of $x$ in $X(\{\mu_X\},b)_{\mbfK_p}$ satisfying $\epsilon x=\Phi^mx$ in (3), we mean existence of $u\in G(\Qpb)$ such that $\Int u\circ\xi_p$ is unramified, for which there exists $x$ in the orbit space $G(L)/\mbfKt_p$ with 
\begin{equation} \label{def:admissible_pair_(3)a}
\epsilon'x=\Phi^mx,
\end{equation}
where $\Phi=F^{[\kappa(\wp):\Fp]}$ for $F=(\Int u\circ\xi_p)(s_{\widetilde{\sigma}})=b\rtimes\widetilde{\sigma}$. 
Here, to define $F$, instead of $\xi_p'$, we can also use a Galois $\Qpnr/\Qp$-gerb morphism $\theta^{\nr}$ whose inflation to $\Qpb$ is $\xi_p'$: it does not change $b$ (Lemma \ref{lem:unramified_morphism}).

\begin{rem} \label{rem:admissible_pair}
(1) Suppose that the anisotropic kernel of $Z(G)$ (maximal anisotropic subtorus of $Z(G)$) remains anisotropic over $\R$ (then $Z(G)(\Q)$ is discrete in $Z(G)(\A_f)$); take $\mbfK^p$ to be small enough so that $Z(\Q)\cap \mbfK^p=\{1\}$ and also the condition (1.3.7) of \cite{Kottwitz84b} holds.

Then, the conjecture (\ref{conj:Langlands-Rapoport_conjecture_ver1}) provides the following description of the finite sets $\sS_{\mbfK}(G,X)(\F_{q^m})=[\sS_{\mbfK}(G,X)(\Fpb)]^{\Phi^m}$ for each finite extension $\F_{q^m}$ of $\F_q=\kappa(\wp)$, where $\Phi:=F^{[\kappa(\wp):\Fp]}$ is the relative Frobenius (cf. \cite[1.3-1.5]{Kottwitz84b}, \cite[$\S$5]{Milne92}): There exists a bijection (forming a compatible family for varying $\mbfK^p$)
\begin{equation} \label{eqn:LRconj-ver2} 
\sS_{\mbfK}(G,X)(\F_{q^m})\isom \bigsqcup_{(\phi,\epsilon)} S(\phi,\epsilon), \end{equation} where $S(\phi,\epsilon)=I_{\phi,\epsilon}(\Q)\backslash X_p(\phi,\epsilon)\times X^p(\phi,\epsilon)/K^p$ with 
\begin{eqnarray} \label{eqn:LRconj-ver2_constituents} 
I_{\phi,\epsilon}&=&\text{ centralizer of }\epsilon\text{ in }I_{\phi}(\Q), \\  X^p(\phi,\epsilon)&=&\{ x^p\in X^p(\phi)\ |\ \epsilon x^p=x^p\mod \mbfK^p\}, \nonumber \\  X_p(\phi,\epsilon)&=&\{ x_p\in X(\{\mu_X\},b)_{\mbfK_p}\ |\ \epsilon x_p=\Phi^mx_p\}. \nonumber \end{eqnarray}
In (\ref{eqn:LRconj-ver2}), $(\phi,\epsilon)$ runs over a set of representatives for the conjugacy classes of pairs consisting of an admissible morphism $\phi$ and $\epsilon\in I_{\phi}(\Q)(\subset G(\Qb))$. 

(2) According to this conjecture (\ref{eqn:LRconj-ver2}), an admissible pair $(\phi,\epsilon)$ will contribute to the set $\sS_{\mbfK}(G,X)(\F_{q^m})$ if and only if in the condition (3) one can find $y\in X^p(\phi)(\subset G(\bar{\A}_f^p))$ and a solution $x\in G(L)/\mbfKt_p$ of the equation $\epsilon x=\Phi^mx$, further satisfying that
\begin{equation} \label{eqn:effectivity_admssible_pair_condition}
y^{-1}\epsilon y\in \mbfK^p,\ \text{ and }\ x\in X(\{\mu_X\},b)_{\mbfK_p}
\end{equation}

In regarding this, we emphasize that in the condition (3) of the definition of an admissible pair, $y\in X^p(\phi)$ and $x\in G(L)/\mbfKt_p$ are \emph{not} demanded to satisfy this (seemingly natural) condition (\ref{eqn:effectivity_admssible_pair_condition}); so, Def. 5.8 of \cite{Milne92} is not correct. Indeed, Langlands and Rapoport made this point explicit: \textit{``Es wird \"ubrigens nicht verlangt, dass $x$ in $X_p$ liegt."} \cite[p. 189, line +18]{LR87}. We will just remark that this definition of admissible pair plays an auxiliary role in \cite{LR87} but serves well the intended purpose.

(3) We will call an admissible pair $(\phi,\epsilon)$ $\mbfK^p$-\textit{effective}, respectively $\mbfK_p$-\textit{effective}  if there exist $y\in X^p(\phi)$, respectively $x\in G(L)/\mbfKt_p$ with $\epsilon x=\Phi^mx$, satisfying the conditions (\ref{eqn:effectivity_admssible_pair_condition}). Note that if  $(\phi,\epsilon)$ is $\mbfK^p$-effective and $\epsilon\in G(\A_f^p)$, $\epsilon$ itself lies in a compact open subgroup of $G(\A_f^p)$ (cf. proof of Lemma \ref{lem:invariance_of_(ast(gamma_0))_under_transfer_of_maximal_tori}, (2)).
\end{rem}

Two admissible pairs $(\phi,\epsilon)$, $(\phi',\epsilon')$ are said to be \textit{equivalent} (or \textit{conjugate}) if there exists $g\in G(\Qb)$ such that $(\phi',\epsilon')=g(\phi,\epsilon)g^{-1}$.

\subsubsection{} 
Now, we explain how with any admissible pair $(\phi,\epsilon)$, one can associate a Kottwitz triple $(\gamma_0;(\gamma_l)_{l\neq p},\delta)$, cf. \cite{LR87}, p189. Thus, again we assume that $G^{\der}$ is simply connected; see Lemma 5.13 in  \cite{LR87} and the remarks preceding it where this assumption is really used.

First, we set $\gamma_l:=\gamma(l)$. As when we replace $y$ by $yh, h\in G(\A_f^p)$, $\gamma$ goes over to $h^{-1}\gamma h$, the $G(\A_f^p)$-conjugacy class of $\gamma$ is well-determined by the pair $(\phi,\epsilon)$. Also, one easily sees that if $(\phi',\epsilon')=g(\phi,\epsilon)g^{-1}\ (g\in G(\Qb))$, the corresponding $G(\A_f^p)$-conjugacy classes of $(\gamma_l')_{l\neq p}$, $(\gamma_l)_{l\neq p}$ are the same, since $y\mapsto gy$ gives a bijection $X^p(\phi)\isom X^p(\Int g\circ\phi)$.

Next, to find $\gamma_0$, we use Lemma \ref{lem:LR-Lemma5.23} below, which is a generalization of Lemma 5.23 of \cite{LR87}. If an admissible pair $(\phi,\epsilon)$ satisfies that $\phi=i\circ\psi_{T,\mu_h}$ and $\epsilon\in T(\Q)$ for a special Shimura datum $(T,h)$, then we will say that it is \textit{nested} 
\footnote{This is our translation of the German word \textit{eingeschachtelt} used by Langlands-Rapoport \cite[p.190, line19]{LR87}.}
in $(T,h)$. Hence, this lemma says that every admissible pair is equivalent to an admissible pair that is nested in a special Shimura sub-datum of $(G,X)$, in particular to an admissible pair $(\phi,\epsilon)$ such that $\epsilon\in G(\Q)$.
Then, we define $\gamma_0$ to be any semi-simple rational element of $G$ which is stably conjugate to $\epsilon$ and also lies in a maximal $\Q$-torus $T'$ of $G$ which is elliptic at $\R$, which we now know exists; note that $T'_{\R}$ is elliptic (since $\Inn(h(i))$ defines a Cartan involution of $G^{\ad}_{\R}$ and so of $T'_{\R}/Z_{\R}$ as well), hence this $\gamma_0$ given by Lemma \ref{lem:LR-Lemma5.23} is also semi-simple. By definition, this rational element is well-defined up to stable conjugacy.

Finally, to define $\delta$, choose $u\in G(\Qpb)$ such that $\Int u\circ \xi_p$ is unramified and also satisfying the condition (3) of Def. \ref{def:admissible_pair}.
Let $b\in G(L)$ be defined by $\Int u\circ\xi_p(s_{\widetilde{\sigma}})=b\rtimes\widetilde{\sigma}$, and put $\epsilon'=u\epsilon u^{-1}\in G(L)$.

\begin{lem} \label{lem:Kottwitz84-a1.4.9_b3.3}
(1) Let $\Psi=b\rtimes\sigma^n\in G(L)\rtimes\langle\sigma\rangle$ such that $n\neq0$. Then, $\Psi$ is conjugate under $G(L)$ to an element of $\langle\sigma\rangle$ if and only if $\Psi$ fixes some point in the orbit space $G(L)/\mbfKt_p$. .

(2) Let $H$ be a quasi-split group over a $p$-adic field, either a local field $F$ or the completion $L$ of its maximal unramified extension $F^{\nr}$ in $\overline{F}$. Then, the map $v_{H_L}:H(L)\rightarrow \Hom(X_{\ast}(Z(\widehat{H}))^I,\Z)$ (Subsec \ref{subsubsec:Kottwitz_hom}) vanishes on any special maximal bounded subgroup of $H(L)$ (not just on special parahoric subgroups).
\end{lem}

\begin{proof}
(1) This is Lemma 1.4.9 of \cite{Kottwitz84b} when $x^0$ is a hyperspecial point, and its proof continues to work in our setup: note that the parahoric group scheme over $\Zp^{\nr}$ attached to $x^0$ has connected special fiber, so the result of Greenberg \cite[Prop. 3]{Greenberg63} still applies.

(2) This is (stated and) proved in Lemma 3.3 of \cite{Kottwitz84a} when $H$ is unramified and $F$ is a local field, for the map $\lambda_H$ (recall that in such case, the Kottwitz's definition of $\lambda_H$ coincides with our definition and that $\lambda_H$ is the restriction of $v_{H_L}$ to $H(F)$). But again it is easily seen to hold in general for any quasi-split $H$, for the map $v_{H_L}$ and any special maximal bounded subgroup. More explicitly, the relation (3.3.4) in loc. cit. continues to hold for any quasi-split $H$ over $F$ or $L$, and one just need to note (in the last step of the original argument) that for any maximal split $L$-torus $S$ whose apartment contains a special point $x^0\in\mcB(G,L)$ and $T=\Cent_H(S)$, $T(L)_0:=T(L)\cap K\ (K:=\mathrm{Fix} (x^0))$ (the maximal bounded subgroup of $T(L)$) now equals $\Ker(v_{T})$ (\cite{HainesRapoport08}).  
\end{proof}

\begin{lem} \label{lem:delta_from_b}
There exists $c\in G(L)$ such that $\delta:=cb\sigma(c^{-1})$ lies in $G(L_n)$ and that 
$\Nm_{L_n/\Qp}\delta$ and $\gamma_0$ are conjugate under $G(\Qpnr)$.
\end{lem}

\begin{proof} 
By definition (cf. \ref{def:admissible_pair_(3)a}) and Lemma \ref{lem:Kottwitz84-a1.4.9_b3.3}, there exists $c'\in G(L)$ such that $c'(\epsilon'^{-1}F^n)c'^{-1}=\sigma^n$. If we define $\delta\in G(L)$ by 
\[\delta\rtimes\sigma:=c'Fc'^{-1}=c'(b\rtimes\sigma)c'^{-1},\] 
i.e. $\delta=c'b\sigma(c'^{-1})$, then it follows (as shown on p. 291 of \cite{Kottwitz84b}) that $\delta\in G(L_n)$ and $\sigma^n=c'(\epsilon'^{-1}F^n) c'^{-1}=c'\epsilon'^{-1}c'^{-1}(\delta\rtimes\sigma)^n=(c'\epsilon'^{-1}c'^{-1}\Nm_{L_n/\Qp}\delta)\rtimes\sigma^n$, i.e. $\Nm_{L_n/\Qp}\delta=c'\epsilon'c'^{-1}$, so that 
\begin{equation} \label{eqn:stable_conjugacy_reln_betwn_Nmdelta_and_gamma_0}
\Nm_{L_n/\Qp}\delta=c'\epsilon'c'^{-1}=c'u\epsilon (c'u)^{-1}=c'ug\gamma_0(c'ug)^{-1},
\end{equation}
where $\epsilon= g\gamma_0 g^{-1}$ for $g\in G(\Qb)$. Namely, $\Nm_{L_n/\Qp}\delta$ and $\gamma_0$ are conjugate under $G(\Qpb)$.
As both $\Nm_{L_n/\Qp}\delta$ and $\gamma_0$ belong to $G(\Qp)$, by the Steinberg's theorem ($H^1(\Qpnr, G(\gamma_0))=0$), they are conjugate under $G(\Qpnr)$.
\end{proof}

Choose $\delta\in G(L_n)$ as in Lemma \ref{lem:delta_from_b}: its $\sigma$-conjugacy class in $G(L_n)$ is uniquely determined by the given admissible pair $(\phi,\epsilon)$. 
We will say that the triple $(\gamma_0;(\gamma_l)_{l\neq p},\delta)$ \textit{corresponds to} the admissible pair $(\phi,\epsilon)$; then $\epsilon$ and $\gamma_0$ are stably conjugate, among others. Also note that by definition (cf. (\ref{subsubsec:cls})), the $\sigma$-conjugacy class $\mathrm{cls}(\phi)\in B(G_{\Qp})$ attached to $\phi(p)$ (the pull-back of $\phi$ to $\Gal(\Qpb/\Qp)$) is nothing but $\overline{\delta}\in B(G_{\Qp})$.

\begin{rem} \label{rem:two_different_b's}
Choose $c\in G(L)$ such that $\Nm_{L_n/\Qp}\delta=c\gamma_0 c^{-1}$. Then as
\[\delta^{-1}c\gamma_0 c^{-1}\delta=\sigma(\delta)\cdots\sigma^n(\delta)=\sigma(c\gamma_0 c^{-1})=\sigma(c)\gamma_0\sigma(c^{-1}),\]
we see that $b':=c^{-1}\delta\sigma(c)$ commutes with $\gamma_0$. According to \cite[Lem. 5.15]{LR87}, the $\sigma$-conjugacy of $b'\in G(\gamma_0)(L)$ is basic in $B(G(\gamma_0))$, where $G(\gamma_0)$ denotes the centralizer of $\gamma_0$ in $G$.
In general, this $b'$ (whose definition depends on $c\in G(L)$ wtih $\Nm_{L_n/\Qp}\delta=c\gamma_0 c^{-1}$) is different from the $b$ defined by $\xi_p^{\nr}(s_{\sigma})=b\rtimes\sigma$ (which, together with $\epsilon'$, was the input data used to produce $\delta$). But, when $(\phi,\epsilon)$ is \emph{well-located in a maximal $\Q$-torus $T$}, in the sense that $\epsilon,\phi(\delta_k)\in T(\Q)$ and $\phi$ factors through $\fG_T\subset\fG_G$, we can choose $c$ so that $b'=b$. Indeed, we can take $\epsilon$ for $\gamma_0$ (as $\epsilon\in G(\Q)$) and also we can choose $u$ from $T(\Qpb)$ (as $\phi(p)$ factors through $\fG_{T_{\Qp}}$) so that $\epsilon'=\epsilon$ (as $\epsilon\in T(\Qb)$). 
Then, we have $\Nm_{L_n/\Qp}\delta=c'\epsilon'c'^{-1}=c'\gamma_0c'^{-1}$ in (\ref{eqn:stable_conjugacy_reln_betwn_Nmdelta_and_gamma_0}) (note that $c'$ was determined by $\xi_p^{\nr}(s_{\sigma})=b\rtimes\sigma$ and $\epsilon'$, cf. proof of Lemma \ref{lem:delta_from_b}). In other words, we can take $c=c'$, so $b'=(c')^{-1}\delta\sigma(c')=b$.
\end{rem}

It is easy to see that equivalent admissible pairs give rise to equivalent Kottwitz triples. 
However, non-equivalent admissible pairs can also give equivalent Kottwitz triples; there is a cohomological expression for the number of non-equivalent admissible pairs corresponding to a given Kottwitz triple, cf. Thm. \ref{thm:LR-Satz5.25}.

We also remark that for a triple $(\gamma_0;(\gamma_l)_{l\neq p},\delta)$ constructed now from an admissible pair $(\phi,\epsilon)$, one can still define its Kottwitz invariant using the fact that $(\phi,\epsilon)$ is nested in a torus $(T,h)$, (cf. proof of Satz 5.25 in \cite{LR87}).

\begin{prop}
The triple $(\gamma_0;(\gamma_l)_{l\neq p},\delta)$ is a Kottwitz triple of level $n=[\kappa(\wp):\Fp]m$ (Def. \ref{defn:Kottwitz_triple}) whose Kottwitz invariant vanishes.
\end{prop}

\begin{proof} 
First, by construction, $\gamma_0\in G(\Q)$ lies in a maximal $\Q$-torus $T'$ which is elliptic at $\R$; this gives the condition (i). The condition (iii) follows from the condition (2) of Def. \ref{defn:admissible_morphism} (i.e. $X_l\neq\emptyset$). We have already seen that $\Nm_{L_n/\Qp}\delta$ is conjugate to $\gamma_0$ under $G(\Qpb)$. Therefore, the triple satisfies the conditions (i) - (iii) of Def. \ref{defn:Kottwitz_triple}. That the condition $(\ast(\delta))$ holds is a consequence of $\phi$ being admissible (specifically, the condition (3) of Def. \ref{defn:admissible_morphism}) in view of Theorem A of \cite{He15}.
Finally, it remains to verify vanishing of the Kottwitz invariant $\alpha(\gamma_0;(\gamma_l)_{l\neq p},\delta)$; as observed before (\cite[p.172]{Kottwitz90}), this will also establish the condition (iv) (alternatively, we can take $I_{\phi,\epsilon}$ for $I$ in the condition (iv), (\cite[Lem6.10]{Milne92})).
For this, we can assume that $(\phi,\epsilon)$ is nested in a special Shimura subdatum by Lemma \ref{lem:LR-Lemma5.23}. From this point, the original argument works almost word for word with all the necessary ingredients being established in our situation of parahoric levels.
\end{proof}


\section{Every admissible morphism is conjugate to a special admissible morphism}

We make some comments on various assumptions called upon in this section.
First, we remind ourselves the running assumption that $G_{\Qp}$ is quasi-split. In this section the level subgroup can be arbitrary parahoric subgroup. Also every statement involving the pseudo-motivic Galois gerb $\fP$ (instead of the quasi-motivic Galois gerb $\fQ$) will (for safety) assume that the Serre condition holds for the Shimura datum $(G,X)$ (i.e. the center $Z(G)$ splits over a CM field and the weight homomorphism $w_X$ is defined over $\Q$). At some places (specifically, Thm. \ref{thm:LR-Satz5.3}, (2), Prop. \ref{prop:equivalence_to_special_adimssible_morphism}, and Lemma \ref{lem:LR-Lemma5.23}), we require that $G^{\der}$ is simply-connected. Finally, when we prove non-emptiness of Newton strata, we need to assume that $G_{\Qp}$ splits over a cyclic (tamely ramified) extension of $\Qp$. Of course, in every statement we will make explicit the assumptions that we impose.

\subsection{Transfer of maximal tori and strategy of proof of Theorem \ref{thm:LR-Satz5.3}}

\subsubsection{}
In the work of Langlands-Rapoport, a critical role is played by the notion of \textit{transfer} (or \textit{admissible embedding}) \textit{of maximal torus}. Recall (\cite[$\S$9]{Kottwitz84b}) that for a connected reductive group $G$ over a field $F$, if $\psi:H_{\overline{F}}\isom G_{\overline{F}}$ is an inner-twisting and $T$ is a maximal $F$-torus of $H$, an $F$-embedding $i:T\rightarrow G$ is called \textit{admissible} if it is of the form $\Int g\circ\psi|_T$ for some $g\in G(\overline{F})$ (equiv. of the form $\psi\circ \Int  h|_T$ for some $h\in H(\overline{F})$). Whether an $F$-embedding $i:T\rightarrow G$ is admissible or not depends only on the conjugacy class of the inner-twisting $\psi$. We will also say that $T$ \textit{transfers to} $G$ (with a conjugacy class of inner twistings $\psi:H_{\overline{F}}\isom G_{\overline{F}}$ understood) if there exists an admissible $F$-embedding $T\rightarrow G$ (with respect to the same conjugacy class of inner twistings). Usually, this notion is considered when $H$ is a quasi-split inner form of $G$, due to the well-known fact (\cite[p.340]{PR94}) that \textit{every maximal torus in a reductive group transfers to its quasi-split inner form}, but here we do not necessarily restrict ourselves to such cases. In fact, we will consider more often the identity inner twisting $\mathrm{Id}_G:G_{\overline{F}}=G_{\overline{F}}$, in which case if $\Int g:T_{\overline{F}}\hra G_{\overline{F}}\stackrel{=}{\rightarrow}G_{\overline{F}}$ is an admissible embedding, we will also let $\Int g$ denote both the $\Q$-isomorphism $T\rightarrow \Int g(T)$ of $\Q$-tori and the induced morphism of Galois gerbs $\fG_T\rightarrow \fG_{\Int g(T)}$; of course the latter is also the restriction of $\Int g$, regarded as an automorphism of the Galois gerb $\fG_G$, to $\fG_T$.

\subsubsection{} 
We say that an admissible morphism $\phi:\fP\rightarrow\fG_G$ is \textit{well-located}
\footnote{This is our translation of the German word \textit{g\"unstig gelegen} used by Langlands-Rapoport \cite[p.190, line8]{LR87}.}
if $\phi(\delta_n)\in G(\Q)$ for all sufficiently large $n$. Here, $\delta_n$ is the element in $P(K,m)(\Q)$ introduced in Lemma \ref{lem:Reimann97-B2.3}, (2), for any $\fP(K,m)$ which $\phi$ factors through: $\phi(\delta_n)$ does not depend on the choice of such $\fP(K,m)$ by the compatibility property of $\delta_n$ (loc.cit.). More generally, for a $\Q$-subgroup $H$ of $G$, we will say that an admissible morphism $\phi:\fP\rightarrow\fG_G$ is \textit{well-located in $H$} if $\phi(\delta_n)\in H(\Q)$ (for all sufficiently large $n$) and $\phi$ maps into $\fG_H (\hra \fG_G)$; note the second additional requirement. For example, every special admissible morphism $i\circ\psi_{T,\mu_h}$ (for a special Shimura sub-datum $(T,h)$) is always well-located (in $T$). Indeed, $\delta_n\in P(K,m)(\Q)$ (for sufficiently large $n$ divisible by $m$) and the restriction of $i\circ\psi_{T,\mu_h}$ to kernels is defined over $\Q$ (\cite[p.143, second paragraph]{LR87}). 

For a well-located morphism $\phi:\fP\rightarrow\fG_G$, we denote by 
\[I:=\Cent_G(\phi(\delta_n))\] 
the centralizer $\Q$-group of $\phi(\delta_n)\in G(\Q)$ (for any sufficiently large $n$: $I$ does not depend on the choice of $n$), and by $I_{\phi}$ the inner twisting of $I$ defined by $\phi$. More precisely, if $\phi(q_{\rho})=g_{\rho}\rtimes\rho$ for the chosen section $\rho\mapsto q_{\rho}$ to $\fP\rightarrow\Gal(\Qb/\Q)$ (Remark \ref{rem:comments_on_zeta_v}, (3)), the map  $(\rho\mapsto g_{\rho})\in Z^1(\Q,G)$ is a cocycle defining $I_{\phi}$, namely, there is a natural inner twisting $\Qb$-isomorphism
\begin{equation} \label{eqn:inner-twisting}
\psi:I_{\Qb}\isom (I_{\phi})_{\Qb},
\end{equation}
under which
\[I_{\phi}(\Q)=\{g\in I(\Qb)\ |\ g_{\rho}\rho(g)g_{\rho}^{-1}=g\}=\{g\in G(\Qb)\ |\ \Inn(g)\circ\phi=\phi\}.\]

We will also say that an admissible pair $(\phi,\epsilon)$ is \textit{well-located} if $\phi$ is a well-located admissible morphism and $\epsilon\in I_{\phi}(\Q)$ lies in $G(\Q)$ (a priori, $\epsilon\in I_{\phi}(\Q)$ is only an element of $G(\Qb)$ via $I_{\phi}(\Q)\subset I_{\phi}(\Qb)=I(\Qb)\subset G(\Qb)$). If $\phi$ is well-located in $H$ and $\epsilon\in I_{\phi}(\Q)\cap H(\Q)$ for a $\Q$-subgroup $H$ of $G$, the admissible pair $(\phi,\epsilon)$ will be said to be \textit{well-located in $H$}. Clearly, any admissible pair $(\phi,\epsilon)$ nested in some special Shimura datum $(T,h)$ is well-located (in $T$). 

We also note that for any admissible morphism $\phi$ mapping into $\fG_T$ for a maximal $\Q$-torus $T$, we have that $T$ is also a $\Q$-subgroup of $I_{\phi}$ (so $T(\Q)\subset I_{\phi}(\Q)=\{g\in G(\Qb)\ |\ \Inn(g)\circ\phi=\phi\}$). Indeed, suppose that $\phi:\fP\rightarrow \fG_T$ factors through $\fP(K,m)$ for a CM field $K$ Galois over $\Q$ and $m\in\N$. We need to check that $\epsilon$ commutes with $\phi(\delta_n)$ and $\phi(q_{\rho})$ (for all $n\gg 1$ and $\rho\in\Gal(\Qb/\Q)$). But, the first is obvious since both $\epsilon$,$\phi(\delta_n)$ belong to $T(\Qb)(\subset \fG_T=T(\Qb)\rtimes\Gal(\Qb/\Q))$, and for the second, in general, for $x\in G(\Qb)$, we have
\[\phi(q_{\rho})x\phi(q_{\rho})^{-1}=(g_{\rho}\rtimes\rho)x(g_{\rho}\rtimes\rho)^{-1}=(g_{\rho}\rtimes\rho)x(\rho^{-1}(g_{\rho}^{-1})\rtimes\rho^{-1})=g_{\rho}\rho(x)g_{\rho}^{-1}\rtimes1.\]
So, $\epsilon\in T(\Q)$ and $g_{\rho}\in T(\Qb)$ together imply that $\epsilon\in I_{\phi}(\Q)$.

\begin{thm} \label{thm:LR-Satz5.3}
Assume that $G_{\Qp}$ is quasi-split, and that the Serre condition for $(G,X)$ holds. Let $\mbfK_p$ be a parahoric subgroup of $G(\Qp)$. 

(1) Every admissible morphism $\phi:\fP\rightarrow \fG_G$ is conjugate to a special admissible morphism $i\circ\psi_{T,\mu_h}:\fP\rightarrow \fG_T\hra\fG_G$ (for some special Shimura sub-datum $(T,h\in\Hom(\dS,T_{\R})\cap X)$ and $i:\fG_T\rightarrow\fG_G$ the canonical morphism defined by the inclusion $i:T\hra G$).

(2) Assume that $G^{\der}$ is simply connected. 
Let $\phi:\fP\rightarrow \fG_G$ be a \emph{well-located} (i.e. $\phi(\delta_n)\in G(\Q)$) admissible morphism. Then, for any maximal torus $T$ of $G$, elliptic at $\R$, such that
\begin{itemize}
\item[(i)] $\phi(\delta_n)\in T(\Q)$ for a sufficiently large $n$ (i.e. $T\subset I:=\Cent_G(\phi(\delta_n))$), and
\item[(ii)] $T_{\Qp}\subset I_{\Qp}$ is elliptic,
\end{itemize}
$\phi$ is conjugate to a special admissible morphism $\psi_{T',\mu_{h'}}$, where $T'=\Int g'(T)$ for some transfer of maximal torus $\Int g':T\hra G\ (g'\in G(\Qb))$ (with respect to the identity inner twisting, i.e. the composite map $T_{\Qb}\hra G_{\Qb}\stackrel{\Int  g'}{\lra}G_{\Qb}$ is defined over $\Q$).

\begin{itemize}
\item[(iii)] If furthermore $T_{\Ql}\subset G_{\Ql}$ is elliptic at some prime $l\neq p$,
\end{itemize}
we may find such transfer of maximal torus $\Int g':T\hra G$ which is also conjugation by an element of $G(\Qp)$ (i.e. $\Int g':T_{\Qpb}\rightarrow G_{\Qpb}$ equals $\Int y$ for some $y\in G(\Qp)$).
\end{thm}

The first statement, in the hyperspecial level case, is Satz 5.3 of \cite{LR87}. It does not offer any information on the special admissible morphisms that are conjugate to given admissible morphism, especially on the tori $T$. In contrast, the point of the second statement, which is due to us in any level case, is to shed light on such special admissible morphisms, focusing on the tori in the special Shimura subdata. 

To establish these statements, we will adapt the arguments from the proof of loc. cit., which we now review briefly. It can be divided into three steps: 
\begin{itemize}
\item[I.] The first step is to replace a given admissible morphism $\phi$ by an equivalent one which is well-located (i.e. when we denote the conjugate again by $\phi$, we have $\phi(\delta_n)\in G(\Q)$).
\item[II.] The second step is to show that there is a conjugate $\Int  g\circ\phi:\fP\rightarrow\fG_G\ (g\in G(\Qb))$ of $\phi$ in the step I which factors through $\fG_T$ for \emph{some} maximal $\Q$-torus $T$ (of $G$) that is elliptic over $\R$.
Again, let us denote the conjugate $\Int  g\circ\phi$ by $\phi$. 
\item[III.]
The final step is to find an admissible embedding $\Int g':T \hra G$ (with respect to the identity inner twisting $\mathrm{id}:G_{\Qb}=G_{\Qb}$) with $g'\in I(\Qb)$
such that $\Int  g'\circ\phi:\fP\rightarrow \fG_{\Int  g'(T)}$ becomes a special admissible morphism. 
\end{itemize}

We point out that while the first two steps are established by arguments in Galois cohomology which do not use level subgroups, we need to validate the third step for parahoric level subgroups: see Lemma \ref{lem:LR-Lemma5.11} below. For the finer claim (2), it is also necessary to strengthen the second and the third steps. 
For that purpose (and some other applications as well), we formalize (with some improvements incorporated) these steps into two propositions respectively: Prop. \ref{prop:existence_of_admissible_morphism_factoring_thru_given_maximal_torus}, \ref{prop:equivalence_to_special_adimssible_morphism}. We first introduce these propositions, postponing their proofs to the next subsection.

\begin{prop} \label{prop:existence_of_admissible_morphism_factoring_thru_given_maximal_torus}
Let $\phi:\fP\rightarrow\fG_G$ be an admissible morphism and $T$ a maximal $\Q$-torus of $G$, elliptic over $\R$, fulfilling the properties (i) and (ii) of Thm. \ref{thm:LR-Satz5.3}, (2). Then, there exists $g\in I(\Qb)$ such that $\Int  g\circ\phi$ maps $\fP$ into $\fG_T(\hra \fG_G)$. 
\end{prop}

Notice that here we only modify $\phi$ (while keeping $T$).
This proposition can be regarded as a soup-up version of the second step above, in the following sense. In the original second step (i.e. \cite{LR87}, p. 176, from line 1 to -5), we are given an admissible morphism $\phi$ and want to find a maximal $\Q$-torus $T$ such that some conjugate of $\phi$ maps into $\fG_{T}\subset \fG_G$. There, $\phi$ is considered somewhat as given and one looks for $T$ with this property. As such, choice of $T$ is restricted by $\phi$, namely, for arbitrary $\phi$ there is very little room of choice for such $T$. In our case, however, we start with some fixed torus $T$ (notably, the tori $T'$ as in Thm. \ref{thm:LR-Satz5.3}, (2) and (3)), and demand that a conjugate of $\phi$ factors through $\fG_T$. It turns out that this becomes possible if $T$ satisfies the properties (i) and (ii) of Thm. \ref{thm:LR-Satz5.3}, (2).

Also we note that the new pair $(\Int g\circ\phi,T)$ still enjoys the properties (i) and (ii),
since $\Int g\circ\phi(\delta_n)=g\phi(\delta_n)g^{-1}=\phi(\delta_n)$ (as $g\in I(\Qb)$). 

The next proposition is also an enhanced version of the third step above.

\begin{prop} \label{prop:equivalence_to_special_adimssible_morphism}
Assume that $G^{\der}$ is simply connected.
If an admissible morphism $\phi:\fP\rightarrow\fG_G$ is well-located in a maximal $\Q$-torus $T$ of $G$ that is elliptic over $\R$, there exists an admissible embedding of maximal torus $\Int g'|_{T}:T\hra G$ such that $\Int  g'\circ\phi:\fP\rightarrow\fG_{T'}$ is special, where $T':=\Int  g'(T)$ (a $\Q$-torus), i.e. $\Int  g'\circ\phi=\psi_{T',\mu_{h'}}$ for some $h'\in X\cap\Hom(\dS,T'_{\R})$.

If $T$ further fulfills the condition (iii) of Thm. \ref{thm:LR-Satz5.3}, one can find such transfer of maximal torus $\Int g'$ (i.e. $\Int g'\circ\phi$ becomes special admissible) which also satisfies that
\begin{itemize}
\item $\Int g'|_{T_{\Qpb}}:T_{\Qpb}\hra G_{\Qpb}$ equals $\Int y|_{T_{\Qp}}:T_{\Qp}\hra G_{\Qp}$ for some $y\in G(\Qp)$. 
\end{itemize}
In particular, $T_{\Qp}$ and $T'_{\Qp}$ are conjugate under $G(\Qp)$.
\end{prop}

The first statement of this proposition is proved in \cite{LR87}, for hyperspecial $\mbfK_p$, in the course of proving Satz 5.3 (more precisely, in the part beginning from Lemma 5.11 until the end of the proof of Satz 5.3). We remark that the idea of exploiting a transfer of maximal torus $i:T\isom T'$ which becomes $G(\Qp)$-conjugacy first appeared in our previous work (\cite[Thm. 4.1.1]{Lee14}), and will find similar application here later.

Finally, we point out that the three properties (i) - (iii) of Theorem \ref{thm:LR-Satz5.3} remain intact under any transfer of maximal torus. Indeed, an inner automorphism does not interfere with the center and a transfer of maximal torus $i:T\rightarrow G$ also restricts to a $\Q$-isomorphism $i:Z(\Cent_{G}(t))\isom Z(\Cent_{G}(i(t)))$ for any $t\in T(\Q)$.

In the next subsection, we present the proofs of Prop.
\ref{prop:existence_of_admissible_morphism_factoring_thru_given_maximal_torus}, \ref{prop:equivalence_to_special_adimssible_morphism}, and Theorem \ref{thm:LR-Satz5.3}.
\footnote{The original arguments in \cite{LR87} use the quasi-motivic Galois gerb $\fQ$ instead of the pseudo-motivic Galois gerb $\fP$, but as we mentioned before, the definition of $\fQ$ given there is wrong. However, the whole arguments remain valid with $\fQ$ replaced by $\fP$, as long as the (admissible) morphisms in question factor through $\fP$.}

\subsection{Proof of Propositions \ref{prop:existence_of_admissible_morphism_factoring_thru_given_maximal_torus},  \ref{prop:equivalence_to_special_adimssible_morphism}, and Theorem \ref{thm:LR-Satz5.3}}

Recall that we have fixed a continuous section $\rho\mapsto q_{\rho}$ to the projection $\fP\thra\Gal(\Qb/\Q)$ (Remark \ref{rem:comments_on_zeta_v}, (3)).

\begin{lem} \label{lem:criterion_for_admissible_morphism_to_land_in_torus}
Let $\phi:\fP\rightarrow\fG_G$ be a well-located admissible morphism and $T$ a maximal $\Q$-torus of $G$. Let $I=\Cent_G(\phi(\delta_n))$ and $I_{\phi}$ the inner twist of $I$ (\ref{eqn:inner-twisting}).

(1) Suppose that there exist a $\Q$-subgroup $T'$ of $I_{\phi}$ and $a\in G(\Qb)$ such that
$\Int a^{-1}:T_{\Qb}\hra G_{\Qb}$ is a $\Q$-isomorphism from $T$ to $T'$, where $T'\subset I_{\phi}$ is considered as a $\Qb$-subgroup of $G_{\Qb}$ via the inner twisting $I_{\Qb}\isom (I_{\phi})_{\Qb}$ (\ref{eqn:inner-twisting}). Then, $\Int  a\circ\phi$ maps $\fP$ to $\fG_T$. 

(2) If $T$ is a $\Q$-subgroup of $I$ and $a\in I(\Qb)$, then $\Int  a\circ\phi$ maps $\fP$ to $\fG_T$ if and only if $T_{\Qb}\hookrightarrow I_{\Qb}\stackrel{\Int  a^{-1}}{\lra}I_{\Qb}\stackrel{\psi}{\rightarrow} (I_{\phi})_{\Qb}$ is defined over $\Q$ (i.e., $\psi\circ\Int  a^{-1}|_T:T\hra I_{\phi}$ is a transfer of maximal torus with respect to the inner twisting $I_{\Qb}\isom (I_{\phi})_{\Qb}$). 

In other words, a conjugate of $\phi$ maps into $\fG_T$ if and only if $T\subset I$ transfers to $I_{\phi}$ with respect to the inner twisting $I_{\Qb}\isom (I_{\phi})_{\Qb}$ (\ref{eqn:inner-twisting}). 
\end{lem}

\begin{proof} (1) First, since $\psi:I_{\Qb}\isom (I_{\phi})_{\Qb}$ is an inner twist, it restricts to a $\Q$-isomorphism $Z(I)\isom Z(I_{\phi})$. So, if $\phi(\delta_n)\in G(\Q)$, $\phi(\delta_n)\in Z(I)(\Q)=Z(I_{\phi})(\Q)\subset T'(\Q)$, that is, $\Int a\circ\phi(\delta_n)\in  T(\Q)$. This is equivalent to that $(\Int a\circ\phi)^{\Delta}$ maps $P$ to $T$, as $\{\delta_n^k\}_{k\in\N}$ is Zariski-dense in $P(K,n)$ (for any suitable CM field $K$) (Lemma \ref{lem:Reimann97-B2.3} (2)).
Next, via $\psi$ we identify $I_{\phi}(\Qb)$ with $I(\Qb)\subset G(\Qb)$ endowed with the twisted Galois action $g\mapsto g_{\rho}\rho(g)g_{\rho}^{-1}$, where $g\mapsto\rho(g)$ is the original Galois action on $G(\Qb)$. 
Let $\phi(q_{\rho})=g_{\rho}\rtimes\rho$. Then, the condition means that 
\begin{equation} \label{eqn:transfer_from_I_to_I_phi}
g_{\rho}\rho(a^{-1}ta)g_{\rho}^{-1}=a^{-1}\rho(t)a
\end{equation}
for all $t\in T(\Qb)$, which is the same as that $ag_{\rho}\rho(a)^{-1}\in T(\Qb)$, as $\Cent_G(T)=T$. As $a\phi a^{-1}(q_{\rho})=ag_{\rho}\rho(a)^{-1}\rtimes \rho$, this implies the assertion.

(2) This is similar to (1). 
\end{proof}

\subsubsection{Proof of Proposition \ref{prop:existence_of_admissible_morphism_factoring_thru_given_maximal_torus}.}

Let $I=\Cent_G(\phi(\delta_n))$ and $\psi:I_{\Qb}\isom (I_{\phi})_{\Qb}$ the inner twisting defined by $\phi$ (\ref{eqn:inner-twisting}). By Lemma \ref{lem:criterion_for_admissible_morphism_to_land_in_torus}, (2), it suffices to prove that $T$ transfers to $I_{\phi}$ (with respect to the conjugacy class of $\psi$). For this we choose a quasi-split inner-twisting $\psi^{\ast}:I^{\ast}_{\Qb} \rightarrow I_{\Qb}$, where $I^{\ast}$ is the quasi-split inner form of $I$. Then, it is well-known (\cite[p.340]{PR94}) that any maximal torus of $I$ (in particular, $T$) always transfers to the quasi-split inner form $I^{\ast}$. So by replacing $\psi^{\ast}$ by a conjugate of it, we may assume that $(\psi^{\ast})^{-1}|_T:T_{\Qb}\rightarrow I^{\ast}_{\Qb}$ is defined over $\Q$; let $T^{\ast}$ be its image (so a maximal $\Q$-subtorus of $I^{\ast}$). Obviously, to prove Proposition \ref{prop:existence_of_admissible_morphism_factoring_thru_given_maximal_torus}, we only need to show that $T^{\ast}$ transfers to $I_{\phi}$ (with respect to the inner twisting $\psi\circ\psi^{\ast}:I^{\ast}_{\Qb}\isom (I_{\phi})_{\Qb}$). 

First, note that all of $I$, $I_{\phi}$, and $I^{\ast}$, being inner forms of each other, share the same center, and by construction $T^{\ast}\approx T$ over $\Q$. So, for $v=\infty$ and $p$, $T^{\ast}_{\Qv}$ is an elliptic maximal torus of $I^{\ast}_{\Qv}$. Indeed, as $T_{\Qv}$ is an elliptic maximal torus of $I_{\Qv}$ (which is the same as that $Z(I_{\Qv})$ contains the maximal split $\Qv$-subtorus of $T_{\Qv}$), $Z(I^{\ast}_{\Qv})$ contains a split $\Qv$-torus of same rank, which then must be the maximal split subtorus of $T^{\ast}_{\Qv}\approx T_{\Qv}$
Therefore, according to \cite[Lemma 5.6]{LR87}, it suffices to check that $T^{\ast}$ transfers to $I_{\phi}$ everywhere locally. At $v=\infty, p$, this already follows from the fact that $T^{\ast}_{\Qv}$ is an elliptic torus of $I^{\ast}_{\Qv}$ (\cite[Lemma 5.8, 5.9]{LR87}). At a finite place $v\neq p$, since $\phi\circ\zeta_v$ is conjugate to the canonical trivialization $\xi_v:\fG_v\rightarrow \fG_G(v)$ (Def. \ref{defn:admissible_morphism}, (2)), the inner-twisting $\psi_{\Qlb}:I_{\Qvb}\isom (I_{\phi})_{\Qvb}$ (via fixed embedding $\Qb\hra\Qvb$) descends to an $\Qv$-isomorphism 
\[I_{\Qv}\isom (I_{\phi})_{\Qv}.\]
But, by our assumption that $\phi(\delta_n)\in T(\Q)$, $T$ is a subgroup of $I$. Therefore, $T^{\ast}_{\Qv}\approx T_{\Qv}$ transfers to $(I_{\phi})_{\Qv}$ with respect to the conjugacy class of $\psi_{\Qvb}$. \end{proof}

As was explained after the statement of Theorem \ref{thm:LR-Satz5.3}, Proposition \ref{prop:equivalence_to_special_adimssible_morphism} is a strengthening of the third step in the proof of Satz 5.3 in \cite{LR87}. 
The proof of this step in loc. cit. itself proceeds in three steps: Lemma 5.11, Lemma 5.12, and the rest of the proof of Satz 5.3 (p.181, line 1-19 of \cite{LR87}). 
Again we will prove our proposition along the same line. First, we need some facts from the Bruhat-Tits theory.

\begin{lem} \label{lem:specaial_parahoric_in_Levi}
(1) Let $G$ be a (connected) reductive group over a field $F$. Then, for any $F$-split $F$-torus $A_M$ in $G$, its centralizer $M:=\Cent_G(A_M)$ is an $F$-Levi subgroup of $G$ (i.e. a Levi factor defined over $F$ of an $F$-parabolic subgroup of $G$). If $G$ is quasi-split, then so is $M$. 

From now on, we suppose that $F$ is a complete discrete valued field with perfect residue field (mainly, local fields or $\mfk=\mathrm{Frac}(W(\Fpb))$), and $G$ a connected reductive group over $F$.
As before, let $M=\Cent_G(A_M)$ for a split $F$-torus $A_M$, and fix a maximal split $F$-torus $S$ of $G$ containing $A_M$. Let $\mcA^G=\mcA^G(S,F)$ and $\mcA^M=\mcA^M(S,F)$ denote respectively the apartments in the buildings $\mcB(G,F)$ and $\mcB(M,F)$ corresponding to $S$.

(2) Every affine root $\alpha$ of $\mcA^G$ whose vector part $a=v(\alpha)\in \Phi(G,S)$ is a root in $\Phi(M,S)$ is also an affine root of $\mcA^M$.

(3) For any special maximal parahoric subgroup $K$ of $G(F)$ associated with a special point in $\mcA^G$, the intersection $K\cap M(F)$ is also a special maximal parahoric subgroup of $M(F)$. 
\end{lem}

\begin{proof} (1) These are standard. For the first claim, see \cite[Thm. 4.15]{BT65}.
The second claim is easily seen. We use the fact that for a (connected) reductive group $H$ over a field $F$, $H$ is quasi-split if and only if for a (equiv. any) maximal $F$-split torus of $H$, its centralizer in $H$ is a (maximal) torus (cf. proof of Prop. 16.2.2 of \cite{Springer98}). Now, as any torus containing $A_M$ is a subgroup of $M=\Cent_{G}(A_M)$, so is any maximal $F$-split torus of $G$ containing $A_M$; choose one and call it $S$. As $G$ is quasi-split, the centralizer $T:=\Cent_{G}(S)$ of $S$ is a torus of $G$, thus is itself contained in $M$. But, then $T$ is also the centralizer of $S$ in $M$. 

(2) This also follows readily from definition. First, we recall that the relative root datum $\Phi(M,S)=(X_{\ast}(S),R_{\ast}(M),X^{\ast}(S),R^{\ast}(M))$ for $(M,S)$ is a closed sub-datum of the root datum $\Phi(G,S)=(X_{\ast}(S),R_{\ast}(G),X^{\ast}(S),R^{\ast}(G))$ defined by a subset $I$ of the set $\Delta=\{a_1,\cdots,a_n\}$ of simple roots (for some ordering on $R^{\ast}(G)$) (\cite[Thm. 4.15]{BT65}): 
\begin{equation} \label{eqn:basis_of_simple_roots_for_Levi}
R^{\ast}(M)=R^{\ast}(G)\cap \sum_{a_i\in I}\Z a_i,
\end{equation} 
and $A_M=(\cap_{\alpha\in I}\Ker(\alpha))^0$ (the largest split $F$-torus in the center $Z(M)$). 
Next, for an affine function $\alpha$ on $\mathcal{A}(S,F)\cong X_{\ast}(S)_{\R}$ (regarded as a common affine space without any apartment structure) whose vector part belongs to $\Phi$, let $X_{\alpha}^G$ be defined as in \cite[1.4]{Tits79} with respect to $G$, i.e.
\[X_{\alpha}^G=\{u\in U_{v(\alpha)}(F)\ |\ u=1\text{ or }\alpha(v(\alpha),u)\geq \alpha\}.\]
Here, for $a\in R^{\ast}(G,S)$, $U_a$ refers to the associated root group. This an unipotent $F$-group, which was denoted by ${}_FU_a$ or $U_{(a)}$ in \cite{BT65}, 5.2.
\footnote{In turn, this is the group that was denoted by $G^{\ast (S)}_{(a)}$ (or $G^{\ast}_{(a)}$) in loc. cit. 3.8. Namely, when we choose  a maximal $\overline{F}$-torus $T$ of $G_{\overline{F}}$ containing $S$, it is the group generated by ${}_{\overline{F}}U_b$ (the ``absolute" root group in $G_{\overline{F}}$ defined with respect to $(G_{\overline{F}},T)$) for the absolute roots in $R^{\ast}(G_{\overline{F}},T)$ whose restriction to $S$ belong to $(a)$, the set of relative roots in $R^{\ast}(G,S)$ that are positive integer multiples of $a$.}
When $a\in R^{\ast}(M,S)$, as $U_{a}\subset M$ for $a\in R^{\ast}(M)$, it follows from definition that this $U_a$ is the same group as that defined regarding $a$ as a root for $(M,S)$. Similarly, if $v(\alpha)\in R^{\ast}(M)$, the same is also true of the affine function $\alpha(v(\alpha),u)$. In more detail, its definition uses only the properties of \textit{root (group) datum} (of type some root system) in the sense of \cite[6.1]{BT72}. For $S$ and $U_{a}\ (a\in R^{\ast}(G,S))$ as above, there exist certain $S$-right cosets $\{M_a\}_{a\in R^{\ast}(G,S)}$ such that the family of subgroups 
\[\{S,\{U_{a},M_{a}\}_{a\in \Phi(G,S)}\}\] 
becomes a root group datum of type $\Phi(G,S)$ (in $G$) (in fact $M_a$ is then a subset of the group generated by $\{S,U_{a},U_{-a}\}$, cf. \cite[(6.1.2), (9)]{BT72}). In particular, the element $m(u)$ (for each $u\in U_a(F)\backslash\{1\}$) appearing in \cite[1.4]{Tits79} belongs to $M_a$ and is determined solely by the root group datum $\{S,\{U_{a},M_{a}\}_{a\in \Phi(G,S)}\}$ (\cite[(6.1.2), (2)]{BT72}). But, as $\Phi(M,S)$ is \textit{(quasi-)closed} in $\Phi(G,S)$ and $U_{a}\subset M$ for $a\in R^{\ast}(M)$, the subset $\{T,\{U_{\alpha},M_{\alpha}\}_{\alpha\in \Phi(M,S)}\}$ is also a root datum of type $\Phi(M,S)$ (in $G$), cf. \cite[7.6]{BT72}. Hence we can drop the superscript $G$ in $X_{\alpha}^G$ without ambiguity.
Now, we recall the definition (\cite[1.6]{Tits79}) that an affine function $\alpha$ is an \textit{affine root} of $G$ (relative to $S$ and $F$) if $X_{\alpha}$ is not contained in $X_{\alpha+\epsilon}\cdot U_{2v(\alpha)}\ (=X_{\alpha+\epsilon}\text{ if }2v(\alpha)\notin\Phi$) for any strictly positive constant $\epsilon$. The claim in question is immediate from this definition and the above discussions. 

(3) It is shown in \cite[Lemma  4.1.1]{HainesRostami10} that $K\cap M(F)$ is a parahoric subgroup of $M(F)$. So we just have to show that it is special maximal parahoric. 
Using the special point $\mbfo\in \mcA^G(S,F)$, we may embed $\mcB(M,F)$ into $\mcB(G,F)$ such that $\mbfo$ lies in the image (\cite[7.6.4]{BT72}, \cite[4.2.17-18]{BT84}). Let $\mcA((M^{\der}\cap S)^0,F)$ be the apartment corresponding to the maximal split $F$-torus $(M^{\der}\cap S)^0$ of $M^{\der}$.
As the affine hyperplanes in $\mcA^{M}$ form a subset of those in $\mcA^{G}$, it is obvious that $\mbfo$ is contained in a unique facet $\mbfa_{\mbfo}^{M}$ in $\mcA^{M}$, i.e. in
\[\mbfa^{M}_{\mbfo}\cong X_{\ast}(Z(M))_{\R}\times \{v_{\mbfo}\}.\]
for some unique facet $v_{\mbfo}$ in the apartment $\mcA((M^{\der}\cap S)^0,F)$ (recall that $M$ is the centralizer of a split torus $Z(M)$ which then must be the center). 
Now, we claim that $v_{\mbfo}$ is a vertex of $\mcA((M^{\der}\cap S)^0,F)$.
Indeed, as $\mbfo$ is a special point, we may identify the affine space $\mcA^{G}$ with the vector space $X_{\ast}(S)_{\R}$ ($\mbfo$ becoming the origin) so that the root hyperplanes $\{H_{\alpha}\}_{\alpha\in R^{\ast}(G,S)}$ are all affine hyperplanes. Clearly, $\mbfo=\cap_{\alpha\in R^{\ast}(G,S)}H_{\alpha}\cap X_{\ast}((G^{\der}\cap S)^0)_{\R} $. Let $I\subset \Delta$ be the subset defining the root datum $\Phi(M,S)$ as in (2); so, the center $Z(M)$ is equasl to $\cap_{\alpha\in I}H_{\alpha}$ (intersection of root hyperplanes in $X_{\ast}(S)_{\R}$).
But, by (2), $\{H_{\alpha}\}_{\alpha\in I}$ is also a subset of affine hyperplanes in the apartment $\mcA^M$ of $(M,S)$, and $\mbfo$ is contained in the intersection of these linearly independent affine hyperplanes in $\mcA(M,S)$, whose dimension is thus equal to rank $r_M$ of $Z(M)$. Hence, $\mbfo$ is contained in a facet of $\mcA^M$ of dimension at most $r_M$, which implies that the facet $v_{\mbfo}$ is of zero-dimension, i.e. a vertex in the building for $M^{\der}$.

Once we know that $v_{\mbfo}$ is a vertex, the fact that it is a special vertex also follows readily from (2). Indeed, 
by definition (\cite[1.9]{Tits79}), we need to check that every root of $a\in\Phi(M,S)$ is the vector part of an affine roof of $(M,S)$ vanishing at $v_{\mbfo}(\in \mcA^{M^{\der}})$. We know that $a$ also belongs to $\Phi(G,S)$, thus since $\mbfo$ is special, there exists an affine root $\alpha$ of $(G,S)$ with vector part $v(\alpha)=a$ and vanishing at $\mbfo$. But, by (2), such $\alpha$ is also an affine root of $\mcA^M$, and as such, it must vanish on the facet in $\mcA^M$ containing $\mbfo$, i.e. on $\mbfa^M_{\mbfo}\cong X_{\ast}(Z(M))_{\R}\times \{v_{\mbfo}\}$. \end{proof}

\begin{lem} \label{lem:LR-Lemma5.11}
Let $T_1\subset G_{\Qp}$ be a maximal $\Qp$-torus, split by a finite Galois extension $K$ of $\Qp$, $b\in T_1(\mfk)$, and $\{\mu\}$ a $G(\Qpb)$-conjugacy class of minuscule cocharacters of $G_{\Qpb}$. Let $\mbfK_p$ be a (not necessarily special maximal) parahoric subgroup of $G(\Qp)$.
If $X(\{\mu\},b)_{\mbfK_p}\neq\emptyset$, there exists $\mu\in X_{\ast}(T_1)\cap \{\mu\}$ such that
\begin{equation} \label{eq:equality_on_the_kernel}
\Nm_{K/\Qp}\mu=[K:\Qp]\nu_b,
\end{equation}
where $\nu_b\in X_{\ast}(T_1)_{\Q}$ is the Newton homomorphism attached to $b$.

In particular, if $\phi$ is an admissible morphism well-located in a maximal $\Q$-torus $T$ of $G$ that is elliptic over $\R$,  there exists a $\mu\in X_{\ast}(T)\cap\{\mu_X\}$ such that $\phi$ and $\psi_{T,\mu}$ coincide on the kernel of $\fP$. 
\end{lem}

Here, $i\circ\psi_{T,\mu}$ is not necessarily admissible, because $\mu$ may not be $\mu_{h'}$ for some $h'\in X$. 

\begin{proof} This is proved in Lemma 5.11 of \cite{LR87} when $\mbfK_p$ is a hyperspecial subgroup. 
We will adapt its argument for a general parahoric subgroup. 
We first show how the first statement implies the second one.
Since the kernel of the Galois gerb $\fP$ is the projective limit of $P(L,m)(\Qb)$, where $L$ runs through CM Galois extensions of $\Q$ and $m\in\N$ varies with respect to divisibility (cf. Subsubsec. \ref{subsubsec:pseudo-motivic_Galois_gerb}), we only need to show it after restricting $\phi$, $\psi_{T,\mu}$ to $P(L,m)$ (for all sufficiently large Galois CM field $L$ and $m$). Let $L$ be a CM-field splitting $T$. Then, the Galois-gerb morphisms $\psi_{T,\mu}:\fP\rightarrow \fG_T$ and $\zeta_p:\fG_p\rightarrow\fP(p)$ factor through $\fP^L$ and $\fG_p^{L_{v_2}}$, respectively (Lemma \ref{lem:defn_of_psi_T,mu}, (2) and Remark \ref{rem:comments_on_zeta_v}), where as usual $v_2$ denotes (by abuse of notation) the place of $L$ induced by the fixed embedding $L\hra \Qb$. Let $\zeta_p^{L_{v_2}}:\fG_p^{L_{v_2}}\rightarrow\fP(p)$ denote the induced morphism.
Then, according to the definition of $\psi_{T,\mu}$ (cf. \cite{LR87}, p. 143-144), when $\mu+\iota\mu$ is defined over $\Q$, $\psi_{T,\mu}(\delta_m)$ is the unique element $t$ in $T(\Q)$ such that for all $\lambda\in X^{\ast}(T)$, $\lambda(t)$ is a Weil number and 
\[|\prod_{\sigma\in\Gal(L_{v_2}/\Q_p)}\sigma\lambda(t)|_p=q^{-\langle\lambda,\Nm_{L_{v_2}/\Qp}\mu\rangle} \]
holds with $q=p^m$.
Further, $\psi_{T,\mu}(p)\circ\zeta_p:\fG_p^{L_{v_2}}\rightarrow \fG_T$ is conjugate to $\xi_{-\mu}^{L_{v_2}}$ by an element of $T(\Qpb)$ (Lemma \ref{lem:properties_of_psi_T,mu}, (2)), where $\xi_{-\mu}^{L_{v_2}}:\fG_p^{L_{v_2}}\rightarrow \fG_T(p)$ is the morphism defined in Definition \ref{defn:psi_T,mu} for $(T_{\Qp},\mu,L_{v_2})$.
On the other hand, by enlarging $L$ if necessary, we may assume that $\phi$ also factors through $\fP^L$, and that there exists a Galois $\Qpnr/\Qp$-gerb morphism $\xi_p':\fD_l\rightarrow\fG_{T_{\Qp}}^{\nr}$ whose inflation to $\Qpb$ is $T(\Qpb)$-conjugate to $\xi_p=\phi(p)\circ\zeta_p$, where $l:=[L_{v_2}:\Qp]$, (Lemma \ref{lem:unramified_morphism}, (2), or Lemma \ref{lem:unramified_conj_of_special_morphism}). Then, for $b\rtimes\sigma:=\xi_p'(s_{\sigma}^l)$, we have that 
\[-[L_{v_2}:\Qp]\nu_b=\nu_p:=\phi(p)^{\Delta}\circ(\zeta_p^{L_{v_2}})^{\Delta}.\]
Next, recall (cf. Subsec. \ref{subsubsec:pseudo-motivic_Galois_gerb}) that each character of $P(L,m)$ is regarded as a Weil $q=p^m$-number in $L$, with the correspondence being realized in terms of $\delta_m$ by $\chi\mapsto \chi(\delta_m)$ (Lemma \ref{lem:Reimann97-B2.3}, (2)), and that for a Weil $q$-number $\pi$, $\chi_{\pi}$ is the notation regarding it as a character of $P(L,m)$. For $\lambda\in X^{\ast}(T)$, writing $\phi^{\ast}(\lambda)(\delta_m)$ as $\pi_{\lambda}$ for short (so that $\chi_{\pi_{\lambda}}=\phi^{\ast}(\lambda)$), we see that
\begin{eqnarray*}
|\prod_{\sigma\in\Gal(L_{v_2}/\Q_p)}\sigma\lambda(\phi(\delta_m))|_p&=&
|\prod_{\sigma\in\Gal(L_{v_2}/\Q_p)}\sigma\phi^{\ast}(\lambda)(\delta_m)|_p =|\prod_{\sigma\in\Gal(L_{v_2}/\Q_p)}\sigma\pi_{\lambda}|_p\\
&=&q^{\langle \chi_{\pi_{\lambda}},\nu_2^{L_{v_2}}\rangle}=q^{\langle\phi^{\ast}(\lambda),(\zeta_p^{L_{v_2}})^{\Delta}\rangle} \\
&=&q^{\langle \lambda,\nu_p\rangle}.
\end{eqnarray*}
Here, the third equality is the property of $\nu_2^{L_{v_2}}=(\zeta_p^{L_{v_2}})^{\Delta}$ (Def. \ref{defn:Weil-number_torus}, (\ref{eqn:cocharacters_nu^K}), cf. Subsec. \ref{subsubsec:pseudo-motivic_Galois_gerb}). 
This shows that the first statement implies the second claim.

Now, we establish the first statement. We remark that we will reduce the general parahoric subgroup case to a situation involving only a special maximal parahoric subgroup.

Let $T_1^{\textrm{split}}$ be the maximal $\Qp$-split subtorus of $T_1$, $M$ the centralizer of $T_1^{\textrm{split}}$; thus $T_1\subset M$, and $M$ is a $\Qp$-Levi subgroup (Lemma \ref{lem:specaial_parahoric_in_Levi}, (1)).  Below, there will be given a special point $\mbfo$ of the Bruhat-Tits building $\mcB(G,\mfk)$. Then, one can find a $\Qp$-torus $S'$ of $M$ whose extension to $\Qpnr$ becomes a maximal $\Qpnr$-split torus of $M_{\Qpnr}$, and a $M(\mfk)\rtimes\Gal(\Qpnr/\Qp)$-equivariant embedding $\mcB(M,\mfk)\hra\mcB(G,\mfk)$ such that $\mbfo$ lies in the image of the apartment $\mcA^{M}_{\mfk}\subset\mcB(M,\mfk)$ corresponding to $S'$.
Note that the centralizer $T':=\Cent_{G_{\Qp}}(S')$ is a maximal torus of $G_{\Qp}$ (thus, a maximal torus of $M$ as well), as $G_{\mfk}$ is quasi-split by a theorem of Steinberg.

We recall that for a facet $\mbff^{\sigma}$ in $\mcB(G_{\Qp},\Qp)$, there exists a unique $\sigma$-stable facet $\mbff$ in $\mcB(G_{\Qp},\mfk)$ with $\mbff^{\langle\sigma\rangle}=\mbff^{\sigma}$ (\cite[5.1.28]{BT84}).
Let $\mcG_{\mbff}^{\mro}$ be the smooth $\cO_{\mfk}$-group scheme canonically attached to $\mbff$, so that it has connected geometric fibers and the elements of $\mcG_{\mbff}^{\mro}(\cO_{\mfk})$ fixes $\mbff$ pointwise (cf. \cite[5.2]{BT84}). Then, 
\[K_{\mbff}(\mfk):=\mcG_{\mbff}^{\mro}(\cO_{\mfk}),\quad K_{\mbff}(\Qp):=\mcG_{\mbfo}^{\mro}(\cO_{\mfk})^{\sigma}\]
are the pararhoric groups associated with the facet $\mbff$ (or $\mbff^{\sigma}$) (cf. \cite{HainesRapoport08}, Prop. 3). 

Let $\mcA^{G_{\Qp}}_{\mfk}$ be the apartment corresponding to $S'$. By conjugation, we assume that the given $\sigma$-stable facet $\mbff$ defining $\mbfK_p$ (i.e. $\mbfK_p=K_{\mbff}(\Qp)$) lies in $\mcA^{G_{\Qp}}_{\mfk}$.
We fix a $\sigma$-stable alcove $\mbfa$ in $\mcA^{G_{\Qp}}_{\mfk}$ whose closure contains $\mbff$, and let $K_{\mbfa}(\mfk)$ be the corresponding Iwahoric subgroup of $G(\mfk)$. The alcove $\mbfa$ then must contain some (not necessarily $\sigma$-stable) special point $\mbfo$ in its closure. If $K_{\mbfo}(\mfk)\subset G(\mfk)$ denotes the associated special maximal parahoric subgroup, we have that $K_{\mbfa}(\mfk)\subset K_{\mbfo}(\mfk)$ since $\mbfo$ is in the closure of $\mbfa$.
Now, as both  $\mbfa$ and $\mbff$ are $\sigma$-stable, according to \cite[Thm.1.1]{He15}, the condition $X(\{\mu\},b)_{K_{\mbff}(\Qp)}\neq\emptyset$ implies that $X(\{\mu\},b)_{K_{\mbfa}(\Qp)}\neq\emptyset$. 
Let $\mu_B$ be a dominant representative of $\{\mu\}$, where we choose the dominant Weyl chamber \emph{opposite} to the unique Weyl chamber containing the base alcove $\mbfa$ with apex at the special vertex $\mbfo$ (following the convention of \cite{HeRapoport15}). Also, recall (\ref{eqn:splitting_of_EAWG2}) that the choice of a base alcove $\mbfa$ presents the extended affine Weyl group $\widetilde{W}$ as the semidirect product $W_a\rtimes \Omega_{\mbfa}$ of the affine Weyl group $W_a$ (attached to $S'$) with the normalizer subgroup $\Omega_{\mbfa}\subset \widetilde{W}$ of $\mbfa$, thereby fixes a Bruhat order $\leq$ on $\widetilde{W}$ as well.

Let $g_1\in G(\mfk)$ be such that $g_1 K_{\mbfa}(\mfk)\in X(\{\mu_X\},b)_{K_{\mbfa}(\mfk)}$, i.e. if 
\begin{equation} \label{eqn:Iwahori_invariant} 
\mathrm{inv}_{K_{\mbfa}(\mfk)}(g_1,b\sigma(g_1))=\widetilde{W}_{K_{\mbfa}(\mfk)}\cdot w_1\cdot \widetilde{W}_{K_{\mbfa}(\mfk)}\quad (w_1\in \widetilde{W}),
\end{equation}
under the isomorphism $K_{\mbfa}(\mfk)\backslash G(\mfk)/K_{\mbfa}(\mfk)\simeq \widetilde{W}_{K_{\mbfa}(\mfk)}\backslash\widetilde{W}/\widetilde{W}_{K_{\mbfa}(\mfk)}$, there exists $\mu'\in X_{\ast}(T')\cap W_0\cdot\mu_B$ that 
\begin{eqnarray} \label{eqn:Iwahori_inequality} 
\widetilde{W}_{K_{\mbfa}(\mfk)}\cdot w_1\cdot \widetilde{W}_{K_{\mbfa}(\mfk)} &\leq&\widetilde{W}_{K_{\mbfa}(\mfk)}\cdot t^{\underline{\mu'}}\cdot \widetilde{W}_{K_{\mbfa}(\mfk)}. 
\end{eqnarray} 
Here, we used the notations from (\ref{subsubsec:mu-admissible_set}); namely, $\underline{\mu'}$ is the image of $\mu'$ in $X_{\ast}(T')_{\Gamma_{\mfk}}$, and for $\lambda\in X_{\ast}(T')_{\Gamma_{\mfk}}$, $t^{\lambda}$ denotes the corresponding element of $\widetilde{W}$ via $X_{\ast}(T')_{\Gamma_{\mfk}}\cong T'(\mfk)/T'(\mfk)_1\subset \widetilde{W}$.
Since $K_{\mbfa}(\mfk)\subset K_{\mbfo}(\mfk)$, the same relations (\ref{eqn:Iwahori_invariant}), (\ref{eqn:Iwahori_inequality}) continue to hold with $K_{\mbfa}(\mfk)$ replaced by $K_{\mbfo}(\mfk)$ (see \cite[(3.5)]{Rapoport05} for (\ref{eqn:Iwahori_inequality})). Be warned that this does not mean that $X(\{\mu\},b)_{K_{\mbfo}(\mfk)}\neq\emptyset$: the latter definition makes sense only when the point $\mbfo$ is $\sigma$-stable.

Therefore, we are given a string of $\Qp$-subgroups of $G_{\Qp}$:
\[S'\subset T'\subset M,\quad T_1\subset M,\]
where
\begin{itemize}
\item[(a)] $M$ is a $\Qp$-Levi subgroup of $G_{\Qp}$ (i.e. $M$ is the centralizer of a $\Qp$-split torus of $G_{\Qp}$);
\item[(b)] $S'$ is a $\Qp$-torus of $M$ whose extension to $\Qpnr$ becomes a maximal $\Qpnr$-split torus of $M_{\Qpnr}$ (thus $S'$ is also such torus for $G_{\Qp}$);
\item[(c)] $T'=\Cent_{G_{\Qp}}(S')$ (thus, a subgroup of $M$);
\item[(d)] $T_1$ is an elliptic maximal torus of $M$ which $\nu_b$ factors through.
\end{itemize}
These satisfy the following properties: There exists a special point $\mbfo$ of $\mcB(G,\mfk)$ which lies in the image of the apartment $\mcA^{M}_{\mfk}\subset\mcB(M,\mfk)$ corresponding to $S'$, under a suitable embedding $\mcB(M,\mfk)\hra\mcB(G,\mfk)$. Also, the relations (\ref{eqn:Iwahori_invariant}), (\ref{eqn:Iwahori_inequality}) hold with $K_{\mbfa}(\mfk)$ replaced by $K_{\mbfo}(\mfk)$ (for some $g_1$, $w_1$, $\mu'$ as in there). In this set up, we establish the existence of $\mu\in X_{\ast}(T_1)\cap \{\mu\}$ satisfying (\ref{eq:equality_on_the_kernel}). We proceed in the following steps:

\begin{itemize}
\item[(1)] Let $Q$ be a $\Qp$-parabolic subgroup of $G_{\Qp}$ of which $M$ is a Levi factor. Then, 
by Iwasawa decomposition $G(\mfk)=Q(\mfk) K_{\mbfo}(\mfk)$, we may assume $g_1\in Q(\mfk)$: this follows from the classical Iwasawa decomposition $G(\mfk)=Q(\mfk) \mathrm{Fix}(\mbfo)$ (\cite[3.3.2]{Tits79}) and that $\mathrm{Fix}(\mbfo)\subset T'(\mfk)\cdot K_{\mbfo}(\mfk)$. Indeed, $\mathrm{Fix}(\mbfo)\subset G(\mfk)=K_{\mbfo}(\mfk)T'(\mfk) K_{\mbfo}(\mfk)$, so any $g\in \mathrm{Fix}(\mbfo)$ is written as $k_1 tk_2$ with $k_1,k_2\in K_{\mbfo}(\mfk)$ and $t\in T'(\mfk)\cap \mathrm{Fix}(\mbfo)$. But, as $K_{\mbfo}(\mfk)$ is normal in $\mathrm{Fix}(\mbfo)$, we see that $g\in T'(\mfk)\cdot K_{\mbfo}(\mfk)$.
When one writes $g_1=nm$ with $m\in M(\mfk)$ and $n\in N_Q(\mfk)$ ($N_Q$ being the unipotent radical of $Q$), one has that
\[g_1^{-1}b\sigma(g_1)=m^{-1}b\sigma(m)n',\]
where 
\[n'=\sigma(m)^{-1}b^{-1}n^{-1}b\sigma(n)\sigma(m).\]
One readily checks that $n'\in N(\mfk)$.

\item[(2)] Define $\mu''\in X_{\ast}(T')$ by
\[m^{-1}b\sigma(m)\in (K_{\mbfo}(\mfk)\cap M(\mfk))\ t^{\underline{\mu''}}\ (K_{\mbfo}(\mfk)\cap M(\mfk)),\]
using the Cartan decomposition for $(M,K_{\mbfo}(\mfk)\cap M(\mfk))$ (as $K_{\mbfo}(\mfk)\cap M(\mfk)$ is a special maximal parahoric subgroup of $M(\mfk)$, by Lemma \ref{lem:specaial_parahoric_in_Levi}, (3)). Note the equality: 
\[w_{G_{\mfk}}(t^{\underline{\mu''}})=w_{G_{\mfk}}(m^{-1}b\sigma(m))=w_{G_{\mfk}}(m^{-1}b\sigma(m)n')=w_{G_{\mfk}}(w_1)=w_{G_{\mfk}}(t^{\underline{\mu'}}),\]
where the last equality holds since by definition of the Bruhat order on $\widetilde{W}=W_a\rtimes \Omega_{\mbfa}$, $w_1$ and $t^{\underline{\mu'}}$ have the same component in $\Omega_{\mbfa}$ and $w_{G_{\mfk}}$ is trivial on the image in $G_{\Qp}(\mfk)$ of $G_{\Qp}^{\uc}(\mfk)$ ($G_{\Qp}^{\uc}$ being the universal covering of $G_{\Qp}^{\der}$) (\cite[7.4]{Kottwitz97}). It follows that the images of $\mu',\mu''\in X_{\ast}(T')$ in $\pi_1(G)_{\Gamma_{\mfk}}$ are the same. 

On the other hand, if $k_1,k_2\in K_0(\mfk)\cap M(\mfk)$ are such that $m^{-1}b\sigma(m)=k_1t^{\underline{\mu''}}k_2$, then
$m^{-1}b\sigma(m)n'=k_1t^{\underline{\mu''}}k_2n'=k_1t^{\underline{\mu''}}n''k_2$
for some $n''\in N(\mfk)$.
Consequently, we see that
\[K_{\mbfo}(\mfk)\ t^{\underline{\mu''}}\ N_Q(\mfk)\cap K_{\mbfo}(\mfk)w_1 K_{\mbfo}(\mfk)\neq\emptyset.\]
By (\ref{eqn:Iwahori_inequality}), this implies that with respect to the Bruhat order on $\widetilde{W}_{K_{\mbfo}(\mfk)}\backslash\widetilde{W}/\widetilde{W}_{K_{\mbfo}(\mfk)}$, \[\widetilde{W}_{K_{\mbfo}(\mfk)}\cdot t^{\underline{\mu''}}\cdot \widetilde{W}_{K_{\mbfo}(\mfk)} \leq \widetilde{W}_{K_{\mbfo}(\mfk)}\cdot t^{\underline{\mu'}}\cdot \widetilde{W}_{K_{\mbfo}(\mfk)}.\] 
Indeed, the argument of the proof of \cite{HainesRostami10}, Lemma 10.2 establishes the following fact: for $x,y\in \widetilde{W}$, if $K_{\mbfo}(\mfk)yN_Q(\mfk)\cap K_{\mbfo}(\mfk)x K_{\mbfo}(\mfk)\neq\emptyset$, then $y\leq x'$ for some $x'\in\widetilde{W}_{K_{\mbfo}(\mfk)}\cdot x\cdot \widetilde{W}_{K_{\mbfo}(\mfk)}$. 
Also, if $x\leq y$ in the Bruhat order on $\widetilde{W}$, then $\widetilde{W}_{K_{\mbfo}(\mfk)}\cdot x \cdot\widetilde{W}_{K_{\mbfo}(\mfk)} \leq \widetilde{W}_{K_{\mbfo}(\mfk)}\cdot y\cdot \widetilde{W}_{K_{\mbfo}(\mfk)}$ (\cite[8.3]{KR00}).

\item[(3)] 
Let $\underline{\mu'}_0$, $\underline{\mu''}_0$ denote the dominant representatives of $W_0\cdot\underline{\mu'}$, $W_0\cdot\underline{\mu''}\subset X_{\ast}(T')_I$ in $X_{\ast}(T')_I\otimes\R=X_{\ast}(S')_{\R}$, where $I=\Gamma_{\mfk}=\Gal(\overline{\mfk}/\mfk)$. Then, we claim that 
\[\underline{\mu''}_0\leq \underline{\mu'}_0.\]
for the dominance order on $X_{\ast}(S')_{\R}$ (determined by the choice of the alcove $\mbfa$).

Indeed, by (\cite[1.7]{Tits79}), with our choice of the special vertex $\mbfo$, the affine space $\mcA(S',\mfk)$ is identified with the real vector space $V:=X_{\ast}(S')_{\R}=X_{\ast}(T')_I\otimes\R$ (with the origin $\mbfo$), and there exists a reduced root system ${}^{\mbfo}\Sigma$ such that $W_a$ can be identified with its affine Weyl group 
\[W_a=Q^{\vee}({}^{\mbfo}\Sigma)\rtimes W({}^{\mbfo}\Sigma)\]
(i.e. $Q^{\vee}({}^{\mbfo}\Sigma)=X_{\ast}(T^{\uc})_I$ and $W({}^{\mbfo}\Sigma)=W_0$).
Also, the choice of the alcove $\mbfa$ containing $\mbfo$ determines a set $\mathrm{S}_{\mbfa}$ of simple affine roots on $\mcA(S',\mfk)$, and $\widetilde{W}_{K_{\mbfo}(\mfk)}\simeq W_0$ is the subgroup of $W_a$ generated by the subset ${}^{\mbfo}\Delta$ consisting of the simple affine roots whose corresponding affine hyperplanes pass through $\mbfo$ (thus, ${}^{\mbfo}\Delta$ is a set of simple roots for ${}^{\mbfo}\Sigma$). In this set-up of the Coxeter group $W_a$ endowed with a set of generators $\mathrm{S}_{\mbfa}$, the claim follows from Lemma \ref{lem:Bruhat_order_on_Coxeter_group}, applied with the choice $\theta=\mbfo$, noting the following two facts: 
First, for $w=t^{\underline{\nu}}\in W_a$ with $\underline{\nu}\in X_{\ast}(T^{\uc})_I$, $w\theta$ is identified with $\underline{\nu}\in V$. Secondly, $W_0 t^{\underline{\mu''}} W_0\leq W_0 t^{\underline{\mu'}} W_0$ if and only if $w''\leq w'$, where $w''$ (resp. $w'$) is the (unique) element of minimal length in the coset $W_0 t^{\underline{\nu''}} W_0$ (resp. $W_0 t^{\underline{\nu'}} W_0$) with $\underline{\nu''}\in X_{\ast}(T^{\uc})_I$ being the component of $\underline{\mu''}\in X_{\ast}(T)_I(\subset \widetilde{W}=W_a\rtimes \Omega_{\mbfa})$ (resp. $\underline{\nu'}\in X_{\ast}(T^{\uc})_I$ being the component of $\underline{\mu'}\in X_{\ast}(T)_I$).

Hence, by  \cite[Lemma 2.2]{RR96}, we have that
\[W_0\cdot\underline{\mu''}\leq \underline{\mu'}_0.\]

But, as $\underline{\mu'}-w\underline{\mu''}\in X_{\ast}(T'^{\uc})_I$ for any $w\in W_0$, it follows from \cite{Kottwitz84b}, Lemma 2.3.3 that \[\underline{\mu''}\in W_0\cdot\underline{\mu'},\] 
because $\mu'$ is a minuscule coweight. 
Namely, there exist $w\in W_0$ and $\mu_1\in \langle \tau x-x\ |\ \tau\in\Gal(\overline{\mfk}/\mfk), x\in X_{\ast}(T')\rangle$
such that 
\[\mu''=w\mu'\cdot\mu_1\] 
(multiplicative notation).

\item[(4)] So far, we have not used $T_1$ at all. Now, we will use the condition that $T_1$ is elliptic in $M$. Let $\mu\in X_{\ast}(T_1)$ (arbitrary cocharacter for a moment). Let us put $\nu_p:=[K:\Qp]\nu_b\in X_{\ast}(T_1)$.
As $\Nm_{K/\Qp}\mu$ and $\nu_p$ are both $\Qp$-rational, the required equation (\ref{eq:equality_on_the_kernel}) holds if and only if
\begin{equation} \label{eqn:defining_property_of_mu}
[K:\Qp]\langle\chi,\mu\rangle=\langle\chi,\nu_p\rangle
\end{equation}
for every $\Qp$-rational character $\chi$ of $T_1$. 

But, since $T_1$ is \emph{elliptic} in $M$, any $\Qp$-rational character $\chi$ of $T_1$ can be regarded as a $\Qp$-rational character $\chi^{\ab}$ of $M^{\ab}=M/M^{\der}\cong T_1/(T_1\cap M^{\der})$, thus also as that of $M$ (via the canonical projection $p:M\rightarrow M^{\ab}$) such that $\langle\chi,\mu\rangle=\langle\chi^{\ab},p\circ\mu\rangle$
(the first pairing is defined for $X^{\ast}(T_1)\times X_{\ast}(T_1)$ and the second one for $X^{\ast}(M^{\ab})\times X_{\ast}(M^{\ab})$). For a cocharacter $\nu$ of $M$, we will often write $\langle\chi^{\ab},p\circ\nu\rangle$ simply as $\langle\chi^{\ab},\nu\rangle$. 

Now, as $T_1$, $T'$ are both maximal tori of $M$, there exists $g\in M(\overline{\mfk})$ with $T_1=gT'g^{-1}$. Set
\begin{equation}
\mu:=gw(\mu')g^{-1}\in X_{\ast}(T_1)\cap W\cdot\{\mu\}
\end{equation}
Then, as $\chi$ is $\Qp$-rational, we have that $\langle\chi,g\mu_1 g^{-1}\rangle=\langle\chi^{\ab},p\circ g\mu_1 g^{-1}\rangle=\langle\chi^{\ab},p\circ\mu_1\rangle=0$, thus
\[\langle\chi,\mu\rangle=\langle\chi^{\ab},p\circ\mu\rangle=\langle\chi^{\ab},p\circ gw(\mu')g^{-1}\cdot g\mu_1g^{-1}\rangle=\langle\chi^{\ab},p\circ\mu''\rangle.\]

On the other hand, by definition of $\mu''$, for any $\Qp$-rational character $\lambda$ of $M$, we have that
\[|\lambda(m^{-1}b\sigma(m))|=p^{-\langle\lambda,\mu''\rangle}.\]
Since $|\lambda(b)|=p^{-\langle\lambda,\nu_b\rangle}$, we also get that wtih $l:=[K:\Qp]$, 
\[p^{-[K:\Qp]\langle\lambda,\mu''\rangle}=|\lambda(m^{-1}b\sigma(m)\cdots\sigma^{-(l-1)}(m)\sigma^{l-1}(b)\sigma^l(m))|=p^{-\langle\lambda,\nu_p\rangle}.\]
Thus by substituting $\lambda=\chi^{\ab}$ (for a $\Qp$-rational character $\chi$ of $T_1$) and using that $\langle\chi^{\ab},\nu_p\rangle=\langle\chi,\nu_p\rangle$, one obtains the equality (\ref{eqn:defining_property_of_mu}).
\end{itemize}

This completes the proof. 
\end{proof}

\begin{rem}
(1) Suppose that $\psi_{T,\mu}(p)\circ\zeta_p$ factors through $\fG_p^{K}$.
Since $\psi_{T,\mu}(p)\circ\zeta_p$ is conjugate to $\xi_{-\mu}$ under $T(\Qpb)$, their restrictions to the kernel $\Gm$ are the same (Lemma \ref{lem:properties_of_psi_T,mu}, (2)), i.e. equals $-\Nm_{K/\Qp}\mu$. So, the condition (\ref{eq:equality_on_the_kernel}) means that the two $\Qpb/\Qp$-Galois gerb morphisms $\psi_{T,\mu}(p)\circ\zeta_p$, $\phi(p)\circ\zeta_p$ have the same restrictions to the kernel. Then, by definition of $\psi_{T,\mu}$ (cf. \cite{LR87}, p.143-144), this implies that $(i\circ\psi_{T,\mu})^{\Delta}=\phi^{\Delta}$.

(2) In our proof, we had a maximal $\Qp$-torus $T'=\mathrm{Cent}_{G_{\Qp}}(S')$. In the proof of \cite{LR87}, Lemma 5.11, there appears (on p. 177, line -5) a maximal torus of $M$ which is also denoted by the same symbol $T'$. From our perspective, the role of their $T'$ is that of providing a common affine space underlying the apartments $\mcB(G_{\Qp},\mfk)$, $\mcB(M,\mfk)$ (which contains the given (hyper)special point), whose job in our situation is done by $S'$. Meanwhile, our $T'_{\mfk}$ is the centralizer of a maximal $\mfk$-split torus $S'_{\mfk}$ and enters the proof as such, for example, via the Iwasawa and Cartan decompositions (cf. \cite[3.3.2]{Tits79}).
On the other hand, when $G_{\Qp}$ is unramified, their $T'$ is unramified and this is how the unramified condition (or word) in their proof shows up.
\end{rem}

The next lemma is our strengthening of Lemma 5.12 of \cite{LR87}. Its proof does not involve the level subgroup $\mbfK_p$.

\begin{lem} \label{lem:LR-Lemma5.12}
Let $\phi$, $\psi_{T,\mu}$ be as in Lemma \ref{lem:LR-Lemma5.11}.
Then, there exists an admissible embedding of maximal torus $\Int g':T\hra G\ (g'\in G(\Qb))$ (with respect to the identity inner twisting $G_{\Qb}=G_{\Qb}$) such that 
\begin{itemize} 
\item[(i)] $\Int g'\circ\phi$ equals $\psi_{T',\mu_{h'}}$ on the kernel of $\fP$, for some $h'\in X_{\ast}(T')\cap X$.
\end{itemize}
Moreover, if $T_{\Ql}$ is elliptic in $G_{\Ql}$ for some prime $l\neq p$,
there exist $g'\in G(\Qb)$ and $h'\in X$ satisfying, in addition to (i), that 
\begin{itemize}
\item[(ii)] there exists $y\in G(\Qp)$ such that $(T'_{\Qp},\mu_{h'})=\Int y(T_{\Qp},\mu)$.
\end{itemize}
\end{lem}

The first statement, i.e. existence of a transfer of maximal torus $\Int g':T\hra G$ with the property (i) is the assertion of Lemma 5.12 of \cite{LR87} which we reproduce now, while the existence of such element with the additional property (ii) (under the given assumption on $T_{\Ql}$) is due to the author.

\begin{proof}
Let $T^{\uc}$ denote the inverse image of $T\cap G^{\der}$ under the isogeny $G^{\uc}\rightarrow G^{\der}$; it is a maximal torus of $G^{\uc}$. Choose $w\in N_{G}(T)(\C)$ such that $\mu=w(\mu_h)$. We have the cocycle $\alpha^{\infty}\in Z^1(\Gal(\C/\R),G^{\uc}(\C))$ defined by 
\[\alpha^{\infty}_{\iota}=w\cdot\iota(w^{-1}).\] 
One readily checks that this has values in $T^{\uc}(\C)$. Indeed, according to \cite[Prop. 2.2]{Shelstad79}, the automorphism $\Int(w^{-1})$ of $T^{\uc}_{\C}$ is defined over $\R$, so $\Int(w^{-1})(\iota(t))=\iota(\Int(w^{-1})t)=\Int(\iota(w^{-1}))(\iota(t))$ for all $t\in T^{\uc}(\C)$, i.e. $\iota(w)w^{-1}\in \mathrm{Cent}_{G^{\uc}}(T^{\uc})(\C)=T^{\uc}(\C)$, and so is $w\iota(w^{-1})=\iota(\iota(w)w^{-1})$.
Let $\phi$, $\psi_{T,\mu}$ be as in Lemma \ref{lem:LR-Lemma5.11}. Then, according to Lemma 7.16 of \cite{Langlands83}, one can find a global cocycle $\alpha\in Z^1(\Q,T^{\uc})$ mapping to $\alpha^{\infty}\in  H^1(\Q_{\infty},T^{\uc})$. If furthermore $T_{\Ql}$ is elliptic in $G_{\Ql}$ for some prime $l\neq p$, 
we can choose $\alpha\in Z^1(\Q,T^{\uc})$ mapping to $\alpha^{\infty}\in  H^1(\Q_{\infty},T^{\uc})$ and having trivial image in $H^1(\Qp,T^{\uc})$, according to \cite[Lemma 4.1.2]{Lee14}, which we now recall with its proof. It is a variant of the original argument of \cite[Lemma 5.12]{LR87}.

\begin{lem} \cite[Lemma 4.1.2]{Lee14} \label{lem:Lee14-lem.4.1.2}
Let $T$ be a maximal $\Q$-torus of $G$ which is elliptic at some finite place $l\neq p$.

(1) The natural map $(\pi_1(T^{\uc})_{\Gamma(l)})_{\mathrm{tors}} \rightarrow (\pi_1(T^{\uc})_{\Gamma})_{\mathrm{tors}}$
is surjective.

(2) The diagonal map $H^1(\Q,T^{\uc})\rightarrow H^1(\R,T^{\uc})\oplus H^1(\Qp,T^{\uc})$ is surjective.
\end{lem}

\begin{proof} 
(1) This map equals the composite:
\[(\pi_1(T^{\uc})_{\Gamma(l)})_{\mathrm{tors}}\hra \pi_1(T^{\uc})_{\Gamma(l)}\twoheadrightarrow \pi_1(T^{\uc})_{\Gamma}\twoheadrightarrow (\pi_1(T^{\uc})_{\Gamma})_{\mathrm{tors}},\]
where the last two maps are obviously surjective. 
Therefore, it is enough to show that $\pi_1(T^{\uc})_{\Gamma(l)}$ is a torsion group. 
But, as $T^{\uc}_{\Q_l}$ is anisotropic, $\widehat{T^{\uc}}^{\Gamma(l)}$ is a finite group, and so is $\pi_1(T^{\uc})_{\Gamma(l)}=X^{\ast}(\widehat{T^{\uc}}^{\Gamma(l)})=\Hom(\widehat{T^{\uc}}^{\Gamma(l)},\C^{\times})$.

(2) For every place $v$ of $\Q$, non-archimedean or not, there exists a canonical isomorphism (\cite[(3.3.1)]{Kottwitz84a})
\[H^1(\Qv,T^{\uc})\isom \pi_0(\widehat{T^{\uc}}^{\Gamma(v)})^D=\Hom(\pi_0(\widehat{T^{\uc}}^{\Gamma(v)}),\Q/\Z)\cong X^{\ast}(\widehat{T^{\uc}}^{\Gamma(v)})_{\mathrm{tors}}\cong(X_{\ast}(T^{\uc})_{\Gamma(v)})_{\mathrm{tors}},\]
and a short exact sequence (\cite[Prop.2.6]{Kottwitz86})
\[H^1(\Q,T^{\uc})\rightarrow H^1(\Q,T^{\uc}(\overline{\A}))=\oplus_v H^1(\Qv,T^{\uc})
\stackrel{\theta}{\rightarrow} \pi_0(Z(\widehat{T^{\uc}})^{\Gamma})^D=(\pi_1(T^{\uc})_{\Gamma})_{\mathrm{tors}},\]
where $\theta$ is the composite 
\[\oplus_v H^1(\Qv,T^{\uc})\isom \oplus_v \pi_0(Z(\widehat{T^{\uc}})^{\Gamma(v)})^D\rightarrow \pi_0(Z(\widehat{T^{\uc}})^{\Gamma})^D\] (the second map is the direct sum of the maps considered in (1)). 
Let $(\gamma^{\infty},\gamma^p)\in H^1(\R,T^{\uc})\oplus H^1(\Qp,T^{\uc})$. By (1), there exists a class $\gamma^l\in H^1(\Q_l,T^{\uc})$ with $\sum_{v=l,\infty,p}\theta(\gamma^{v})=0$. Then, the element $(\beta^v)_v\in H^1(\Q,T^{\uc}(\overline{\A}))$ such that $\beta^{v}=\gamma^{v}$ for $v=l,\infty,p$ and $\beta^{v}=0$ for $v\neq l,\infty,p$ goes to zero in $ \pi_0(Z(\widehat{T^{\uc}})^{\Gamma})^D$. By exactness of the sequence, we find a class $\gamma$ in $H^1(\Q,T^{\uc})$ which maps to the class $(\beta^v)_v$. 
\end{proof}

Now, by changing $w$ (to another $w'\in N_G(T)(\C)$) if necessary, we may further assume that $\alpha^{\infty}$ is equal (as cocycles) to the restriction of $\alpha$ to $\Gal(\C/\R)$. Then, since the restriction map $H^1(\Q,G^{\uc})\rightarrow H^1(\Q_{\infty},G^{\uc})$ is injective (the Hasse principle), $\alpha$ becomes trivial as a cohomology class in $G_{\uc}(\Qb)$: 
\[\xymatrix{
H^1(\Q_{\infty},T^{\uc})\ar[r] & H^1(\Q_{\infty},G^{\uc}) & \alpha^{\infty}_{\iota}=w\iota(w^{-1}) \ar@{|->}[r] & 0 & \\
H^1(\Q,T^{\uc})\ar[r] \ar[u] & H^1(\Q,G^{\uc}) \ar@{^{(}->}[u] & \alpha \ar@{|->}[u]   \ar@{|->}[r] & \alpha'\ \ar@{^{(}->}[u] \ar@{=>}[r] & \alpha'=0} 
\]
In other words, there exists $u\in G^{\uc}(\Qb)$ such that
 \[\alpha_{\rho}=u^{-1}\rho(u),\]
for all $\rho\in\Gal(\Qb/\Q)$. It then follows that $\Int  u: (T)_{\Qb}\rightarrow G_{\Qb}$ is an admissible embedding of maximal torus with respect to the identity inner twisting of $G_{\Qb}$ (i.e. the homomorphism $\Int u$ and thus the torus $T'=\Int(T)$ as well are defined over $\Q$). We also note that since the restriction of $\Int u$ to $Z(G)$ is the identity, $T'_{\R}$ is also elliptic in $G_{\R}$. When one replaces $\phi$ by $\phi'=\Int  u\circ\phi$ and hence $\gamma_n$ by $\gamma_n'=u\gamma_n u^{-1}$, for the corresponding cocharacter $\mu'$ of $T'$ and the homomorphism $\psi_{T',\mu'}$, we have $\psi_{T',\mu'}=\phi'$ on the kernel. Furthermore, since 
\[u^{-1}\iota(u)=\alpha_{\iota}=\alpha^{\infty}_{\iota}=w\iota(w^{-1})\] 
for $\iota\in\Gal(\C/\R)$, one has that $uw\in G^{\uc}(\R)$ and $\mu'=\Int  u(\mu)=\mu_{h'}$ for $h':=\Int  (uw)(h)\in X$. This establishes the first claim (existence of a transfer of maximal torus $\Int g':T\hra G$ with the property (i)). 

Next, when we assume that $T_{\Ql}$ is elliptic for some $l\neq p$, by Lemma \ref{lem:Lee14-lem.4.1.2}, 
we may choose $\alpha\in Z^1(\Q,T^{\uc})$ such that it maps to $\alpha^{\infty}\in H^1(\Q_{\infty},T^{\uc})$ and to zero in $H^1(\Qp,T^{\uc})$. Then, by repeating the argument above, we find $u\in G^{\uc}(\Qb)$ such that $\alpha_{\rho}=u^{-1}\rho(u)$ for all $\rho\in\Gal(\Qb/\Q)$.
As $\alpha|_{\Gal/\Qpb/\Qp)}$ is trivial, there exists $x\in T(\Qpb)$ such that $x\rho(x^{-1})=\alpha_{\rho}=u^{-1}\rho(u)$ for all $\rho\in \Gal(\Qpb/\Qp)$, in other words. $y:=ux\in G(\Qp)$. But, the homomorphism $\Int u:T_{\Qpb}\rightarrow T'_{\Qpb}$ also equals $\Int u=\Int y$; in particular, it is defined over $\Qp$. This proves (ii) and finishes the proof of Lemma \ref{lem:LR-Lemma5.12}. 
\end{proof}

\subsubsection{Proof of Proposition \ref{prop:equivalence_to_special_adimssible_morphism}.}

We proceed in parallel with the arguments on p.181, line 1-19 of \cite{LR87}. By Lemma \ref{lem:LR-Lemma5.11} and Lemma \ref{lem:LR-Lemma5.12}, after some transfer of tori (always with respect to the identity inner twist $\mathrm{Id}_G$) whose restriction to the torus becomes a conjugation by an element of $G(\Qp)$ when $T_{\Ql}$ is elliptic in $G_{\Ql}$ for some $l\neq p$, we may assume that $\phi$ coincide with $i\circ\psi_{T,\mu_h}$ on the kernel for some $h\in X$ factoring through $T_{\R}$. Then, one readily checks that the map $\Gal(\Qb/\Q)\rightarrow T(\Qb):\rho\mapsto b_{\rho}$ defined by
\[\phi(q_{\rho})=b_{\rho}i\circ\psi_{T,\mu_h}(q_{\rho})\]
is a cocycle, where $\rho\mapsto q_{\rho}$ is the chosen section to the projection $\fP\rightarrow\Gal(\Qb/\Q)$ (Remark \ref{rem:comments_on_zeta_v}). We claim that its image in $H^1(\Q,G)$ under the natural map $H^1(\Q,T)\rightarrow H^1(\Q,G)$ is trivial.  As before, a diagram helps to visualize the proof:
\[\xymatrix{
& H^1(\Q_{\infty},G') & & & 0 &  \\
H^1(\Q_{\infty},T)\ar@{^{(}->}[ur] \ar[r]  & H^1(\Q_{\infty},G) & f(H^1(\Q_{\infty},G^{\der})) \ar@{_{(}->}[l] & 0=b^{\infty}_{\rho} \ar@{^{(}->}[ur] \ar@{|->}[rr] & & 0 \\
H^1(\Q,T)\ar[r] \ar[u] & H^1(\Q,G) \ar[u] & f(H^1(\Q,G^{\der})) \ar@{_{(}->}[l]  \ar@{^{(}->}[u] & b_{\rho} \ar@{|->}[u]  \ar@{|->}[rr]  & & b_{\rho}'=0\  \ar@{^{(}->}[u] } 
\]
Here, $G'$ is the inner twist of $G$ by $\phi$ (i.e. by the cocycle $\rho\mapsto g_{\rho}\in Z^1(\Q,G)$ with $\phi(q_{\rho})=g_{\rho}\rtimes\rho$) and 
$f$ is the natural map $H^1(\Q,G^{\der})\rightarrow H^1(\Q,G)$. 
The restriction of $[b_{\rho}]\in H^1(\Q,T)$ to $\R=\Q_{\infty}$ is trivial, since it maps to zero in $H^1(\Q_{\infty},G')$ and that map is injective (\cite[Lemma 5.14]{LR87}). The image $b_{\rho}'$ of $b_{\rho}$ under the map $H^1(\Q,T) \rightarrow H^1(\Q,G)$ lies in the image of $H^1(\Q,G^{\der})$ in $H^1(\Q,G)$. But, the Hasse principle holds for such image (\cite[Lemma 5.13]{LR87}; this lemma assumes that $G^{\der}$ is simply connected), so we deduce that $b_{\rho}'\in H^1(\Q,G)$ is zero. 
If \[b_{\rho}=v\rho(v^{-1}),\quad v\in G^{\uc}(\Qb),\]
then, $\Int v^{-1}:T_{\Qb}\hra G_{\Qb}$ is an admissible embedding of maximal torus (with respect to the identity twisting $G_{\Qb}=G_{\Qb}$), i.e. the image $T':=v^{-1}Tv$ and the isomorphism $\Int v^{-1}:T_{\Qb}\hra T'_{\Qb}$ are all defined over $\Q$. One has to check that $\mu'$ is $\mu_{h'}$ for some $h'\in X$. This can be seen as follows. Since the cohomology class $[b_{\sigma}]\in H^1(\Q_{\infty},T)$ is trivial, there exists $t_{\infty}\in T(\C)$ such that 
$t_{\infty}^{-1}\iota(t_{\infty})=b_{\iota}=v\iota(v^{-1})$, which implies that
$t_{\infty}v\in G(\R)$. Then, 
\[\mu'=v^{-1}\mu_h v=(t_{\infty}v)^{-1}\cdot\mu_h\cdot (t_{\infty}v)=\mu_{h'}\] 
for $h':=(t_{\infty}v)^{-1}\cdot h\cdot (t_{\infty}v)\in X$. 

Finally, by Lemma \ref{lem:equality_restrictions_to_kernels_imply_conjugacy} below (applied to $\phi(p)\circ\zeta_p$, $i\circ\psi_{T,\mu_h}(p)\circ\zeta_p$), there is $x\in T(\Qpb)$ such that $i\circ\psi_{T,\mu_h}(p)\circ\zeta_p=x(\phi(p)\circ\zeta_p)x^{-1}$ (as $\Qpb/\Qp$-Galois gerb morphisms $\fG_p\rightarrow\fG_{T_{\Qp}}$). But, as $\phi$ and $i\circ\psi_{T,\mu_h}$ are the same on the kernel $P$, the two $\Qpb/\Qp$-Galois gerb morphisms 
\[i\circ\psi_{T,\mu_h}(p),\ x\phi(p)x^{-1}\ :\fP(p)\rightarrow \fG_{T_{\Qp}}\]
agree on the kernel $P_{\Qp}(\Qpb)$. It follows that $i\circ\psi_{T,\mu_h}(p)$ and $x\phi(p)x^{-1}$ are equal on the whole $\fP(p)$. In other words, the restriction of $b_{\rho}$ to $\Gal(\Qpb/\Qp)$ is zero: 
\[x^{-1}\rho(x)=b_{\rho}=v\rho(v^{-1})\] 
for all $\rho\in \Gal(\Qpb/\Qp)$, and $xv\in G(\Qp)$. But, the homomorphism $\Int v^{-1}:T_{\Qpb}\rightarrow T'_{\Qpb}$ also equals $\Int v^{-1}=\Int (xv)^{-1}$. 
It follows from this and the discussion in the beginning of the proof that if the initial torus $T$ satisfies that $T_{\Ql}\subset G_{\Ql}$ is elliptic at some $l\neq p$, we may find a transfer of maximal torus $\Int g':T\hra G$ such that $\Int g'\circ\phi$ is special admissible and that $\Int g'|_{T_{\Qp}}=\Int y|_{T_{\Qpb}}$ for some $y\in G_{\Qp}$.
This completes the proof of Proposition \ref{prop:equivalence_to_special_adimssible_morphism}. $\square$

\begin{lem} \label{lem:equality_restrictions_to_kernels_imply_conjugacy}
Let $T$ be a $\Qp$-torus, and for $i=1,2$, $\theta_i:\fG_p\rightarrow\fG_T$ a morphism of $\Qpb/\Qp$-Galois gerbs.
If the restrictions $\theta_i^{\Delta}$ of $\theta_i$ to the kernel $\mathbb{D}$ are equal, then $\theta_1$ and $\theta_2$ are conjugate under $T(\Qpb)$.
\end{lem}

\begin{proof}
By conjugating by elements of $T(\Qpb)$, we may assume that each $\theta_i$ is unramified; we still have that $\theta_1^{\Delta}=\theta_2^{\Delta}$.
We claim that this equality implies that $\theta_1$ and $\theta_2$ are conjugate under $T(\Qpb)$.

For this, we may replace each $\theta_i$ by a morphism $\theta_i^{\nr}:\fD\rightarrow\fG_{T_{\Qp}}^{\nr}$ of $\Qpnr/\Qp$-Galois gerbs such that $\theta_i$ is conjugate to the inflation $\overline{\theta_i^{\nr}}$ under $T(\Qpb)$. To ease the notations, we continue to use $\theta_i$ for such $\theta_i^{\nr}$. 
By Lemma \ref{lem:unramified_morphism}, $\mathrm{cls}_{T_{\Qp}}(\theta_i)\ (i=1,2)$ is the $\sigma$-conjugacy class of $b_i\in T(\Qpnr)$ with $\theta_i(s_{\sigma})=b_i\rtimes\sigma$. But, according to Lemma \ref{lem:Newton_hom_attached_to_unramified_morphism}, the equality $\theta_1^{\Delta}=\theta_2^{\Delta}$ tells us that $\mathrm{cls}(\theta_1)=\mathrm{cls}(\theta_2)$ in $B(T_{\Qp})$, that is, there exists $t_p\in T(L)$ such that $b_2=t_pb_1\sigma(t_p^{-1})$.

We will show that there exists $x_p\in T(\Qpnr)$ with  $\theta_2=\Int(x_p)\circ\theta_1$, i.e. such that 
\[\theta_2(s_{\tau})=x_p\theta_1(s_{\tau})x_p^{-1}\]
for all $\tau\in\Gal(\Qpnr/\Qp)$. The map $\tau\mapsto b_{\tau}:\Gal(\Qpnr/\Qp)\rightarrow T(\Qpnr)$ defined by 
\[\theta_2(s_{\tau})=b_{\tau}\theta_1(s_{\tau})\] 
is a cocycle in $Z^1(\Gal(\Qpnr/\Qp),T(\Qpnr))$ with $b_{\sigma}=b_2b_1^{-1}=t_p\sigma(t_p^{-1})$; here, we have to use again the condition that $\theta_1^{\Delta}=\theta_2^{\Delta}$. Let $\langle\sigma\rangle$ be the infinite cyclic group $\langle\sigma\rangle$ generated by $\sigma$ (endowed with discrete topology) and $W(\Lb/\Qp)$ the Weil group of continuous automorphisms of $\Lb$ which fix $\Qp$ pointwise and which induce on the residue field of $\Lb$ an integral power of the Frobenius automorphism. The exact sequence $1\rightarrow\Gal(\Lb/L)\rightarrow W(\Lb/\Qp)\rightarrow \langle\sigma\rangle\rightarrow1$
gives rise, via inflation, to a bijection (\cite[(1.8.1)]{Kottwitz85})
\begin{equation} \label{eq:Kottwitz85,(1.8.1)}
B(T_{\Qp})=H^1(\langle\sigma\rangle,T(L))\isom H^1(W(\Lb/\Qp),T(\Lb)).
\end{equation}
Similarly, the restrictions $W(\Lb/\Qp)\rightarrow \Gal(\Qpb/\Qp)\rightarrow \Gal(\Qpnr/\Qp)$ combined with the inclusions $T(\Qpnr)\hra T(\Qpb)\hra T(\Lb)$ induce (again via inflation) maps
\[H^1(\Gal(\Qpnr/\Qp),T(\Qpnr))\hra H^1(\Gal(\Qpb/\Qp),T(\Qpb))\hra H^1(W(\Lb/\Qp),T(\Lb)),\]
which are all injective (\cite[(1.8.2)]{Kottwitz85}).
The image of our cohomology class $[b_{\tau}]\in H^1(\Gal(\Qpnr/\Qp),T(\Qpnr))$ under this composite map is equal to the image under the isomorphism (\ref{eq:Kottwitz85,(1.8.1)}) of its restriction to $\langle\sigma\rangle$, hence is trivial. Therefore, its image in $H^1(\Gal(\Qpnr/\Qp),T(\Qpnr))$ is already trivial. Clearly, this is the claimed statement. 
\end{proof}

\textsc{Proof of Theorem \ref{thm:LR-Satz5.3}.} 
(1) By \cite[Lemma 3.7.7]{Kisin13}, we may assume that $G^{\der}$ is simply connected (cf. proof of Thm. 3.7.8 of loc. cit.).
We follow the original proof of Satz 5.3, as explained after statement of Theorem \ref{thm:LR-Satz5.3}. 
The first step is to replace given $\phi$ by a conjugate $\phi_0=\Int g_0\circ \phi\ (g_0\in G(\Qb))$ of it whose restriction to the kernel $\phi_0^{\Delta}:P_{\Qb}\rightarrow G_{\Qb}$ is defined over $\Q$ (which amounts to that $\phi_0(\delta_n)\in G(\Q)$ for all sufficiently large $n\in\N$, \cite[Lemma 5.5]{LR87}). This is Lemma 5.4 of \cite{LR87}. This lemma is a statement just concerned with the restriction $\phi^{\Delta}$, whose proof only requires that $G_{\Qp}$ is \emph{quasi-split} and does not use the level subgroup at all.

The second step is to find a conjugate $\phi_1=\Int g_1\circ\phi_0$ of $\phi_0$ (produced in the first step) that factors through $\fG_{T_1}$ for some maximal $\Q$-torus $T_1$ (elliptic over $\R$, as usual). As discussed before (after statement of Prop. \ref{prop:existence_of_admissible_morphism_factoring_thru_given_maximal_torus}), this is shown on p. 176, from line 1 to -5 of loc. cit., and the arguments given there again do not make any use of the level (hyperspecial) subgroup and thus carries over to our situation. The basic idea is, in view of Lemma \ref{lem:criterion_for_admissible_morphism_to_land_in_torus}, (2), to find a maximal torus $T_1$ of $I=\Cent_G(\phi_0(\delta_n))$ that can transfer to $I_{\phi}$. (An argument in similar style appears in the proof of Prop. \ref{prop:existence_of_admissible_morphism_factoring_thru_given_maximal_torus}). 

The final step is to find a conjugate $\phi:\fP\rightarrow\fG_{T}$ of $\phi_1:\fP\rightarrow\fG_{T_1}$ which becomes a special admissible morphism $i\circ\psi_{T,\mu_h}$ (for some special Shimura sub-datum $(T,h)$ and the canonical morphism $i:\fG_T\rightarrow\fG_G$ defined by the inclusion $i:T\hra G$). This is accomplished by successive admissible embeddings of maximal tori. It begins with showing existence of $\mu_1\in X_{\ast}(T_1)$ lying in the conjugacy class $\{\mu_h\}$ such that $\phi_1:\fP\rightarrow\fG_{T_1}$ coincides with $\psi_{T_1,\mu_1}$ on the \emph{kernel} of $\fP$. This is done in Lemma 5.11 of loc. cit. This lemma is the only place in the original proof where the level subgroup is involved in an explicit manner (through non-emptiness of the set $X_p(\phi)$). But, as stated in our Lemma \ref{lem:LR-Lemma5.11}, it continues to hold for general parahoric subgroup $\mbfK_p$ with the new condition $X(\{\mu_X\},b)_{\mbfK_p}\neq\emptyset$. 
After this, the rest of the proof is the same as the original one. In a bit more detail, one performs two more admissible embeddings. First, we need to find an admissible embedding of maximal torus $\Int g_2: T_1\hra G$ such that $\Int g_2\circ\phi_1$ equals a special admissible morphism $\psi_{T_2,\mu_{h_2}}$ again on the \emph{kernel} of $\fP$ (here $T_2=\Int g_2(T_1)$ and $(T_2,h_2)$ is a special Shimura sub-datum); the difference from the previous step is that in the previous step, $\mu_1$ did not need to be $\mu_h$ for some $h\in X$. This is shown in Lemma 5.12 of loc. cit., whose argument we adapted to prove our Lemma \ref{lem:LR-Lemma5.12} (which is a refinement of that lemma). Let $\phi_2:=\Int g_2\circ\phi_1:\fP\rightarrow\fG_{T_2}$ be the admissible morphism just obtained. Then, one looks for a (last) admissible embedding of maximal torus $\Int(g_3):T_2\hra G$ making finally $\Int(g_3)\circ\phi_2:\fP\rightarrow\fG_{\Int g_3(T_2)}$ special admissible. This is carried out in loc. cit., from after Lemma 5.12 to the rest of the proof of Satz 5.3; this part of the argument was adapted to prove our Prop. \ref{prop:equivalence_to_special_adimssible_morphism}. Now, we see that
\[(T,\phi):=(\Int g_3(T_2),\Int(g_3)\circ\phi_2)=(\Int (\prod_{i=2}^3g_i)(T_1),\Int(\prod_{i=0}^3g_i)\circ\phi)\]
is a special admissible morphism which is a conjugate of $\phi$.

(2) This is proved by the same argument from (1) with applying Prop. \ref{prop:existence_of_admissible_morphism_factoring_thru_given_maximal_torus} in the second step and then Prop. \ref{prop:equivalence_to_special_adimssible_morphism} in the third step. Note that as explained before, the three properties (i) - (iii) continue to hold under any transfer of maximal torus.
$\square$

\begin{lem} \label{lem:LR-Lemma5.23}
Retain the assumptions of Theorem \ref{thm:LR-Satz5.3}, and assume that $G^{\der}$ is simply connected. 
Then, for each admissible pair $(\phi,\epsilon),\epsilon\in I_{\phi}(\Q)$, one can find an equivalent pair $(\phi',\epsilon')$ and $T'$ and $h'$, such that $(\phi',\epsilon')$ is nested in $(T',h')$.
\end{lem}
 
\begin{proof} In the original setting of \cite{LR87}, this is their Lemma 5.23. The proof given in loc. cit. works in our situation without any significant change. Most importantly, the level subgroup $\mbfK_p$ enters the proof only through Lemma 5.11 of \cite{LR87} (which is generalized by our Lemma \ref{lem:LR-Lemma5.11}). Hence here we only give a sketch of proof. We may assume that $\phi$ is well-located.
It will suffice to find a maximal $\Q$-torus $T$ of $I_{\phi}$ containing $\epsilon$ (i.e. $\epsilon\in T(\Q)$) and $g\in G(\Qb)$ such that
\[ G_{\Qb} \stackrel{\Int g}{\leftarrow} G_{\Qb}\hookleftarrow T_{\Qb}\] 
is defined over $\Q$. Here, $T_{\Qb}$ is regarded as a $\Qb$-subgroup of $G_{\Qb}$ via the inner twisting $I_{\Qb}\isom (I_{\phi})_{\Qb}$ (\ref{eqn:inner-twisting}). Indeed, then $T_1:=\Int g(T)$ is a $\Q$-subgroup of $\Cent_G(\Int g(\epsilon))\subset G$, and $\phi_1:=\Int g\circ\phi$ maps into $\fG_{T_1}$ by Lemma \ref{lem:criterion_for_admissible_morphism_to_land_in_torus}, (1) (applied with $(g,T_1)$ in the place of $(a,T)$). Hence, $\Int g(\epsilon)\in I_{\phi_1}(\Q)\cap T_1(\Qb)=T_1(\Q)$ (the equality holds since $\phi_1$ maps into $\fG_{T_1}$).
Therefore, we can apply Prop. \ref{prop:equivalence_to_special_adimssible_morphism} to $(\Int g\circ\phi, \Int g(T))$ and obtain the desired pair $(\phi'=\psi_{T',\mu_{h'}},T')$ by another admissible embedding; as $\epsilon\in \Int g(T)(\Q)$, we also have $\epsilon\in T'(\Q)$. 

Now, to find such torus $T\subset I_{\phi}$ and $g\in G(\Qb)$, we choose an element $\epsilon_1$ of $I_{\phi}(\Q)(\subset G(\Qb))$, whose centralizers in $I_{\phi}, I_{\Qb}\subset G_{\Qb}$ are maximal tori of $G_{\Qb}$ containing $\epsilon$. These centralizers are then the same (via the inner twisting  (\ref{eqn:inner-twisting})) which will be our $T$; $T$ is a $\Q$-subgroup of $I_{\phi}$, but only a $\Qb$-subgroup of $I_{\Qb}$ in general. 
The conjugacy class of $\epsilon_1$ in $G(\Qb)$ is rational, as $I$ is an inner form of $I_{\phi}$. So, according to a variant of Steinberg's theorem \cite[Thm.4.4]{Kottwitz82} (which assumes that $G^{\der}$ is simply connected), we can find an inner twist $\psi: G_{\Qb}\rightarrow G^{\ast}_{\Qb}$ with $G^{\ast}$ being quasi-split, such that $\epsilon_1^{\ast}=\psi(\epsilon_1)$ is rational. This implies that $\psi:T_{\Qb}\rightarrow T^{\ast}_{\Qb}$ is $\Q$-rational, where $T^{\ast}(=\psi(T))$ denotes the centralizer of $\epsilon_1^{\ast}$ in $G^{\ast}$. Indeed, if $G_{\phi}$ is the inner twist of $G$ twisted by $\phi$, i.e. defined by $G_{\phi}(\Q)=\{g\in G(\Qb) \ |\ g_{\rho}\rho(g)g_{\rho}^{-1}=g\}$, where $\phi(q_{\rho})=g_{\rho}\rtimes\rho$ (so that $T_{\Qb}\subset G_{\Qb}$ becomes a $\Q$-torus of $G_{\phi}$ and $\epsilon_1\in G_{\phi}(\Q)$), the inner twist $\psi:G_{\Qb}\isom G^{\ast}_{\Qb}$ gives rise to an inner twist $\psi_{\phi}:(G_{\phi})_{\Qb}\isom G^{\ast}_{\Qb}$ with $\psi_{\phi}(\epsilon_1)=\epsilon_1^{\ast}$. So the claim follows the easily verified fact that for an inner twist $\psi:H_1\rightarrow H_2$ with $\phi(\epsilon_1)=\epsilon_2$ for some regular semisimple $\epsilon_i\in H_i(\Q)\ (i=1,2)$, its restriction to the centralizers $I_i=\Cent_{H_i}(\epsilon_i)$ becomes a $\Q$-isomorphism. Now, by the proof of \cite[Lemma 5.23]{LR87} (more precisely by the argument in the last paragraph on p.190), there exists a transfer of maximal torus $\Int g:T^{\ast}\hra G$ with respect to the inner twist $\psi^{-1}$, namely 
\[\Int g\circ \psi^{-1}:T^{\ast}_{\Qb}\hra G^{\ast}_{\Qb}\stackrel{\psi^{-1}}{\rightarrow}G_{\Qb}\stackrel{\Int g}{\rightarrow} G_{\Qb}\] 
is defined over $\Q$; here one uses the assumption that $G_{\Qp}$ is quasi-split. Therefore, $T_1:=\Int g(T)=\Int g\circ\psi^{-1}(T^{\ast})$ is a $\Q$-torus of $G$, and $T_{\Qb}\hra G_{\Qb}\stackrel{\Int g}{\rightarrow} G_{\Qb}$ is a $\Q$-rational homomorphism (where $T_{\Qb}$ is regarded as a $\Qb$-subgroup of $G_{\Qb}$ via the inner twisting $(I_{\phi})_{\Qb}\isom I_{\Qb}$), as we wanted. \end{proof}

Now, as an application, we establish non-emptiness of Newton strata for general parahoric levels, when $G_{\Qp}$ is quasi-split. To talk about the reduction at a prime, we need to choose an integral model, i.e. a flat model $\sS_{\mbfK}$ over $\cO_{E_{\wp}}$ with generic fiber being the canonical model $\Sh_{\mbfK}(G,X)_{E_{\wp}}$. For the following result, it is enough to fix an integral model over $\cO_{E_{\wp}}$ having the extension property that every $F$-point of $\Sh_{\mbfK}(G,X)$ for a finite extension $F$ of $E_{\wp}$ extends uniquely to $\sS_{\mbfK}$ over its local ring (for example, a normal integral model); see \cite{KP15} for a construction of such integral model.

\begin{thm} \label{thm:non-emptiness_of_NS}
Suppose that $G_{\Qp}$ is quasi-split. Let $\mbfK_p$ be a parahoric subgroup of $G(\Qp)$ and $\mbfK=\mbfK_p\mbfK^p$ for a compact open subgroup $\mbfK^p$ of $G(\A_f^p)$. 

(1) Then, for any $[b]\in B(G_{\Qp},\{\mu_X\}$) (Subsec. \ref{subsubsec:B(G,{mu})}), there exists a special Shimura sub-datum $(T,h\in\Hom(\dS,T_{\R})\cap X)$ such that the Newton homomorphism $\nu_{G_{\Qp}}([b])$ equals the $G(L)$-conjugacy class of
\[\frac{1}{[K_{v_2}:\Qp]}\Nm_{K_{v_2}/\Qp}\mu_h\quad (\in X_{\ast}(T)_{\Q}),\]
where $K_{v_2}\subset\Qpb$ is any finite extension of $\Qp$ splitting $T$.

In particular, if $(G,X)$ is of Hodge type, for $g_f\in G(\A_f)$, the reduction in $\sS_{\mbfK}\otimes\Fpb$ of the special point $[h,g_f\cdot\mbfK]\in \Sh_{\mbfK}(G,X)(\Qb)$ has the $F$-isocrystal represented by
\[\Nm_{K_{v_2}/K_0}(\mu_h(\pi)),\]
where $K_0\subset K_{v_2}$ is the maximal unramified subextension and $\pi$ is a uniformizer of $K_{v_2}$.

(2) Suppose that $\mbfK_p$ is \emph{special maximal} parahoric. Moreover, assume that $G$ splits over a tamely ramified cyclic extension of $\Qp$ and is of classical Lie type. Then, one can choose a special Shimura datum $(T,h\in\Hom(\dS,T_{\R})\cap X)$ as in (1) such that furthermore the unique parahoric subgroup of $T(\Qp)$ is contained in $\mbfK_p$.

(3) Suppose that $(G,X)$ is a Shimura datum of Hodge type. Then the reduction $\sS_{\mbfK}(G,X)\otimes\Fpb$ has non-empty ordinary locus if and only if $\wp$ has absolute height one (i.e. $E(G,X)_{\wp}=\Q_p$). 
\end{thm}

\begin{proof}
(1) We follow the strategy of our proof of the corresponding result in the hyperspecial case given in \cite{Lee14}, Thm. 4.1.1 and Thm. 4.3.1. Let $[b]\in B(G_{\Qp},\{\mu\}$). 
Since $G_{\Qp}$ is quasi-split, there exist a representative $b\in G(L)$ of $[b]$ and a maximal torus $T_p$ of $G_{\Qp}$ such that the Newton homomorphism $\nu_b:\mathbb{D}\rightarrow G_{\Qpnr}$ is $\Qp$-rational and factors through $T_p$ (\cite[Prop. 6.2]{Kottwitz86}). By the argument of Step 1 in the proof of \cite{Lee14}, Thm. 4.1.1, we may further assume that $T_p=(T_0)_{\Qp}$ for a maximal $\Q$-torus $T_0$ of $G$ such that $(T_0)_{\Qv}\subset G_{\Qv}$ is elliptic maximal for $v=\infty$ and some prime $v=l\neq p$. Then, Lemma \ref{lem:LR-Lemma5.11} tells us that there exists 
$\mu'\in X_{\ast}(T_0)\cap \{\mu\}$ such that the relation (\ref{eq:equality_on_the_kernel}) holds in $X_{\ast}(T_0)$:
\[\Nm_{K_{v_2}/\Qp}\mu'=[K_{v_2}:\Qp]\ \nu_b,\]
where $K$ is a finite Galois extension of $\Q$ splitting $T_0$ and $v_2$ is the place of $K$ induced by the pre-chosen embedding $\Qb\hra\Qpb$ (here, the sign is correct by Lemma \ref{lem:Newton_hom_attached_to_unramified_morphism}).
Next, by the argument of Step 2 in loc. cit. (which corresponds to that of Lemma \ref{lem:LR-Lemma5.12}), we can find a transfer of maximal torus $\Int u:T_0\hra G$ such that 
$\Int u (\mu')=\mu_h$ for some $h\in X\cap \Hom(\dS,T_{\R})$, where $T=\Int u(T_0)$ (again, be wary of the sign difference from \cite{Lee14}), and that $\Int u|_{(T_0)_{\Qpb}}=\Int y$ for some $y\in G(\Qp)$. By the latter property, for $(T,\mu_h, \Int u (b))$ we still have that 
\[\Nm_{K_{v_2}/\Qp}\mu_h=[K_{v_2}:\Qp]\ \nu_{yb\sigma(y)^{-1}}\]
(here, $\Nm_{K_{v_2}/\Qp}$ is taken on $X_{\ast}(T)$). This proves the first statement of (1).
According to Lemma \ref{lem:unramified_conj_of_special_morphism} and \cite[Thm. 1.15]{RR96}, the element of $T(L)$
\[b_T:=\Nm_{K_{v_2}/K_0}(\mu_h(\pi))\]
has the Newton homomorphism $\nu_{b_T}=\frac{1}{[K_{v_2}:\Qp]}\Nm_{K_{v_2}/\Qp}\mu_h=\nu_{yb\sigma(y)^{-1}}$.
As $\kappa_{T_{\Qp}}(b_T)=\mu^{\natural}\in X_{\ast}(T)_{\Gal(\Qpb/\Qp)}$ and $(\overline{\nu},\kappa):B(G_{\Qp})\rightarrow \mathcal{N}(G_{\Qp})\times\pi_1(G)_{\Gal(\Qpb/\Qp)}$ is injective (\cite[4.13]{Kottwitz97}), we see the equality of isocrystals $[b]=[b_T]\in B(G_{\Qp})$. 
Given this, the second statement is proved in the same fashion as in the hyperspecial case, using 
\cite[Lem. 3.24]{Lee14}.

(2) Let $(T_1,h_1)$ be a special Shimura sub-datum produced in (1). Thanks to our additional assumptions and Prop. \ref{prop:existence_of_elliptic_tori_in_special_parahorics}, in its construction, we could have started with a maximal torus $T_p$ of $G_{\Qp}$ such that the unique parahoric subgroup of $T_p(\Qp)$ is contained in a $G(\Qp)$-conjugate of $\mbfK_p$. Then, also by the fine property of our methods (it uses only transfers of maximal tori which become conjugacy by $G(\Qp)$-elements), the torus $T_1$ produced in (1) can be assumed to further satisfy that the unique parahoric subgroup of $T_1(\Qp)$ is contained in $g_p\mbfK_pg_p^{-1}$ for some $g_p\in G(\Qp)$. As $G_{\Qp}$ splits over a cyclic extension of $\Qp$, $G(\Q)$ is dense in $G(\Qp)$ by a theorem of Sansuc (\cite[Lem. 4.10]{Milne94}), thus there exists $g_0\in G(\Q)\cap \mbfK_p\cdot g_p^{-1}$. Then, one easily checks that the new special Shimura datum $(T,h):=\Int(g_0)(T_1,h_1)$ satisfies the required properties.

(3) Again. the proof is the same as that in the hyperspecial case given in \cite[Cor. 4.3.2]{Lee14}. In more detail, as was observed in loc. cit., it suffices to construct a special Shimura sub-datum $(T,\{h\})$ with the property that
there exists a $\Qp$-Borel subgroup $B$ of $G_{\Qp}$ containing $T_{\Qp}$ such that $\mu_h\in X_{\ast}(T)$ lies in the closed Weyl chamber determined by $(T_{\Qp},B)$. Indeed, then we have $E(T,h)_{\mathfrak{p}}=E(G,X)_{\wp}$, where $\mathfrak{p}$ and $\wp$ denote respectively the places of each reflex field induced by the given embedding $\Qb\hra\Qpb$. We remark that this is the property (ii) found in the proof of loc. cit., and for our conclusion one does not really need the property (i) from it.
But, since $G_{\Qp}$ is quasi-split, there exists a Borel subgroup $B'$ defined over $\Qp$. Moreover, by the same argument as was used in (1) (i.e. Step 1 in the proof of \cite[Thm. 4.1.1]{Lee14}), we may assume that $B$ contains $T'_{\Qp}$ for a maximal $\Q$-torus $T'$ of $G$ such that $T'_{\Qv}\subset G_{\Qv}$ is elliptic for $v=\infty$ and some prime $v=l\neq p$. Let $\mu'\in \{\mu_X\}\cap X_{\ast}(T')$ be the cocharacter lying in the closed Weyl chamber determined by $(T'_{\Qp},B')$. Then, the argument in (1) again produces a special Shimura sub-datum $(T,\{h\})$ such that $(T,\mu_h)=\Int y(T',\mu_{h'})$ for some $y\in G(\Qp)$, and $(T,\{h\})$ is the looked-for special Shimura sub-datum. Note that as we do not need the property (i) in the original proof of \cite[Thm. 4.1.1]{Lee14}, the condition in (2) on splitting of $G_{\Qp}$ is no longer necessary. 
\end{proof}

\begin{rem}
(1) For more on the Newton stratification, we refer to the recent survey article \cite{Viehmann15}. 

(2) As is obvious from the proof, the assumption that $G$ splits over a tamely ramified (not necessarily cyclic) extension of $\Qp$ and is of classical Lie type and the one that $G_{\Qp}$ splits over a \emph{cyclic} extension of $\Qp$ are needed only to invoke, respectively, Prop. \ref{prop:existence_of_elliptic_tori_in_special_parahorics} and the fact that $G(\Q)$ is dense in $G(\Qp)$. Therefore, any weaker assumption guaranteeing those facts could be used in their places. On the other hand, when $(G,X)$ is of abelian type, $G$ is always of classical Lie type.
\end{rem}


\section{Admissible pairs and Kottwitz triples}
 
Here. again our running assumption is that $G_{\Qp}$ is quasi-split. Otherwise, please refer to each statement for precise assumptions made there.

\begin{prop} \label{prop:phi(delta)=gamma_0_up_to_center}
Assume that $G_{\Qp}$ is quasi-split. 
Let $(\phi,\gamma_0)$ be an admissible pair that is well-located in a maximal $\Q$-torus $T$ of $G$ which is elliptic over $\R$ and such that the image of $\gamma_0$ in $G^{\ad}(\A_f^p)$ lies in a compact open subgroup of $G^{\ad}(\A_f^p)$ (equiv. in a compact open subgroup of $T/Z(G)(\A_f^p)$). Suppose that $(\phi,\gamma_0)$ is of level $n$ (cf. Def. \ref{def:admissible_pair}). 
 
(1) For sufficiently large $k\in\N$ divisible by $n$, the element $\phi(\delta_k)\cdot \gamma_0^{-\frac{k}{n}}$ of $G(\Q)$ lies in the center of $G$. 

(2) Suppose further that $\phi=i\circ\psi_{T,\mu_h}$ for some $h\in \Hom(\dS,T_{\R})\cap X$ and let $K\supset \Qp$ be a finite Galois extension splitting $T_{\Qp}$.
If $H$ is the centralizer in $G_{\Qp}$ of the maximal $\Qp$-spit torus $A$ in the center of $\Cent_G(\gamma_0)_{\Qp}$, then $\Nm_{K/\Qp}\mu_h$ maps into the center of $H$ (which equals $A$).

(3) Assume that the weight homomorphism $w_X=\mu_h\cdot\iota(\mu_h)\ (h\in X)$ is rational. 
If $\gamma_0$ is a Weil $q=p^n$-number of weight $w=w_X$, in the sense that for every character $\chi$ of $T$, $\chi(\gamma_0)\in\Qb$ is a Weil $q=p^n$-number of weight $\langle \chi,w_X\rangle\in\Z$ in the usual sense, then $\gamma_0^{\frac{k}{n}}=\phi(\delta_k)$ for sufficiently large $k\in\N$.

(4) If the anisotropic kernel of $Z(G)$ remains anisotropic over $\R$ and $\gamma_0\in G(\A_f^p)$ lies in a compact open subgroup of $G(\A_f^p)$, then $\gamma_0^{\frac{k}{n}}=\phi(\delta_k)$ for sufficiently large $k\in\N$.
\end{prop}

\begin{proof} The first statement (1) is asserted in \cite{LR87}, p.194 (line -8) - p.195 (line 12) with a sketchy proof. Here we will give a detailed proof. To show the proposition, we need a fact which was stated in \cite{LR87}, p.195, line 5-9,  but without an explanation:

\begin{lem} \label{lem:equality_of_two_Newton_maps}
(1) For an admissible morphism $\phi:\fP\rightarrow \fG_G$ well-located in a maximal $\Q$-torus $T$ of $G$, elliptic over $\R$, let $\xi_p':\fD\rightarrow\fG_{T_{\Qp}}^{\nr}$ be an unramified conjugate of $\xi_p=\phi(p)\circ\zeta_p$ under $T(\Qpb)$, and $b\in T(\mfk)$ be defined by $\xi'_p(s_{\sigma})=b\rtimes\sigma$. 
Then, the two Newton homomorphisms $\nu_{\phi(\delta_k)}$, $\nu_b\in X_{\ast}(T)_{\Q}^{\Gal_{\Qp}}$ attached to $\phi(\delta_k)$, $b$ are related by: $\nu_{\phi(\delta_k)}=k\nu_b$.

(2) Let $\epsilon\in T_{\Qp}(\Qp)$ for a maximal $\Qp$-torus $T_{\Qp}$ of $G_{\Qp}$ and suppose there exists a $\delta\in G(L_n)$ such that $\Nm_{L_n/\Qp}\delta=c\epsilon c^{-1}$ for some $c\in G(\Qpnr)$, where $L_n=\mathrm{Frac}(W(\F_{p^n}))$; it follows that $b':=c^{-1}\delta\sigma(c)$ belongs to $G(\epsilon)(\Qpnr)$. Then, the two Newton homomorphisms $\nu_{\epsilon}$, $\nu_{b'}\in X_{\ast}(T)_{\Q}^{\Gal_{\Qp}}$ are related by: $\nu_{\epsilon}=n\nu_{b'}$. 

(3) If $(\phi,\epsilon\in I_{\phi}(\Q))$ is an admissible pair well-located in a maximal $\Q$-torus $T$ of $G$ that is elliptic over $\R$, the two Newton homomorphisms $\nu_b$, $\nu_{b'}$ in (1) and (2) are the same: $\nu_b=\nu_{b'}$.
\end{lem}

For (1), recall that for an admissible morphism $\phi:\fP\rightarrow \fG_G$ factoring through $\fG_T$ for a maximal $\Q$-torus $T$ of $G$, $b\in T(\mfk)$ defined by an unramified conjugate $\xi_p'$ \emph{under $T(\Qpb)$} of $\xi_p=\phi(p)\circ\zeta_p$, its Newton homomorphism $\nu_b\in X_{\ast}(T)_{\Q}^{\Gamma(p)}$ does not depend on the choice of $\xi_p'$ (as long as $\xi_p'$ is a conjugate of $\xi_p$ by an element of $T(\Qpb)$).

\begin{proof}
(1) Recall that if $\phi:\fP(L,m)\rightarrow \fG_T$ for some CM field $L$ Galois over $\Q$ and $m\in\N$, for $\lambda\in X^{\ast}(T)$, the $\Qb$-character $\lambda\circ\phi^{\Delta}$ is also a Weil $p^m$-number $\pi$, in which case when we write $\chi_{\pi}:=\lambda\circ\phi^{\Delta}$, we have $\chi_{\pi}(\delta_k)=\pi^{\frac{k}{m}}$ for every $k\in\N$ divisible by $m$. Next, when we regard any $\lambda\in X^{\ast}(T_{\Qp})^{\Gal_{\Qp}}$ as a $\Qb$-character of $T$ via the chosen embedding $\iota_p:\Qb\hra\Qpb$, 
\[|\lambda(\phi(\delta_k))|_p^{[K_w:\Qp]}=|\lambda(\phi(\delta_k))|_w=|\chi_{\pi}(\delta_k)|_w=|\pi|_w^{\frac{k}{m}}=p^{k\nu_2(\pi,w)}=p^{k\langle\chi_{\pi},\nu_2\rangle}=p^{k\langle\lambda,\xi_p^{\Delta}\rangle},\]
where $\xi_p^{\Delta}=\phi(p)^{\Delta}\circ\nu_2$ and $w$ is the place of $K$ induced by $\iota_p$. (Notice that $\lambda$ can be considered as a $\Qp$-character in the first three expressions, while the $3^{\text{rd}}$ equality makes sense only when $\chi_{\pi}$ is a $\Qb$-character.) 
This shows (\cite[2.8, 4.4]{Kottwitz85}) that the Newton homomorphism $\nu_{\phi(\delta_k)}\in X_{\ast}(T)_{\Q}^{\Gal_{\Qp}}$ attached to $\phi(\delta_k)\in T(\Qp)$ is $-\frac{k}{[K_w:\Qp]}\xi_p^{\Delta}$.

On the other hand, if $\xi_p'$ factors through $\fD_n$, the Newton homomorphism $\nu_b$ attached to $b$ is $\frac{1}{n}(\xi_p')^{-1}|_{\Gm}$ (Lemma \ref{lem:Newton_hom_attached_to_unramified_morphism}, cf. \cite{LR87}, Anmerkung).
In our case, we may assume that $n=[K_w:\Qp]$, by Remark \ref{rem:comments_on_zeta_v}, (2) and Lemma \ref{lem:unramified_morphism}, (2). Also, we have $(\xi_p')^{\Delta}=(\xi_p)^{\Delta}$ (restriction to the kernel $\Delta$ of $\fD_n$) as $\phi^{\Delta}$ maps to the torus $T$. Therefore, we get $\nu_b=-\frac{1}{[K_w:\Qp]}\xi_p^{\Delta}$, which proves the claim.

(2) This is proved in Lemma 5.15 of \cite{LR87}. We briefly sketch its arguments. First, we observe that for any $c\in G(\mfk)$ and $n'\in N$, $c_{n'}:=c^{-1}\sigma^{n'}(c)$ lies in any small neighborhood of $1$ as $n'$ becomes large (in fact, even becomes $1$ if $c\in G(\Qpnr)$). Secondly, set $\nu_1:=\frac{1}{n}\nu_{\epsilon}\in X_{\ast}(T)_{\Q}^{\Gal_{\Qp}}$ for the Newton (quasi-)cocharacter $\nu_{\epsilon}$ of $\epsilon\in T_{\Qp}(\Qp)$ so that
\[|\lambda(\epsilon)|_p=p^{-\langle\lambda,n\nu_1\rangle}\]
holds for every $\Qp$-rational $\lambda\in X_{\ast}(T)$ (\cite[4.4]{Kottwitz85}). It follows from this equation that $\nu_1$ maps into the center of $G(\epsilon)$ and that $p^{-n\nu_1}\epsilon\in T_{\Qp}(\Qp)$ lies in the maximal compact subgroup of $T_{\Qp}(\Qp)$. Especially, $(p^{-n\nu_1}\epsilon)^k$ also lies in any small neighborhood of $1$ as $k$ becomes large. Therefore, according to \cite[Prop. 3]{Greenberg63}, for sufficiently large $k\in\N$, there exists $d\in G(\mfk)$ such that with $n'=nk$, 
\[\epsilon^k c_{n'}=p^{n'\nu_1}d^{-1}\sigma^{n'}(d).\]
Here, we used the fact that $G_{\Qpnr}$ admits a smooth $\cO_{\Qpnr}$-integral model with connected special fiber (e.g. parahoric group schemes, cf. \cite{HainesRapoport08}).
Finally, from $\Nm_{L_n/\Qp}\delta=c\epsilon c^{-1}$, one easily checks that $b':=c^{-1}\delta\sigma(c)$ and $c_{n'}=c^{-1}\sigma^{n'}(c)$ belong to $G(\epsilon)(\Qpnr)$, and the equality
\[\Nm_{L_{n'}/\Qp}b'=\epsilon^k c_{n'}(=p^{n'\nu_1}d^{-1}\sigma^{n'}(d))\]
holds.
Therefore, the definition (\cite[4.3]{Kottwitz85}) tells us that the Newton homomorphism $\nu_{b'}\in\Hom_{\mfk}(\mathbb{D},G)$ of $b'$ is equal to $\nu_1$. 

(3) This follows from Remark \ref{rem:two_different_b's}. 
\end{proof}

\textsc{Proof of Proposition \ref{prop:phi(delta)=gamma_0_up_to_center} continued.} 
(1) As $\delta_{kd}=\delta_k^{d}$, it is enough to show that for some $k\in\N$ (divisible by $n$), the image of $\phi(\delta_k)\cdot \gamma_0^{-\frac{k}{n}}$ in $G^{\ad}(\Q)$ is a torsion element. For that, we use the fact that for any linear algebraic group $G$ over a number field $F$, $G(F)$ is discrete in $G(\A_F)$, so for any compact subgroup $K\subset G(\A_F)$, $G(F)\cap K$ will be finite, particularly, a torsion group. 
We will check that the image of $\phi(\delta_k)\cdot \gamma_0^{-\frac{k}{n}}$ in $T/Z(G)(\Qv)$ lies in a compact (open) subgroup of $T/Z(G)(\Qv)$ for every place $v$ of $\Q$. Recall that for an $F$-torus $T$ and any finite place $v$ of $F$, a subgroup $H$ of $T(F_v)$ is compact if and only if $H$ is contained in 
\[\bigcap_{\chi\in X^{\ast}(T),\ F_v-\text{rational}}\mathrm{Ker}(\mathrm{val}_v\circ\chi),\]
where $\mathrm{val}_v$ is the (normalized) valuation on $F_v$. 
For every finite place $l\neq p$, the image of $\gamma_0$ in $T/Z(G)(\Ql)$ is a unit (i.e. lies in a compact subgroup) by assumption, and so is $\phi(\delta_k)$ by definition of $\delta_k$ (in fact, $\phi(\delta_k)$ is itself a unit in $T(\Ql)$ for every $l\neq p$). As $T/Z(G)$ is anisotropic over $\R$, the claim is trivial for the archimedean place. Hence, it suffices to show that for every $\Qp$-rational character $\chi$ of $T/Z(G)$, $|\chi(\phi(\delta_k))|_p=|\chi(\gamma_0^{\frac{k}{n}})|_p$. In fact, we will show this for $\Qp$-rational characters $\chi$ of $T$.
Choose a finite Galois CM-extension $L$ of $\Q$ and $m\in\N$ such that $\phi$ factors through $\fP(L,m)$. Then, for all sufficiently large $k\in\N$ divisible by $[L:\Q]n$ and for any $\Qp$-rational character $\chi$ of $T$, one has 
\[ |\chi(\phi(\delta_k))|_p=p^{-k\langle\chi,\nu_b\rangle}=p^{-\frac{k}{n}\langle\chi,n\nu_{b'}\rangle}=p^{-\frac{k}{n}\langle\chi,\nu_{\gamma_0}\rangle}=|\chi(\gamma_0)|_p^{\frac{k}{n}}. \]
Here, the first (resp. the second, and the third) equality holds by Prop. \ref{lem:equality_of_two_Newton_maps}, (1) (resp. (3), and (2)). 

(2) After enlarging $K$, if necessary, the Galois $\Qpb/\Qp$-gerb morphisms $\xi_p:=\psi_{T,\mu_h}(p)\circ\zeta_p$, $\xi_{-\mu}:\fG_p\rightarrow \fG_{T_{\Qp}}$ factor through $\fG_p^K$, and $\xi_p$ is conjugate to $\xi_{-\mu}^K$ under $T(\Qpb)$, so $\xi_p^{\Delta}=(\xi_{-\mu}^K)^{\Delta}=-\Nm_{K/\Qp}\mu_h$.
By (1), the centralizer $\Cent_G(\gamma_0)$ equals that of $\phi(\delta_k)\in T(\Q)$ which is in turn equal to the centralizer $\Cent_G(\phi^{\Delta})$ of $\phi^{\Delta}$ (as $\{\delta_k^n\}_{n\in\N}$ is Zariski-dense in $P$). But, obviously $-\Nm_{K/\Qp}\mu_h=(\phi(p)\circ \zeta_p)^{\Delta}=(\phi^{\Delta})_{\Qp}\circ \zeta_p^{\Delta}$ maps into a $\Qp$-split torus in $\mathrm{Im}(\phi^{\Delta})_{\Qp}$, thus a posteriori into a $\Qp$-split torus in the center of $\Cent_G(\phi^{\Delta})_{\Qp}$, from which the claim follows. 

(3) The additional assumption tells us that $|\chi(\phi(\delta_k))|_{\infty}=|\chi(\gamma_0^{\frac{k}{n}})|_{\infty}$ for every $\Q_{\infty}$-rational character $\chi$ of $T$, and also implies that $\gamma_0\in G(\A_f^p)$ itself lies in a compact open subgroup of $G(\A_f^p)$.
Hence, by the argument of (1), $\phi(\delta_k)\cdot \gamma_0^{-\frac{k}{n}}\in G(\Q)$ is a torsion element.

(4) It is well-known that the stated condition implies that for any maximal $\Q$-torus $T_0$ of $G$, elliptic over $\R$, $T_0(\Q)$ is discrete in $T_0(\A_f)$; this is the condition what Kisin called \emph{the Serre condition for $T_0$}, \cite[(3.7.3)]{Kisin13}. Then, again we resort to the argument of (1).
\end{proof}

\begin{lem} \label{lem:invariance_of_(ast(gamma_0))_under_transfer_of_maximal_tori}
Let $\gamma_0, \gamma_0'\in G(\Q)$ be semi-simple elements. 

(1) If there exists a transfer of maximal tori $\Int g:T\hra G$ such that $\gamma_0\in T(\Q)$ and $\gamma_0'=\Int g(\gamma_0)$, then the condition $(\ast(\gamma_0))$ in Subsec. \ref{subsubsec:pre-Kottwitz_triple} holds for $(\gamma_0, T)$ and some $\mu\in X_{\ast}(T)\cap \{\mu_X\}$ if and only if it holds for $(\gamma_0', T',\mu'):=\Int g(\gamma_0,T,\mu)$.

(2) If $\gamma_0':=\Int g(\gamma_0)$ for some $g\in G(\Qb)$, for any place $v$ of $\Q$, the image of $\gamma_0$ in $G^{\ad}(\Ql)$ lies in a compact open subgroup of $G^{\ad}(\Ql)$ if and only if $\gamma_0'$ is so.
\end{lem}

\begin{proof}
(1) Let $K\supset \Qp$ be a splitting field of $T_{\Qp}$ (and thus of $T'_{\Qp}\approx T_{\Qp}$ too).
Then, using that $\Int g:T\isom T'$ is a $\Q$-isomorphism, we see that
$\Int g(\Nm_{K/\Qp}\mu)=\Nm_{K/\Qp}\mu'$, which implies that
$\Nm_{K/\Qp}\mu$ maps into the center of $\Cent_G(\gamma_0)$ if and only if $\Nm_{K/\Qp}\mu'$ does so for $\Cent_G(\gamma_0')$. 
Next, let $H$ be the centralizer of the maximal $\Qp$-split subtorus of the center of $\Cent_G(\epsilon)_{\Qp}$ and $H'$ the similarly defined group for $\gamma_0'$. 
Since $H^{\der}$ is simply connected, $\lambda_H(\gamma_0)=\lambda_{H^{\ab}}(\gamma_0)$, so $\lambda_H(\gamma_0)=[K:\Qp]\underline{\mu}$ if and only if $\lambda_{H^{\ab}}(\gamma_0)=[K:\Qp]\underline{\mu}$, where the latter $\underline{\mu}$ is the image of $\mu$ in $X_{\ast}(H^{\ab})$.
But, $\Int g:\Cent_G(\gamma_0)_{\Qb}\isom \Cent_G(\gamma_0')_{\Qb}$ restricts to a $\Q$-isomorphism between their centers (as they are contained in $T$ and $T'$, respectively), so also induces a $\Qp$-isomorphism between the maximal $\Qp$-split subtori of their base-change to $\Qp$. Hence, $\Int g:H_{\Qpb}\isom H'_{\Qpb}$ and this induces a $\Qp$-isomorphism $H^{\ab}=H/H^{\der}\isom H'^{\ab}=H'/H'^{\der}$. Clearly, this proves the claim. 

(2) Let $P$ and $P'$ be the $\Q$-subgroups of $G$ generated by $\gamma_0$, $\gamma_0'$, respectively.
Then, the $\Qb$-isomorphism $\Int g:G_{\Qb}\isom G_{\Qb}$ restricts to a $\Q$-isomorphism $P\isom P'$, thus if the image of $\gamma_0$ in $G^{\ad}(\Ql)$ lies in a compact open subgroup of $G^{\ad}(\Ql)$, then as it also lies in a compact open subgroup of $Q(\Ql)$, where $Q$ is the image of $P$ in $G^{\ad}$, the same property holds for $\gamma_0'$.
\end{proof}

\begin{thm} \label{thm:LR-Satz5.21}
Suppose that $G_{\Qp}$ is quasi-split and tamely ramified, $G$ is of classical Lie type with simply connected $G^{\der}$, and that the Serre condition for $(G,X)$ holds. Let $\mbfK_p$ be special maximal parahoric. 

(1)  Let $\epsilon\in G(\Q)$ be a rational element, elliptic over $\R$, and whose image in $G^{\ad}(\Ql)$ lies in a compact open subgroup of $G^{\ad}(\Ql)$ for every finite place $l\neq p$. If there exists an admissible morphism $\phi$ making $(\phi,\epsilon)$ an admissible pair, then the stable conjugacy class of $\epsilon$ contains an element $\gamma_0\in G(\Q)$ which satisfies the condition $(\ast(\gamma_0))$ of Subsec. \ref{subsubsec:pre-Kottwitz_triple} with the same level as $(\phi,\epsilon)$.

Conversely, if $\gamma_0\in G(\Q)$ satisfies the condition $(\ast(\gamma_0))$ (especially, $\gamma_0$ is elliptic over $\R$), there exists an admissible pair $(\phi,\epsilon)$ with $\epsilon$ stably conjugate to $\gamma_0$. In fact,
we can also find a $\mbfK_p$-effective admissible pair $(\phi,\epsilon)$ with $\epsilon$ stably conjugate to $\gamma_0^t$ for some $t\in\N$ (cf. Remark \ref{rem:admissible_pair}).

(2) Suppose that $\gamma_0\in G(\Q)$ is elliptic over $\R$ and that there exists $\delta\in G(L_n)$ such that $c\gamma_0c^{-1}=\Nm_{L_n/\Qp}(\delta)$ for some $c\in G(\Qpb)$ and $X(\delta,\{\mu_X\})_{\mbfK_p}\neq \emptyset$.
Then, the condition $(\ast(\gamma_0))$ of Subsec. \ref{subsubsec:pre-Kottwitz_triple} holds with level $n=m[\kappa(\wp):\Fp]$. In particular, there exists an admissible pair $(\phi,\epsilon)$ with $\epsilon$ stably conjugate to $\gamma_0$.
\end{thm}
 
The claim (1) in the hyperspecial level case is Satz 5.21 of \cite{LR87}, except for two differences: in the original theorem, there is no condition on $\epsilon\in G^{\ad}(\A_f^p)$ and a condition (denoted by $(\ast(\epsilon))$) different from our condition $(\ast(\gamma_0))$ here is used (cf. Remark \ref{rem:condition_(ast(gamma_0))}, (2)). However, the proof (especially, of the second implication) given in loc. cit. based on that condition seemed to us to be incomplete: see Footnotes \ref{ftn:pseudo_Cartan_decomp}, \ref{ftn:1st_gap}, \ref{ftn:2nd_gap}, and \ref{ftn:3rd_gap} for explanation of the points that we believe to be gaps. On the other hand, hopefully the claim (2) (which is due to us) and thm. \ref{thm:LR-Satz5.25} (which follows from it) should justify the introduction of the new condition, cf. Remark \ref{rem:LR-Satz5.25}. Also, we remark that the main point of the condition $(\ast(\epsilon))$ in the original work \cite{LR87} was Satz 5.21, which corresponds to this theorem.

\begin{proof}
To a large extent, we follow the original strategy, but using our condition $(\ast(\gamma_0))$ (instead of the original one $(\ast(\epsilon))$) as well as some of those facts that were established in our general setting of (special maximal) parhoric level, especially Prop. \ref{prop:existence_of_elliptic_tori_in_special_parahorics} and Lemma \ref{lem:unramified_conj_of_special_morphism}. 

Before entering into the proof, we discuss an explicit expression of the Frobenius automorphism $\Phi=F^n$ attached to a special admissible morphism.
Let $\phi=i\circ\psi_{T,\mu_h}$ be a special admissible morphism, where $(T,h)$ is a special Shimura sub-datum and $i:\fG_T\rightarrow\fG_G$ is as usual the canonical morphism induced by the inclusion $T\hra G$. Put $\xi_p:=\phi(p)\circ\zeta_p$; so $\xi_p$ and $\xi_{-\mu_h}$ are conjugate under $T(\Qpb)$ (and in particular, $\xi_p^{\Delta}=\xi_{-\mu_h}^{\Delta}$). The centralizer $J=\Cent_{G_{\Qp}}(\xi_p^{\Delta})$ of (the image of) $\xi_p^{\Delta}:\mathbb{D}\rightarrow G_{\Qp}$ is a quasi-split $\Qp$-Levi subgroup of $G_{\Qp}$. 
Suppose given an elliptic maximal torus $T'$ of $J$. We choose $j\in J'(\Qpb)$ with $T'=\Int (j)(T_{\Qp})$, and set
\[\mu':=\Int (j)(\mu_h)\quad \in X_{\ast}(T').\]
Let $K$ be a finite Galois extension of $\Qp$ splitting $T'$, $\pi$ a uniformizer of $K$, and $K_0$ the maximal unramified subextension of $K$. We take $K$ to be big enough such that there exists a finite Galois extension $K_h$ of the same degree (over $\Qp$) as $K$ and splitting $T$. Note that in doing so, we can always assume $K$ to be an unramified extension of the splitting filed of $T'$.
If $K_1$ is the composite of $K$ and $L_s$ with $s=[K:\Qp]$ and $\xi_{-\mu'}^{K_1}$ is the pull-back of $\xi_{-\mu'}^{K}$ to $\fG_p^{K_1}$, 
according to Lemma \ref{lem:unramified_conj_of_special_morphism}, there exists $t_p\in T'(\Qpb)$ such that $\Int(t_p)(\xi_{-\mu'}^{K_1})$ is an unramified morphism mapping into $\fG_{T'}$ and factors through $\fG_p^{L_s}$. Moreover, we can choose $t_p$ further such that if $\xi_p':\fG_p^{L_s}\rightarrow \fG_{T_{\Qp}}$ denotes the induced unramified Galois $\Qpb/\Qp$-gerb morphism,
\[\xi_p'(s_{\sigma}^{L_s})=\Nm_{K/K_0}(\mu'(\pi))\rtimes \sigma.\] 
From now on, we write $\xi_{-\mu_h}$, $\xi_{-\mu'}$ for $\xi_{-\mu_h}^{K_h}$, $\xi_{-\mu'}^{K}$, respectively. 
In this set-up, we note that if we put 
\[\qquad \nu_p':=-\Nm_{K/\Qp}\mu'\] 
($\Qp$-rational cocharacter of $T'$), there holds the equality: 
\begin{equation} \label{eqn:equality_of_two_norm_cochars} 
\nu_p'=-\Nm_{K_h/\Qp}\mu_h. 
\end{equation}
This was already noted in our proof of Lemma \ref{lem:LR-Lemma5.2}.
Indeed, first we see that they both map into the center of $J$: for $\Nm_{K_h/\Qp}\mu_h$, this is by definition of $J$ (and as $-\Nm_{K_h/\Qp}\mu_h=(\xi_{-\mu_h}^{K_h})^{\Delta}=(\phi(p)\circ\zeta_p)^{\Delta}$), while $\nu_p'=-\Nm_{K/\Qp}\mu'$ maps into a $\Qp$-split sub-torus of the elliptic maximal torus $T'$ of $J$, so factors through the center. But, also their projections into $J^{\ab}=J/J^{\der}$ are the same, as $\mu'$ and $\mu_h$ are conjugate under $J(\Qpb)$ and $J^{\ab}$ is $\Qp$-split (and since $[K:\Qp]=[K_h:\Qp]$). Clearly this proves the claim.

Next, let $[K:K_0]=e_K$, $[K_0:\Qp]=f_K$. For any $j\in \N$, if it is divisible by $[K:\Qp]=e_Kf_K$, say $j=t'[K:\Qp]$, we have the following expression for $F^j$ in terms of $\nu_p'$:
\begin{eqnarray} \label{eqn:F^j}
F^{j}&=&(\Nm_{K/K_0}(\mu'(\pi))\rtimes \sigma)^j=\Nm_{K/\Qp}(\mu'(\pi^{e_K}))^{t'}\rtimes \sigma^{j} \\
&=&(p^{\Nm_{K/\Qp}\mu'}\cdot \Nm_{K/\Qp}(\mu'(u)))^{t'}\rtimes \sigma^{j} \nonumber \\
&=&p^{-t'\nu_p'}\cdot u_0^{t'}\rtimes\sigma^{j} \nonumber
\end{eqnarray}
where $\pi^{e_K}=pu$ for $u\in \cO_K^{\times}$ and 
\[u_0:=\Nm_{K/\Qp}(\mu'(u)).\] 
A priori, $u_0\in T'(\Qp)_0(=\Ker(v_{T'_L})\cap T'(\Qp))$ (maximal compact subgroup of $T'(\Qp)$), but in fact it belongs to $T'(\Qp)_1(=\Ker (w_{T'_L})\cap T'(\Qp))$. To see that, by funtoriality for tori $T$ endowed with a cocharacter $\mu\in X_{\ast}(T)$, we can take $T'=\Res_{K/\Qp}\Gm$ and $\mu'=\mu_K$, the cocharacter of $T'_K=\Gm^{\Hom(K,K)}$ corresponding to the identity embedding $K\hra K$. But in this case, $X_{\ast}(T')$ is an induced $\Gal(K/\Qp)$-module, so $w_{T'_L}=v_{T'_L}$, and clearly $u_0\in \Ker_{v_{T'_L}}$.

(1) Let us turn to the proof and first establish the necessity of the condition $(\ast(\gamma_0))$. 
By a suitable transfer of maximal tori (Lemma \ref{lem:LR-Lemma5.23}), we may assume that the admissible pair $(\phi,\epsilon)$ is nested in some Shimura sub-datum $(T,h)$, i.e. $\phi=i\circ\psi_{T,\mu_h}$ and $\epsilon\in T(\Q)$. We will verify the condition $(\ast(\gamma_0))$ for $(\gamma_0:=\epsilon,T,\mu_h)$. 
First, according to Lemma \ref{lem:invariance_of_(ast(gamma_0))_under_transfer_of_maximal_tori}, (2), this new $\epsilon$ still satisfies the assumption, i.e. its image in $T/Z(G)(\Ql)$ is a unit (i.e. lies in a compact open subgroup) for every (finite) place $l\neq p$. Therefore, if $K$ splits $T$, by Prop. \ref{prop:phi(delta)=gamma_0_up_to_center}, (2), $\Nm_{K/\Qp}\mu_h$ maps into the center of $\Cent_G(\epsilon)$.
For the special admissible morphism $\phi=i\circ\psi_{T,\mu_h}$, let $\xi_p$, $J$, $T'$, $\mu'$, $\xi_p'$, ... be defined as discussed in the beginning of the proof. It was shown in the proof of Lemma \ref{lem:LR-Lemma5.2} (using that $T'\subset J$ is elliptic) that the two Galois $\Qpb/\Qp$-gerb morphisms $\xi_{-\mu_h}$, $\xi_{-\mu'}$ are conjugate under $J(\Qpb)$. So, since $\xi_p=i\circ\psi_{T,\mu_h}\circ\zeta_p$ and $\xi_{-\mu_h}$ are conjugate under $T(\Qpb)$, there exists $v\in J(\Qpb)$ such that 
\[\Int (v)(\xi_p)=\xi_p'(=\Int (t_p)(\xi_{-\mu'})).\]
Moreover, we could choose $v$ such that $T'=\Int(v)(T_{\Qp})$ (so that $\epsilon'\in T'(\Qpnr)$). Indeed, pick an arbitrary $v_1\in J(\Qpb)$ with this property. Here, we used the condition that $\nu_p'=-\Nm_{K_h/\Qp}\mu_h$ maps into the center of $\Cent_G(\epsilon)_{\Qp}$, since then $(T\subset)H\subset J$. By Lemma \ref{lem:equality_restrictions_to_kernels_imply_conjugacy}, we just need to show that $\Int(v_1)(\xi_p)^{\Delta}=(\xi_p')^{\Delta}(=\xi_{-\mu'}^{\Delta})$. But, $\Int(v_1)(\xi_p)^{\Delta}=\Int(v_1)(\xi_p^{\Delta})=\Int(v_1)(-\Nm_{K_h/\Qp}\mu_h)=-\Nm_{K_h/\Qp}\mu_h=-\Nm_{K/\Qp}\mu'(=\xi_{-\mu'}^{\Delta})$ by (\ref{eqn:equality_of_two_norm_cochars}).
Using such $v$, we set
\[\epsilon':=\Int (v)(\epsilon).\] 
This is an element of $T'(\Qpnr)$: a priori, this is only an element of $T'(\Qpb)$, but since it commutes with an unramified morphism $\xi_p'=\Int(v)(\xi_p)$, it belongs to $T'(\Qpnr)$.

Now, by definition of admissible pair, there exists $x\in G(L)/\mbfKt_p$ such that $\epsilon'x=\Phi^mx$. But then $(\epsilon')^tx=\Phi^{tm}x$ for every $t\in\Z$, since $\epsilon'$ commutes with $\Phi=F^r$ (as $\epsilon$ commutes with $\xi_p=\phi(p)\circ\zeta_p$). Hence, for sufficiently large $t\in\N$ such that $trm$ is divisible by $e_Kf_K$, say $trm=t'[K:\Qp]$, if $k\in G(L)$ is defined by that $k\rtimes \sigma^{tn}=(\epsilon')^{-t}\Phi^{tm}$, i.e. 
\begin{equation} \label{eqn:k_lies_in_cpt_open}
k=(\epsilon')^{-t}\cdot \Nm_{K/\Qp}(\mu'(\pi))^{e_Kt'} =(\epsilon')^{-t}\cdot p^{-t'\nu_p'}\cdot u_0^{t'},
\end{equation}
$k$ lies in the special maximal bounded subgroup $\mathrm{Stab}(x)$ of $G(L)$.
Moreover, $k\in T'(L)$ as both $\epsilon'$ and $\Nm_{K/\Qp}(\mu'(\pi))$ are so. Therefore, $k$ lies in the maximal bounded subgroup $T'(L)_0=\Ker(v_{T'_L})$ of $T'(L)$, and so is $\epsilon'^{-t}\cdot p^{-t'\nu_p'}=k\cdot u_0^{-t'}$.
\footnote{Noe that this is a consequence of our assumption that  $\nu_p'=-\Nm_{K_h/\Qp}\mu_h$ maps into the center of $\Cent_G(\epsilon)_{\Qp}$. Without such additional condition, it seems not clear whether $v_{H'}((\epsilon')^{-t}\cdot p^{-t'\nu_p'})=0$ for $H':=v'(H_L)v'^{-1}$.
On the other hand, in the original proof, Langlands and Rapoport deduce this condition by appealing to uniqueness of ``a kind of polar decomposition" $(\epsilon')^{t}= p^{-t'\nu_p'}\cdot k^{-1}$ in (\ref{eqn:k_lies_in_cpt_open}) (in the hyperspecial case, $u_0=1$). By this they must mean a Cartan decomposition for $\Cent_{G_L}(\epsilon')$ (since $\nu_p'$ is wanted to map into its center), which needs that $k$ lies in a \emph{special} maximal bounded subgroup of $\Cent_{G_L}(\epsilon')(L)$. But, this seems not quite obvious either: one only knows that $k$ lies in \emph{some} maximal bounded subgroup of $\Cent_{G_L}(\epsilon')(L)$. Since $v_{H'}$ vanishes on special maximal bounded subgroups (Lemma \ref{lem:Kottwitz84-a1.4.9_b3.3}), hence we see that the condition in question is in fact ``equivalent" to the property we need.\label{ftn:pseudo_Cartan_decomp}}
Since $\Int (v')$ is an $L$-isomorphism from $H_L$ to its image $H'$ which must contain $T'$, and $v'\in J(\Qpb)=\Cent_{G_{\Qp}}(\nu_p')$ (\ref{eqn:equality_of_two_norm_cochars}), it follows that 
\begin{equation} \label{eqn:lambda(epsilon)=nu_p'}
\epsilon^{-t}\cdot p^{-t'\nu_p'}\stackrel{}{=}v'^{-1}(\epsilon'^{-t}\cdot p^{-t'\nu_p'})v'\in \Ker(v_{H_L}),
\end{equation}
that is, $\lambda_H(\epsilon^{-t}\cdot p^{-t'\nu_p'})=0$.
Hence, noting that the canonical $\Gamma(p)$-action on $\pi_1(H)=X_{\ast}(H^{\ab})$ is trivial, i.e. $\pi_1(H)_I=\pi_1(H)=X_{\ast}(H^{\ab})$ ($H^{\ab}$ splits over $\Qp$), we see that
\[t\lambda_{H}(\epsilon)=\lambda_{H}(\epsilon^t)\stackrel{(\ref{eqn:lambda(epsilon)=nu_p'})}{=}\lambda_{H}(p^{-t'\nu_p'})\stackrel{(\ref{eqn:equality_of_two_norm_cochars})}{=}\lambda_{H}(p^{t'\sum_{K/\Qp}\mu'})=t'\underline{\Nm_{K/\Qp}\mu'}=tn\underline{\mu'},\]
Here, as usual for a maximal torus $T$ of $H$ and $x\in X_{\ast}(T)$, $\underline{x}$ denotes the image of $x$ in $\pi_1(H)_{\Gal_{\Qpnr}}=\pi_1(H)$.
\footnote{In \cite[p.192, line 3-6]{LR87}, Langlands and Rapoport take up a $G(\Qpb)$-conjugate 
$\mu''$ of $\mu'$ (equiv. of $\mu_h$) factoring through $H$ and defined over $L_n$ (which exists, according to the definition of the reflex field and \cite[Lemma 1.1.3]{Kottwitz84b}) and insist that $\underline{\mu'}_{\pi_1(H)}=\underline{\mu''}_{\pi_1(H)}$. But this equality is questionable, since $\mu'$ and $\mu''$ may not be conjugate under $H$. Fortunately, for our main application of the theorem at hand, this property does not play a significant role.\label{ftn:1st_gap}}

Next, assume that $\gamma_0\in G(\Q)$ satisfies the two conditions in the statement of the theorem:
There exists a maximal $\Q$-torus $T_0$ of $\Cent_G(\gamma_0)$ that is elliptic at $\R$, and a cocharacter $\mu_0$ of $T_0$ that satisfies the condition $(\ast(\gamma_0))$ of Subsec. \ref{subsubsec:pre-Kottwitz_triple}  (with some level $n$ divisible by $[\kappa(\wp):\Fp]$).
Then, with such choice, one applies Lemma \ref{lem:LR-Lemma5.12} (or Lemma 5.12 of \cite{LR87}) to find an admissible embedding of maximal torus $\Int g_0:T_0\hra G$ such that $\Int g_0\circ i\circ\psi_{T_0,\mu_0}$ equals $i\circ\psi_{T,\mu_h}$ for some special Shimura sub-datum $(T,h)$. By Lemma \ref{lem:invariance_of_(ast(gamma_0))_under_transfer_of_maximal_tori}, the condition $(\ast(\gamma_0))$ continues to hold for $(\epsilon:=\Int g_0(\gamma_0),T,\mu_h)$.
Now, we check that for the resulting pair $(\phi,\epsilon \in I_{\phi}(\Q))=(i\circ\psi_{T,\mu_h},\Int g_0(\gamma_0))$, $(\phi,\epsilon^t)$ is $\mbfK_p$-effective admissible for some $t\in\N$ and admissible with $t=1$ (Def. \ref{def:admissible_pair}, Remark \ref{rem:admissible_pair}). As we are working with special maximal parahoric $\mbfK_p$, by Lemma \ref{lem:LR-Lemma5.2}, $\phi:=\Int g_0\circ i\circ\psi_{T,\mu_h}$ is admissible. The condition (2) is obvious. So, it remains to establish the condition (3).

As $H$ is a quasi-split $\Qp$-Levi subgroup of $G_{\Qp}$, there exists a $G(\Qp)$-conjugate ${}^gH:=\Int g(H)$ ($g\in G(\Qp))$ of $H$ such that ${}^gH(\Qp)\cap \mbfK_p$ is a special maximal parahoric subgroup of ${}^gH(\Qp)$ (Lemma \ref{lem:specaial_parahoric_in_Levi}, cf. proof of Lemma \ref{lem:LR-Lemma5.2}). Then, we apply Prop. \ref{prop:existence_of_elliptic_tori_in_special_parahorics} to ${}^gH(\Qp)\cap \mbfK_p^g$ and choose an elliptic maximal $\Qp$-torus $T'$ of ${}^gH$ such that $T'_{\Qpnr}$ contains (equiv. is the centralizer of) a maximal $\Qpnr$-split $\Qpnr$-torus of ${}^gH_{\Qpnr}$ and that the (unique) parahoric subgroup $T'(L)_1$ of $T'(L)$ is contained in $T'(L)\cap \mbfKt_p$ (as usual, $\mbfKt_p$ being the parahoric subgroup of $G(L)$ corresponding to $\mbfK_p$); in fact, $T'(L)_1=T'(L)\cap \mbfK_p$. So, we are in the situation discussed in the beginning of the proof.
Let $\mu'$, $\xi_p'$, $K$, $K_0$, $K_h$, ... be defined as before, except for the following changes, caused by that $T'$ is now an elliptic maximal torus of ${}^gH$, not $J$: We take $\mu'\in X_{\ast}(T')$ to be a conjugate of ${}^g\mu_h:=\Int g(\mu_h)(\in X_{\ast}({}^gT))$ \emph{under ${}^gH(\Qpb)$}. Then, because $\Nm_{K/\Qp}{}\mu_h$ maps into the center of $H$ by our assumption, the same argument above establishing (\ref{eqn:equality_of_two_norm_cochars}) gives that $\Nm_{K/\Qp}\mu'=\Nm_{K_h/\Qp}{}^g\mu_h$. Then, in turn using this equality, again the proof of Lemma \ref{lem:LR-Lemma5.2} establishes that the two Galois $\Qpb/\Qp$-gerb morphisms into $\fG_{{}^gH}$, $\xi_{-{}^g\mu_h}$ and $\xi_{-\mu'}$, are conjugate under ${}^gH(\Qpb)$, and so are ${}^g\xi_p={}^g\phi(p)\circ\zeta_p$ and $\xi_{-\mu'}$.
\footnote{This is the first reason why we need the new assumption that $\Nm_{K_h/\Qp}\mu_h$ maps into the center of $\Cent_G(\epsilon)$: without it, there is no gurantee that $\xi_{-\mu'}$ (which defines $F$) is conjugate to $\xi_{-{}^g\mu_h}$ (so conjugate to $\xi_p=\phi(p)\circ\zeta_p$). \label{ftn:2nd_gap}}
Hence, there exists a $u\in {}^gH(\Qpb)$ such that $\xi_p':=\Int (ug)(\xi_p)$ is an unramified conjugate of ${}^g\xi_p$ which factors through $\fG_p^{L_s}$. As was noted above, we can choose such $u$ so as to satisfy further that $T'=\Int (ug)(T_{\Qp})$ and $\xi_p'$ maps into $\fG_{T'}$; in particular, $\epsilon':=\Int (ug)(\epsilon)\in T'(\Qpnr)$. 
Also, as $\epsilon',{}^g\epsilon:=\Int g(\epsilon)\in {}^gH(\Qpnr)$ are conjugate under ${}^gH(\Qpb)$, they are so under ${}^gH(\Qpnr)$, by a theorem of Steinberg; suppose $\epsilon'=\Int (u'g)(\epsilon)$ for $u'\in {}^gH(\Qpnr)$.
With this preparation, for $j\in\N$, let us define $b_j\in T'(K_0)$ by
\[ b_j\rtimes\sigma^{j}:=F^j=(\Nm_{K/K_0}(\mu'(\pi))\rtimes\sigma)^{j}. \]
So, we have that $b_n = \prod_{i=1}^n\sigma^i(\Nm_{K/K_0}(\mu'(\pi)))$.
Also, if $[K:\Qp]$ divides $j$, by (\ref{eqn:F^j}), we have that $b_j = (p^{-\nu_p'}\cdot u_0)^{\frac{j}{[K:\Qp]}}$,
for $\nu_p'=-\Nm_{K/\Qp}\mu'$ and $u_0\in T'(\Qp)_1$; in particular, $b_j\in T'(\Qp)$. 
We will also write $b$ for $b_1$.

Now, by assumption $(\ast(\gamma_0))$, we have that 
\[[K:\Qp]\lambda_{H}(\epsilon)=[K:\Qp]n\underline{\mu_h}=\lambda_{H}(p^{n\Nm_{K_h/\Qp}\mu_h}),\]
and that $\Nm_{K_h/\Qp}\mu_h$ lies in the \emph{connected} center $Z_{\epsilon}$ of $\Cent_G(\epsilon)$.
It follows from this that the element of $T(\Qp)\ (\subset H(\Qp)\cap\Cent_G(\epsilon)(\Qp))$:
\[k_0=\epsilon^{-[K:\Qp]}\cdot p^{-n\Nm_{K_h/\Qp}\mu_h}\]
lies in $\Ker(\lambda_H)\cap Z(\Cent_G(\epsilon))$. We claim that for some $a\in\N$,
\[k_0^{a}\in \Ker(\lambda_{Z_{\epsilon}}),\]
i.e. $k_0^a$ lies in the maximal compact open subgroup of $Z_{\epsilon}(\Qp)\subset T(\Qp)$.
First, take $a_1\in\N$ with $\epsilon^{a_1}\in Z_{\epsilon}(\Q)$, so that 
\[k_0^{a_1}\in Z_{\epsilon}(\Qp).\]
As the natural map $Z(H)\twoheadrightarrow H^{\ab}$ is an isogeny of split tori, there exists $a_2\in\N$ such that $H^{\ab}(\Qp)_0^{a_2}$ is contained in the image of $Z(H)(\Qp)_0$ under this isogeny (clearly, for $a_2$ we can take the degree of the isogeny). Then, since the kernel of the isogeny $Z_{\epsilon}\hra H\twoheadrightarrow H^{\ab}$ is $Z_{\epsilon}\cap H^{\der}$, we can find $x_0\in Z(H)(\Qp)_0(:=\Ker(\lambda_{Z(H)}))$ and $a_3\in \N$ such that 
\[(k_0^{a_1a_2}\cdot x_0^{-1})^{a_3}\in Z_{\epsilon}'(\Qp),\]
where $Z_{\epsilon}':=((Z_{\epsilon})_{\Qp}\cap H^{\der})^0$ (it suffices that $a_3$ kills the finite group $\pi_0((Z_{\epsilon})_{\Qp}\cap H^{\der})$).
But, by definition of $H$, $Z(H)$ is the maximal $\Qp$-split sub-torus of $(Z_{\epsilon})_{\Qp}$, hence
$Z_{\epsilon}'$ is the anisotropic kernel of $(Z_{\epsilon})_{\Qp}$ (the maximal anisotropic subtorus of $(Z_{\epsilon})_{\Qp}$), and $Z_{\epsilon}'(\Qp)=Z_{\epsilon}'(\Qp)_0(:=\Ker(\lambda_{Z_{\epsilon}'}))$. Therefore,  $a:=\prod_{i=1}^{3}a_i$ satisfies that $k^{a}\in \Ker(\lambda_{Z_{\epsilon}})$.

It follows that for sufficiently large $t$ divisible by $a[K:\Qp]$, and for $v':=u'g$, the element of $T'(\Qpnr)$:
\begin{eqnarray} \label{eqn:epsilon'^{-1}F^n}
(\epsilon'^{-1}F^n)^{t}&=&v' \epsilon^{-t} v'^{-1} \cdot (p^{-\nu_p'}\cdot u_0)^{\frac{nt}{[K:\Qp]}}\rtimes\sigma^{nt}\\
&=&v' (\epsilon^{-[K:\Qp]}\cdot p^{n\Nm_{K_h/\Qp}\mu_h})^{\frac{t}{[K:\Qp]}} v'^{-1} \cdot u_0^{\frac{nt}{[K:\Qp]}}\rtimes\sigma^{nt}  \nonumber \\
&=& v' k_0^{\frac{t}{[K:\Qp]}} v'^{-1}\cdot (u_0)^{\frac{nt}{[K:\Qp]}}\rtimes\sigma^{nt} \nonumber
\end{eqnarray}
lies in any given neighborhood of $1$ in $T'(L)$, in particular in the unique (special maximal) parahoric subgroup $T'(L)_1$ of $T'(L)$ which is, by our choice of ${}^gH$ and $T'$, equal to $ T'(L)\cap \mbfKt_p$. Consequently by \cite[Prop. 3]{Greenberg63}, for sufficiently large $t$, there exists $h\in T'(L)\cap \mbfKt_p$ such that
\begin{equation} \label{eqn:epsilon'^{-1}F^n_has_fixed_pt}
(\epsilon'^{-1}\Phi^m)^{t}=h^{-1}\sigma^{tn}(h)\rtimes\sigma^{tn}.
\end{equation}
We note in passing that for sufficiently large $t$ with $\nu':=-\frac{nt}{[K:\Qp]}\nu_p'\in X_{\ast}(T'),$
\[(\epsilon')^t= p^{\nu'}\cdot k\]
with $k:=(\epsilon')^t\cdot p^{-\nu'}=v'(\epsilon^{[K:\Qp]}\cdot p^{n\Nm_{K_h/\Qp}\mu_h})^{\frac{t}{[K:\Qp]}} v'^{-1}=v' k_0^{-\frac{t}{[K:\Qp]}} v'^{-1}$ lying in $\Ker(v_{{}^gH_L})$.
\footnote{In \cite[p.193, line 9]{LR87}, Langlands and Rapoport state that this decomposition $\epsilon'^{t}=p^{\nu'}k$ is ``a polar decomposition" with $p^{\nu'}$ belonging to the center (of ${}^gH$); by this they must mean some Cartan decomposition and, as such, claim that $k$ lies in some (hyper)special subgroup, or at least in a compact open subgroup of ${}^gH(\Qpnr)$. But it is not clear how to show this from their condition $(\ast(\epsilon))$ only, which just tells us that  $v_{{}^gH_L}(k)=v_{{}^gH^{\ab}_L}(k)=0$. This is our second reason for introducing the new assumption $(\ast(\gamma_0))$. \label{ftn:3rd_gap}}

We fix $t\in\N$ for which the equation (\ref{eqn:epsilon'^{-1}F^n_has_fixed_pt}) has a solution $h\in T'(L)\cap \mbfKt_p$. Then, $\epsilon'^{-t}\Phi^{mt}$ fixes $h x^0=x^0$ ($x^0=1\cdot \mbfKt_p$ being the base point of $G(L)/\mbfKt_p$).
But, we have seen in the proof of Lemma \ref{lem:LR-Lemma5.2} that for our choice of $F$ and $x^0$,
$\mathrm{inv}_{\mbfKt_p}(x^0,F x^0)=\Adm_{\mbfKt_p}(\{\mu_X\})$. This proves that $(\phi,\epsilon^t)$ is $\mbfK_p$-effective admissible.

Next, we show that there exists $e\in G(L)$ such that \[e^{-1}(\epsilon')^{-1}\Phi^m e=\sigma^n.\] Obviously, this will establish the sufficiency direction of the admissibility of $(\phi,\epsilon)$. 
We first claim that there exists $c\in T'(L)$ such that \[c^{-1}(\epsilon')^{-1}\Phi^m c\in T'^{\der}(L)\times\sigma^n,\] where $T'^{\der}:={}^gH\cap T'$ (which is connected, i.e. a torus, since $H^{\der}$ is simply connected). 
Since $T'$ splits over $L$, we make take $K$ to be an unramified extension of $\Qp$. So, by (\ref{eqn:comparison_of_two_norms}), 
\begin{eqnarray*}
[K:K_0]w_{T'_L}(b_n)&=&[K:K_0]\sum_{i=1}^n\sigma^i(w_{T'_L}(\Nm_{K/K_0}(\mu'(\pi)))) \\
&=&\sum_{i=1}^n\sigma^i(w_{T'_L}(\Nm_{K/K_0}(\mu')(p))) \\
&=& \sum_{i=1}^n\sigma^i(\underline{\Nm_{K/K_0}(\mu')}),
\end{eqnarray*}
thus we have that $v_{{}^gH^{\ab}_L}(b_n)=n\underline{\mu'}$ (as $H^{\ab}$ is $\Qp$-split).
Then, since $\mu'$ and $\mu_h$ are ${}^gH$-conjugate, by assumption $(\ast(\gamma_0))$, it follows that $v_{{}^gH^{\ab}_L}(\epsilon'^{-1}b_n)=v_{{}^gH_L}(\epsilon'^{-1}b_n)=0$. 
So, there exists $f\in H^{\ab}(L)_0$ such that $\mathrm{pr}_H(\epsilon'^{-1}b_n)=f\sigma^n(f^{-1})$, where $\mathrm{pr}_H:H\rightarrow H^{\ab}$ is the natural projection. Since $T'^{\der}$ is connected, the natural map $T'(L)\rightarrow {}^gH^{\ab}(L)$ is surjective, thus there exists a $c\in T'(L)$ mapping to $f$. Then, we proceed as in the proof (on p. 193) of \cite[Satz 5.21]{LR87} to find $d\in {}^gH^{\uc}(L)\cap \mbfKt_p$ such that $d^{-1}c^{-1}\epsilon'^{-1}\Phi^mcd=\sigma^n$. By \cite[Prop. 5.4]{Kottwitz85}, It suffices to show that $c^{-1}(\epsilon')^{-1}\Phi^m c\in {}^gH^{\der}(L)\rtimes\sigma^n$ is basic, which is by definition (\cite[(4.3.3)]{Kottwitz85}) the same as the existence of $d'\in {}^gH^{\der}(L)$ with
$c^{-1}(\epsilon'^{-1}\Phi^m)^tc=d'\sigma^{tn}(d'^{-1})$ 
for some sufficiently large $t$, which obviously follows from the equation (\ref{eqn:epsilon'^{-1}F^n}). Again, here we need that $k$ (equiv. $k_0$) lies in a compact open subgroup of ${}^g(H)(\Qpnr)$.

(2) By Steinberg's theorem, we can assume that $c\in G(\Qpnr)$. 
Let $b:=c^{-1}\delta \sigma(c)$. Then, as was discussed before, one can readily check that $b\in \Cent_G(\gamma_0)(\Qpnr)$ (cf. \cite[p.183]{LR87}), and that $b\in \Cent_G(\gamma_0)(L)$ is a basic element (\cite[Lemma 5.15]{LR87}), namely $\nu_b$ maps into the center of $\Cent_G(\gamma_0)_L$.
In fact, what that proof establishes is that the Newton homomorphisms $\nu_{\gamma_0}$, $\nu_b$ attached to $\gamma_0$, $b$ are related by:
\[\nu_{\gamma_0}=n\nu_b,\]
from which follows the fact that $b$ is basic in $\Cent_G(\gamma_0)(L)$, cf. proof of Lemma \ref{lem:equality_of_two_Newton_maps}, (2).
Let us choose a maximal $\Q$-torus $T_0$ of $G$ that contains $\gamma_0$ and elliptic over $\R$.
Then, thanks to our assumption that $X(b,\{\mu_X\})_{\mbfK_p}\neq \emptyset$ and Lemma \ref{lem:LR-Lemma5.11}, we can find $\mu\in X_{\ast}(T_0)\cap \{\mu_X\}$ such that if $K$ is a finite Galois extension of $\Qp$ splitting $T_0$, $\Nm_{K/\Qp}\mu= [K:\Qp]\nu_b$.
Also $\gamma_0^t\cdot p^{-t\nu_{\gamma_0}}\in T_0(\Qp)_0$ for any $t\in\N$ with $t\nu_{\gamma_0}\in X_{\ast}(T_0)$ by definition of $\nu_{\gamma_0}$. Hence, $[K:\Qp]\nu_{\gamma_0}=n\Nm_{K/\Qp}\mu\in X_{\ast}(T_0)$, so
\[[K:\Qp]\lambda_{T_0}(\gamma_0)=[K:\Qp]\nu_{\gamma_0}=n\underline{\Nm_{K/\Qp}\mu}=n[K:\Qp]\underline{\mu},\]
and $\lambda_{H}(\gamma_0)=n\underline{\mu}$.
Therefore, the condition $(\ast(\gamma_0))$ of Subsec. \ref{subsubsec:pre-Kottwitz_triple} holds (with level $n$), and the claim follows from (1). This completes the proof. 
\end{proof}


\subsection{Criterion for a Kottwitz triple to come from an admissible pair}

In this subsection, it is assumed that $G^{\der}$ is simply connected, since we will work with Kottwitz triples.
With the works done so far, the proofs of the statements in this section are even closer to the original proofs in the hyperspecial case, so from now on we will be contented with a sketch of the arguments.
We continue to work in the same set-up from the previous subsection.

\begin{thm} \label{thm:LR-Satz5.25} Keep the assumptions of Theorem \ref{thm:LR-Satz5.21}.
Let $(\gamma_0;(\gamma_l)_{l\neq p},\delta)$ be a Kottwitz triple with trivial Kottwitz invariant and such that $\gamma_0$ (equiv. $(\gamma_l)_l$) lies in a compact open subgroup of $G(\A_f^p)$ and $X(\delta,\{\mu_X\})_{\mbfK_p}\neq \emptyset$.
Then, there exists an admissible pair $(\phi,\epsilon)$ that gives rise to the triple $(\gamma_0;(\gamma_l)_{l\neq p},\delta)$. In this case, there are exactly $i(\gamma_0)$ non-equivalent pairs corresponding to $(\gamma_l)$ and $\delta$. Here, $i(\gamma_0)$ is the cardinality of the kernel of the natural map
\[ H^1(\Q,\Cent_G(\gamma_0))\rightarrow H^1(\Q,G^{\ab})\times \prod_v H^1(\Q_v,\Cent_G(\gamma_0)).\]
\end{thm}

\begin{rem} \label{rem:LR-Satz5.25}
This is essentially Satz 5.25 of \cite{LR87} (in the hyperspecial level case), except for the following differences: there, the triple $(\gamma_0;(\gamma_l)_{l\neq p},\delta)$ is still required to satisfy their condition $(\ast(\epsilon))$ while here we do not demand it, nor our condition $(\ast(\gamma_0))$ (recall that our definition of a Kottwitz triple does not include any of these conditions). Instead in this theorem we added conditions at each finite place $l\neq p$ and replaced the original $(\ast(\epsilon))$ (a condition at $p$) by $X(\delta,\{\mu_X\})_{\mbfK_p}\neq \emptyset$. This is not only harmless but also more natural, because according to the conjecture of Langlands-Rapoport, these conditions are a necessary condition for a triple to come from an admissible pair contributing to $\sS_{\mbfK}(\F_{q^m})$ for some $m$  (cf. Remark \ref{rem:admissible_pair}), and are also a sufficient condition according to this theorem.
\end{rem}

\begin{proof} 
By Theorem \ref{thm:LR-Satz5.21}, (2), there exists an admissible pair $(\phi_1,\epsilon)$ with $\epsilon$ stably conjugate to $\gamma_0$, which we may further assume to be nested in a maximal $\Q$-torus, elliptic over $\R$, i.e. $\phi_1=\psi_{T,\mu_h}$ for a special Shimura sub-datum $(T,h)$ and $\epsilon\in T(\Q)$. Then, we have the important fact that the restriction of $\phi$ to the kernel of $\fP$ is determined by $\epsilon$ alone, and that its image $\mathrm{Im}(\phi_1^{\Delta})$ lies in the center of $\Cent_G(\epsilon)$. 
With our assumption on $\gamma_0\in G(\A_f^p)$, this follows from Prop. \ref{prop:phi(delta)=gamma_0_up_to_center}, (1), since $\phi_{\ab}:\fP\rightarrow \fG_{G^{\ab}}$ is uniquely determined (as $G^{\der}$ is simply connected).
The next step in the proof is to modify $\phi_1$ using a suitable cocycle, to get a new admissible morphism $\phi$ so that $(\phi,\epsilon)$ becomes an admissible pair which gives rise to the Kottwitz triple $(\gamma_0;(\gamma_l)_{l\neq p},\delta)$. 
This modification of an admissible morphism $\phi_1$ by a cocycle $a$ in $Z^1(\Gal(\Qb/\Q),G(\epsilon)')$ enjoys the crucial property that it does not change the restriction of $\phi_1$ to the kernels (i.e. $\phi_1^{\Delta}:P\rightarrow G$). In other words, $\phi=a\phi_1$ coincides with $\phi_1$ on $P$, hence if $\phi_1$ is well-located in a maximal torus $T$, elliptic over $\R$, $\phi:\fP\rightarrow\fG_G$ is again well-located in $T$, i.e. $\phi(\delta_n)=\phi_1(\delta_n)\in T(\Q)$.
Such modification is carried out in the proof of Satz 5.25 in \cite{LR87} which works for any kind of level subgroup, and we provide its summary in the next lemma. 
\end{proof}

\begin{lem}
Let $(\phi_1,\epsilon)$ be a well-located admissible pair. Let $G(\epsilon)'$ be the twist of $G(\epsilon):=\Cent_G(\epsilon)$ defined by $\phi_1$ (i.e. $G(\epsilon)'$ is the inner form of $G(\epsilon)$ defined by the cocycle $[\rho\mapsto \phi_1(q_{\rho})]\in Z^1(\Gal(\Qb/\Q),\Cent_G(\epsilon))$). 

(1) For any cocycle $a=\{a_{\sigma}\}$ of $\Gal(\Qb/\Q)$ with values in $G(\epsilon)'$, the map $\phi:\fP\rightarrow\fG_G$ defined by
$\phi|_P=\phi_1|_P$ and $\phi(q_{\rho})=a_{\rho}\phi_1(q_{\rho})$ is a morphism of Galois gerbs over $\Q$; we write $\phi=a\phi_1$.

(2) $(\phi,\epsilon)$ is admissible if and only if the images of $a$ in $H^1(\Q,G^{\ab})$ and in $H^1(\R,G(\epsilon)')$ are trivial.

(3) If $(\epsilon,(\gamma_l),\delta)$ is a Kottwitz triple with trivial Kottwitz invariant, there exists a cocycle 
$a=\{a_{\sigma}\}$ such that the pair $(\phi=a\phi_1,\epsilon)$ is an admissible pair corresponding to $(\epsilon,(\gamma_l),\delta)$.
\end{lem}

\begin{proof} The first statement is a straightforward verification. The second one is Lemma 5.26 of \cite{LR87}. Finally, the third claim is established in the proof of Satz 5.25 of loc. cit., more precisely, in the part from the third paragraph in p.195 to the end of the proof. 
\end{proof}


\subsection{Effectivity criteria for stable conjugacy classes and Kottwitz triples}

Let $T$ be a $\Q$-torus and $\mu\in X_{\ast}(T)$ a cocharacter defined over a CM field $E\subset\Q$.  Let $\mfp$ the prime ideal of $\cO_E$ corresponding to the place $v_2$ of $E$ (the place determined by the chosen embedding $\Qb\hra\Qpb$). We fix an element $\pi$ of $E$ such that $(\pi)=\mfp^r$ holds for some $r\in\N$: it suffices that the class number of $\cO_E$ divides $r$. Note that $\pi$ is a unit in $E_w$ for any finite place $w\neq v_2$ of $E$, and $|\Nm_{E_{\mfp}/\Qp}\pi|_p=p^{-m}$ with $m=[E_{\mfp}:\Qp]r$. Thus, for fixed $r$, any two such numbers $\pi$, $\pi'$ differ by a root of unity. Then, with a choice of $(E,\pi)$, we define a triple $(\gamma_0;(\gamma_l)_{l\neq p},\delta)\in T(\Q)\times T(\A_f)\times T(\Qpnr)$ as follows:
Consider the reciprocity map (cf. \cite[2.2]{Deligne77}):
\[ N(\mu):E^{\times}=\Res_{E/\Q}\Gm\stackrel{\Res_{E/\Q}\mu}{\longrightarrow}\Res_{E/\Q}(T_E)\stackrel{\Nm_{E/\Q}}{\longrightarrow}T.\]
First, we set 
\[\gamma_0:=N(\mu)(\pi)=\Nm_{E/\Q}(\mu(\pi)),\]
where $\Nm_{E/\Q}$ is the norm map $T(E)\rightarrow T(\Q)$, and $\gamma_l:=\gamma_0$ for $l\neq p$. For $\delta$, let 
$E_{\mfp,0}$ be the maximal unramified subextension of $E_{\mfp}$ and $n:=[E_{\mfp,0}:\Qp]$. Then, $\Nm_{E/\Q}(\mu(\pi))=\prod_{w|p} N_{E_{w}/\Qp}(\mu(\pi))$ via $(\Nm_{E/\Qp})_{\Qp}=\prod_{w|p} N_{E_{w}/\Qp}:T(E\otimes\Qp)=\prod_w T(E_w)\rightarrow T(\Qp)$, and for each place $w\neq v_2$ of $E$ above $p$, $w_{T_L}(N_{E_{w}/\Qp}(\mu(\pi)))=0$ (this was seen in the proof of Thm. \ref{thm:LR-Satz5.21}). Thus, we can find
$c\in T(L)$ such that $c N_{E_{\mfp}/\Qp}(\mu(\pi)) \sigma^n(c)^{-1}= \gamma_0$. If we set 
\[\delta:=c\cdot N_{E_{\mfp}/E_{\mfp,0}}(\mu(\pi))\cdot \sigma(c^{-1}),\]
one easily checks that $\delta\in T(L_n)$ and $\Nm_{L_n/\Qp}\delta=\gamma_0$. 
In the case of need for distinction, we write $\gamma_0(T,\mu)$ (or $\gamma_0(T,h)$ when $\mu=\mu_h$ for a special Shimura subdatum $(T,h)$) for the $\gamma_0$ defined this way (with a choice of $(E,\pi)$ understood).
It follows from the definition of $\psi_{T,\mu}$ (\cite[p.144]{LR87}) that 
\[\gamma_0(T,\mu)=\psi_{T,\mu}(\delta_m),\] 
when we use the same $E$ and $\pi$ in both definitions.
Here, $\delta_m\in P(E,m)$ is as in Lemma \ref{lem:Reimann97-B2.3}.
Note that a different choice $(E\subset E',\pi'|\pi)$ of $(E,\pi)$ changes $\gamma_0(T,\mu)$ to its power.

Next, let $(G,X)$ be a Shimura datum meeting our running assumptions in this section (as imposed in Thm. \ref{thm:LR-Satz5.21}).

\begin{lem}
When $(T,\mu)$ comes from a special Shimura subdatum $(T,h)$ of $(G,X)$ (i.e. $\mu=\mu_h$), the triple $(\gamma_0;(\gamma_l)_{l\neq p},\delta)$ just constructed is a Kottwitz triple, i.e, satisfies the conditions (iv), $(\ast(\delta))$ of Subsec. \ref{subsubsec:pre-Kottwitz_triple}. It also satisfies the condition $(\ast(\gamma_0))$ and its Kottwitz invariant is equal to $1$.
\end{lem}
\begin{proof}
By Remark \ref{rem:Kottwitz_triples}, the condition (iv) will be implied by the vanishing of the Kottwitz invariant $\alpha(\gamma_0;(\gamma_l)_l,\delta)$, for which we refer to the proof of Thm. 5.3.1 of \cite{Lee14} (more specifically, Step 4). The condition $(\ast(\delta))$ is Lemma \ref{lem:unramified_conj_of_special_morphism}. Finally, $(\ast(\gamma_0))$ follows from Thm. \ref{thm:LR-Satz5.21}, (3), in view of Lemma \ref{lem:unramified_conj_of_special_morphism} and Lemma \ref{lem:LR-Lemma5.2}.
\end{proof}

Suppose that $(G,X)$ is of Hodge type, endowed with an embedding $\rho:G\hra\mathrm{GSp}(V,\langle\ ,\ \rangle)$. Fix a level subgroup $\mbfK=\mbfK_p\times\mbfK^p\subset G(\A_f)$ and also an integral model $\sS_{\mbfK_p}$ over $\cO_{E(G,X)_{\wp}}$ with the usual extension property (e.g. \cite[Thm. 0.1]{KP15}). Then, for a special Shimura sub-datum $(T,h)$ and $g_F\in G(\A_f)$, the corresponding CM point $x=[(h,g_f\cdot \mbfK)] \in \Sh_{\mbfK}(G,X)(\C)$ is defined over $\Qb(\subset\C)$, thus via the chosen embedding $\Qb\hra\Qpb$, reduces to a point $x_0$ defined over some finite field $\F_{q^m}$. Using $\rho$, we may find a CM-algebra $L$ of degree $2\dim V$, containing $\rho(T)(\subset \mathrm{End}(V))$, and an embedding $L\hra\mathrm{End}(A_x)_{\Q}$. There is a natural CM type $\Phi$ of $L$ determined by $h$. Then we have the inclusions $E(L,\Phi)\subset E\subset E(x)$, where $E(L,\Phi)$ is the reflex field of $(L,\Phi)$ and $E(x)\subset\Qb$ is a field over which $x$ is defined (which can be uniquely determined by $g_f\mbfK$ from the canonical model theory, cf. \cite[2.2]{Deligne77}), and $j\circ N(\mu)\circ \Nm_{E(x)/E}:E(x)^{\times}\rightarrow L^{\times}$ is the usual reflex norm map attached to the CM abelian variety $A_x/E(x)$, where $j:T\hra L^{\times}$ is the inclusion (cf. \cite{CCO14}). Then, the theory of Complex Multiplication (\cite[A.2.5.7, A.2.5.8]{CCO14}) tells us that for large $t\in\N$, $\gamma_0(T,\mu_h)^t\in L^{\times}$ acts as the relative Frobenius on $A_{x_0}/\F_{p^{mt}}$.

When $\mbfK_p$ is hyperspecial, Kisin \cite[Cor. 2.3.1]{Kisin13} also attaches a Kottwitz triple to any point $x\in \sS_{\mbfK}(\F_{q^m})$ over a finite field, using his work on the Tate theorem and the CM-lifting theorem (\cite[Cor. 2.2.5]{Kisin13}). This is equivalent to the one attached to $(T,\mu_h)$ when $x$ is the reduction of a CM point defined by the datum $(T,h)$ (\cite[Prop. 4.6.3]{Kisin13}).

By definition, a \textit{stable conjugacy class} in $G(\Q)$ is an equivalence class in $G(\Q)$ with respect to the stable conjugation relation. For a stable conjugacy class $\mathcal{C}\subset G(\Q)$ and $t\in\N$, $\mathcal{C}^t$ is the stable conjugacy class containing the set $\{g^t\ |\ g\in \mathcal{C}\}$.

\begin{thm} \label{thm:effectivity_criteria} 
Keep the assumptions of Theorem \ref{thm:LR-Satz5.21}.
Assume further that $(G,X)$ is of Hodge type and the anisotropic kernel of the center $Z(G)$ remains anisotropic over $\R$. Fix a special maximal parahoric subgroup $\mbfK_p\subset G(\Qp)$.
 
(1) Let $\mathcal{C}\subset G(\Q)$ be a stable conjugacy class. Then, some power of $\mathcal{C}$ contains the relative Frobenius of the reduction of a CM point of $\Sh_{\mbfK}(G,X)(\C)$, i.e. a power of $\gamma_0(T,h)$ for some special Shimura sub-datum $(T,h)$, if and only if a power of $\mathcal{C}$ contains some $\gamma_0\in G(\Q)$ which satisfies the condition $(\ast(\gamma_0))$ of Subsec. \ref{subsubsec:pre-Kottwitz_triple} and lies in a compact open subgroup of $G(\A_f^p)$.

(2) Let $(\gamma_0;(\gamma_l)_{l\neq p},\delta)$ be a Kottwitz triple with trivial Kottwitz invariant and such that $\gamma_0$ (equiv. $(\gamma_l)_l$) lies in a compact open subgroup of $G(\A_f^p)$ and
\[Y_p:=\{x\in G(L)/\mbfKt_p\ |\ \sigma^n x=x,\ \mathrm{inv}_{\mbfKt_p}(x,\delta\sigma x)\in \Adm_{\mbfKt_p}(\{\mu_X\})\}\neq \emptyset.\]
Then, there exists a special Shimura datum $(T,h)$ such that the reduction of the CM point $[h,1\cdot\mbfK]\in \Sh_{\mbfK}(G,X)(\Qb)$ has the associated Kottwitz triple equal to $(\gamma_0;(\gamma_l)_{l\neq p},\delta)$, up to powers.
\end{thm}

Note that these facts are consequences of the Langlands-Rapoport conjecture and the conjecture (\ref{eqn:LRconj-ver2}) derived from it. 

\begin{proof} 
(1) Suppose $\gamma_0(T,h)^s\in\mathcal{C}^t$ for some $s,t\in\N$. Obviously $(\phi,\epsilon):=(\psi_{T,\mu_h},\psi_{T,\mu_h}(\delta_m))$ (for any $m\in\N$) is an admissible pair such that $\epsilon$ lies in a compact open subgroup of $G(\A_f^p)$, so $\psi_{T,\mu_h}(\delta_m)$ satisfies the stated condition by Thm. \ref{thm:LR-Satz5.21}, (1). But, $\gamma_0(T,h)=\psi_{T,\mu_h}(\delta_m)$. Now, the claim follows since the condition $(\ast(\gamma_0))$ holds for $\gamma_0$ if and only if it holds for $\gamma_0^s$ for any $s\in\N$.

Conversely, if $\gamma_0$ is an element in $G(\Q)\cap \mathcal{C}^r$ satisfying the condition $(\ast(\gamma_0))$ and lying in a compact open subgroup of $G(\A_f^p)$, by Thm. \ref{thm:LR-Satz5.21}, 
there exists an admissible pair $(\phi,\epsilon)$ with $\epsilon$ stably conjugate to $\gamma_0$, which we may further assume to be nested in some special Shimura sub-datum $(T,h)$, i.e. $\phi=\psi_{T,\mu_h}$ and $\epsilon\in T(\Q)$. Then, by Prop. \ref{prop:phi(delta)=gamma_0_up_to_center}, (4), $\epsilon^k=\psi_{T,\mu_h}(\delta_k)^n$ for some $k,n\in\N$. So, $\gamma_0^{kt}\in \mathcal{C}^{rkt}$ is stably conjugate to $\psi_{T,\mu_h}(\delta_k)^{nt}$, which is, for large $t$, the relative Frobenius of the reduction of a CM point, as we have seen above.

(2) The same proof as in (1) works, using Thm. \ref{thm:LR-Satz5.21}, (2). 
\end{proof}
 
This generalizes \cite[Lemma 18.1]{Kottwitz92} which concerns the Shimura varieties of PEL-type whose Lie types are either $A$ or $C$ and with hyperspecial level $\mbfK_p$. It says that a triple $(\gamma_0;(\gamma_l)_{l\neq p},\delta)$ as in Def. \ref{defn:Kottwitz_triple} arises from a point over the finite field $\F_{p^n}$ if and only if it is a Kottwitz triple of level $n$ (in the same sense as ours) with $\gamma_0$ being a Weil $p^n$-number, has the trivial Kottwitz invariant, and the affine Deligne-Lusztig variety $X(\{\mu_X\},\delta)_{\mbfK_p}$ is non-empty.


\begin{appendix}

\section{The quasi-motivic Galois gerb} \label{sec:quasi-motivic_Galois_gerb}

As explained by Reimann, the original definition of the quasi-motivic Galois gerb contains errors. He gave a correct construction in his book \cite[$\S$B2]{Reimann97}. We provide a review of it in loc. cit. 

For every place $l$ of $\Q$, we fix an embedding $\Qb\subset\Qlb$. 
For every Galois extension $L\subset\Qb$ of $\Q$, we first define a $\Q$-torus $Q^L$ which will serve as the kernel of the quasi-motivic Galois gerb to be defined later. 
We set
\[L(l):=L\cap \Ql,\quad L_l:=L\cdot \Ql\]
(intersection and composite in $\Qlb$, respectively). Let 
$Q^L$ be the $\Q$-torus defined by the exact sequence of $\Q$-tori 
\[1\lra \Gm\stackrel{\Delta}{\lra} L(\infty)^{\times}\times L(p)^{\times}\lra Q^L\lra 1\]
where $\Delta$ is induced by the diagonal embedding of $\Q$ into $L(\infty)\times L(p)$, and let
\[\psi^L:Q^L\lra L^{\times}\]
be the homomorphism of $\Q$-tori induced by the natural inclusion of $L(p)^{\times}$ into $L^{\times}$ and the \textit{inverse} of the natural inclusion of $L(\infty)^{\times}$ into $L^{\times}$. On the other hand, the embedding $L(l)\hra \Ql$ determines a place of $L(l)$ and accordingly a factor of $L(l)\otimes_{\Q}\Ql$ isomorphic to $\Ql$, namely a cocharacter of $L(l)^{\times}$ which is defined over $\Ql$. For $l\in\{\infty,p\}$, let $\nu(l)^L$ denote the resulting cocharacter of $Q^L$.

The $\Q$-torus $Q^L$ equipped with two cocharacters $\nu(\infty)^L$, $\nu(p)^L$ is also characterized by certain universal property.

\begin{lem} \cite[B2.2]{Reimann97} \label{lem:Reimann97-B2.2}
For every Galois extension $L\subset\Qb$ of $\Q$, $(Q^L,\nu(\infty)^L,\nu(p)^L)$ is an initial object in the category of all triples $(T,\nu_{\infty},\nu_p)$ where $T$ is a $\Q$-torus which splits over $L$, and, for $l\in\{\infty,p\}$, $\nu_l$ is a cocharacter of $T$ which is defined over $\Ql$, and such that 
\[\sum_{l\in\{\infty,p\}}[L_l:\Ql]^{-1}\Tr_{L/\Q}(\nu_l)=0.\]
The $\Q$-torus $Q^L$ satisfies the Hasse principle for $H^1$, i.e. the natural map
\[H^1(\Q,Q^L)\rightarrow \bigoplus_{all\ l}H^1(\Ql,Q^L)\]
is injective.
\end{lem}

Note that $\nu_l$, being a cocharacter of $T$, is also defined over $L$, which gives meaning to the expression $\Tr_{L/\Q}(\nu_l)$. 

For every successive Galois extensions $L\subset L'\subset\Qb$ of $\Q$, there exists a surjective homomorphism $\omega_{L'|L}:Q^{L'}\rightarrow Q^L$ which commutes with the maps $\psi^{L'}:Q^{L'}\rightarrow L'^{\times}$, $\psi^{L}:Q^{L}\rightarrow L^{\times}$, and $\Nm_{L'|L}:Q^{L'}\rightarrow Q^L$, 
Hence, we get two projective systems $\{Q^L\}_{L}$, $\{L^{\times}\}_{L}$
with respective transition maps $\omega_{L'|L}$, $\Nm_{L'|L}$, and the maps $\{\psi^L\}_{L}$ induce a homomorphism between the corresponding pro-tori
\begin{equation} \label{eq:morphism_from_Q_to_R}
\psi^{\Delta}:Q:=\varprojlim_{L}Q^L\rightarrow R:=\varprojlim_{L}L^{\times}.
\end{equation}
The character group of $R$ can be naturally identified with the set of all continuous maps $f:\Gal(\overline{\Q}/\Q)\rightarrow\Z$, the Galois action being given by $\rho(f)(\tau)=f(\rho^{-1}\tau)$ for all $\rho,\tau\in\Gal(\overline{\Q}/\Q)$. Further easy properties of the system $\omega_{L'|L}$ show that there exist a cocharacter 
\[\nu(\infty):\mathbb{G}_{\mathrm{m},\R}\rightarrow Q_{\R}\]
whose composite with $Q\rightarrow Q^L$, for totally imaginary $L$, is $\nu(\infty)^L$, and 
a cocharacter
\[\nu(p):\mathbb{D}\rightarrow Q_{\Qp}.\]

\begin{defn} \cite[B2.7]{Reimann97} \label{defn:Reimann97-B2.7}
A \textit{quasi-motivic groupoid} is the data of a Galois gerb $\fQ$ over $\Q$ together with morphisms $\zeta_l=\zeta_l^Q:\fG_v\rightarrow \fQ(l)$ for all places $l$ of $Q$ such that
\begin{itemize}
\item[(1)] $(\fQ^{\Delta},\zeta_{\infty}^{\Delta},\zeta_{p}^{\Delta})=(Q,\nu(\infty),\nu(p))$, the identification $\fQ^{\Delta}=Q(\Qb)$ being compatible with the action of $\Gal(\overline{\Q}/\Q)$.
\item[(2)] the morphisms $\zeta_l$, for all $l\neq \infty,p$, are induced by a section of $\fQ$ over $\Spec(\overline{\A_f^p}\otimes_{\A_f^p}\overline{\A_f^p})$;
\end{itemize}
where $\overline{\A_f^p}$ denotes the image of the map $\Qb\otimes_{\Q}\A_f^p\rightarrow \prod_{l\neq\infty,p}\Qlb$.
\end{defn}

For more concrete equivalent description of the condition (2), see (3.1.7) of \cite{Kisin13}.

\begin{thm} \label{thm:Reimann97-B.2.8}
There exists a quasi-motivic Galois gerb $(\fQ,(\zeta_l=\zeta_l^Q))$. If $(\fQ',(\zeta_l'))$ is another quasi-motivic Galois gerb, there is an isomorphism $\alpha:\fQ\rightarrow\fQ'$ such that, for all places $l$ of $\Q$, $\zeta_l'$ is algebraically equivalent to $\alpha\circ\zeta_l$. Moreover, there is a morphism $\psi:\fQ\rightarrow\fG_R$, whose restriction to the kernel is the homomorphism $\psi^{\Delta}$ (\ref{eq:morphism_from_Q_to_R}); $\psi$ is unique up to algebraic equivalence.
\end{thm}

\section{Existence of elliptic tori inside special maximal parahoric group schemes}

\begin{prop} \label{prop:existence_of_elliptic_tori_in_special_parahorics}
Let $k$ be a local field with residue characteristic not equal to $2$ and $L$ the completion of the maximal unramified extension $k^{\nr}$ in an algebraic closure $\overline{k}$ of $k$ with respective rings of integers $\cO_k$ and $\cO_L$. Let $G$ be a connected reductive group over $k$ of classical Lie type. Assume that $G$ is quasi-split and splits over a tamely ramified extension of $k$. 

Then, for any special parahoric subgroup $K$ of $G(k)$, there exists a maximal elliptic $k$-torus $T$ of $G$ such that $T_{k^{\nr}}$ contains (equiv. is the centralizer of) a maximal ($k^{\nr}$-)split $k^{\nr}$-torus $S_1$ of $G_{k^{\nr}}$ and that the unique parahoric subgroup of $T(L)$ is contained in $\tilde{K}$, the parahoric subgroup of $G(L)$ corresponding to $K$.
\end{prop}

\begin{rem} \label{rem:properties_of_certain_elliptic_tori_in_special_parahorics}
(1) The unique parahoric subgroup of $T(L)$ (resp. of $T(k)$) is $T(L)_1:=\Ker(w_{T_L})$ (resp. $T(k)_1:=\Ker(w_{T_L})\cap T(k)$).

(2) With $S_1$ and $T$ as in the statement, the second property of $T$ can be translated into a statement about the Bruhat-Tits building:
any special point $\mbfv\in \mcB(G,k)$ giving $K$ (i.e. $K=\mathrm{Stab}_{G(F)}(\mbfv)\cap\Ker(w_{G_L})$) lies in an apartment $\mcA(S_1,L)$ attached to $S_1$.
Indeed, $S_1(L)_1\subset \tilde{K}$ if and only if $\mbfv\in\mcA(S_1,L)$. This is because $\mcA(S_1,L)$ is the full fixed point set of the pararhoric subgroup $S_1(L)_1$, which in turn is due to \cite[3.6.1]{Tits79} since every relative root $a$ of the root datum $(G_L,(S_1)_L)$, being a non-trivial character (of a split torus), satisfies that $a(S_1(L)_1)\nsubseteq 1+\pi_F$ (i.e. $\bar{a}(\overline{S_1(L)_1})\neq 1$), as the residue field of $\cO_L$ is infinite. But, $S_1(L)_1\subset\tilde{K}$ if and only if $T(L)_1\subset\tilde{K}$, since $T(L)=\Cent_{G_L}(S_1)(L)$ acts on $\mcA(S_1,L)$ via $w_{T_L}\otimes\R:T(L)\rightarrow X_{\ast}(T)_{\Gal(\overline{L}/L)}\otimes\R=X_{\ast}(S_1)_{\R}$ (\cite[1.2.(1)]{Tits79}), so if $S_1(L)_1\subset\tilde{K}$, $T(L)_1$ fixes $\mbfv(\in \mcA(S_1,L))$.
\end{rem}

\begin{proof} 
We reduce a general case to the case of absolutely almost-simple groups. First, under the isomorphism
\[\mcB(G,k)\simeq \mcB(G^{\der},k)\times X_{\ast}(A(G))_{\R},\] 
where $A(G)$ is the maximal split $F$-torus in the center $Z(G)$, a special point in $\mcA(G,k)$ corresponds to $(v,x)$ for a special vertex $v$ in $\mcB(G^{\der})$ and a point $x\in X_{\ast}(A(G))_{\R}$. This implies that we may assume that $G$ is semisimple, and by the same reasoning further that $G$ is simply-connected.
Then as $G(=G^{\der})$ is a product of almost-simple groups of the same kind (by which we mean quasi-split, tamely ramified, classical groups), we may also assume that $G$ is almost-simple. Hence, $G=\Res_{F/k}(H)$ for an absolutely-(almost)simple, quasi-split, classical, tamely ramified, semi-simple group $H$ over a finite extension $F$ of $k$. 
The building $\mcB(G,k)$ (resp. $\mcB(G,L)$) is canonically isomorphic to $\mcB(H,F)$ (resp. $\prod_{\sigma\in\Hom_k(F_0,L)}\mcB(H,F\otimes_{F_0,\sigma}L))$, where $F_0$ is the maximal unamified subextension of $F$); this is true for any finite separable extension $F$ of $k$ (\cite[2.1]{Tits79}). 
This shows that the claim for $G$ follows from the claim for $H$. Let $L$ and $\cO_L$ now be the completion of the maximal unramified extension $F^{\nr}$ in an algebraic closure $\overline{F}$ of $F$ and its ring of integers.

From this point, we prove the proposition case by case. We use the complete list of isomorphism classes of quasi-split, tamely ramified, classical groups over local fields as provided in \cite{Gross12}, where Gross gives a complete list of isomorphism classes of (not necessarily quasi-split or classical) tamely ramified groups over local fields and it is fairly easy to determine the quasi-split, classical ones from that list. There are totally ten such isomorphism classes, among which the first seven are unramified ones (including four split ones). 

\begin{itemize}
\item[$A_m$]$(m\geq1)$: $H=\mathrm{SL}_m$ (split group).
\item[${}^2A'_{2m}$]$(m\geq2)$: Let $E$ be the unramified quadratic extension of $F$ and let $W$ be a non-degenerate Hermitian space of odd rank $n=2m+1$ over $E$ (its Witt-index must be $m$). Then $H=\SU(W)$ (non-split unramified group).
\item[${}^2A'_{2m-1}$]$(m\geq2)$: Let $E$ be the unramified quadratic extension of $F$ and let $W$ be a non-degenerate Hermitian space of even rank $n=2m$ over $E$ which contains an isotropic subspace of dimension $m$. Then $H=\SU(W)$ (non-split unramified group).
\item[$B_m$]$(m\geq3)$: Let $W$ be a non-degenerate orthogonal space of odd dimension $2m+1$ over $F$ which contains an isotropic subspace of dimension $n$. Then $H=\Spin(W)$ (split group).
\item[$C_m$]$(m\geq2)$: Let $W$ be a non-degenerate symplectic space of dimension $2m$ over $F$. Then $H=\Sp(W)$ (split group).
\item[$D_m$]$(m\geq4)$: Let $W$ be a non-degenerate orthogonal space of dimension $2m$ over $F$ which contains a (maximal) isotropic subspace of dimension $m$. Then, the center of the Clifford algebra is the split \'etale quadratic extension $E=F\oplus F$ of $F$, and $H=\Spin(W)$ is a split group.
\item[${}^2D_m$]$(m\geq4)$: Let $W$ be a non-degenerate orthogonal space of even dimension $2m$ over $F$ where the center of the Clifford algebra is the unramified quadratic extension $E$ of $F$. Then $H=\Spin(W)$ (non-split unramified group).
\item[$B\operatorname{-}C_m$]$(m\geq3)$: Let $E$ be a tamely ramified quadratic extension of $F$ and let $W$ be a non-degenerate Hermitian space of even rank $n=2m$ over $E$ which contains an isotropic subspace of dimension $m$. Then $H = SU(W)$ (ramified group)
\item[$C\operatorname{-}BC_m$]$(m\geq2)$: Let $E$ be a tamely ramified quadratic extension of $F$ and let $W$ be a non-degenerate Hermitian space of odd rank $n=2m+1$ over $E$. Then $H = SU(W)$ (ramified group).
\item[$C\operatorname{-}B_m$]$(m\geq2)$: Let $W$ be a non-degenerate orthogonal space of even dimension $2m$ over $F$ where the center of the Clifford algebra is a tamely ramified quadratic extension $E$ of $F$. Then $H=\Spin(W)$ (ramified group).
\end{itemize}

Here, each heading is the name for the corresponding isomorphism class that is used in \cite{Tits79}, tables 4.2 and 4.3. 

Now, when the group is unramified and given special vertex is hyperspecial, the claim is known (\cite{Lee14}, Appendix 1.0.4). For convenience, we split these cases into three kinds. The first kind consists of the split groups, thus, all special vertices are automatically hyperspecial: these are $A_m$, $B_m$, $C_m$, $D_m$. The second case is when there are two special vertices and both of them are hyperspecial: they consist of ${}^2A'_{2m-1}$, ${}^2D_m$. The remaining ones constitute the last case, i.e. there is some non-hyperspecial, special vertex.
So, we only need to take care of the last case: 
\[{}^2A'_{2m},\quad B\operatorname{-}C_m,\quad C\operatorname{-}BC_m,\ \text{ and }\ C\operatorname{-}B_m.\] 
Note that except for the last one, all these are (special) unitary groups.

(1) First, we treat the special unitary groups of \emph{even} absolute rank (i.e. $H_{\overline{F}}\simeq \SU(2m+1)$ for an algebraic closure $\overline{F}$ of $F$). We will reduce the proof in this case to the special unitary groups of \emph{odd} absolute rank.
For a moment, we let $E$ be an arbitrary quadratic extension of $F$ with respective rings of integers $\cO_E$, $\cO_F$ (we assume that the residue characteristic of $\cO_F$ is not $2$). We choose a uniformizer $\pi$ of $\cO_E$ such that $\pi+\overline{\pi}=0$ for the non-trivial automorphism $\overline{\cdot}$ of $E/F$.
Let $(W,\phi:W\times W\rightarrow E)$ be a non-degenerate Hermitian space of dimension $n=2m+1\ (m\geq1)$. As is well-known, $\psi$ has maximal Witt-index $m$, so there exists a Witt basis $\{e_{-m},\cdots,e_{m}\}$, i.e. such that 
\[\phi(e_i,e_j)=\delta_{i,-j},\quad\text{for }-m\leq i,j\leq m.\] 
For $i=0,\cdots,m$, we define an $\cO_E$-lattice in $W$:
\[\Lambda_i:=\mathrm{span}_{\cO_E}\{\pi^{-1}e_{-m},\cdots,\pi^{-1}e_{-i-1},e_{-i},\cdots,e_m\}.\]
(here, $\Lambda_{m}=\mathrm{span}_{\cO_E}\{e_{-m},\cdots,\cdots,e_m\}$.)

Set $H=\SU(W,\phi)$ (algebraic group over $F$). Its (minimal) splitting field is $E$. For a non-empty subset $I$ of $\{0,\cdots,m\}$, we consider the subgroup of $H(F)$ 
\[P_I:=\{g\in\SU(W,\phi)\ |\ g\Lambda_i\subset\Lambda_i,\ \forall i\in I\}.\]
For each $i=0,\cdots,m$, $P_{\{ i\}}$ (stabilizer of a single lattice $\Lambda_i$) is also the stabilizer of a point $v_i$ of the apartment $\mcA$ of a maximal $F$-split torus $S$ of $H$, which in turn can be matched with the $(i+1)$-th vertex of the local Dynkin diagram \cite[1.15, 3.11]{Tits79}. 

From this correspondence and the information found in loc. cit. (1.15 (9), 4.3 in the unramified case, and 1.15 (10), 4.2 in the ramified case), we deduce the following facts.

\textit{Every $P_I$ is a parahoric subgroup of $\SU(W,\phi)$ and any parahoric subgroup of $\SU(W,\phi)$ is conjugate to $P_I$ for a unique subset $I$. If $E$ is unramified (i.e. of type ${}^2A'_{2m}$), there are two special vertices, one hyperspecial and one non-hyperspecial. The group $P_{\{0\}}$ (resp. $P_{\{m\}}$) is the non-hyperspecial, special (resp. the hyperspecial) parahoric subgroup. If $E$ is ramified (i.e. of  type $C\operatorname{-}BC_m$), there are two special vertices, both non-hyperspecial, which correspond to $I=\{0\}$ and $I=\{m\}$. The corresponding parahoric subgroup $\tilde{P}_{\{i\}}$ of $H(L)$ is 
\[ \tilde{P}_{\{i\}}=\{g\in\SU(W,\phi)(L)\ |\ g\Lambda_i\otimes_{\cO_F}\cO_L\subset\Lambda_i\otimes_{\cO_F}\cO_L\}.\]}
The statement on $\tilde{P}_{\{i\}}$ follows from the fact that the stabilizer $\cO_L$-group scheme $\mcG_i$ of the vertex $v_i$ (defined by Bruhat-Tits \cite[3.4.1]{Tits79}) equals the $\cO_L$-structure on $H_L$ induced by the lattice (\cite[3.11]{Tits79}), and the characterization of parahoric groups given by Haines-Rapoport  \cite{PappasRapoport08} ($\SU(W,\phi)$ being a simply-connnected semi-simple group, the Kottwitz homomorphism $w_{\SU(W,\phi)_L}$ is trivial).

(a) The case $\mathbf{{}^2A'_{2m}}$: 
Suppose that $E$ is unramified. Then, we only need to consider the non-hyperspecial, special parahoric subgroup $P_{\{0\}}$. The Hermitian space $(W,\phi)$ splits as the direct sum of two Hermitian subspaces, that is, totally isotropic subspace and anisotriopic subspace: 
\[(W,\psi)=(W_{\mathbf{iso}},\phi_{\mathbf{iso}})\oplus (W_{\mathbf{an}},\phi_{\mathbf{an}}),\]
where
\[ W_{\mathbf{iso}}:=\langle e_l\ |\ l\neq 0\rangle,\quad W_{\mathbf{an}}:=E\cdot e_{0},\]
and $\phi_{\mathbf{iso}}=\phi|_{W_{\mathbf{iso}}}$ and $\phi_{\mathbf{an}}=\phi|_{W_{\mathbf{an}}}$.
There is the corresponding lattice decomposition 
\[\Lambda_{\{0\}}=\Lambda_{\{0\}}'\oplus \Lambda_{\{0\}}'',\] 
where $\Lambda_{\{0\}}':=\mathrm{span}_{\cO_E}\{\pi^{-1}e_{-m},\cdots,\pi^{-1}e_{-1},e_{1},\cdots,e_{m}\}$ and $\Lambda_{\{0\}}'':=\cO_E\cdot e_{0}$. 
Using this decomposition, we reduce the construction of the torus looked for into construction of similar tori for the groups $\SU(W_{\bullet},\psi_{\bullet})$ ($\bullet=\mathbf{iso}$, $\mathbf{an}$).
Let us write for short $\SU_{\bullet}$ and $\U_{\bullet}$ for $\SU(W_{\bullet},\psi_{\bullet})$ and $\U(W_{\bullet},\psi_{\bullet})$ respectively ($\bullet=\mathbf{iso}$, $\mathbf{an}$).
Suppose that $T_{\mathbf{iso}}$ is an $F$-torus of $\SU_{\mathbf{iso}}$ with the property in question and let $Z_{\mathbf{iso}}$ be the (connected) center of $\U_{\mathbf{iso}}$ (so that the subgroup $T_{\mathbf{iso}}\cdot Z_{\mathbf{iso}}$ generated by the two groups is a maximal torus of $U_{\mathbf{iso}}$, among others). This center is isomorphic to the anisotpropic $F$-torus $\underline{E}^{\times}$, whose set of $R$-points, for an $F$-algebra $R$, is 
\[\underline{E}^{\times}_c(R):=\Ker(\Nm_{E/F}:(E\otimes R)^{\times}\rightarrow (F\otimes R)^{\times}).\] 
This is also identified in a natural way with the group $\U_{\mathbf{an}}$. We claim that 
\[T:=S((T_{\mathbf{iso}} \cdot Z_{\mathbf{iso}}) \times \U_{\mathbf{an}})=(T_{\mathbf{iso}}\times\{1\})\cdot \underline{E}^{\times}_c,\] 
where $S(-)$ means the intersection of the group inside the parenthesis (subgroup of $\U(W,\phi)$) with $\SU(W,\phi)$ and $\underline{E}^{\times}_c$ is identified with $S(Z_{\mathbf{iso}} \times \U_{\mathbf{an}})$ via $x\mapsto (x,x^{-2m})$, is a maximal torus of $H=\SU(W,\psi)$ with the same required properties. First, clearly this is anisotropic, and $T_{E}$ is a split maximal torus of $H_E$.
Next, we verify that $T(L)_1$ maps into the parahoric subgroup $\tilde{P}_{\{0\}}$ of $H(L)$. By the description of $\tilde{P}_{\{0\}}$ above, we have to show that $T(L)_1$ leaves stable $\Lambda_{\{0\}}\otimes_{\cO_F}\cO_L$. But, $T(L)_1=T_{\mathbf{iso}}(L)_1\cdot \underline{E}^{\times}_c(L)_1$, and $\underline{E}^{\times}_c(L)_1$ acts on $\Lambda_{\{0\}}'\otimes\cO_L\oplus\Lambda_{\{0\}}''\otimes\cO_L$ through the map $x\mapsto (x,x^{-2m})$ above. So, clearly it suffices to checks that $\underline{E}^{\times}_c(L)_1$ leaves stable each rank-$1$ lattice $\cO_L\otimes_{\cO_F}(\cO_E\cdot e_i)$. But in the case $E$ is unramified over $F$, we have that $\underline{E}^{\times}_c(L)_1=\{(x,x^{-1})\in \cO_L^{\times}\times \cO_L^{\times}\}$ under the isomorphism $(E\otimes L)^{\times}=L^{\times}\times L^{\times}$, thus leaves stable $\cO_L\otimes_{\cO_F}(\cO_E\cdot e_i)=(\cO_L\oplus \cO_L)\cdot e_i$. 
Next, the fact that $T_{\mathbf{iso}}(L)_1$ leaves stable $\Lambda_{\{0\}}'\otimes\cO_L$ will be one of the defining properties of the torus $T_{\mathbf{iso}}$. Indeed, 
$\SU(W_{\mathbf{iso}},\psi_{\mathbf{iso}})$ is a group of type ${}^2A_{2m-1}'$ in the above list, and the stabilizer $P_{\{0\}}'$ of the lattice $\Lambda_{\{0\}}'$ is a hyperspecial subgroup of $\SU(W_{\mathbf{iso}},\psi_{\mathbf{iso}})$ (cf. \cite[4.3]{Tits79}). So, we already know that there exists an elliptic maximal $F$-torus $T_{\mathbf{iso}}$ of $\SU_{\mathbf{iso}}$ such that $T_{\mathbf{iso}}(L)_1$ is contained in the parahoric subgroup $\tilde{P}_{\{0\}}'$ of $\SU_{\mathbf{iso}}(L)$ corresponding to $P_{\{0\}}'$ which is the stabilizer in $\SU_{\mathbf{iso}}(L)$ of the lattice $\Lambda_{\{0\}}'\otimes\cO_L$. This finishes the proof in the case $E$ is unramified over $F$.

(b) The case $\mathbf{C\operatorname{-}BC_m}$: 
When $E$ is ramified, there are two cases: $I=\{0\}$ and $I=\{m\}$. 
Then, the same strategy just used  (i.e. for unramified unitary groups of odd absolute rank) works again, reducing proof to the ramified unitary groups of odd absolute rank $2m-1$ (the type $\mathbf{B\operatorname{-}C_m}$), which will be discussed next. We just note that in this case with the same notations as above,  the (minimal) splitting fields of $H$, $\SU_{\mathbf{iso}}$, $Z_{\mathbf{iso}}$, and $\U_{\mathbf{an}}$ are all $E$, and that $(\underline{E}^{\times}_c)_L$ is anisotropic, so $\underline{E}^{\times}_c(L)$ is its own parahoric subgroup and, being a subgroup of $(\cO_E\otimes_{\cO_F}\cO_L)^{\times}$, leaves stable $\cO_L\otimes_{\cO_F}(\cO_E\cdot e_i)$. 

(2) Let $E$, $W$, and $n$ be as in the previous description (1)-(b), except that the parity of $n$ is even (i fact, it can be arbitrary for a moment). Let $\cO_E$ and $\cO_F$ be the integer rings of $E$ and $F$, respectively. We fix uniformizers $\pi_F$, $\pi=\pi_E$ of $\cO_F$ and $\cO_E$ such that $\pi^2=\pi_F$ (so again $\pi+\overline{\pi}=0$). Let $\phi:W\times W\rightarrow E$ be a non-degenerate Hermitian form and put $H=\SU(W,\phi)$. Again we assume (forced by the quasi-split condition, in the even dimensional case) that $\psi$ has maximal Witt-index, namely when one writes $n=2m$ (or $n=2m+1$), it is $m$. As $E$ is ramified over $F$, the rank of $H$ is the same as that of $H_L$.
 
Following \cite{PappasRapoport08}, $\S$4, we use a different indexing in the coming discussion.
Choose a Witt basis $\{e_{1},\cdots,e_{n}\}$ such that $\phi(e_i,e_j)=\delta_{i,n+1-j}$ for $1\leq i,j\leq n$. 

Suppose that $n=2m$. For $i\in\{1,\cdots,m-2\}\cup\{m\}$, we define an $\cO_E$-lattice $\Lambda_i$:
\[\Lambda_i:=\mathrm{span}_{\cO_E}\{\pi^{-1}e_{1},\cdots,\pi^{-1}e_{i},e_{i+1},\cdots,e_n\},\]
In the place of $i=m-1$, we introduce a new lattice $\Lambda_{m'}$ defined by:
\[\Lambda_{m'}:=\mathrm{span}_{\cO_E}\{\pi^{-1}e_{1},\cdots,\pi^{-1}e_{m-1},e_{m},\pi^{-1}e_{m+1},e_{m+2},\cdots,e_{n}\}.\]
Here, $m'$ is regarded as a symbol like other numbers.

Set $H:=\SU(W,\phi)$. For a non-empty subset $I$ of $\{1,\cdots,m-2,m',m\}$, the associated stabilizer subgroup $P_I$ has the same definition as in the previous case.

When $n=2m$, the group $\SU(W,\phi)$ has the local Dynkin diagram $B\operatorname{-}C_m$ for $m\geq3$, and $C\operatorname{-}B_2$ for $m=2$ (for $m=1$, $\SU(W,\psi)\simeq\mathrm{SL}_2$). 
Then, we have a similar statement (\cite{PappasRapoport08}, $\S$4), namely that

\textit{the subgroup $P_I$ is a parahoric subgroup of $\SU(W,\phi)$ and any parahoric subgroup of $\SU(W,\phi)$ is conjugate to $P_I$ for a unique subset $I$, and the special maximal parahoric subgroups are $P_{\{m\}}$, $P_{\{m'\}}$. The same description is true for the parahoric subgroup $\tilde{P}_I$ of $\SU(W,\phi)(L)$ associated with $P_I$.}

(c) The case $\mathbf{B\operatorname{-}C_m}$: We have $n=2m$. First, let us consider the case $I=\{m\}$. 
The Hermitian space $W$ is the direct sum of $m$ hyperbolic subspaces 
\[\mathbb{H}_i:=E\langle e_i,e_{n+1-i}\rangle\subset W\ (i=1,\cdots,m).\] 
Then, we claim that when we identify $\SU(\mathbb{H}_i)$ with $\SU(\mathbb{H}_i)\times\mathrm{1}_{\oplus_{j\neq i}\mathbb{H}_j}\subset\SU(W,\phi)$, 
\[\SU(\mathbb{H}_i)(F)\cap P_{\{m\}}=\{g\in \SU(\mathbb{H}_i)(F)\ |\ g(\mathbb{H}_i\cap \Lambda_{\{m\}})=\mathbb{H}_i\cap \Lambda_{\{m\}}\}\] 
is a special maximal parahoric subgroup of $\SU(\mathbb{H}_i)(F)\simeq\mathrm{SL}_2(F)$ (recall that there are two $\mathrm{SL}_2(F)$-conjugacy classes of special parahoric subgroups of $\mathrm{SL}_2(F)$, which are however conjugate under $\mathrm{GL}_2(F)$). This can be proved, e.g. using an explicit isomorphism between $\SU(\mathbb{H}_i)\simeq\mathrm{SL}_{2,F}$, one such being
\[ g=\left(\begin{array}{cc}
a& b\\ c& d
\end{array}\right) \mapsto \left(\begin{array}{cc}
a& \pi^{-1}b\\ \pi c& d
\end{array}\right)\]
(check that when $g(e_i)=ae_i+ce_{n+1-i}, g(e_{n+1-i})=be_i+de_{n+1-i}$, $\phi(gv,gw)=\phi(v,w)$ implies that $a,d\in F^{\times}$, $b+\overline{b}=c+\overline{c}=0$). But, there is another (rather indirect) way of seeing this. Let $S_i$ be a maximal ($F$-)split $F$-subtorus of $\SU(\mathbb{H}_i)$ (so that $S:=\prod_i S_I$ is a maximal ($F$-)split $F$-torus of $\SU(W,\phi)$, and also of $\U(W,\phi)$); $S_i$ is contained in a unique maximal torus $T_i(=\Cent_{\U(\mathbb{H}_i)}(S_i))\simeq E^{\times}$ of $\U(\mathbb{H}_i)$. The subgroup 
\[M_i:=\U(\mathbb{H}_i)\times\prod_{j\neq i}T_j\] 
of $\U(W,\phi)$, being the centralizer of $\{1\}\times\prod_{j\neq i}S_j$, is an $F$-Levi subgroup of $\U(W,\phi)$. 
For a subgroup $M$ of $\U(W,\phi)$, let $\mathrm{S}M$ denote the intersection $M\cap \SU(W,\phi)$. Then, as $\{1\}\times\prod_{j\neq i}S_j\subset\SU(W)$, $\mathrm{S}M_i$ is an $F$-Levi subgroup of $\SU(W)$.
Hence, by Lemma \ref{lem:specaial_parahoric_in_Levi}, (3), $\mathrm{S}M_i(F)\cap P_{\{m\}}$ is a special maximal parahoric subgroup of $M_i(F)$. But, obviously one has that (as $P_{\{m\}}\subset\SU(W)(F)$)
\[\mathrm{S}M_i(F)\cap P_{\{m\}}=(\U(\mathbb{H}_i)(F)\times\prod_{j\neq i}T_j(F))\cap P_{\{m\}}
=(\U(\mathbb{H}_i)(F)\cap P_{\{m\}})\times \prod_{j\neq i}(T_j(F)\cap P_{\{m\}}). \]
So, each $\SU(\mathbb{H}_i)(F)\cap P_{\{m\}}=\U(\mathbb{H}_i)(F)\cap P_{\{m\}}$ is a special maximal parahoric subgroup of $\SU(\mathbb{H}_i)(F)\simeq\mathrm{SL}_2(F)$. 

But, the two special vertices in the local Dynkin diagram of $\mathrm{SL}_{2,F}$ are hyperspecial, hence we know that there exists an anisotropic maximal torus $S_i'$ of $\SU(\mathbb{H}_i)$, splitting over $F^{\nr}$, such that the unique parahoric subgroup $S_1'(L)_1$ of $S_i'(L)$ is contained in $\SU(\mathbb{H}_i)(L)\cap \tilde{P}_{\{m\}}$, i.e. leaves stable the $\cO_E\otimes_{\cO_F}\cO_L$-lattice $(\mathbb{H}_i\cap \Lambda_{\{m\}})\otimes_{\cO_F}\cO_L=\mathrm{span}_{\cO_E\otimes_{\cO_F}\cO_L}\{\pi^{-i}e_{i},e_{n+1-i}\}$ of $\mathbb{H}_i\otimes L$. 
Therefore, for the center $Z_i\simeq \underline{E}^{\times}_c$ of $\U(\mathbb{H}_i)$,
\[T_i':=S_i'\cdot Z_i\] 
is an anisotropic maximal torus of $\U(\mathbb{H}_i)$, whose $L$-rank equals $1$ and whose group of $L$-points also leaves stable the $\cO_E\otimes_{\cO_F}\cO_L$-lattice $(\mathbb{H}_i\cap \Lambda_{\{m\}})_{\cO_L}$, as $Z_i$ remains anisotropic over $L$ and $Z_i(L)=(\cO_E\otimes_{\cO_F}\cO_L)_1$. Finally, the torus $T:=(\prod_iT_i')\cap\SU(W,\phi)$ is an anisotropic maximal torus of $\SU(W,\phi)$ with the same required property (its $L$-rank is $m$, equal to the $L$-rank of $H_L$).

The case $I=\{m'\}$ can be treated in a completely analogous way, once we switch the basis vectors $e_{m}$ and $e_{m+1}$; although such permutation does not lie in $\SU(W,\psi)$, obviously it is allowed when applying the previous argument. 

(d) We are left with the case $\mathbf{C\operatorname{-}B_{m+1}\ (m\geq1)}$. This is also similar to the above cases.
Let $E$ be a (ramified) quadratic extension of $F$ and $W=F^{\oplus m}\oplus E\oplus F^{\oplus m}$, viewed as a vector space over $F$. We consider the quadratic form on $W$ expressed in terms of a basis $\{e_{-m},\cdots,e_{m}\}\cup\{e_0\}$ by
\[q(\sum_{1\leq|i|\leq m} x_i e_i+x_0e_0)=\sum_{i=1}^{m}x_{-i} x_i+\Nm_{E/F}x_0,\quad (x_i\in F,x_0\in E).\] 
For $i=0,\cdots,m$, we define a lattice $\Lambda_i$ as before:
\[\Lambda_i:=\mathrm{span}_{\cO_F}\{\pi^{-1}e_{-m},\cdots,\pi^{-1}e_{-i-1},e_{-i},\cdots,\check{e}_0,\cdots,e_m\}\oplus\cO_Ee_0,\]
where $\check{e}_0$ means as usual that it is omitted from the list
(so, $\Lambda_{0}=\pi^{-1}\mathrm{span}_{\cO_F}\{e_{-m},\cdots,e_{-1}\}\oplus\cO_Ee_0\oplus\mathrm{span}_{\cO_F}\{e_{1},\cdots,e_{m}\}$).

It is obvious that the claim at hand holds for the special orthogonal group $\mathrm{SO}(W,q)$ if and only if it does so for the universal covering of $\mathrm{SO}(W,q)$, i.e. the spin group $\mathrm{Spin}(W,q)$. We will show that for any special maximal parahoric subgroup $K$ of $\mathrm{SO}(W,q)$, there exists an anisotropic maximal $F$-torus $T$ such that $T_L$ contains a maximal ($L$-)split $L$-torus of $\mathrm{SO}(W,q)$ and $T(L)_1$ is contained in $\tilde{K}$, the special parahoric subgroup of $\mathrm{SO}(W,q)(L)$ corresponding to $K$. Note that the $L$-rank of $H$ is $m$ as $E$ is ramified over $F$.

For a non-empty subset $I$ of $\{0,\cdots,m\}$, let $P_I$ denote the stabilizer subgroup:
\[P_I:=\{g\in \mathrm{SO}(W,q)(F)\ |\ g\Lambda_i\subset\Lambda_i,\ \forall i\in I\}.\]
We know (deduced from \cite{Tits79}, 1.16, 4.2, cf. 3.12) that

\textit{the subgroup $P_I$ is a parahoric subgroup of $\SU(W,\phi)$ and that any parahoric subgroup of $\SU(W,\phi)$ is conjugate to $P_I$ for a unique subset $I$, and that the special maximal parahoric subgroups are $P_{\{0\}}$ and $P_{\{m\}}$. The corresponding parahoric subgroup $\tilde{P}_{\{i\}}$ of $H(L)$ is 
\[ \tilde{P}_{\{i\}}=\{g\in\SO(W,\phi)(L)\ |\ g\Lambda_i\otimes_{\cO_F}\cO_L\subset\Lambda_i\otimes_{\cO_F}\cO_L\}.\]} 

The idea used above for the ramified special unitary group of odd relative rank (i.e. of type $\mathbf{B\operatorname{-}C_{m}}$) works here, too. Namely, the quadratic space $W$ decomposes into the direct sum of maximally isotropic subspaces $W_{\mathbf{iso}}=F\langle e_i\ |\ 1\leq|i|\leq m\rangle$ and the anisotropic subspace $(E\cdot e_0,\Nm_{E/F})$. Then, $\SO(W_{\mathbf{iso}})$ is a split group, so for each $j=0,m$, the parahoric subgroup $P_{\{j\}}\cap \SO(W_{\mathbf{iso}})(F)$ is a a hyperspecial subgroup. Hence, there exists an anisotropic maximal torus $T'_j$ of $\SO(W_{\mathbf{iso}})$ of $L$-rank $m$, such that $T'_j(L)_1$ is contained in $\SO(W_{\mathbf{iso}})(L)\cap \tilde{P}_{\{j\}}$; the latter means that $T'_j(L)_1$ leaves stable the $\cO_L$-lattice
\[(W_{\mathbf{iso}}\cap \Lambda_{j})\otimes\cO_L\]
of $W_{\mathbf{iso}}\otimes L$. Now, it is easy to see that the anisotropic torus $T_j:=T'_j$ is a maximal torus of $\SO(W,q)$ (which also has the $L$-rank $m$) with the same properties for $\Lambda_{j}$. 

This completes the proof of the proposition. \end{proof}

\section{Some facts on the Bruhat order of Coxeter groups}
Our presentation here follows that of \cite{Stembridge05}, Ch. 1. 
Let $(W,S)$ be a Coxeter system; so there exists a real vector space $V$ equipped with an inner product $\langle\ ,\ \rangle$ such that $W$ is a group of isometries with a generating set $S$ and a centrally-symmetric, $W$-invariant subset $\Phi\subset V-\{0\}$ (\textit{the root system}) such that the isometries in $W$ are the reflections $s_{\beta}=\lambda\rightarrow \langle\beta,\lambda\rangle\beta^{\vee}$ for varying $\beta\in\Phi$, where $\beta^{\vee}:=2\beta/\langle\beta,\beta\rangle$ (the coroot corresponding to $\beta$). We fix a set $\Delta$ of simple roots such that $S=\{s_{\alpha}\}_{\alpha\in\Delta}$ and denote by $\Phi^+$ the set of positive roots (the set of roots in $\Phi$ which are linear combinations of $\Delta$ with positive coefficients). 

For $w\in W$, let $l(w)$ denote the length of $w$, i.e. the minimal $l\in\N$ for which $w=s_1\cdots s_l$ for some $s_i\in S$. The \textit{Bruhat order} $<_B$ of $W$ (with respect to $S$) is the transitive closure of the relations:
\[w<_B s_{\beta}w\]
for all $w\in W$ and $\beta\in \Phi^+$ satisfying either of the following equivalent conditions:
\[l(w)<l(s_{\beta}w)\quad \Leftrightarrow\quad w^{-1}(\beta)\in\Phi^+.\]

For each $J\subset S$, we let $W_J$ denote the (parabolic) subgroup of $W$ generated by $J$, and $\Phi_J\subset \Phi$ the corresponding root subsystem. The subset
\[W^J:=\{ w\in W\ |\ l(w)<l(ws)\text{ for all }s\in J\}.\]
of $W$ consists of the coset representatives of $W/W_J$ of minimal length. The similarly defined subset ${}^JW$ of $W$ has the same properties. We put
\[{}^IW^J:={}^IW\cap W^J.\]

\begin{lem} \label{lem:Bruhat_order_on_Coxeter_group}
Suppose that $\theta\in V$ is dominant with stabilizer $W_J$.
(1) The evaluation map 
\[(W,<_B)\rightarrow (W\theta,<_B)\ :\ w\mapsto w\theta\] is order-preserving  and restricts to a poset isomorphism $(W^J,<_B)\rightarrow (W\theta,<_B)$.

(2) For any $I\subset S$, the evaluation map restricts to a bijection between ${}^IW^J$ and 
\[(W\theta)_I:=\{\mu\in W\theta\ |\ \langle\alpha,\mu\rangle>0\text{ for all }\alpha\in\Phi^+_I\}.\]
\end{lem}

The first statement is Prop. 1.1 of \cite{Stembridge05}, and the second statement follows from Prop. 1.5 and Prop. 1.8.

\end{appendix}

Center for Geometry and Physics, Institute for Basic Science (IBS), 

Pohang 790-784, Republic of 
Korea 

\textit{Email:} dulee@ibs.re.kr

\end{document}